\documentclass[aos]{imsart}

\RequirePackage{amsthm,amsmath,amsfonts,amssymb}
\RequirePackage[numbers]{natbib}
\RequirePackage[colorlinks,citecolor=blue,urlcolor=blue]{hyperref}
\RequirePackage{graphicx}

\usepackage{caption}
\usepackage{subcaption}
\usepackage{mathrsfs}

\startlocaldefs
\theoremstyle{plain}

\newtheorem{theorem}{Theorem}[section]
\newtheorem{lemma}[theorem]{Lemma}
\theoremstyle{remark}

\newtheorem{remark}{Remark}
\newtheorem{defn}[theorem]{Definition}

\def \bx{\mathbf{x}}

\def \bX{\mathbf{X}}

\def \be{\begin{align*}}
\def \ee{\end{align*}}

\def \E{\mathbb{E}}
\def \P{\mathbb{P}}

\newcommand{\indep}{\perp \!\!\! \!\perp}

\endlocaldefs

\begin{document}

\begin{frontmatter}
\title{A new central limit theorem for the augmented IPW estimator: variance inflation, cross-fit covariance and beyond}\runtitle{A new CLT for the augmented IPW estimator}

\begin{aug}
\author[A]{\fnms{Kuanhao}~\snm{Jiang}\ead[label=e1]{kuanhaojiang@g.harvard.edu}},
\author[B]{\fnms{Rajarshi}~\snm{Mukherjee}\ead[label=e2]{ram521@mail.harvard.edu}}\thanks{Author names sorted in alphabetical order. Corresponding author email: pragya@fas.harvard.edu},
\author[A]{\fnms{Subhabrata}~\snm{Sen}\ead[label=e3]{subhabratasen@fas.harvard.edu}}\footnotemark[1]
\and
\author[A]{\fnms{Pragya}~\snm{Sur}\ead[label=e4]{pragya@fas.harvard.edu}}\footnotemark[1]

\address[A]{Department of Statistics, Harvard University\printead[presep={,\ }]{e1,e3,e4}}

\address[B]{Department of Biostatistics, Harvard T.H. Chan School of Public Health\printead[presep={,\ }]{e2}}
\end{aug}

\begin{abstract}
		Estimation of the average treatment effect (ATE) is a central problem in causal inference. In recent times, 
	inference for the ATE in presence of high-dimensional covariates has been extensively studied. 
		Among diverse approaches that have been proposed,
 augmented inverse propensity weighting (AIPW) with cross-fitting has emerged a popular choice in practice.
		In this work, we study this cross-fit AIPW estimator under well-specified outcome regression and propensity score models in a high-dimensional regime where the number of features and samples are both large and comparable.
Under assumptions on the covariate distribution, we establish a new central limit theorem for the suitably scaled cross-fit AIPW that applies without \emph{any} sparsity assumptions on the underlying high-dimensional parameters. Our CLT uncovers two crucial phenomena among others: (i) the AIPW  exhibits a substantial variance inflation that can be precisely quantified in terms of the signal-to-noise ratio and other problem parameters, (ii) the asymptotic covariance between the pre-cross-fit estimators is non-negligible even on the $\sqrt{n}$ scale. 
These findings are strikingly different from their classical counterparts. 
On the technical front, our work utilizes a novel interplay between three distinct tools—approximate message passing theory, the theory of deterministic equivalents, and the leave-one-out approach. We believe our proof techniques should be useful for analyzing other two-stage estimators in this high-dimensional regime. Finally, we complement our theoretical results with simulations that demonstrate both the finite sample efficacy of our CLT and its robustness to our assumptions.
		\end{abstract}

\end{frontmatter}

\section{Introduction}

Causal inference based on observational studies poses a problem of intrinsic interest in the natural and social sciences. 
Unmeasured confounders pose a major challenge in this regard.
We recall that an unmeasured confounder is a variable that affects both the exposure and the outcome  of interest, and thus invalidates causal effect estimates based on observational data. Fortunately, with rapid advances in modern data collection technologies, the statistician often hopes to overcome this barrier by collecting data on a large number of potential confounders. While this provides an attractive strategy to mitigate unmeasured confounding, it  necessitates causal effect estimation in presence of high-dimensional confounders. This has inspired rapid methodological advances over the past decade  at the intersection of statistics, machine learning, computer science, epidemiology etc.~on high-dimensional causal inference. This manuscript contributes to this crucial area of research.


Before proceeding further, we describe our problem of interest formally. 
We observe $n$ i.i.d.~observations of a tuple $(y,A,x)$  from some joint distribution $\mathbb{P}$, where
 $y \in \mathbb{R}$ denotes an outcome of interest, $A \in \{0,1\}$ denotes the binary exposure or treatment and $x \in \mathbb{R}^p$ denotes the measured confounders. We seek to estimate the Average Treatment Effect (ATE), a canonical estimand in this context. The ATE is defined as $\E(y(1)-y(0))$, where $y(a)$ denotes the potential outcome corresponding to $A=a\in \{0,1\}$ \citep{pearl2009causality,hernan2010causal,imbens2015causal}. Throughout the manuscript, we assume that conditions necessary for identification  of the ATE are satisfied, that is, we have (i) no unmeasured confounding,  ($y(1),y(0)\indep A|x$), (ii) consistency, ($y=Ay(1)+(1-A)y(0)$) and (iii) positivity ($\P(A=1|x)>0$ for all $x\in \mathbb{R}^p$).  Under these assumptions, the ATE can be identified from the observed data distribution using  $\E(y(1)-y(0))=\E(\E(Y|A=1,x)-\E(Y|A=0,x))$ \citep{pearl2009causality,hernan2010causal,imbens2015causal}. 
 
Varied approaches exist for ATE estimation \citep{athey2019machine} and 
two  \textit{nuisance functions} arise naturally in this context---(i) the \textit{Outcome Regression} (OR) given by $m(A,x)=\E(y|A,x)$ and (ii) the \textit{Propensity Score} (PS) given by $\pi(x)=\E(A|x)=\P(A=1|x)$. 
The regression coefficient vectors underlying the PS and OR models form nuisance parameters for the problem of ATE estimation.
Classical approaches include those based on outcome regression \cite{robins1986new,hernan2010causal,snowden2011implementation, vansteelandt2011invited}, propensity score \citep{horvitz1952generalization,rosenbaum1983central,hahn1998role,hirano2003efficient},  augmented inverse probability weighting (AIPW) \cite{bang2005doubly,scharfstein1999adjusting}, to name a few---these utilize suitable modeling assumptions for at least one of the nuisance functions.  Recent state-of-the-art methods, including Double Machine Learning \cite{chernozhukov2017double}, Covariate Balancing \cite{imai2014covariate, li2018balancing, athey2018approximate,zubizarreta2015stable,fong2018covariate,ning2020robust},  Matching methods \cite{rubin1973use,rosenbaum1984reducing,rubin1996matching,stuart2010matching, abadie2011bias, abadie2016matching}, and calibration-based procedures \cite{sun2021high,tan2020regularized,tan2020model}, also estimate at least one of the nuisance functions on way to estimating the ATE. In particular, these approaches assume some structure, e.g., sparsity, in one (or both) of the nuisance parameters in high dimensions. 


 Among the aforementioned approaches, Double Machine Learning style estimators allow particular flexibility in choosing nuisance functions.
To accommodate this flexibility and 
facilitate theoretical analyses in high dimensions, one additionally employs the idea of \emph{cross-fitting} \cite{chernozhukov2017double,newey2018cross,smucler2019unifying}. In this scheme, the statistician initially splits the observed data into (a few) distinct folds. The nuisances are computed based on one fold and the ATE estimate is obtained from an independent fold. The nuisance estimates from the initial step are plugged into the final ATE estimate as appropriate.  Subsequently, additional estimators are obtained by permuting roles of the folds. The final cross-fitted estimator is obtained by averaging these distinct estimators. 
Under high-dimensional sparse models, this strategy allows one to establish consistency and asymptotic normality of the proposed estimator. Remarkably, this approach yields \emph{efficient} estimators in high dimensions under appropriate sparsity assumptions \cite{chernozhukov2017double}. 
(see Section \ref{sec:background} for a detailed review).

We note that verifying structural assumptions, e.g. sparsity, in the relevant nuisance parameters can be difficult in practice in high dimensions. In addition, the results obtained under such assumptions may suffer from gross inaccuracies when  the assumptions are violated. To illustrate, we present Figure \ref{variance_infaltion_plot}. Here we focus on the augmented inverse probability weighting (AIPW) estimator \cite{scharfstein1999adjusting,bang2005doubly} that exhibits a number of fascinating features, and is arguably the most widely used Double Machine Learning style estimators in practice (see \eqref{eq:AIPW} for a formal definition). In low dimensions, the estimator has the desirable double robustness property, that is, one can estimate the ATE consistently even if one of the OR or PS models is misspecified \cite{scharfstein1999adjusting,bang2005doubly}. Furthermore,  cross-fitted versions of this estimator retain similar robustness properties in ultra-high-dimensions when suitable conditions on signal sparsities are met (see Section \ref{sec:background} for details).

In Figure \ref{variance_infaltion_plot}, we consider a setting with $n=10000$ i.i.d.~samples and $p=700$ covariate dimension. We plot the
standard errors of a centered and scaled cross-fit AIPW  where the following cross-fitting mechanism is employed: 
split the data into three equal folds, estimate PS, OR from  separate folds, plug into the third fold to estimate the AIPW, switch the role of folds, and average the resulting $3!$ estimators. For this figure, the coordinates of the covariates are drawn i.i.d. from suitably normalized mean zero standard  gaussian and the nuisance parameter coordinates (for both the OR and PS models) are drawn i.i.d.~from mean-zero gaussian and subsequently considered fixed. We use a linear-logistic model specification for the OR and PS respectively. Theoretical calculations using existing results \cite{bang2005doubly,smucler2019unifying,chernozhukov2017double} show that the classical estimate of SE in this setting equals 2.06 (the red line). Note these works showed that the theoretical value remains the same in low dimensions, and ultra-high-dimensions under suitable sparsity assumptions. The blue histogram represents the true empirical variability of the estimator. We observe the empirical SE to be much larger, concentrating around 5.5. Thus, as soon as the dimension is moderate compared to the sample size, the cross-fitted AIPW estimator exhibits a massive variance inflation compared to its classical variance, when the underlying signals are not sparse.
This means if we use the red line to provide uncertainty quantification,
it will lead to gross errors in settings where assumptions from \cite{chernozhukov2017double,smucler2019unifying} might be violated. 
Therefore,  there is an urgent need for theory and methods that explain Figure \ref{variance_infaltion_plot}, and allow for causal effect estimation with suitable uncertainty quantification in analogous such settings. In this paper, we fill this critical gap in the literature.

\begin{figure}[!ht]
    \includegraphics[scale = 0.5]{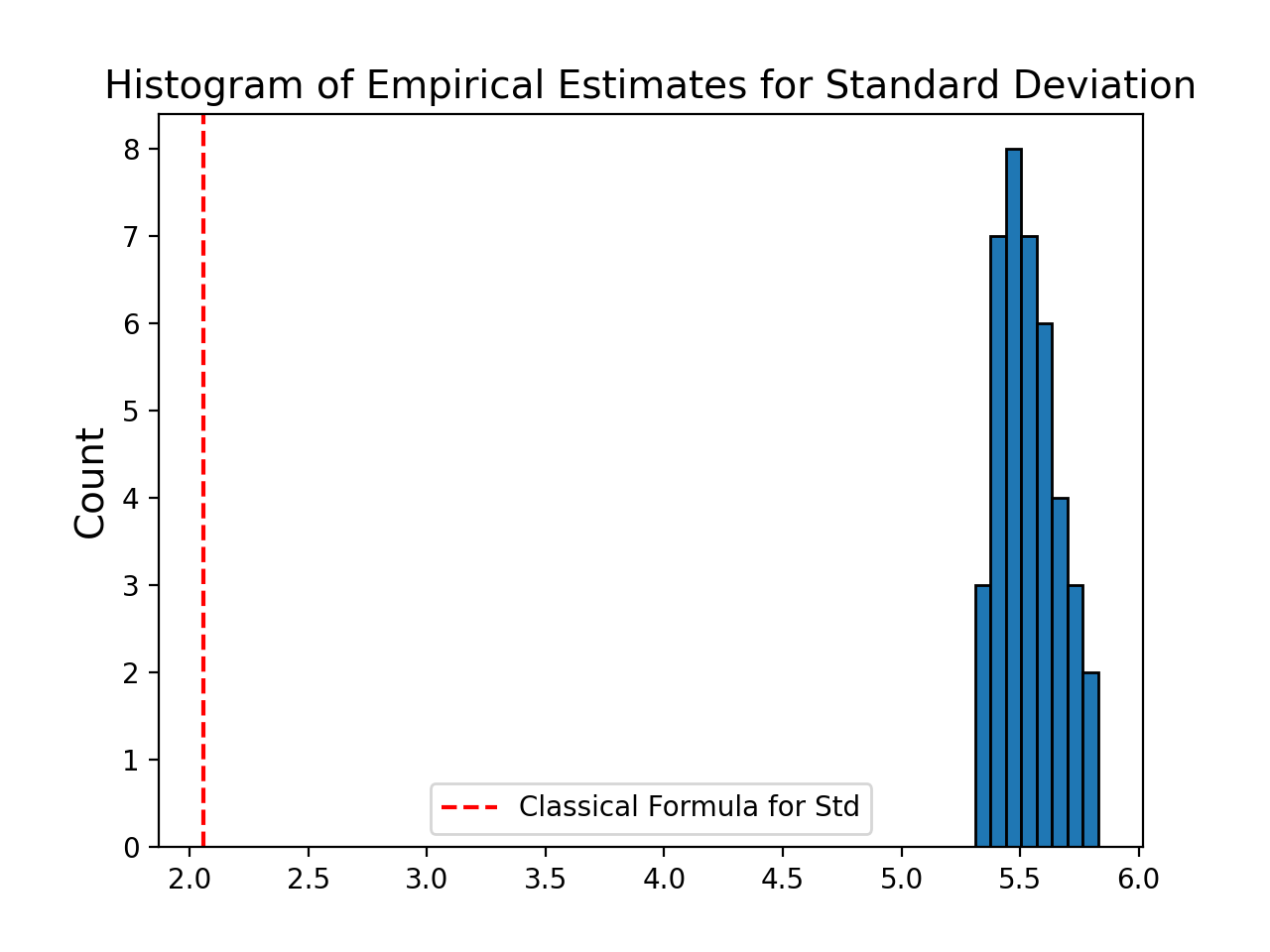}
    \caption{Histogram of standard errors of the 3-split version of cross-fit AIPW. For each SE, we fix the parameters underlying the PS and OR models, generate $1000$ i.i.d.~copies of the data, and compute standard error of the cross-fit AIPW across these replicates. We repeat this experiment many times to obtain the histogram of SEs.
    The classical SE value is shown in red. Clearly, classical theory underestimates the true variability. The parameter values here remain the same as for Figure \ref{normal_qq_plot_original_vs_trimmed}. We thus defer the readers to Section \ref{sec:main_results} for further details on the simulation setting.} 
    \label{variance_infaltion_plot}
\end{figure}
We study ATE estimation in the absence of  sparsity-type assumptions on nuisance parameters, in an arguably high-dimensional regime.
%
%
In our subsequent analysis, we assume a linear model for the outcome regression, and a logistic model for the propensity score. 
We analyze the cross-fitted AIPW estimator in the ``proportional asymptotic regime", where the number of observations $n$ and features $p$ both diverge, with the ratio $p/n$ converging to some constant $\kappa>0$. This regime has attracted considerable recent attention in high-dimensional statistics \cite{johnstone2001distribution,donoho2009message,bayati2011lasso,el2013robust,bean2013optimal,wang2017bridge,el2018impact,donoho2016high,thrampoulidis2018precise,lei2018asymptotics,sur2019likelihood,bellec2019biasing,candes2020phase,sur2019modern,celentano2020estimation,celentano2020lasso,feng2021unifying,salehi2019impact,bu2019algorithmic,hu2019asymptotics,bellec2021asymptotic,bellec2022observable,xu2019consistent,rad2018scalable,javanmard2013state,bellec2020out,thrampoulidis2015regularized,dobriban2018high,patil2021uniform,yadlowsky2022causal}, statistical machine learning and analysis of algorithms \cite{mei2018mean,montanari2019generalization,mei2022generalization,deng2019model,liang2020precise,javanmard2020precise,li2021minimum,chandrasekher2021sharp,mignacco2020role}, econometrics \cite{bekker1994alternative,hahn2002optimal,andrews2005identification,cattaneo2018inference,cattaneo2018alternative,cattaneo2019two,anatolyev2019many} etc, and shares roots with probability theory and statistical physics \cite{zdeborova2016statistical,montanari2022short}. Asymptotic approximations derived under this regime demonstrate commendable performance even under moderate sample sizes (c.f.~\cite{sur2019modern,liang2020precise} as well as the aforementioned references)---this renders the proportional asymptotics regime particularly attractive from a practical perspective.

In our analysis, we trade structural assumptions on model parameters for specific distributional assumptions on observed covariates;
intuitively, our setup complements the sparse models studied in the recent literature.
Under sparsity, one assumes that either the OR or the PS is governed by relatively few strong features. In contrast, our setting allows both to be potentially  influenced by all the covariates, but their individual influences must be of a comparable scale (Section \ref{sec:setup} formalizes this notion). We emphasize that we do not debate the relative merits of these two classes of assumptions. Instead, we seek to provide novel alternate approximations that can be valuable to practitioners in settings where sparsity assumptions from the recent high-dimensional causal inference literature may be violated.


In our framework, consistent estimation of the high-dimensional nuisance parameter vectors (e.g. the regression coefficient vectors for the PS and OR) are impossible in $L_2$ norm. However, low-dimensional functionals such as the ATE can still be estimated at the classical $O_p(1/\sqrt{n})$ rate. The recent work \cite{yadlowsky2022causal} noted this possibility and compared certain high-level properties of common ATE estimators in the absence of sparsity under proportional asymptotics. However, \cite{yadlowsky2022causal} focused only on the possibility of $\sqrt{n}$-consistent estimation, without any uncertainty quantification. In this paper, we derive an explicit CLT for the AIPW estimator, and in sharp contrast to the analysis of \cite{yadlowsky2022causal}, we study the AIPW estimator \emph{with cross-fitting}. 
The analysis of the cross-fitted AIPW estimator is relatively straight-forward under sparsity---several pairs of estimators obtained from permuting the splits turn out to be (asymptotically) independent. The averaging operation therefore reduces the variance by a constant factor to gain back the efficiency lost due to sample splitting \cite{chernozhukov2017double}. In our setting, the behavior shows far more nuances---
these estimators exhibit non-trivial dependencies across the splits that we characterize precisely. To the best of our knowledge, this is the first instance where such non-trivial \emph{cross-covariances} have been identified.
%
%
%
Indeed, we believe this to be one of our main contributions. We hope our analysis will inspire follow-up  analyses of similar two-stage estimators  under this proportional asymptotics regime.

Throughout this paper, we analyze the 3-split version of the cross-fitted  estimator that we used for Figure \ref{variance_infaltion_plot}. 
To keep things tractable, we consider that the OR is fit using maximum likelihood whereas the PS is fit using either maximum likelihood or its ridge regularized version.  
We next describe our main contributions in this paper.


\subsection{Our contributions}
Our main contributions are as follows:
\begin{enumerate}
    \item First, we establish that the cross-fit AIPW estimator converges to a Gaussian limit after centering and $\sqrt{n}$-scaling under the high-dimensional asymptotics $p/n \rightarrow \kappa >0$. Though our assumption on the covariate distribution is stylized, to the best of our knowledge, this is the first CLT for the celebrated AIPW that applies in an arguably high-dimensional regime without any sparsity condition. We hope our analysis will motivate further investigations into properties of other ATE estimators in this regime.
    
    \item We provide a precise characterization of the asymptotic variance of the appropriately centered and scaled cross-fit AIPW in terms of the problem parameters. Empirically, we observe that this limiting variance is higher than the classical variance.
This is indeed expected per prior observations noted in \cite{el2018impact,bean2013optimal,el2013robust,donoho2016high,sur2019modern,sur2019likelihood,cattaneo2018alternative,yadlowsky2022causal}. However, the exact form of the variance allows one to carefully study effects of (i) the signal-to-noise ratios of the underlying parameters, (ii) the degree of high-dimensionality as quantified by $\kappa$, and (iii) the relations among the underlying parameters, on the asymptotic variance. 
    
    \item Next, cross-fitting leads to intriguing phenomena in our setting. In the existing ultra-high-dimensional literature, certain pairs of estimators obtained by permuting the folds are asymptotically independent on the $\sqrt{n}$ scale, and cross-fitting leads to constant gains in the asymptotic variance---thus yielding an efficient estimator \cite{chernozhukov2017double}. In sharp contrast,  the corresponding pairs of estimators are asymptotically correlated in our setting.
    We provide an (asymptotically) exact characterization of these cross-covariances as a function of our problem parameters.  This once again allows one to study the effects of the parameters on the magnitude of these cross-covariances. In fact, we uncover that in many settings these cross-covariances are, in fact, negative.  Complementing earlier works in the literature \cite{chernozhukov2017double,newey2018cross,smucler2019unifying}, our work thus suggests further benefits of cross-fitting in high dimensions, at least in some scenarios.
           
   \item  On the technical front, we develop our proofs based on the following three distinct techniques: approximate message passing theory, the theory of deterministic equivalents, and the leave-one-out approach. As the reader will see, dealing with the cross-fit estimator and in particular, characterizing the cross-covariances requires a novel conjunction of \emph{all} of the aforementioned tools. To the best of our knowledge, we have not encountered high-dimensional problems in the literature, broadly speaking, that demand the full strengths of all of these approaches. We expect that the our proof ideas should be useful for studying several other high-dimensional estimators---particularly those involving two-stage procedures that start with nuisance estimation followed by a plug-in step. 
 
 \item   To study the practical merits of this work, we complement our results with substantial simulations that demonstrate the finite sample efficacy of our theory. This is perhaps another fascinating feature of the proportional asymptotics regime--- the asymptotic theory based on this regime usually demonstrates remarkable performance even in moderate sample sizes. The recent literature in high-dimensional statistics shows ample evidence in this regard across a variety of problems, and we observe this once again for the AIPW CLT characterized in our work. 
We also provide extensive comparisons of our work with classical results and demonstrate that we recover classical results when $p/n$ becomes vanishingly small.
Finally, our experiments demonstrate that optimizing for predictive accuracy during  propensity score estimation via ridge-regularized logistic regression fails to yield optimal downstream variance for the AIPW estimator. This calls for other approaches that would be necessary for choosing the optimal regularization parameter in terms of the AIPW variance.  
\end{enumerate}

 
 \noindent 
{ \textbf{Organization:}} The rest of the paper is organized as follows. We describe our precise setting and the recent literature in Section \ref{sec:setup}. We present our main result together with empirical studies on its finite sample performance  in Section \ref{sec:main_results}. We complement this via further simulations in Section \ref{sec:experiments}, where we investigate the effects of cross-fitting in high dimensions and test the robustness of our assumptions.  Finally, we discuss key ideas involved in the proof in Section \ref{sec:proof_ideas}, and finish with a discussion of directions for future research in Section \ref{sec:discussions}. 

\subsection{{\bf{Notation}}} 

The results in this paper are mostly asymptotic (in $n$) in nature and thus requires some standard asymptotic  notations.  If $a_n$ and $b_n$ are two sequences of real numbers then $a_n \gg b_n$  (and $a_n \ll b_n$) implies that ${a_n}/{b_n} \rightarrow \infty$ (and ${a_n}/{b_n} \rightarrow 0$) as $n \rightarrow \infty$, respectively. Similarly $a_n \gtrsim b_n$ (and $a_n \lesssim b_n$) implies that $\liminf_{n \rightarrow \infty} {{a_n}/{b_n}} = C$ for some $C \in (0,\infty]$ (and $\limsup_{n \rightarrow \infty} {{a_n}/{b_n}} =C$ for some $C \in [0,\infty)$). Alternatively, $a_n = o(b_n)$ will also imply $a_n \ll b_n$ and $a_n=O(b_n)$ will imply that $\limsup_{n \rightarrow \infty} \ a_n / b_n = C$ for some $C \in [0,\infty)$). If $C>0$ then we write $a_n=\Theta(b_n)$.  If  $a_n/b_n\rightarrow 1$, then we  say $a_n \sim b_n$. 

We use $\stackrel{p}{\to}$ and $\stackrel{d}{\to}$ to denote convergence in probability and distribution respectively. We use $o_p(1)$ to denote sequences of random variables which converge to zero in probability. For any sequences of probability measures $\mu_n$ and another probability measure $\mu$, we say that $\mu_n \stackrel{W_2}{\to} \mu$ if the following holds: there exists a sequence of couplings $\Pi_n$ with marginals $\mu_n$ and $\mu$ respectively, so that if $(X_n,X) \sim \Pi_n$, then $\mathbb{E}[(X_n - X)^2] \to 0$ as $n\to \infty$.

\section{Setup}
\label{sec:setup} 
We study the AIPW estimator using the following working model. Throughout we assume that we observe $n$ i.i.d.~samples $\{(y_i, A_i, x_i): 1\leq i \leq n\}$, where the conditional distribution of the treatment given the covariates follows a logistic regression, and the conditional distribution of the outcome given the treatment and the covariates satisfy a linear model. We wish to work in a high-dimensional regime where the covariate dimension is allowed to grow with the sample size. To model this formally, we consider a sequence of problem instances, $\{y_i,A_i,x_i,\epsilon_i^{(0)}, \epsilon_i^{(1)}, 1 \leq i \leq n, \beta(n),\alpha^{(0)}, \alpha^{(1)},\beta^{(0)}(n),\beta^{(1)}(n)\}_{n \geq 1}$ such that 
\begin{align}
A_i  \sim & \text{Ber}(\sigma(x_i^{\top}(n)\beta(n)))\nonumber \\
y_i &  = \alpha^{(A_i)} + x_i(n)^{\top} \beta^{(A_i)}(n) +  \epsilon_i^{(A_i)},
\label{eq:outcome_reg} 
\end{align}
where $\epsilon_i^{(A_i)} \sim \mathcal{N}\left(0,\left(\sigma^{(A_i)}\right)^2\right)$, independent of everything else. 
Above, $x_i(n),1\leq i \leq n, \beta(n),\beta^{(0)}(n),\beta^{(1)}(n)$ all lie in $ \mathbb{R}^{p(n)}$ and we allow $p(n), n \rightarrow \infty$ with $p(n)/n \rightarrow \kappa > 0 $.
We assume that the covariates satisfy $x_i(n) \sim \mathcal{N}(0,I_p/n)$. Naturally, this is a stylized setting, but we will see that the setting uncovers novel high-dimensional phenomena that should motivate further studies into this regime. We also check robustness to our assumption on the covariate distribution in Section \ref{sec:experiments}. In the sequel, we drop the dependence on $n$ whenever it is clear from context.


Under the outcome regression model \eqref{eq:outcome_reg}, the population average treatment effect is given by
\begin{align}
    \Delta =\mathbb{E}\Big[ \mathbb{E}[y_i | A_i=1,x_i] - \mathbb{E}[y_i | A_i=0,x_i]\Big] = \alpha^{(1)} - \alpha^{(0)}. \label{eq:att} 
\end{align}
We seek to study estimation and inference for $\Delta$, \emph{without} invoking sparsity type conditions on the propensity score/outcome regression model parameters. This is of course challenging in high dimensions---thus, to keep the problem meaningful we assume that the signal strengths remain finite in the limit, after appropriate scaling. This reduces to requiring that

\begin{align}
\frac{\|\beta\|^2}{n} \rightarrow \gamma^2,\;\;  \frac{\| \beta^{(0)} \|^2}{p} \rightarrow \sigma^2_{0 \beta}, \;\; \frac{\| \beta^{(1)} \|^2}{p} \rightarrow \sigma^2_{1 \beta}, \;\; \frac{ \left(\beta^{(0)}\right)^\top  \beta^{(1)}}{p} \rightarrow \rho_{01} \sigma_{0 \beta} \sigma_{1 \beta} \;\; \label{eq:norm_limits} 
\end{align} 
for some  $\gamma, \sigma_{0 \beta}, \sigma_{1\beta} \in \mathbb{R}^{+}, \rho_{01} \in [-1,1]$. 
Finally, we require a regularity condition on the structure of the signals given as follows:

\begin{align}
    \frac{1}{p} \sum_{i=1}^{p} \delta_{\beta_i} \stackrel{W_2}{\to} \mu, \,\,\,
    \frac{1}{p} \sum_{i=1}^{p} \delta_{\beta_i^{(0)}} \stackrel{W_2}{\to} \mu_0, \,\,\, 
    \frac{1}{p} \sum_{i=1}^{p} \delta_{\beta_i^{(1)}} \stackrel{W_2}{\to} \mu_1.  \label{eq:limit_dist} 
\end{align}

where $W_2$ denotes Wasserstein-2 convergence. 

This assumption says that the empirical distributions constructed out of the deterministic sequence of vectors $\{\beta(n) \}_{n \geq 1}, \{\beta^{(0)}(n) \}_{n \geq 1}, \{\beta^{(1)}(n) \}_{n \geq 1}$ converges to a weak limit and the corresponding second moments converge. This is a rather common assumption in the proportional asymptotics regime \cite{donoho2009message,bayati2011lasso,javanmard2013state}, and intuitively, it ensures that the entries of each of these vectors do not differ wildly from each other. To keep a specific example in mind, the reader may consider a random effects setting, where each entry of the vector $\beta$ is i.i.d., that is, $\beta_i \stackrel{\text{i.i.d}}{\sim} \mu_{\beta}$ and analogously for $\beta^{(1)}_i, \beta^{(0)}_i $. Note that we can allow  $\mu_{\beta}$ to contain a spike at $0$, meaning that $\beta$ would then be a sparse vector with sparsity linear in $n$ or $p$. Once again, this is true for  $\beta^{(1)}_i, \beta^{(0)}_i $ as well. 

We seek to study the cross-fitted AIPW estimator in the aforementioned regime, focusing on the 3-split version:

\begin{itemize} 
\item[(i)] Split the data into 3 groups $S_1, S_2, S_3$ with sizes $n_1, n_2, n_3$ respectively such that $$n_{1}+n_{2}+n_{3}=n, \quad \lim _{n \rightarrow \infty} \frac{n_{i}}{n}=r_{i} \in(0,1), \quad \lim _{n \rightarrow \infty} \frac{p}{n_{i}} =\kappa_i > 0 \quad \forall i=1,2,3.$$
\item[(ii)] Let $ (a,b,c) $ be a permutation of $(1,2,3)$. 

\begin{enumerate} 
\item 
Use $S_{a}$ to obtain an estimate for $\beta$. Here we consider either the logistic MLE or its ridge regularized counterpart. We denote these using $\hat{\beta}_{S_a}$ or $\hat{\beta}_{S_a}^{(\lambda)}$ respectively. 
Note that $\hat{\beta}_{S_a}^{(\lambda)}$ is obtained by solving the following strongly convex minimization problem

$$ \hat{\beta}_{S_a}^{(\lambda)} =\text{argmin}_{b \in \mathbb{R}^{p}} \sum_{i \in S_a}\left\{\log \left(1 + e^{{x}_{i}^{\top} {b}}\right)-A_{i}\left({x}_{i}^{\top} {b}\right)\right\} + \frac{\lambda}{2} \|b\|^2 .$$ 


   
   \item 
   Use $S_b$ to estimate ${\alpha}^{(0)}, {\alpha}^{(1)}, {\beta}^{(0)}, {\beta}^{(1)}$. In particular, we consider the least squares estimators
 \begin{align}
     (\hat{\alpha}^{(0)}, \hat{\beta}^{(0)}) = \mathrm{argmin}_{(\alpha, \beta)} \sum_{i \in S_b} (1-A_i) (y_i - \alpha - x_i^{\top} \beta)^2 , \nonumber \\
     (\hat{\alpha}^{(1)}, \hat{\beta}^{(1)}) = \mathrm{argmin}_{(\alpha, \beta)} \sum_{i \in S_b} A_i (y_i - \alpha - x_i^{\top} \beta)^2. \label{eq:penalized_est} 
 \end{align}
 
 \item  Use $S_c$ to obtain the final estimator 
 \begin{equation}\label{eq:AIPW}
     \hat{\Delta}_{AIPW}=\hat{\Delta}_{AIPW,{1}}-\hat{\Delta}_{AIPW,0}
 \end{equation}
for the ATE, where $$ \hat{\Delta}_{AIPW,{1}} = \frac{1}{n_c} \sum_{i \in S_c} \left \{ \frac{A_iy_i}{\sigma \left( x_i^{\top} \hat{\beta}_{S_a} \right )} -  \frac{A_i - \sigma\left( x_i^{\top} \hat{\beta}_{S_a} \right)}{\sigma\left( x_i^{\top} \hat{\beta}_{S_a}\right)} \left(\hat{\alpha}^{(1)}_{S_b} +x_i^{\top}\hat{\beta}^{(1)}_{S_b}\right)\right \} ,$$
$$ \hat{\Delta}_{AIPW,{0}} = \frac{1}{n_c} \sum_{i \in S_c}
\left \{  \frac{\left(1-A_i\right)y_i}{1-\sigma\left( x_i^{\top} \hat{\beta}_{S_a}\right)} +  \frac{A_i - \sigma\left( x_i^{\top} \hat{\beta}_{S_a}\right)}{1-\sigma\left( x_i^{\top} \hat{\beta}_{S_a}\right)} \left (\hat{\alpha}_{S_b}^{(0)} +x_i^{\top}\hat{\beta}^{(0)}_{S_b}\right )\right \} .$$
\end{enumerate}

\item[(iii)] For each permutation of $(1,2,3)$, we obtain an estimator $\hat{\Delta}_{AIPW}$. The final estimator of the population treatment effect is obtained by averaging all such estimators. We denote the cross-fitted estimator as $\hat{\Delta}_{cf}$.


\end{itemize} 

Note that we use OLS estimators for $(\alpha^{(0)}, \beta^{(0)}, \alpha^{(1)}, \beta^{(1)})$, so we need to restrict to a regime where these are unique. 
Of course this is not guaranteed, especially when the feature dimension $p$ is reasonably large compared to sample size $n$. In Theorem \ref{thm:ols_existence} below, we derive an explicit characterization of the regime where unique OLS estimators exist with high probability for our aforementioned problem.  The theorem shows that it suffices to have $\kappa_i < 1/2$ for $i=1,2,3$. We will implicitly  restrict ourselves to this region in the rest of the paper. Similarly, when we use the logistic MLE we will restrict to a regime where it exists w.h.p. We will clarify this further in Section \ref{sec:main_results}.

\subsection{\bf Background} 
\label{sec:background} 
In this section, we review strategies for ATE estimation, focusing primarily on the recent literature on ATE estimation with high-dimensional covariates. Along the way, we describe some key  ideas facilitating these recent methodological breakthroughs, and contrast them with our approach. 

\noindent
{\bf ATE estimation in low dimensions:} In the classical setting ($p$-fixed, $n \to \infty$), the ATE can be estimated at the $\sqrt{n}$ rate, and asymptotically normal semi-parametric efficient estimators are well-known.  In this context, AIPW estimators are particularly attractive \cite{scharfstein1999adjusting,bang2005doubly,van2006targeted}. These estimators were originally introduced for mean estimation in missing data problems \cite{robins1994estimation,robins1995analysis,robins1995semiparametric,scharfstein1999adjusting}, before being used for causal effect estimation. The interest in these estimators stems from the well-known "Double Robustness" (DR) property. Formally, AIPW estimators facilitate consistent estimation of the ATE even if one of the PS or OR is misspecified. Additionally, such estimators are also asymptotically gaussian under potential model misspecifications described above \cite{scharfstein1999adjusting,bang2005doubly,van2006targeted}, and thus facilitates robust inference of the ATE. Indeed, this attractive combination of properties has established AIPW estimators as a trusted tool for causal effect estimation in the modern statistician's toolkit. 

\noindent
{\bf ATE estimation in high dimensions:}
We now turn to the extensive recent advances in causal effect estimation in high dimensions (i.e. both $n,p \to \infty$). Ideally, one still wishes to design estimators that enable consistent and asymptotically normal (CAN) inference for the ATE under misspecification of either the PS or OR model. Unfortunately, this presents challenges in high dimensions, and such estimators are usually available under strong structural assumptions on the PS and/or OR models. Over the past decade, the scope of allowed model misspecifications expanded significantly and at the same time, structural constraints imposed on the ``well-specified" part of the model reduced steadily. Such remarkable progress occurred due to a number of creative methodological ideas such as penalized regression followed by de-biasing, sample splitting and cross-fitting etc. In the subsequent discussion, we will touch upon some of these key ideas, and discuss why they fail to apply in our setting.

First, we review the state-of-the-art in terms of allowed model misspecification, and survey the modern causal effect estimators that enjoy these robustness guarantees (along the way, we will indicate the structural assumptions imposed on the well-specified part of the model by these respective strategies). 
In terms of tolerated model misspecifications, two recent notions have gained prominence: (i) \emph{rate double robustness}---here one assumes that both the PS and OR models have approximately sparse expansions, and establishes that CAN estimation is possible as long as the product of the underlying sparsity parameters is sufficiently small, (ii) \emph{model double robustness}---here one allows one of the PS or OR model to be misspecified, as long as the other well-specified nuisance component is sufficiently sparse.

\noindent
\emph{Rate Double Robustness:} In the context of rate double  robustness, \cite{belloni2014inference,farrell2015robust,chernozhukov2017double,chernozhukov2018biased,smucler2019unifying,chernozhukov2021automatic}
employ somewhat parallel strategies where one first estimates 
the nuisance functions and thereby requires the product of their  errors (in root mean squared error) in estimating the true functions to be $o_{p}(1)$. Translating to exact sparsity classes, since one can typically estimate OR and PS at a rate $\sqrt{s_{m}\log{p}/n}$ and $\sqrt{s_{\pi}\log{p}/n}$ (see e.g. \citep{buhlmann2011statistics}) respectively (where $s_{\pi}$ is the sparsity of $\pi(\bx)$ and $s_{m}$ is the maximum sparsity of $m(1,\bx)$ and $m(0,\bx)$ respectively), one obtains a requirement of $s_{m}\vee s_{\pi}\ll \sqrt{n}/\log{p}$ for CAN estimation of ATE. More carefully constructed estimators have obtained sharper results through various approaches that lower the requirement on the sparsities of $s_m$ and $s_{\pi}$. For instance, \cite{bradic2019sparsity} constructs an estimator that requires either $(s_{\pi}\ll n/\log{p}, s_m\ll \sqrt{n}/\log{p})$ or $(s_{\pi}\ll \sqrt{n}/\log{p}, s_m\ll n^{3/4}/\log{p})$.

\noindent
\emph{Model Double Robustness:} We now turn to the model double robustness literature. 
In this regard, (a) \cite{athey2018approximate} bypasses  correct specification on PS by exploiting the structure of the bias in estimation of the sparse OR (which is required to satisfy $s_{m}\ll \sqrt{n}/\log{p}$); (b) \cite{wang2020debiased} bypasses correct specification of OR by correcting the bias in estimation of the sparse PS
(which is required to satisfy $s_{\pi}\ll \sqrt{n}/\log{p}$); 
(d) \cite{tan2020model} constructs estimators of ATE based on calibrated OR and PS estimation. This allows valid CAN inference on ATE when the PS model is correctly specified and the OR model is  misspecified (under a linear representation in a feature space), but the product of sparsities of the PS and the limit of the OR estimator is smaller than $n/\log^2{p}$; (e) \cite{ning2020robust} employs a covariate balancing technique to allow for similar results to \cite{tan2020model} but also provides asymptotic normality of their estimator at a rate slower than $\sqrt{n}$ when the PS model is misspecified; and (f) \cite{smucler2019unifying} provides a unified view of construction of rate and model doubly robust estimators of quantities similar in essence to ATE, using ideas from semiparametric theory.

\noindent
{\bf Key Methodological Ingredients and Principles:} The impressive advances surveyed above rest on a few key insights. First, the aforementioned estimators allow $\sqrt{n}$-consistent, asymptotically normal estimation of the ATE, as long as at least one of the PS or OR models is consistently estimable \cite{belloni2014inference,athey2018approximate,tan2020model,tan2020regularized,bradic2019sparsity,wang2020debiased,smucler2019unifying,chernozhukov2017double,farrell2015robust} in $L_2$ norm. Furthermore, while constructing CAN estimators using Neyman orthogonalization, an approach that encompasses AIPW-type estimators, one first establishes an asymptotic expansion 
\cite{farrell2015robust,chernozhukov2017double,smucler2019unifying} under suitable regularity conditions (e.g.~sparsity). This expansion implies a limiting gaussian distribution for the estimator prior to cross fitting. Finally, one establishes that the individual estimators obtained from the permutation of the splits are asymptotically independent on the $\sqrt{n}$ scale, and thus a CLT for the cross-fit estimator follows immediately (see e.g. \cite{chernozhukov2017double,kennedy2022semiparametric}).

\noindent
{\bf Key distinctions in our setting:} It is particularly instructive to evaluate the utility of the aforementioned ideas in our context. First and foremost, consistent estimation of the PS and OR models in $L_2$ norm is impossible in our framework \textcolor{blue}{\citep{mourtada2019exact,sur2019modern,donoho2016high}}. This immediately invalidates the technical ingredients underlying the prior methods. Moreover, the aforementioned expansion of the AIPW estimator fails to hold in our case.
Finally, as mentioned previously, the estimators obtained from permuting different splits are asymptotically dependent in our setting. This crucially affects our analysis, and necessitates a radically different approach. We emphasize that although we assume well-specified PS and OR models, CAN estimation of the ATE is known to be challenging even under these additional simplifications \cite{bang2005doubly,wager2016high,ning2020robust,tan2020model}.

\section{Main Results}
\label{sec:main_results} 
Recall from Section \ref{sec:setup} that we use OLS for fitting the outcome regression model. As a first step, 
we characterize the sample size regimes that ensure the existence of these least squares estimators  with high probability. 

\begin{theorem}
\label{thm:ols_existence}
For any $(a,b,c)$ permutation of $(1,2,3)$, the estimates $(\hat{\alpha}^{(0)}, \hat{\beta}^{(0)})$, $(\hat{\alpha}^{(1)}, \hat{\beta}^{(1)})$ are unique with high probability if and only if $\kappa_b<1/2$. 
\end{theorem}

When we use maximum likelihood for the propensity score estimation, we need to ensure that this exists in our setting. The precise asymptotic threshold for the existence of the logistic MLE  has been recently characterized in \cite{candes2020phase}. Specifically, \cite{candes2020phase} provides an explicit formula for a function $h(\cdot)$ such that when $\kappa = \lim p/n < h(\gamma^2)$ (resp.~$\kappa > h(\gamma^2)$), the logistic MLE exists (resp.~does not exist) with high probability. 
Combining these two requirements, we introduce the notion of a feasible tuple that refers to any combination of problem parameters for which both the OLS for the outcome regression model and the MLE for the propensity score model exist w.h.p.  
\begin{defn}[Feasible]
We call a tuple $(n,p,r_1, r_2, r_3,\beta,\alpha^{(0)},\beta^{(0)}, \alpha^{(1)}, \beta^{(1)})$ to be \emph{feasible} if 
\begin{itemize}
\item[(i)] The logistic regression MLE estimates $\hat{\beta}_{S_1}$, $\hat{\beta}_{S_2}$, $\hat{\beta}_{S_3}$ exist with probability converging to $1$, and 
\item[(ii)] The OLS estimates $\{(\hat{\alpha}^{(0)}_{S_i}, \hat{\beta}^{(0)}_{S_i}): i =1,2,3\}$ and $\{(\hat{\alpha}^{(1)}_{S_i}, \hat{\beta}^{(1)}_{S_i}): i =1,2,3\}$ exist with probability converging to $1$.
\end{itemize}
\end{defn}

We now introduce the first of our two main results that establishes the asymptotic distribution of the cross-fit AIPW estimator for every feasible tuple, when the propensity score model is fit using maximum likelihood.

\begin{theorem}
\label{thm:dr_distribution} 
Assume that the tuple $(n,p,r_1, r_2, r_3,\beta,\alpha^{(0)},\beta^{(0)}, \alpha^{(1)}, \beta^{(1)})$ is feasible and that the logistic MLE is used for propensity score estimation. Under the conditions specified in Section \ref{sec:setup},  as $p,n \to \infty$ with $p/n \rightarrow \kappa > 0$, 
\begin{align}\label{eq:clt}
  & \sqrt{n}(\hat{\Delta}_{cf} - \Delta) \stackrel{d}{\to} \mathcal{N}(0, \sigma_{cf}^2), \quad \text{with} \\
   & \sigma_{cf}^2  =\left[ \left(\sigma^{(0)}\right)^2 + \left(\sigma^{(1)}\right)^2 \right] f(\kappa,\gamma^2)+     \frac{\kappa}{9} \left( \sigma_{0\beta}^2 +  \sigma_{1\beta}^2 - 2 \rho_{01} \sigma_{0 \beta} \sigma_{1 \beta}  \right) \Big( \frac{1}{r_{1}} + \frac{1}{r_2} + \frac{1}{r_3} \Big).\nonumber
\end{align}
\end{theorem}

The effect of fitting the propensity score and the noise level in the observed outcomes appear in the first summand in the variance, while the second summand concerns the signal strengths underlying the two outcome regression models. 
The function $f(\cdot)$ takes a complicated form so we defer its details to Appendix \ref{subsec:fullformula} (Eqn.~\eqref{eq:fkgamma}). 

Our new formula \eqref{eq:clt} warrants an immediate comparison with its classical counterpart. To this end, we consider a simplified setting where $r_i = 1/3$ and $\sigma^{(0)}=\sigma^{(1)}=\sigma_{\varepsilon}$.  If the dimension were fixed, the classical asymptotic (in large sample limit) variance for the AIPW \cite{bang2005doubly} in this case reduces to 
\begin{align}\label{eq:classical}
    \sigma^2_{\text{classical}} = 2\sigma^2_{\varepsilon}\mathbb{E}[\frac{1}{\sigma(x_i^{\top}\beta)}]+\text{Var}\{x_i^{\top}(\beta^{(1)}-\beta^{(0)}   )\}. 
\end{align}
The ultra-high-dimensional settings in \cite{chernozhukov2017double,smucler2019unifying} also admit the same variance form, apart from an additional limit (in $p$) on the RHS to account for the divergence of $p$.  
Here, we restrict our discussion to the fixed $p$ case for simplicity. Note that the second term in $\sigma^2_{\text{cf}}$ (Eq.~\eqref{eq:clt}) is the limit, under our regime, of $ \text{Var}\{x_{i}^{\top}(\beta^{(1)}-\beta^{(0)}) \}$, the second term in the classical formula \eqref{eq:classical}.  Thus, the differences induced by our high-dimensional regime manifests through differences between $f(\kappa,\gamma^2)$ from \eqref{eq:clt} and $\mathbb{E}[1/\sigma(x_i^{\top}\beta)]$ from \eqref{eq:classical}. To visualize this difference, we plot the ratio $\log(f(\kappa,\gamma^2)/\mathbb{E}[1/\sigma(x_i^{\top}\beta)])$ as a function of $p/n$, for a few choices of $\gamma$ in Figure \ref{fig:classvshd}. Note that the ratio tends to zero as $p/n$ approaches zero, indicating that our variance formula recovers the classical formula when the dimensionality decreases. Whereas the ratio deviates further from 1 as $p/n$ grows larger. We investigate our formula for $f(\kappa,\gamma^2)$ further and formally show in Appendix \ref{eq:formalcompare} that $f(\kappa,\gamma^2)$ reduces to $\mathbb{E}[1/\sigma(x_i^{\top}\beta)]$ in the classical regime (fixed $p$, large $n$). We further plot the ratio between the total variance in our regime versus the classical regime in Figure \ref{fig:classvshd}, and observe similar trends.

\begin{figure}
\centering
\begin{subfigure}{.5\textwidth}
  \centering
  \includegraphics[width=\linewidth]{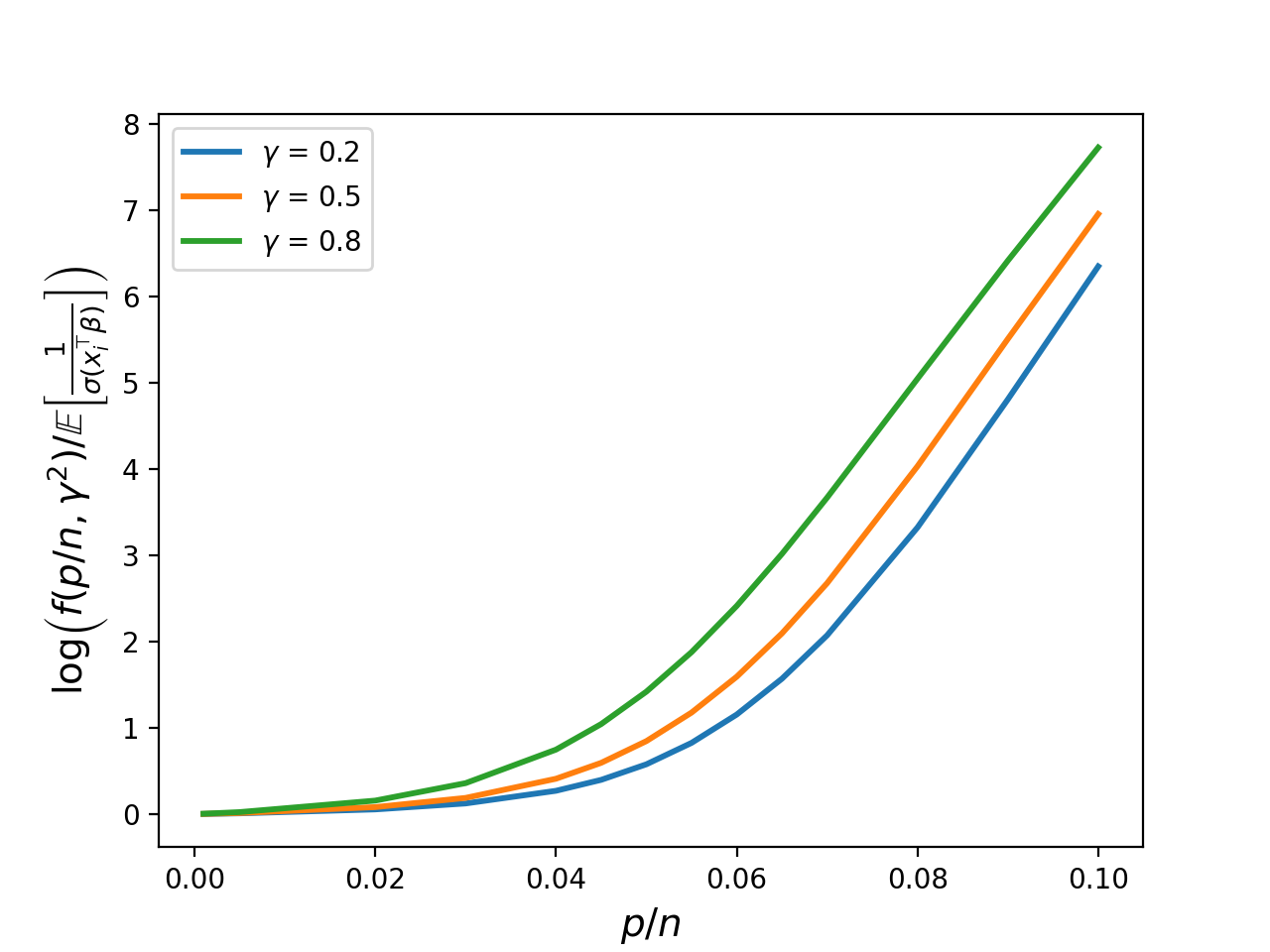}
  \label{fig:sub1}
\end{subfigure}%
\begin{subfigure}{.5\textwidth}
  \centering
  \includegraphics[width=\linewidth]{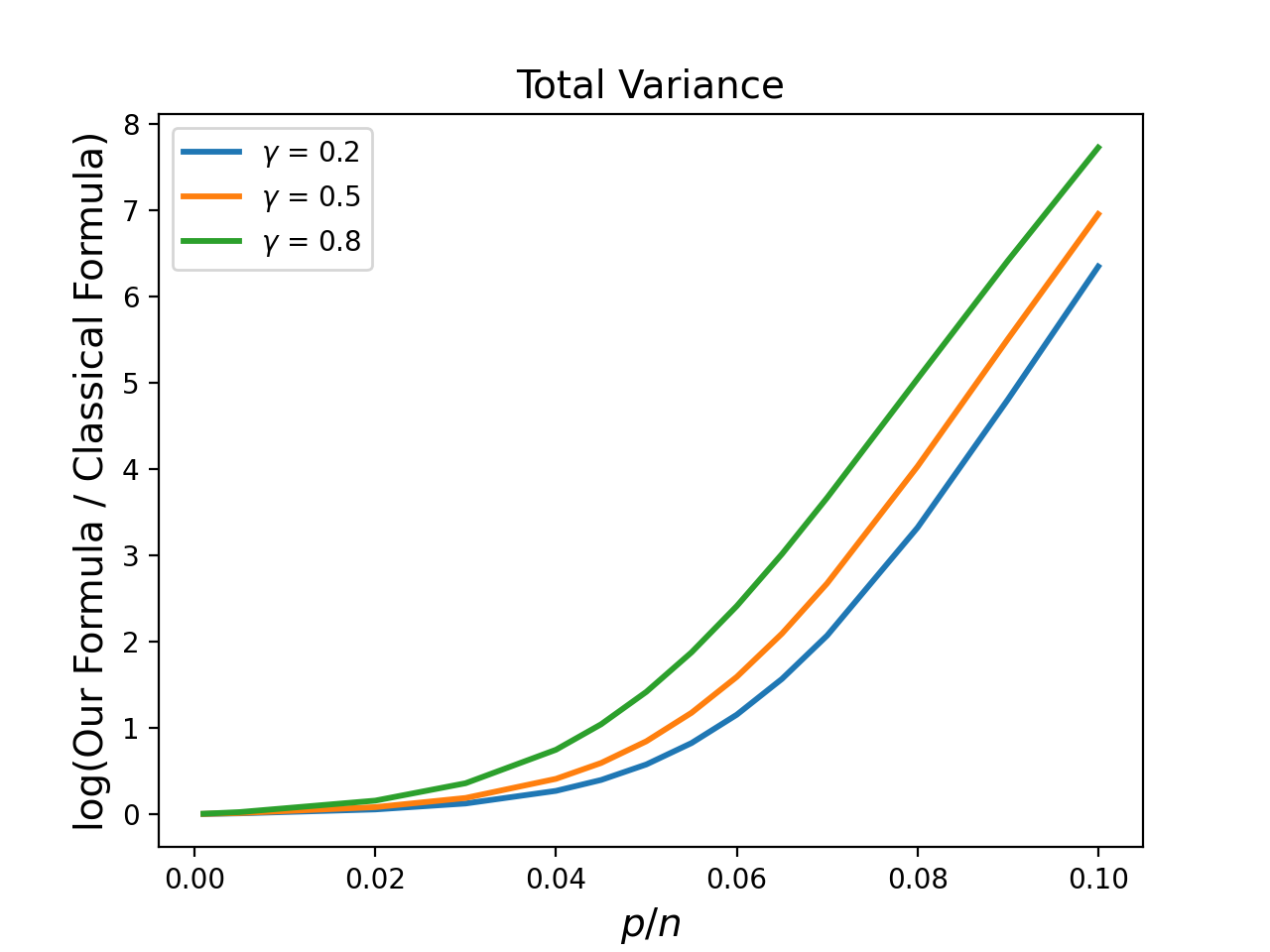}
  \label{fig:sub2}
\end{subfigure}
\caption{x-axis: $ p/n$; y-axis: (Left) The ratio $\log(f(\kappa,\gamma^2)/\mathbb{E}[1/\sigma(x_i^{\top}\beta)])$, where the numerator is from \eqref{eq:clt} and formally defined in Appendix \ref{subsec:fullformula}; (Right) $\log(\sigma^2_{\text{cf}}/\sigma^2_{\text{classical}})$, where these are defined as in  \eqref{eq:clt} and \eqref{eq:classical}. The other problem parameters assume the following values: $\alpha^{(0)}=0,\alpha^{(1)}=2,$ so that the ATE equals $2$; the error variances in both OR models equal $1$; $ \sigma_{0\beta}  = \sigma_{1\beta}  = 0.1/\sqrt{\kappa}, \rho_{01}= 0.2.$ Note that $f(\kappa,\gamma^2)$ equals $\mathbb{E}[1/\sigma(x_i^{\top}\beta)]$ as $p/n$ converges to zero, and similar for the total variance, suggesting that our theory recovers the classical theory in this limiting case. For formal calculations along this line, we defer the readers to Appendix \ref{eq:formalcompare}.}
\label{fig:classvshd}
\end{figure}

Note that Theorem \ref{thm:dr_distribution} uses maximum likelihood for both the OR and PS models, thereby restricting the parameter range where the Theorem applies.
To overcome this restriction, we next establish an analogous CLT where the propensity scores are estimated via ridge regularized logistic regression.

\begin{theorem}
\label{thm:dr_distribution_ridge} 
Fix any $\lambda \in \mathbb{R}^{+}$. Assume that the OLS estimates $\{(\hat{\alpha}^{(0)}_{S_i}, \hat{\beta}^{(0)}_{S_i}): i =1,2,3\}$ and $\{(\hat{\alpha}^{(1)}_{S_i}, \hat{\beta}^{(1)}_{S_i}): i =1,2,3\}$ exist with probability converging to $1$, that is, $\kappa_i < 1/2$ for all $i$.
Under the conditions specified in Section \ref{sec:setup},  as $p,n \to \infty$ with $p/n \rightarrow \kappa > 0$, 
\begin{align}
    \sqrt{n}(\hat{\Delta}_{cf}^{(\lambda)} - \Delta) \stackrel{d}{\to} \mathcal{N}\left(0, \left(\sigma_{cf}^{(\lambda)}\right)^2\right),\nonumber
\end{align}
where $$\left(\sigma_{cf}^{(\lambda)}\right)^2 =\left[ \left(\sigma^{(0)}\right)^2 + \left(\sigma^{(1)}\right)^2 \right] f^{(\lambda)}(\kappa,\gamma^2)+     \frac{\kappa}{9} \left( \sigma_{0\beta}^2 +  \sigma_{1\beta}^2 - 2 \rho_{01} \sigma_{0 \beta} \sigma_{1 \beta}  \right) \Big( \frac{1}{r_{1}} + \frac{1}{r_2} + \frac{1}{r_3} \Big).$$
\end{theorem}

Once again, $f^{(\lambda)}(\cdot)$ takes a complicated
form so we defer its details to Appendix \ref{appendix g}. Note the limiting variance has a similar structure as in Theorem \ref{thm:dr_distribution}. On examining the Appendix one would observe that $f^{(\lambda)}(\kappa,\gamma^2)$ equals $f(\kappa,\gamma^2)$ when $\lambda=0$, as we would expect.

We next study the finite sample efficacy of our result.
Through the rest of this section and the subsequent section, we set $n = 10,000, n_1=3,333, n_2 = 3,333, n_3 = 3,334, p = 700$ so that the “dimensionalities” $p/n_1, p / n_2, p/n_3$ are approximately 0.21. 
The matrix of covariates has i.i.d. $ \mathcal{N}\left(0, 1/n\right)$ entries unless otherwise specified, and the  regression coefficients $\beta, \beta^{(1)}, \beta^{(0)}$ are drawn from normal distributions with zero mean and scaled in such that $  \gamma=0.1$ and $ \sigma_{0\beta},\sigma_{1\beta},\rho_{01}$ remain the same as in Figure \ref{fig:classvshd}.

In the aforementioned setting,
Figure \ref{normal_qq_plot_original_vs_trimmed} shows two overlaid normal Q-Q plots of $\sqrt{n}\left(\hat{\Delta}_{c f}-\Delta\right)$. In both cases, we compute the sample quantiles from 30,000 simulation runs. The darker blue points represent the theoretical quantiles based on our theory, when the logistic MLE is used for propensity score estimation, while the lighter cyan points represent those computed based on the classical theory. Observe that our theory captures the true sample quantiles accurately. The plot exhibits some deviation from the reference line near the tails. This occurs due to the presence of $\sigma(\cdot)$ and $1-\sigma(\cdot)$ in the denominator of the AIPW estimator. It is expected that if either of these terms is extremely small, this would manifest as outliers in the QQ-plot. To alleviate this issue, we winsorize the sigmoid function to satisfy $0.005 \leq \sigma(\cdot) \leq 0.995$. This winsorizing step is commonly used in the implementation of the AIPW estimator. Figure \ref{normal_qq_plot_original_vs_trimmed} demonstrates that after winsorizing, our theoretical variance matches the empirical value exceptionally well. We discuss the possibilities of rigorously quantifying an analogous CLT for the winsorized estimator in Section \ref{sec:discussions}. In Section \ref{sec:experiments}, we further study the effects of regularized estimation of the propensity scores (Theorem \ref{thm:dr_distribution_ridge}). 

\begin{figure}[!ht]
   \includegraphics[scale = 0.4]{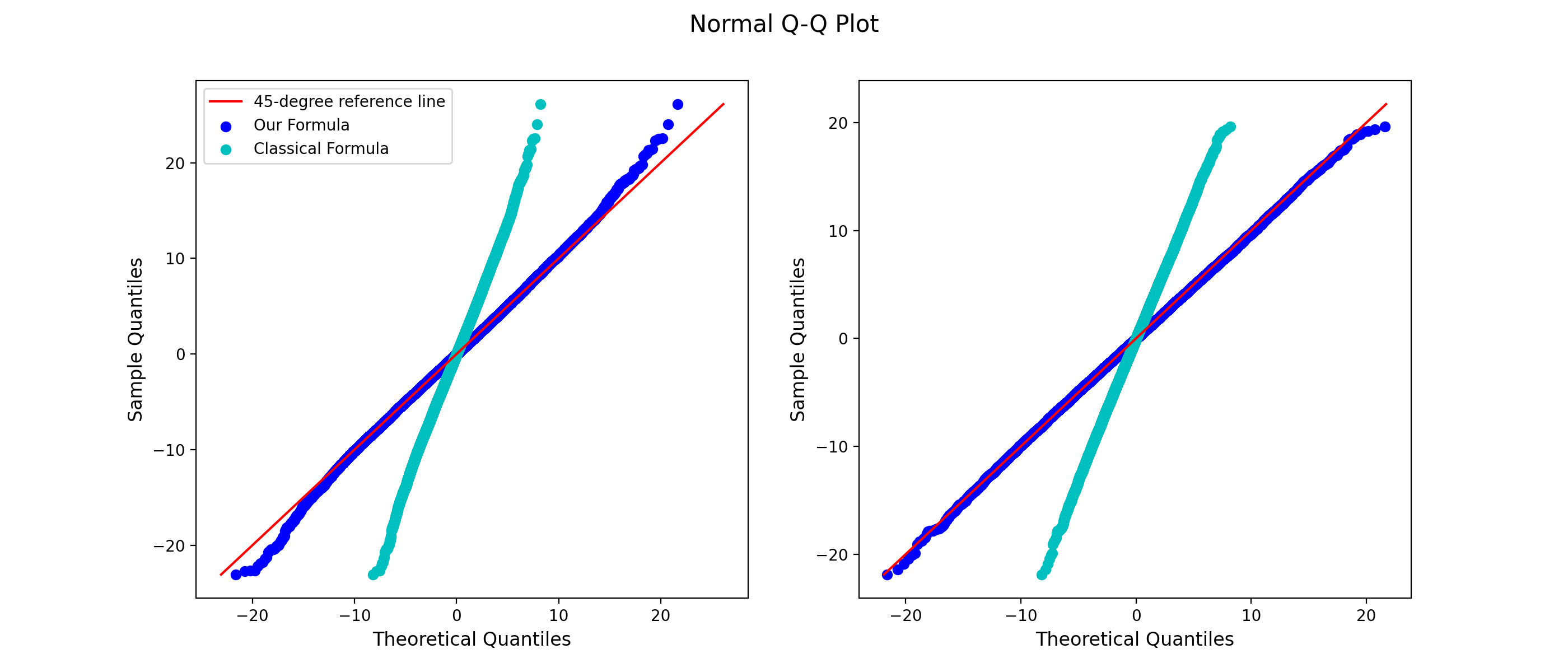}
    \caption{(Left) Normal Q-Q plots of $\sqrt{n}\left(\hat{\Delta}_{cf} - \Delta \right)$, where the sample quantiles are computed from 30,000 samples. The normal Q-Q plot where the theoretical quantiles are computed according to the classical theory is shown in cyan, while the normal Q-Q plot where the theoretical quantiles are computed according to our result is shown in blue. (Right) the settings are the same as the figure on the left, except that the estimated propensity scores are winsorized with threshold 0.005.}
    \label{normal_qq_plot_original_vs_trimmed}
\end{figure}

\section{Numerical Experiments}\label{sec:experiments} 
This section pursues important additional empirical investigations. First, we present a crucial phenomenon that can be studied as an upshot of our theory. Next, we study the effects of cross-validation, and finally, we test the robustness of our results to the covariate distribution assumptions.

 \textbf{Effects of Cross-fitting in high dimensions?} The existing literature on cross-fit AIPW tells us the following important fact: at the $\sqrt{n}$-scale, the covariances between certain pairs of pre-cross-fit estimators are asymptotically negligible. Thus, averaging the pre-cross-fit estimators leads to a variance reduction. In our setting, we observe that these cross-covariances admit non-trivial limits, and our proof for Theorem \ref{thm:dr_distribution} precisely characterizes the asymptotic values of these cross-covariances.
 
To describe further,
 denote $\hat{\Delta}_{\text{pre$\_$fit}}(S_a, S_b, S_c)$ to be the pre-cross-fit AIPW estimator where the PS is estimated using $S_a$, the OR is estimated using $S_b$, and the AIPW is calculated on $S_c$, plugging in the preceding nuisance estimates. Suppose we group the $3 !$ pre-cross-fit AIPW estimators into 3 pairs, where each pair consists of two estimators of the form $\Big(\hat{\Delta}_{\text{pre$\_$fit}}(S_a, S_b, S_c), \hat{\Delta}_{\text{pre$\_$fit}}(S_b, S_a, S_c)\Big)$. With this grouping, we may split our asymptotic variance $\sigma^2_{\text{cf}}$ into the following parts:
\begin{align}
&\sigma^2_{\text{cf}}  = \sum_{(a,b,c) \in \mathscr{S}_3} \text{Var}(\hat{\Delta}_{\text{pre$\_$fit}}(S_a, S_b, S_c)) \label{eq:variance}\\
&+ \frac{1}{2}\sum_{a,b,c \in [3], a\neq b\neq c}\text{Cov}\Big(\hat{\Delta}_{\text{pre$\_$fit}}(S_a, S_b, S_c), \hat{\Delta}_{\text{pre$\_$fit}}(S_b, S_a, S_c)\Big)\label{eq:wpcov}\\
&+\sum_{a,b,c \in [3], a\neq b\neq c}\text{Cov}\Big(\hat{\Delta}_{\text{pre$\_$fit}}(S_a, S_b, S_c), \hat{\Delta}_{\text{pre$\_$fit}}(S_c, S_a, S_b)\Big) \label{eq:bpcov1}\\
&+ \sum_{a,b,c \in [3], a\neq b\neq c}\text{Cov}\Big(\hat{\Delta}_{\text{pre$\_$fit}}(S_a, S_b, S_c), \hat{\Delta}_{\text{pre$\_$fit}}(S_b, S_c, S_a)\Big),\label{eq:bpcov2}
\end{align}
where the second covariance term captures sum of the total covariance within each pair, and the sum of the last two  terms capture the overall between-pair covariances. On examining each term in the decomposition \eqref{eq:variance}--\eqref{eq:bpcov2}, we observe that both \eqref{eq:variance} and \eqref{eq:wpcov} contribute in our setting and in the classical low-dimensional setting. But, their magnitude is higher in our regime due to high-dimensional effects. In fact, if we were to plot ratios of these terms under the two regimes, we would once again observe trends similar to those reported in Figure \ref{fig:classvshd}. Thus, we refrain from investigating these further and instead turn to the between-pair covariance, that is, sum of \eqref{eq:bpcov1} and \eqref{eq:bpcov2}.

In the classical regime, the total between-pair covariance is negligible at the $\sqrt{n}$-scale. However, these contribute non-trivially in our regime even in the large sample and large dimensional limit. The reader should view this phenomenon as an additional effect of cross-fitting in high dimensions. When we fit propensity scores using maximum likelihood, we observe that the total between-pair covariance is negative, as demonstrated via Figure \ref{fig:bpcov}. This illustrates that cross-fitting helps in high dimensions in such settings, in addition to its usual advantages discussed in \cite{chernozhukov2017double,newey2018cross}. However, on using ridge regression for estimating the propensity score, we observe that this between-pair covariance could be positive in some cases. Thus, one needs to investigate this phenomenon further to characterize the interplay between the problem parameters, e.g. signal strength, tuning parameter, etc. that determines regimes where the between-pair covariance is negative in our high-dimensional setting. We defer these additional investigations to future work. To our knowledge, our work uncovers such non-trivial  between-pair covariances for the first time in the literature on high-dimensional causal inference and cross-fitting.

\begin{figure}
\centering
    \includegraphics[scale=0.5,keepaspectratio]{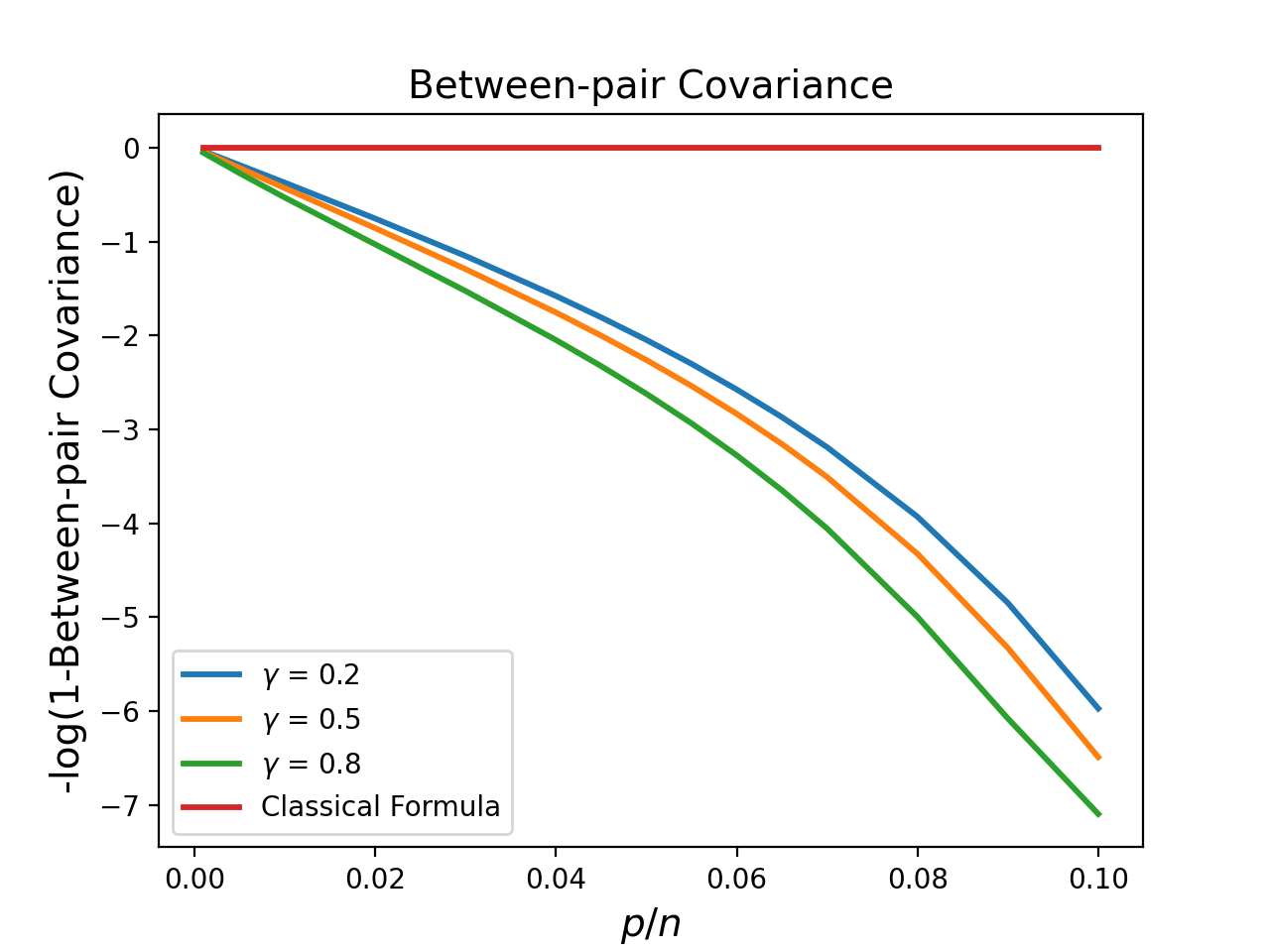} 
    \caption{x-axis: $ p/n$; y-axis: Here the Between-pair covariance refers to sum of the terms in \eqref{eq:bpcov1} and \eqref{eq:bpcov2} in our regime.  When $p/n$ converges to zero, the curves  corresponding to $\gamma=0.2,0.5,0.8$ approach zero, suggesting that the Between-pair covariances end up negligible, as is the case under classical asymptotics. We plot negative of the logarithm to visually depict that the Between-pair covariance is negative as soon as $p/n > 0$. This is in stark contrast  to both the classical regime and the existing ultra-high-dimensional literature that assumes sparsity in at least one  nuisance model. The setting is the same as in Figure \ref{fig:classvshd}.}\label{fig:bpcov}
\end{figure}

 \textbf{Does Cross-validation Find the Optimal Regularization Parameter?} 
 In this paper, we allow regularized estimation of the propensity score model via ridge penalized logistic regression. This naturally requires suitable choice of the tuning parameter. In traditional supervised learning, one seeks to tune the regularization parameters to optimize the out-of-sample prediction accuracy. In this context, it is well-known that tuning parameter selection approaches such as k-fold cross-validation (CV) suffer from large biases in high dimensions (c.f \cite{rad2020scalable}), whereas leave-one-out cross validation (LOOCV) exhibits desirable properties \cite{patil2021uniform}. 
 
 Note that in our setting, the tuning parameter should not be selected to optimize prediction accuracy on a test point, but rather to minimize the variance of the downstream AIPW estimator. However, traditional CV based approaches are still widely utilized in this setting. Here, we explore the impact of this choice on the ATE estimation task.  Formally, we study the effects of using LOOCV for choosing the  tuning parameter for the propensity score model on the downstream performance of the AIPW estimator. Note that LOOCV is computationally expensive, so we work with an approximation obtained as follows.  For any given $\lambda$, LOOCV involves computing all $n$ possible leave-one-out estimates $\{\hat{\beta}^{(\lambda)}_{S_a}\}^{(-i)}$.  Now, \cite[Lemma 21]{Sur14516} relates such leave-one-out estimates to the original estimator, when one uses the logistic MLE. Using the exact same computation, an analogous expression can be derived for the ridge regularized problem. This connects $\{\hat{\beta}^{(\lambda)}_{S_a}\}^{(-i)}$'s to the original ridge estimate $\{\hat{\beta}^{(\lambda)}_{S_a}\}$. Utilizing this formula, one can bypass the computational overload induced by the leave-one-out operation and obtain an approximation that is asymptotically equivalent to LOOCV (\cite{rad2020scalable,wang2018approximate} studies such approximations for a variety of problems). We  implement this approximate LOOCV in Figure \ref{ridge_theo_vs_cv}---the dotted red line shows the standard deviation of the AIPW estimator corresponding to the tuning parameter chosen via this approximated LOOCV. The solid blue line shows the variation in the standard deviation as a function of the tuning parameter. The optimal tuning parameter (in terms of the standard deviation) reduces the variance significantly compared to the MLE, as one would expect. However, the LOOCV  tuned estimator is highly sub-optimal. This clearly illustrates that optimizing the propensity score fit for predictive accuracy at the first stage does not guarantee optimal sampling variance downstream. 

 On the other hand, if one can develop consistent estimators for the signal strength parameters $\gamma, \sigma_{0\beta}, \sigma_{1\beta}$ and $\rho_{01}$, our theory provides an alternate route to select tuning parameters (thereby minimizing the downstream variance).  We defer further discussions on the possibility of developing such estimators to Section \ref{sec:discussions}.

\begin{figure}[!ht]
  \includegraphics[scale = 0.5]{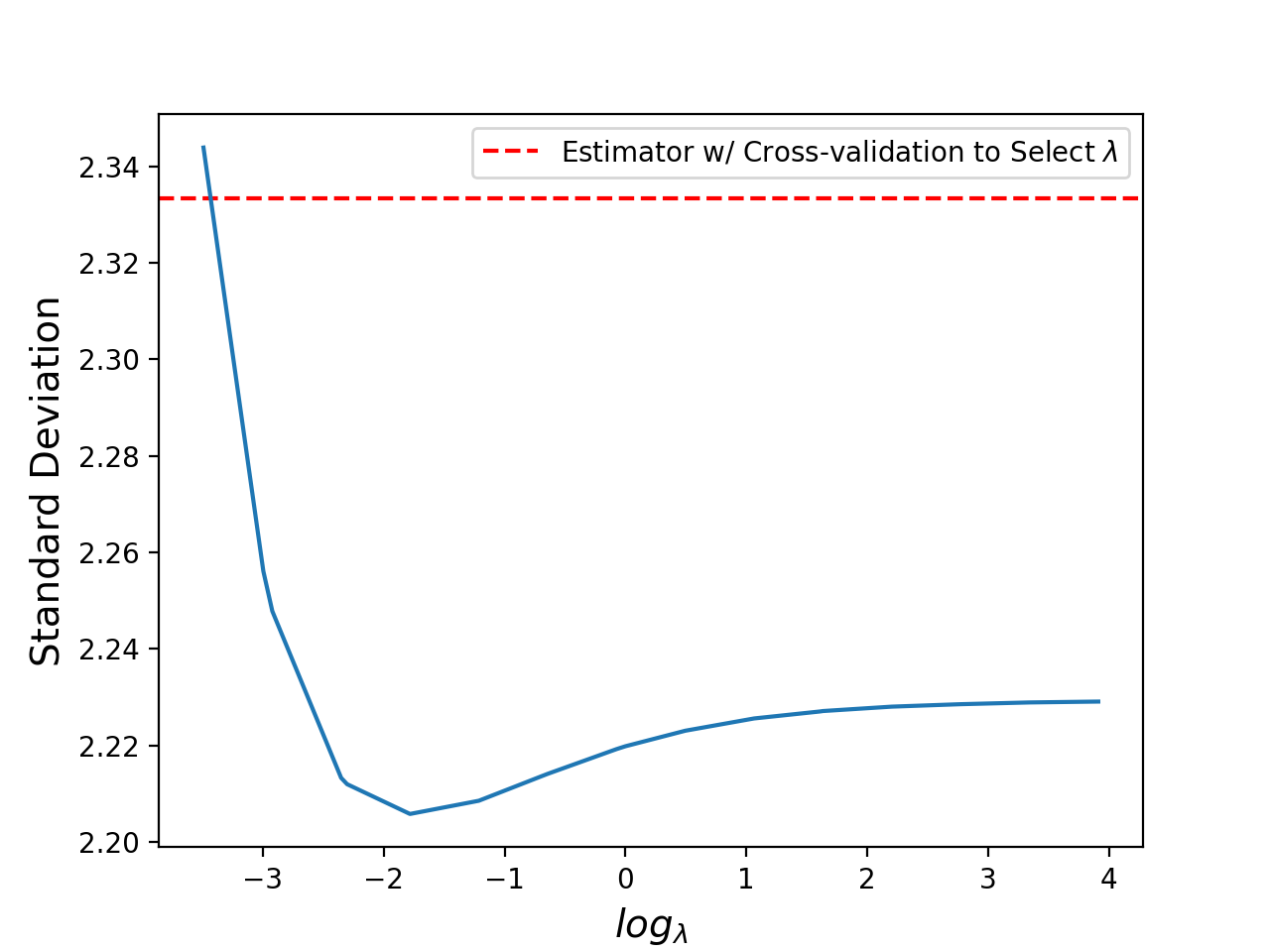}
    \caption{The blue curve shows the theoretical standard deviation of $\hat{\Delta}_{c f}^{(\lambda)}$ at different values of $\lambda$. The red dashed line indicates the SE of $\hat{\Delta}_{c f}^{(\lambda_{LOOCV})}$, where $\lambda_{LOOCV}$ is chosen by approximated LOOCV from a list of ten numbers equally spaced between 0.01 and 100 on the log scale. The SE is estimated using 5,000 samples. Clearly, LOOCV chooses a sub-optimal value of $\lambda$ in terms of the standard deviation.}
    \label{ridge_theo_vs_cv}  
\end{figure}

\textbf{Robustness to Normality Assumptions?} To conclude our empirical investigations, we test the validity of our theory under non-Gaussian covariate distributions. We consider two settings: (i) a simple Uniform distribution and (ii) a discrete distribution inspired by applications in statistical genetics. Figure \ref{robust_uniform} shows two overlaid normal Q-Q plots of $\sqrt{n}\left(\hat{\Delta}_{c f}-\Delta\right)$, where the matrix of covariates has i.i.d. Uniform$\left(-\sqrt{\frac{3}{n}}, \sqrt{\frac{3}{n}}\right)$ entries. Observe that although our theory fails to cover this setting for the time being,
the theoretical predictions match the empirical behavior of the cross-fit AIPW remarkably well. To test the validity of our theory further, Figure \ref{robust_discrete} considers a design matrix where the $j$th feature takes values in \{0,1,2\} with probabilities $p_{j}^{2}, 2 p_{j}\left(1-p_{j}\right),\left(1-p_{j}\right)^{2}$; here, $p_j \in [0.25, 0.75]$ and $ p_{j} \neq p_{k} \text { for } j \neq k$. Features are then centered and rescaled to have unit
variance. The setting is otherwise the same as for Figure \ref{normal_qq_plot_original_vs_trimmed}. The left plot depicts the quantiles for the cross-fit AIPW and compares with our theory. We see that the suitably scaled and centered cross-fit estimator still follows an approximate normal distribution whose variance can be characterized by our results far better than the classical variance. This time, we do observe deviations from our theory---this is indeed expected since several of the estimated propensity scores are either too small or too large for this particular setting. This prompts us to consider the winsorized version of the estimator, where as before, $0.005 < \sigma(\cdot) < 0.995$. The right plot shows the winsorized cross-fit AIPW, and once again we observe the empirical quantiles match those based on our CLT extremely well. This set of simulations raises
an interesting question: can one characterize the class of covariate distributions under which our same CLT applies? In light of our current fairly involved proofs, we defer theoretical investigations in this direction to future work.

\begin{figure}[!ht]
  \includegraphics[scale = 0.5]{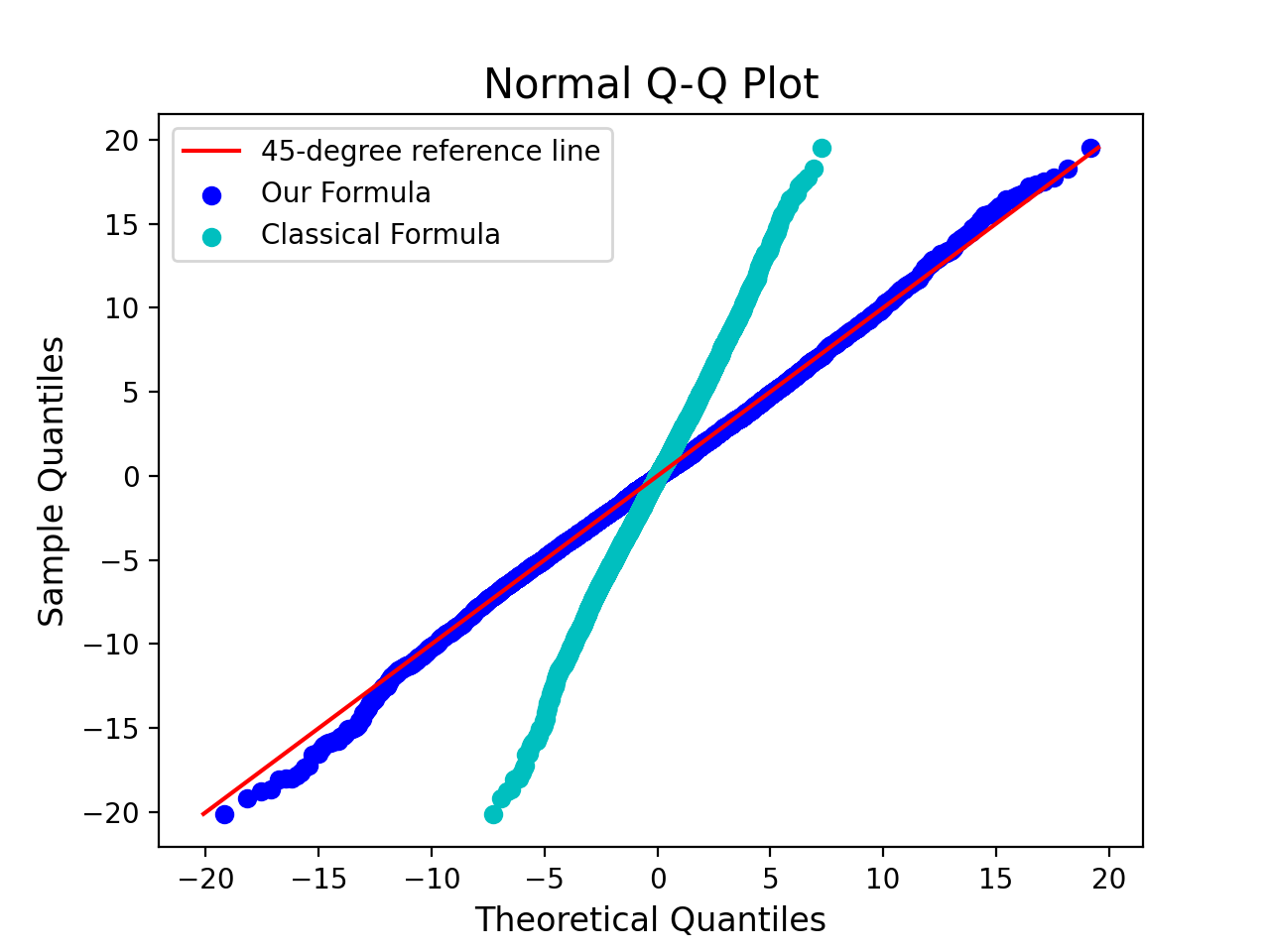}
    \caption{Normal Q-Q plots of $\sqrt{n}(\hat{\Delta}_{cf} - \Delta )$, where the matrix of covariates has i.i.d. Uniform($-\sqrt{\frac{3}{n}}, \sqrt{\frac{3}{n}}$) entries. The setting is otherwise the same as for Figure \ref{normal_qq_plot_original_vs_trimmed}.}
    \label{robust_uniform}
\end{figure}

\begin{figure}[!ht]
  \includegraphics[scale = 0.4]{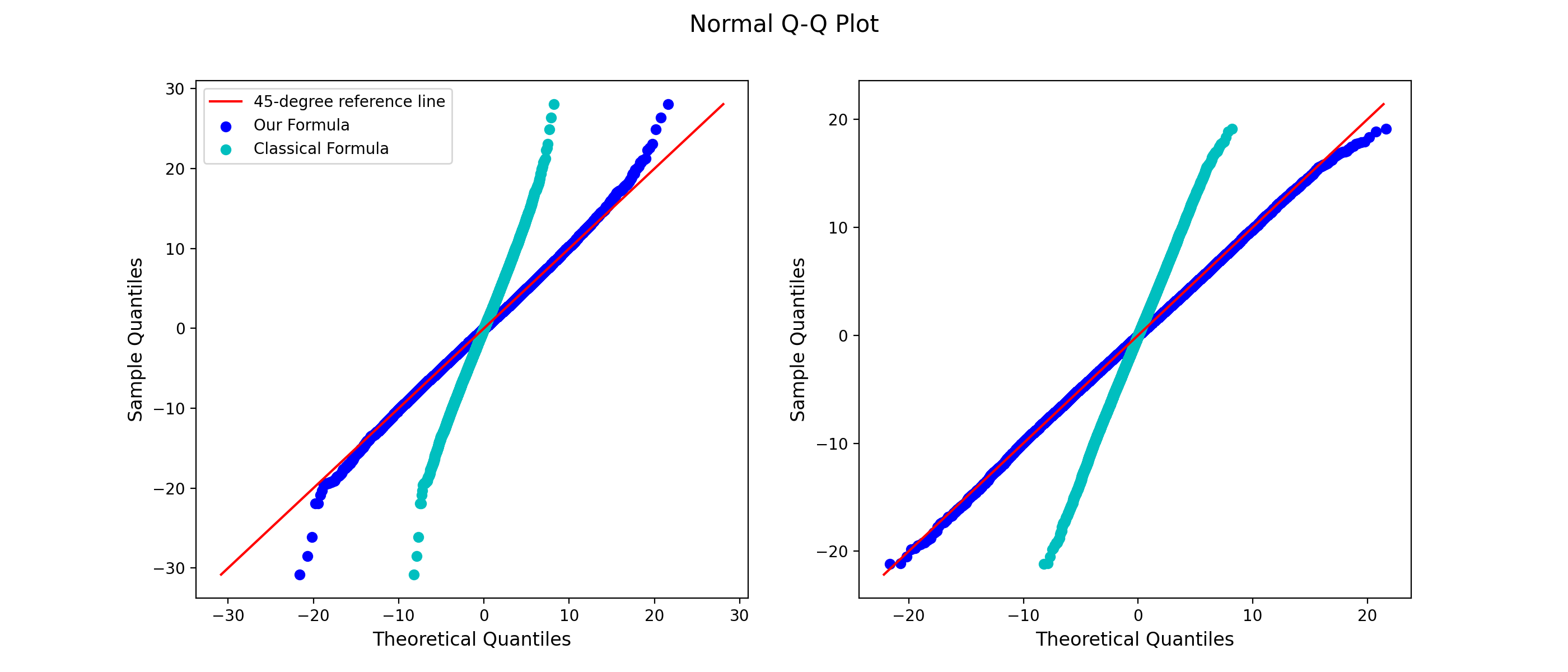}
    \caption{Left: Normal Q-Q plots of $\sqrt{n}(\hat{\Delta}_{cf} - \Delta )$, where the $j$th feature takes values in \{0,1,2\} with probabilities $p_{j}^{2}, 2 p_{j}\left(1-p_{j}\right),\left(1-p_{j}\right)^{2}$; here, $p_j \in [0.25, 0.75]$ and $ p_{j} \neq p_{k} \text { for } j \neq k$. Features are then centered and rescaled to have unit
variance. The setting is otherwise the same as for Figure \ref{normal_qq_plot_original_vs_trimmed}. Right: the settings are the same as the figure on the left, except that the estimated propensity scores are winsorized with threshold 0.005.}
    \label{robust_discrete}
\end{figure}

\section{Proof Outline}
\label{sec:proof_ideas}

In this section, we collect some ideas involved in the proof of Theorem \ref{thm:dr_distribution}, and discuss the main technical ingredients. To this end, we first introduce some notation. 

For $j \in \{1,2,3\}$, let $\mathcal{E}_{S_j}$ denote the vectors containing $ A_i\epsilon_{i}^{\left(1\right)} +(1-A_i)\epsilon_{i}^{\left(0\right)}   $ for all $i \in S_j$.

 For $j,k \in \{1,2,3\}$, define 
  \begin{align}
      g(X_{S_j}) = \frac{1}{\sqrt{n_j}} \sum_{i \in S_{j}} x_i^\top\beta^{(1)}, &\,\,\,\, \tilde{g}\left(X_{S_{j}}\right)=\frac{1}{\sqrt{n_{j}}} \sum_{i \in S_{j}} x_{i}^\top \beta^{(0)}, \nonumber \\
      l_{S_j,S_k} =\frac{1}{\sqrt{n_j}} \begin{pmatrix} 
 \sum_{i \in S_j} \left(\frac{A_i}{\sigma(x_i^{\top} \hat{\beta}_{S_k})}-1 \right)\\
 \sum_{i \in S_j} \left(\frac{A_{i}x_i}{\sigma(x_{i}^{\top} \hat{\beta}_{S_k})}-x_i \right)
\end{pmatrix}, &\,\,\,\, \tilde{l}_{S_{j}, S_{k}}=\frac{1}{\sqrt{n_{j}}}\left(\begin{array}{c}
\sum_{i \in S_{j}}\left(\frac{1-A_{i}}{1-\sigma\left(x_{i}^{\top} \hat{\beta}_{S_k}\right)}-1\right) \\
\sum_{i \in S_{j}}\left(\frac{(1-A_{i}) x_{i}}{1-\sigma\left(x_{i}^{\top} \hat{\beta}_{S_k}\right)}-x_{i}\right)
\end{array}\right), \nonumber \\
f(\mathcal{E}_{S_j},X_{S_j}) = \begin{pmatrix}
 \alpha^{(1)}-\hat{\alpha}^{(1)}_{S_j} \\
\beta^{(1)}-\hat{\beta}^{(1)}_{S_j}
\end{pmatrix}, &\,\,\,\, \tilde{f}\left(\mathcal{E}_{S_{j}}, X_{S_{j}}\right)=\left(\begin{array}{c}
\alpha^{(0)}-\hat{\alpha}^{(0)}_{S_{j}} \\
\beta^{(0)}-\hat{\beta}^{(0)}_{S_{j}}\end{array}\right).\nonumber 
  \end{align}
  Furthermore, let $V_{S_j,S_k},\tilde{ V}_{S_{j}, S_{k}} \in \mathbb{R}^{|S_j|}$ denote the vectors  containing $ \frac{1}{\sqrt{n_j}} \frac{A_i}{ \sigma(x_i^{\top} \hat{\beta}_{S_k}) }$ and $ \frac{1}{\sqrt{n_{j}}} \frac{1-A_{i}}{1-\sigma\left(x_{i}^{\top} \hat{\beta}_{S_k}\right)}$ for all $i \in S_j$ respectively.
  We establish the following representation for the cross-fitted AIPW estimator in Lemma \ref{lemma:representation}.  
  \begin{align}
      &\sqrt{n}(\hat{\Delta}_{cf} - \Delta) = T_1 + T_2, \label{eq:decomp}\\
      &T_1 = \frac{1}{6  }  \sum_{(a,b,c) \in \mathscr{S}_3} \frac{1}{\sqrt{r_c}} \left[ l_{S_c,S_a}^{\top} f(\mathcal{E}_{S_b},X_{S_b}) - \tilde{l}_{S_c,S_a}^{\top} \tilde{f}(\mathcal{E}_{S_b},X_{S_b}) + \mathcal{E}_{S_c}^{\top}V_{S_c,S_a}
      - \mathcal{E}_{S_c}^{\top}\tilde{V}_{S_c,S_a} \right]. \nonumber \\
      &T_2 = \frac{1}{6 }  \sum_{(a,b,c) \in \mathscr{S}_3} \frac{1}{\sqrt{r_c}}[g(X_{S_c}) - \tilde{g}(X_{S_c})], \nonumber  
  \end{align}
In the representation above, $\mathscr{S}_3$ denotes the set of all permutations of $\{1,2,3\}$, and we use $(a,b,c)$ to denote the permutations in this set. 

The representation \eqref{eq:decomp} is critical for analyzing the limiting distribution of the AIPW estimator. Note that conditioned on everything but the $\mathcal{E}$ variables, $T_1$ has a mean-zero gaussian distribution. On the other hand, $T_2$ has a mean-zero gaussian distribution as it is a linear function of the $\{x_i : 1\leq i \leq n\}$ variables. As the $\{x_i : 1\leq i \leq n\}$ and $\mathcal{E}$ variables are independent, it is not hard to see that $T_1$ and $T_2$ are asymptotically independent. As both $T_1$ and $T_2$ are mean zero gaussian and asymptotically independent, the limiting gaussian distribution of the AIPW follows immediately, once we establish that the limiting variance of $T_1$ and $T_2$ converge to well-defined constants. 

The limiting variance of $T_2$ is explicit, and its convergence follows directly from our assumptions \eqref{eq:limit_dist} and \eqref{eq:norm_limits}. The variance of $T_1$ is significantly more involved---we establish that $\mathrm{Var}(T_1|A,X)$ converges to a deterministic constant in the limit $n \to \infty$. This is our main theoretical contribution, and requires the bulk of the technical work in this paper. 

To characterize the limit of $\mathrm{Var}(T_1|A,X)$, we carefully combine several distinct ingredients. We take this opportunity to briefly describe each tool, and motivate its usefulness in our setting. We believe these ideas can be useful for analyzing other estimators in high-dimensions, and should be of independent interest. 

\noindent 
{\bf Approximate Message Passing and state evolution:} Approximate Message Passing (AMP) algorithms were introduced in the study of mean-field spin glasses and in compressed sensing \cite{donoho2009message,bolthausen2014iterative}. In high-dimensional statistics, these algorithms provide a valuable theoretical device---they can be used to ``track" the performance of specific statistical estimators e.g. the LASSO, M-estimators, the MLE etc. 
At a high-level, an AMP algorithm introduces an iterative system $\{\hat{\beta}^{t}: t \geq 1\}$ which ``converges" to the estimator of interest $\hat{\beta}$---formally,  
\begin{align}
    \lim_{t\to \infty} \lim_{p \to \infty} \frac{1}{p} \|\hat{\beta}^t - \hat{\beta}\|^2 \to 0 \quad \text{a.s.} \label{eq:limit_conv}  
\end{align}
AMP algorithms are attractive theoretical devices in high-dimensional statistics, as their empirical distributions can be tracked using low-dimensional scalar recursions, referred to as ``state-evolution". In particular, for well-behaved functions $f:\mathbb{R}^2\to \mathbb{R}$ and any $t \geq 1$,  one obtains explicit expressions for the limits of empirical averages $\frac{1}{p} \sum_{i=1}^{p} f(\beta_i, \hat{\beta}_i^t)$ as $p \to \infty$. Here, $\beta \in \mathbb{R}^p$ refers to the underlying latent parameter of interest. Subsequently, using the AMP convergence property \eqref{eq:limit_conv} and setting $t \to \infty$, one obtains a precise characterization of the empirical distribution of the estimator $\hat{\beta}$. Specifically, this characterizes $ \frac{1}{p} \sum_{i=1}^{p} f(\beta_i, \hat{\beta}_i)$ in the limit $p \to \infty$. We do not provide a more formal discussion of AMP style algorithms and their consequences in this paper, but refer the interested reader to \cite{montanari2012graphical,feng2021unifying} for an in-depth exposition of these ideas. We note in passing that similar characterizations of empirical averages can also be obtained using the parallel approach based on Gaussian comparison inequalities \cite{stojnic2013framework,thrampoulidis2015gaussian,thrampoulidis2018precise}. 

Instead, we turn to the importance of these ideas in our analysis. In our analysis of the conditional variance $\mathrm{Var}(T_1| A,X)$, we naturally have to deal with averages of the form 
\begin{align}
    \frac{1}{n_3} \sum_{i \in S_3} f(x_i^{\top} \beta, x_i^{\top} \hat{\beta}_{S_1}), \label{eq:typical_avg} 
\end{align}
where $\hat{\beta}_{S_1}$ denotes the MLE estimate for the propensity score model based on the sample split $S_1$. Note that as $S_1$ and $S_3$ are disjoint, conditioned on the samples in $S_1$, the empirical average above is an i.i.d. average, with $(x_i^{\top} \beta, x_i^{\top} {\hat{\beta}_{S_1}})$ are bivariate gaussian with mean zero, $\mathrm{Var}(x_i^{\top} \beta|S_1) = \frac{1}{p} \|\beta\|^2$,  $\mathrm{Var}(x_i^{\top} \hat{\beta}_{S_1}|S_1) = \frac{1}{p} \|\hat{\beta}_{S_1}\|^2$, and $\mathrm{cov}(x_i^{\top}\beta, x_i^{\top} \hat{\beta}_{S_1}) = \frac{1}{p}  \beta^{\top} \hat{\beta}_{S_1}$. Thus conditioned on the samples in $S_1$, as $n_3 \to \infty$, 
\begin{align}
    \frac{1}{n_3} \sum_{i \in S_3} f(x_i^{\top} \beta, x_i^{\top} \hat{\beta}_{S_1}) \approx \mathbb{E}[f(Z_1,Z_2)], \nonumber 
\end{align}
where $(Z_1,Z_2)$ is a mean zero bivariate normal with the covariance matrix described above. Observe that if one could establish that the (random) covariance matrix of $(Z_1,Z_2)$ stabilizes to a deterministic limit as $n_1 \to \infty$, it immediately follows that 
\begin{align}
    \frac{1}{n_3} \sum_{i \in S_3} f(x_i^{\top} \beta, x_i^{\top} \hat{\beta}_{S_1}) \to \mathbb{E}[f(Z_1,Z_2)], \nonumber 
\end{align}
where $(Z_1,Z_2)$ is a mean zero bivariate gaussian with the limiting covariance matrix. This is precisely the step where the state-evolution characterization of the MLE is invaluable. Indeed, note that both $\frac{1}{p} \| \hat{\beta}_{S_1} \|^2$ and $\frac{1}{p} \beta^{\top} \hat{\beta}_{S_1}$ are empirical averages of the form described above, and thus have well-defined, explicit, deterministic limits specified by the state-evolution description. This idea is used repeatedly in our proof to characterize the (deterministic) limits of several averages of the form \eqref{eq:typical_avg}.

\noindent 
{\bf Deterministic Equivalents:} 
In classical random matrix theory, the limiting spectral distribution of a random matrix $M \in \mathbb{R}^{n\times n}$ is an object of central interest. The limiting spectral measure of classical random matrix ensembles such as the Wigner and the Wishart ensembles have been characterized using a number of different approaches e.g., the moment method and the method of Stieljes transforms. However, these approaches have some shortcomings---first, they are typically tractable only for very symmetric random matrix models, and second, these approaches do not shed any light on the eigenvectors. Consequently, understanding the eigenvectors often requires significant additional work. 

The theory of deterministic equivalents was inspired by applications in signal processing and wireless communications \cite{hachem2007deterministic,couillet2011deterministic}, but its origins can be traced to the early works of \cite{girko2012theory}. Intuitively, given a random matrix, this non-asymptotic theory identifies a deterministic surrogate which has the same eigenvalue and eigenvector properties. Crucially, this yields rich spectral information about the random matrix of interest at finite problem sizes, without the restriction that these properties converge in the limit. We use the following formal definition of deterministic equivalents in this paper \cite{phdthesis}.

\begin{defn}[Deterministic Equivalent]
We say that $\overline{{Q}} \in \mathbb{R}^{n \times n}$ is a deterministic equivalent for the symmetric random matrix $Q \in \mathbb{R}^{n \times n}$ if, for sequences of deterministic matrix ${A} \in \mathbb{R}^{n \times n}$ and vectors $a, b \in \mathbb{R}^n$ of unit norms (operator and Euclidean, respectively), we have, as $n \rightarrow \infty$, $$\frac{1}{n} \operatorname{tr} {A}({Q}-\overline{{Q}}) \rightarrow 0, \quad {a}^{\top}({Q}-\overline{{Q}}) {b} \rightarrow 0,$$ where the convergence is either in probability or almost sure.
\end{defn}

\noindent 
We refer the interested reader to the recent book \cite{rmtnew} for a survey of the history of deterministic equivalents in random matrix theory, and several applications. We now discuss the relevance of this notion in our analysis. 

Recall that we use OLS to fit the outcome regression parameters. For concreteness, suppose we use the second split $S_2$ to fit the outcome regression. Using properties of OLS, we note that the covariance matrix of $(\hat{\alpha}^{(1)}_{S_2} ,\hat{\beta}^{(1)}_{S_2})$ is $(\sum_{i\in S_2} A_i \tilde{x}_i \tilde{x}_i^{\top})^{-1}$, where $\tilde{x}_i^{\top} = (1, x_i^{\top})$ denotes the vector $x_i$ padded with an additional entry $1$ for the intercept. Similarly, the covariance matrix of $(\hat{\alpha}^{(0)}_{S_2} ,\hat{\beta}^{(0)}_{S_2})$ is $(\sum_{i\in S_2} (1-A_i) \tilde{x}_i \tilde{x}_i^{\top})^{-1}$. 

The variable $T_1$ in \eqref{eq:decomp} involves terms of the form 
\begin{align}
    \ell^{\top} \left( 
    \begin{matrix}
    \alpha^{(1)} - \hat{\alpha}^{(1)}_{S_2} \\
    \beta^{(1)} - \hat{\beta}^{(1)}_{S_2}
    \end{matrix} 
    \right), \,\,\,\,  
    \ell^{\top} \left( 
    \begin{matrix}
    \alpha^{(0)} - \hat{\alpha}^{(0)}_{S_2} \\
    \beta^{(0)} - \hat{\beta}^{(0)}_{S_2}
    \end{matrix} 
    \right)
    \nonumber 
\end{align}
where $\ell$ is a random vector independent of the samples in $S_2$, and a function of the covariates $x_i$'s and the exposure $A_i$'s. Thus the conditional variance $\mathrm{Var}(T_1|A,X)$ involves quadratic forms $\ell^{\top} \Big( \sum_{i\in S_2} A_i \tilde{x}_i \tilde{x}_i^{\top} \Big)^{-1} \ell$ and 
$\ell^{\top} \Big( \sum_{i\in S_2} (1-A_i) \tilde{x}_i \tilde{x}_i^{\top} \Big)^{-1} \ell$. To determine the limit of the conditional variance, it suffices to establish that these quadratic forms converge to deterministic limits as $n,p \to \infty$. To this end, we derive a deterministic equivalent of the covariance matrices---this allows us to replace the quadratic forms $\ell^{\top}\Big( \sum_{i\in S_2} A_i \tilde{x}_i \tilde{x}_i^{\top} \Big)^{-1} \ell$, $\ell^{\top}\Big( \sum_{i\in S_2} (1-A_i) \tilde{x}_i \tilde{x}_i^{\top} \Big)^{-1} \ell$ by quadratic forms with deterministic interaction matrices. This is crucial for our subsequent analysis, and aids us in deriving the limits of these quadratic forms.

\noindent 
{\bf Leave one out:} ``Leave one out" style arguments have been critical in random matrix theory \cite{bai2010spectral}, as well as in high-dimensional statistics \cite{bean2013optimal,el2013robust,el2018impact}. This technique is also related to the cavity method from statistical physics \cite{montanari2022short,mezard2009information}. In random matrix theory, this technique is  ubiquitous, and is used for example in the proof of the limiting spectral distribution of a sample covariance matrix by the Stieljes transform method \cite{bai2010spectral}. This idea has also been critical in establishing asymptotic distribution of classical estimators/test statistics in  high-dimensional inference problems in the proportional asymptotics regime. To the best of our knowledge, this idea was first employed in high-dimensional statistics in the works of El Karoui and collaborators to analyze M-estimators in linear models \cite{bean2013optimal,el2013robust,el2018impact}. Subsequently, it has been crucial for analyzing the MLE, LRT in logistic regression \cite{sur2019modern,sur2019likelihood}, as well as diverse optimization problems \cite{ma2018implicit,chen2021spectral}.  Finally, this technique has been recently used to prove universality of high-dimensional estimation problems to the distribution of the feature vectors \cite{hu2020universality}. 

The nature of the technique as employed in the random matrix literature versus the high-dimensional statistics literature has subtle differences.
Our analysis crucially employs both styles of leave-one-out arguments described above. To highlight the utility of this technique for our proofs, we sketch two intermediate arguments that utilize this idea. 

First, we present the leave-one-out idea applied in the context of random matrices. Lemma \ref{lemma 6} establishes that 
\begin{align}\label{eq:zeta}
    \zeta:=\frac{1}{n} \Big( \sum_{i=1}^{n} x_i \Big)^{\top} \Big( \sum_{i=1}^{n} x_i x_i^{\top} \Big)^{-1} \Big( \sum_{i=1}^{n} x_i \Big) \stackrel{p}{\to} \kappa, 
\end{align}
where $x_i$ are iid random vectors in $\mathbb{R}^{p}$ with iid $\mathcal{N}(0,1/n)$ entries. Setting $T_1 = \Big( \sum_{i=1}^{n} x_i \Big) \Big( \sum_{i=1}^{n} x_i \Big)^{\top}$ and $T_2 = T_1 - \sum_{i=1}^{n} x_i x_i^{\top}$, we have 
\begin{align*}
   \zeta = \frac{1}{n} \mathrm{Tr}\Big[ T_1(T_1 - T_2)^{-1} \Big] = \frac{p}{n} + \frac{1}{n} \mathrm{Tr} [T_2 (T_1 - T_2)^{-1}],  
\end{align*}
where the last step follows from the identity $(A+B)^{-1} = A^{-1} - A^{-1} B (A+B)^{-1}$ for square matrices $A,B$. Thus it suffices to establish that 
\begin{align*}
    \frac{1}{n} \mathrm{Tr}\Big[T_2 \Big(\sum_{i=1}^{n} x_i x_i^{\top} \Big)^{-1} \Big] \stackrel{p}{\to} 0. \nonumber 
\end{align*}
This analysis is involved as both $T_2$ and $\sum_{i=1}^{n} x_i x_i^{\top}$ depend on all the $x_i$ vectors. A natural strategy at this point is to isolate out the dependence of this expression on the individual $x_i$'s. Applying the Sherman-Morrison identities, one obtains that 
\begin{align*}
    \frac{1}{n} \mathrm{Tr} \Big[T_2 \Big( \sum_{i=1}^{n} x_i x_i^{\top} \Big)^{-1} \Big] = \frac{1}{n} \sum_{i=1}^{n} \frac{\Big(x^{\{-i\}}\Big)^{\top} \Big(\sum_{j \neq i} x_j x_j^{\top}  \Big)^{-1} x_i  }{1+ x_i^{\top} \Big( \sum_{j\neq i} x_j x_j^{\top}  \Big)^{-1}  x_i},  
\end{align*}
where $x^{\{-i\}} = \sum_{j \neq i} x_j$. This representation isolates out $x_i$ from the other vectors---the resulting sum is easy to track by direct computation. Indeed, one completes the proof by directly  establishing that the sum above has mean zero and variance converging to zero. This illustrates one instance of the leave-one-out idea in the context of random matrices, as utilized in our proof. We refer the interested reader to the proof of Lemma \ref{lemma 6} for additional details.
While the above application of the leave-one-out is straightforward to the experts
(and the result can be established without this technique for Gaussian covariates),
we chose this example to provide  a simple illustration of the technique in action. Our proofs invoke this technique in a large number of steps and often for  expressions that are far more complicated than \eqref{eq:zeta}. However, the underlying basic principle mostly remains similar to the above.


In addition to the abovementioned application of the leave-one-out, we utilize the technique crucially to track the asymptotic dependence  between estimators used in cross-fitting. We emphasize that this is a major challenge in our proof; in comparison, this dependence is absent in the analysis of the AIPW estimator without cross-fitting, and the associated CLT proof would be significantly simpler. To explain the issue at a high-level, note that the cross-fitted AIPW includes a term where the first split is used to estimate the propensity score model, while the final plug-in is performed on the third split. Simultaneously, it includes a term where the roles of the first and third splits are flipped  (Note that there is nothing special about these two terms---the same issue arises for many pairs of terms obtained from the sample splits.). Naturally, when we compute the variance of the cross-fitted estimator, we have to control all of the cross-covariances among these terms. This covariance is implicit, as the MLE $\hat{\beta}$ is a complicated function of the individual sample points. To compute this limiting covariance, our strategy is to replace the logistic MLE $\hat{\beta}$ by a surrogate $\hat{\beta}^{\{-i\}}$---the MLE on the sample with the $i^{th}$ datapoint left-out. The surrogate is independent of the $i^{th}$ datapoint by construction, and is critical for calculating the covariance. Crucially, one cannot replace the MLE with its surrogate $\hat{\beta}^{\{-i\}}$ without paying a price---the fitted values $x_i^{\top}\hat{\beta}$ and $x_i^{\top}\hat{\beta}^{\{-i\}}$ are different, and this difference shows up in our limiting covariance calculation. This difference has been precisely characterized in \cite{sur2019modern}, and is a crucial We begin our proof by replacing the MLE with the leave-one-out surrogate. However, tracking  the downstream effects of this replacement is highly non-trivial and can be viewed as one of our major technical contributions. Putting these ingredients together yields a fairly explicit expression for the limiting covariances and uncovers the negative cross-covariance phenomenon described in the Introduction.

\section{Discussions and Open Questions}
\label{sec:discussions} 
We discuss follow up questions arising from our results, and collect initial thoughts regarding their resolution. 

\begin{itemize}

    \item[(i)] The effect of winsorizing---It is well known that the finite sample performance of the AIPW might suffer due to the inverse probability weighing involved in its evaluation. To mitigate this issue, practitioners routinely use a winsorized version of the estimator.  Formally, this corresponds to replacing $\sigma(x_i^{\top} \hat{\beta}_{S_a})$ (respectively $1-\sigma(x_i^{\top} \hat{\beta}_{S_a})$ by $\max\{\sigma(x_i^{\top} \hat{\beta}_{S_a}), \varepsilon\}$ (respectively $\max\{1-\sigma(x_i^{\top} \hat{\beta}_{S_a}), \varepsilon \}$) for some small $\varepsilon>0$. We see this finite sample effect also in our simulations (see Figure \ref{normal_qq_plot_original_vs_trimmed}). The winsorizing regularizes the estimator, and removes the outliers in the q-q plot. We believe it should be possible to track the sampling distribution of the winsorized estimator using the tools introduced in this paper. For small $\varepsilon>0$, the limiting distributions of the original estimator and the winsorized one are approximately the same. We thus do not pursue a formal theoretical treatment of the winsorized estimator in this paper.

    \item[(ii)] Constructing confidence intervals for the ATE using our CLT---Given our main result, one immediately wonders if it can yield confidence intervals for the ATE. Of course, this will require a consistent estimate of the sampling variance $\sigma_{cf}^2$. The expression for $\sigma_{cf}^2$ is quite involved, so it is a priori unclear if this is possible. However, on closer inspection we notice that the limiting variance is a function of the signal-to-noise ratio type parameters \eqref{eq:norm_limits}, and consistent estimation of such  quantities are known to be feasible in the proportional asymptotics regime. This has been demonstrated in a variety of prior works \cite{bayati2013estimating,dicker2016maximum,sur2019modern,yadlowsky2021sloe,bellec2022observable,janson2017eigenprism}. A combination of these techniques should yield a consistent variance estimator in our setting. We will explore this direction in-depth in future work. As an aside, we note that  traditional re-sampling approaches such as the bootstrap are known to be inconsistent in simpler statistical problems under the proportional asymptotics regime \cite{el2018can}---we expect similar phenomena to hold in our setup.

    \item[(iii)] Beyond the  assumptions on the covariates---Our result assumes that the covariates $x_i \in \mathbb{R}^p$ are iid gaussian. We believe that the gaussianity is not critical for the validity of this result---indeed, we expect our results to be valid in settings where the entries of $x_i$ are iid,
    as long as these have well-behaved tail properties (e.g. sub-gaussian tails). Several results of this type have by now been established in the proportional asymptotics regime \cite{bayati2015universality,abbasi2019universality,liang2020precise,hu2020universality,montanari2022universality}. Furthermore, our experiments in Figures \ref{robust_uniform} and \ref{robust_discrete} indicate the presence of such universality phenomenon in our setting. On the contrary, extending our results to allow for correlations among the features is less straightforward. 
    Such situations are more natural in practical applications, thus establishing analogues of our results in these settings is of intrinsic interest. We expect this direction to be feasible, at least for special covariance structures or in the case of gaussian correlated covariates, following arguments similar to \cite{liang2020precise,zhao2020asymptotic}. That said, we view this paper  as a stepping stone for analyzing other causal effect estimators in the proportional asymptotics regime. Our proofs are significantly involved even under the stylized covariate distribution assumed herein. In this light, we defer generalizations of this condition to future works. 
    
    \item[(iv)] More general nuisance estimators--- 
In this paper, we focus on simple nuisance estimators such as the MLE or ridge regression. In high dimensions, one typically wishes to employ more sophisticated estimators for the nuisance parameters e.g., those arising from modern machine learning. The performance of the AIPW  with such advanced nuisance estimators has been analyzed in the recent literature \cite{chernozhukov2017double,smucler2019unifying}. To the best of our knowledge, all existing analyses of this flavor assume sparsity of either the propensity score or the outcome regression model, and thus are not directly applicable to our setting. It would be interesting to explore the effect of using powerful Machine Learning based nuisance estimators in our setting. We leave this for future work. 
    
    \item[(v)] Alternative sample splitting schemes---We employ a three sample split strategy in this paper---the two nuisance functions are estimated from distinct sample folds, while the final estimator is constructed based on the third fold. This is certainly not the only possible choice for this problem; in our case, this three sample split strategy is a conscious choice, 
    since this aids our theoretical analysis. However, one might naturally wish to use other splitting strategies e.g., the samples could be split into two parts, the two nuisances being estimated from the first split, and the final estimator being evaluated on the second split. The most extreme example would be to use the whole data to estimate both the nuisances, and the subsequent computation of the AIPW. Analyzing these estimators are significantly more challenging, due to the subtle dependencies among the intermediate estimators. It is apriori unclear which of these sample splitting schemes yields the estimator with the best empirical performance. We believe that extending our results to settings with fewer splits will require new technical ideas, and is an interesting direction for follow-up research. 
    \item[(vi)] The problem of optimal estimation---Our work raises the following natural question: is some version of the AIPW (with cross-fitting) optimal in terms of the asymptotic variance in this setting?
     Note that in the classical low-dimensional setting ($n \to \infty$ and $p$ fixed), the AIPW is semi-parametric efficient \citep{tsiatis2006semiparametric,bickel1993efficient,van2000asymptotic,le2000asymptotics} in a nonparametric model that does not restrict the distribution of the tuple $(y,A,x)$. Our analysis assumes a specific covariate distribution---this assumption  allows for more efficient ATE estimation in the classical regime \citep{kallus2020role}.  One might naturally wonder if this improved ATE estimator might beat the cross-fitted AIPW estimator, and continue to be optimal in our setting. We emphasize that these existing comparisons do not translate directly to our proportional asymptotics regime---in fact, pinning down ``efficient" estimators in our context remains an outstanding question. 
    We defer this direction to future work, and adopt the following perspective here.  The AIPW is one of the most widely used ATE estimators in practice---can its fluctuations be characterized via the classical asymptotic variance when we are neither in a classical setting, nor  in the ultra-high-dimensional regime with sparsity? Our central limit theorems provide an answer in the negative and develop alternate approximations that can be used for inferring the ATE in a large class of problems.

\end{itemize}

    \begin{acks}[Acknowledgments]
PS acknowledges support from NSF DMS-2113426 and SS acknowledges support from a Harvard Dean's Competitive Fund Award. PS would  like to thank Andrea Rotnitzky for helpful discussions on an earlier version of this manuscript.
\end{acks}

\bibliographystyle{imsart-number} 
\bibliography{biblio_ate}       

\newpage
\begin{appendix}
\section{Supplementary material} 
\label{sec:proof_of_main_results} 
We establish Theorem \ref{thm:ols_existence} and Theorem \ref{thm:dr_distribution} in this supplement. We start with some notations. 

\subsection{Notations}
  \label{subsec:notations}
  \begin{enumerate}
      \item For any vector $v$, $\| v \|$ denotes its $l_2$-norm. For any matrix $A$, $\| A \|$ denotes its spectral norm, $\|A\|_F$ denotes its Frobenius norm, and $\lambda_{\min}(A)$ denotes its eigenvalue with the smallest magnitude.
      \item Fix any $p, n \in \mathbb{N}^{+}$. For any matrix $X \in \mathbb{R}^{n \times p}$, $x_i$ denotes the $i$-th row of $X$, where $i = 1,2, \ldots, n$.
      \item For any random variable $X$, $p_X$ denotes its probability density function, if exists.
      \item For $i=1,2,3$, define $ \kappa_i = \lim_{n \rightarrow \infty} \frac{p}{n_i} = \frac{\kappa}{r_i}.$
      \item $\sigma(\cdot)$ denotes the sigmoid function, where $\sigma(x) = \frac{1}{1+e^{-x}}.$ Define $\rho(x)=\log \left(1+e^{x}\right) \;\; \forall x \in \mathbb{R}$.
      \item For any $i =1,2,3$, let $(\alpha_i^*, \sigma_i^*, \lambda_i^*)  $ be the solution to the system of equations (3.5) in \cite{Sur14516}, where the covariates are of dimension $ {n_i \times p}$ with i.i.d. entries $\sim \mathcal{N}(0, \frac{1}{n_i})$, and the signal strength is $\gamma^2$.  
      \item For any $i=1,2,3$, let $X_{S_i} \in \mathbb{R}^{n_i \times p}$ and $\mathcal{E}_{S_{i}} \in \mathbb{R}^{n_i}$ denote the design matrix and Gaussian noises of the samples in the $i$-th split, respectively. Moreover, let $ \tilde{X}_{S_i} \in \mathbb{R}^{n_i \times (p+1)}$ denote the design matrix with an extra column of ones (i.e. the first column of $\tilde{X}_{S_i}$ is a vector of ones, and the remaining columns of $\tilde{X}_{S_i}$ is the same as $X_{S_i}$).
      \item For any $i=1,2,3, \; j = 1,2,$ define $S_{i,j} = \{k \in S_i \mid A_k = j\}$. Also, let $X_{S_{i}, j} $ and $ \mathcal{E}_{S_{i}, j}$ denote the design matrix of the samples in the $i$-th split whose treatment receipt indicator turn out to be $j$, respectively. Moreover, let $ \tilde{X}_{S_i,j} $ be the corresponding design matrix with an extra column of ones.
      \item For any $i=1,2,3$, let $\hat{\beta}_{S_i}$ denote the estimated logistic regression parameter $\beta$ using data $(X_{S_i}, A_{S_i})$. Let $\hat{\alpha}^{(0)}_{S_i}, \hat{\alpha}^{(1)}_{S_i}, \hat{\beta}^{(0)}_{S_i}, \hat{\beta}^{(1)}_{S_i}$ denote the estimated linear regression parameters $\alpha^{(0)}, \alpha^{(1)}, {\beta}^{(0)}, {\beta}^{(1)}$ using data $(X_{S_i}, Y_{S_i})$, respectively.
      \item For any $\gamma > 0, C \in \mathbb{R}$, define constants $$e_{\gamma, C} =\mathbb{E}\left[\frac{z}{1+e^{-\gamma z}}\right], \quad  q_{\gamma, C} = \mathbb{E}\left[\frac{1}{1+e^{-\gamma z}}\right] ,  \quad \text{ where } z \sim \mathcal{N}(C,1),  $$ and for any $i=1,2,3,$ define $$h_i = \mathbb{E}\left[ \frac{Z_\beta(1+e^{-Z_{ \hat{\beta}_{S_i}}})  }{1+e^{-Z_{\beta}}} \right], $$ where \begin{align*}
    & \left(\begin{array}{c}
Z_\beta \\
Z_{\hat{\beta}_{S_{1}}} \\
Z_{\hat{\beta}_{S_{2}}} \\
Z_{\hat{\beta}_{S_{3}}}
\end{array}\right) \sim N\left(\mathbf{0}, \left[\begin{array}{cccc}
\gamma^{2} & \alpha_{1}^{*} \gamma^{2}  & \alpha_{2}^{*} \gamma^{2}  & \alpha_{3}^{*} \gamma^{2} \\
\alpha_{1}^{*} \gamma^{2} & \kappa_{1}\left(\sigma_{1}^{*}\right)^{2}+\left(\alpha_{1}^{*}\right)^{2} \gamma^{2} & \alpha_{1}^{*} \alpha_{2}^{*} \gamma^{2} & \alpha_{1}^{*} \alpha_{3}^{*} \gamma^{2}\\
\alpha_{2}^{*} \gamma^{2} & \alpha_{1}^{*} \alpha_{2}^{*} \gamma^{2} & \kappa_{2}\left(\sigma_{2}^{*}\right)^{2}+\left(\alpha_{2}^{*}\right)^{2} \gamma^{2} & \alpha_{2}^{*} \alpha_{3}^{*} \gamma^{2}\\
\alpha_{3}^{*} \gamma^{2} & \alpha_{1}^{*} \alpha_{3}^{*} \gamma^{2} & \alpha_{2}^{*} \alpha_{3}^{*} \gamma^{2} & \kappa_{3}\left(\sigma_{3}^{*}\right)^{2}+\left(\alpha_{3}^{*}\right)^{2} \gamma^{2}
\end{array}\right] \right).
\end{align*}
      
      \item  We say that $\overline{{Q}} \in \mathbb{R}^{n \times n}$ is a deterministic equivalent for the symmetric random matrix $Q \in \mathbb{R}^{n \times n}$ if, for sequences of deterministic matrix ${A} \in \mathbb{R}^{n \times n}$ and vectors $a, b \in \mathbb{R}^n$ of unit norms (operator and Euclidean, respectively), we have, as $n \rightarrow \infty$, $$\frac{1}{n} \operatorname{tr} {A}({Q}-\overline{{Q}}) \rightarrow 0, \quad {a}^{\top}({Q}-\overline{{Q}}) {b} \rightarrow 0, $$ where the convergence is either in probability or almost sure.
      
      \item For any convex function $f(\cdot)$,  the proximal mapping operator is defined as $$ \operatorname{prox}_{f}(z):=\arg \min _{t \in \mathbb{R}}\left\{f(t)+\frac{1}{2}(t-z)^{2}\right\}.$$
     
  \end{enumerate}
  
  \subsection{Proof of Theorem \ref{thm:ols_existence}} 
  We outline the proof of Theorem \ref{thm:ols_existence} in this section. 
%
Define $z_i^{\mathrm{T}} = ( 1, x_i^{\mathrm{T}})$. Note that the OLS estimates $\hat{\alpha}^{(1)}$, $\hat{\beta}^{(1)}$ are unique with high probability if the matrix $M=\sum_{i=1}^{n} A_i z_i z_i^{\mathrm{T}}$ is invertible. 

Observe that the $\{A_i : 1 \leq i \leq n\}$ are iid random variables, with $\mathbb{P}[A_i = 1] = \sigma (x_i^{\mathrm{T}} \beta)$. First, observe that an equivalent construction of these Bernoulli variables can be accomplished as follows: let $U_1, \cdots, U_n \sim U([0,1])$, and define $\xi_i = \sigma^{-1}(U_i)$. We have, 
\begin{align}
\{A_i : 1\leq i \leq n\} \stackrel{d}{=} \{\mathbf{1}(x_i^{\mathrm{T}}\beta > \xi_i): 1\leq i \leq n\}. \nonumber  
\end{align}

Note that given $V:=\sum_i A_i$, 
\begin{align}
M = 
\left( \begin{matrix}
V & \sum_{i=1}^{V} w_i^{\mathrm{T}} \\
\sum_{i=1}^{V} w_i & \sum_{i=1}^{V} w_i w_i^{\mathrm{T}}
\end{matrix} \right),\nonumber 
\end{align}
where $w_i$ is distributed as $x_i | A_i =1$. Formally, for any measurable subset $O \subset \mathbb{R}^{p}$, 
\begin{align}
\mathbb{P}[w_i \in O ] = \frac{\mathbb{P}[x_i \in O, A_i=1 ]}{\mathbb{P}[A_i=1]} = \frac{\mathbb{E}_{\xi_i}[\mathbb{P}[x_i \in O, x_i^{\mathrm{T}} \beta > \xi_i]] }{\mathbb{E}_{\xi_i}[\mathbb{P}[x_i^{\mathrm{T}} \beta > \xi_i] ]} . \nonumber 
\end{align}
This immediately implies that conditional on $V$, $w_i$ are iid, and their distribution is absolutely continuous with respect to Lebesgue measure. \cite{eaton1973non} implies that if $V \geq p$, then $\sum_{i=1}^{V} w_i w_i^{\mathrm{T}}$ is invertible almost surely. Further,
\begin{align}
V -  \Big( \sum_{i=1}^{V} w_i \Big)^{\mathrm{T}} \Big( \sum_{i=1}^{V} w_i w_i^{\mathrm{T}} \Big)^{-1} \Big( \sum_{i=1}^{V} w_i \Big)  \neq 0 \nonumber 
\end{align}
with probability 1, as the concerned random variables have a density with respect to Lebesgue measure. This establishes that $M$ is invertible almost surely if $V\geq p$. Marginally, $V \sim \mathrm{Bin}(n, \mathbb{E}[\sigma(x_i^{\mathrm{T}}\beta)])$. Thus if $\mathrm{E}[\sigma(\frac{\|\beta\|_2}{\sqrt{n}} Z)] > (1+\varepsilon) \kappa$, then $V > p$ with high probability. In this case, $M$ will be invertible with high probability. Direct calculation reveals that $\mathbb{E}[\sigma(\frac{\|\beta\|_2}{\sqrt{n}} Z)] = \frac{1}{2}$. An analogous computation applies for the OLS estimates $(\hat{\alpha}^{(0)}, \hat{\beta}^{(0)})$, as 
\begin{align}
    1- \mathbb{E}[\sigma(x_i^{\mathrm{T}} \beta)] = \mathbb{E}[\sigma(-x_i^{\mathrm{T}} \beta)] = \mathbb{E}[\sigma(x_i^{\mathrm{T}} \beta)]. \nonumber 
\end{align}

Conversely, if $\mathbb{E}[\sigma(\frac{\|\beta\|_2}{\sqrt{n}} Z)] < (1- \varepsilon) \kappa$, then $V< p-1$ with high probability, and then 
\begin{align}
\sum_{i=1}^{V} w_i w_i^{\mathrm{T}} - \frac{1}{n} \Big( \sum_{i=1}^{V} w_i \Big) \Big( \sum_{i=1}^{V} w_i \Big)^{\mathrm{T}} \nonumber  
\end{align}
is singular with high probability. This completes the proof.

  \subsection{Proof of Theorem \ref{thm:dr_distribution} and comparison with classical results}\label{subsec:fullformula}
  We start with the following lemma. 
  
  \begin{lemma}
  \label{lemma:representation}
  For $j,k \in \{1,2,3\}$, define 
  \begin{align}
      g(X_{S_j}) = \frac{1}{\sqrt{n_j}} \sum_{i \in S_{j}} x_i^\top\beta^{(1)}, &\,\,\,\, \tilde{g}\left(X_{S_{j}}\right)=\frac{1}{\sqrt{n_{j}}} \sum_{i \in S_{j}} x_{i}^\top \beta^{(0)}, \nonumber \\
      l_{S_j,S_k} =\frac{1}{\sqrt{n_j}} \begin{pmatrix} 
 \sum_{i \in S_j} \left(\frac{A_i}{\sigma(x_i^{\top} \hat{\beta}_{S_k})}-1 \right)\\
 \sum_{i \in S_j} \left(\frac{A_{i}x_i}{\sigma(x_{i}^{\top} \hat{\beta}_{S_k})}-x_i \right)
\end{pmatrix}, &\,\,\,\, \tilde{l}_{S_{j}, S_{k}}=\frac{1}{\sqrt{n_{j}}}\left(\begin{array}{c}
\sum_{i \in S_{j}}\left(\frac{1-A_{i}}{1-\sigma\left(x_{i}^{\top} \hat{\beta}_{S_k}\right)}-1\right) \\
\sum_{i \in S_{j}}\left(\frac{(1-A_{i}) x_{i}}{1-\sigma\left(x_{i}^{\top} \hat{\beta}_{S_k}\right)}-x_{i}\right)
\end{array}\right), \nonumber \\
f(\mathcal{E}_{S_j},X_{S_j}) = \begin{pmatrix}
 \alpha^{(1)}-\hat{\alpha}^{(1)}_{S_j} \\
\beta^{(1)}-\hat{\beta}^{(1)}_{S_j}
\end{pmatrix}, &\,\,\,\, \tilde{f}\left(\mathcal{E}_{S_{j}}, X_{S_{j}}\right)=\left(\begin{array}{c}
\alpha^{(0)}-\hat{\alpha}^{(0)}_{S_{j}} \\
\beta^{(0)}-\hat{\beta}^{(0)}_{S_{j}}\end{array}\right).\nonumber 
  \end{align}
  Furthermore, let $V_{S_j,S_k},\tilde{ V}_{S_{j}, S_{k}} \in \mathbb{R}^{|S_j|}$ denote the vectors  containing $ \frac{1}{\sqrt{n_j}} \frac{A_i}{ \sigma(x_i^{\top} \hat{\beta}_{S_k}) }$ and $ \frac{1}{\sqrt{n_{j}}} \frac{1-A_{i}}{1-\sigma\left(x_{i}^{\top} \hat{\beta}_{S_k}\right)}$ for all $i \in S_j$ respectively.
  Then we have, 
  \begin{align}
      &\sqrt{n}(\hat{\Delta}_{cf} - \Delta) = T_1 + T_2, \nonumber\\
      &T_1 = \frac{1}{6  }  \sum_{(a,b,c) \in \mathscr{S}_3} \frac{1}{\sqrt{r_c}} \left[ l_{S_c,S_a}^{\top} f(\mathcal{E}_{S_b},X_{S_b}) - \tilde{l}_{S_c,S_a}^{\top} \tilde{f}(\mathcal{E}_{S_b},X_{S_b}) + \mathcal{E}_{S_c}^{\top}V_{S_c,S_a}
      - \mathcal{E}_{S_c}^{\top}\tilde{V}_{S_c,S_a} \right]. \label{cf_estimator} \\
      &T_2 = \frac{1}{6 }  \sum_{(a,b,c) \in \mathscr{S}_3} \frac{1}{\sqrt{r_c}}[g(X_{S_c}) - \tilde{g}(X_{S_c})], \nonumber
  \end{align}
  \end{lemma}
  
  Our next result characterizes the joint distribution of $(T_1, T_2)$. Before we demonstrate the results, we remind the readers of the definitions of $e_{\gamma,C}, q_{\gamma,C}$ and $\alpha_i^*, \sigma_i^*, \lambda_i^*,h_i, \,\, i=1,2,3 $ from Points 6 and 10 of Subsection \ref{subsec:notations}.
In addition, recall that in this Point 10, we had introduced $(Z_{\beta}, Z_{\hat{\beta}_{S_{1}}}, Z_{\hat{\beta}_{S_{2}}},
Z_{\hat{\beta}_{S_{3}}})$ to be a multivariate normal vector. In this subsection, we will use these random variables repeatedly. Thus, to keep the notation concise, we abbreviate these as follows
$$ Z_0 = Z_\beta, \,\, Z_i =  Z_{\hat{\beta}_{S_{i}}},  \,\, i=1,2,3.$$
We will use this abbreviation only for this subsection.
Furthermore, we define the following quantities for $i,j = 1,2,3$

$$  s_{i} =  \mathbb{E}\left[\frac{\sigma\left(Z_{\beta}\right)}{\sigma\left(Z_{\hat{\beta}_{S_{i}}}\right)}\right], t_i = \left(\frac{r_{i}}{2}-\kappa\right)\left(1-4 e_{\gamma, 0}^{2}\right), $$ 
$$ f_{i,j} = \frac{2 r_{i} h_{j} e_{\gamma, 0}}{\gamma}-4 \kappa e_{\gamma, 0}\left(e_{\gamma, 0}+e^{\frac{\left(\alpha_{j}^{*} \gamma\right)^{2}+\kappa_{j}\left(\sigma_{j}^{*}\right)^{2}}{2}} e_{\gamma,-\alpha_{j}^{*} \gamma}\right)+2 \kappa\left(\frac{1}{2}+e^{\frac{\left(\alpha_{j}^{*} \gamma\right)^{2}+\kappa_{j}\left(\sigma_{j}^{*}\right)^{2}}{2}} q_{\gamma,-\alpha_{j}^{*} \gamma}\right)\, $$ 
$$ g_{i,j} = \mathbb{E}\left[\sigma\left(Z_{0}\right)\left(\frac{1}{\sigma\left(Z_{i}\right)}-1\right)\left(1-\sigma\left(\operatorname{prox}_{\lambda_{j}^{*} \rho}\left(Z_{j}+\lambda_{j}^{*}\right)\right)\right)+\left(1-\sigma\left(Z_{0}\right)\right) \sigma\left(\operatorname{prox}_{\lambda_{j}^{*} \rho}\left(Z_{j}\right)\right)\right]. $$

\begin{lemma}
\label{lemma:variance_stabilization}
We have, as $n\to \infty$, 
\begin{align}
    \left( \begin{matrix}
    T_1 \\
    T_2 
    \end{matrix} \right) \stackrel{d}{\to} 
    \mathcal{N} \left( \left( \begin{matrix} 0 \\ 0 \end{matrix} \right), \left( \begin{matrix} V_{T_1} & 0 \\ 0 & V_{T_2} \end{matrix} \right) \right), \nonumber 
\end{align}
where 
\begin{align}\label{eq:VT1}
    V_{T_1} &:=  \frac{\left(\sigma^{(0)}\right)^2+\left(\sigma^{(1)}\right)^2}{36} \big (V_{\text{var}}+  V_{\text{within}} + V_{\text{between}} \big ),\\
    V_{T_2} &:= \frac{\kappa}{9} \left( \sigma_{0\beta}^2 +  \sigma_{1\beta}^2 - 2 \rho_{01} \sigma_{0 \beta} \sigma_{1 \beta}  \right) \Big( \frac{1}{r_{1}} + \frac{1}{r_2} + \frac{1}{r_3} \Big), \nonumber\\
    V_{\text{var}} &: = 
     \sum_{(a,b,c) \in \mathscr{S}_3} \left \{ \frac{2 \kappa\left(1-2 s_a\right)}{r_c\left(r_b-2 \kappa\right)}+\frac{r_b}{r_c\left(r_b-2 \kappa\right)} \mathbb{E}\left[\frac{\sigma\left(Z_0\right)}{\sigma^2\left(Z_a\right)}\right] + \frac{4 r_c e_{\gamma, 0}^2 h_a^2}{\gamma^2 r_b t_b} -\frac{4 e_{\gamma, 0} h_a}{\gamma t_b}\left(s_a-1\right) \right. \nonumber  \\
    &+\frac{2 \gamma^2}{r_b-2 \kappa} \mathbb{E}^2\left[\left(1-\alpha_a^*\right) \frac{\sigma\left(Z_0\right)}{\sigma\left(Z_a\right)}-\frac{\sigma^2\left(Z_0\right)}{\sigma\left(Z_a\right)}+\frac{\alpha_a^*}{2}\right]  +\frac{2 \kappa_a\left(\sigma_a^*\right)^2}{\left(r_b-2 \kappa\right)}\left(s_a-\frac{1}{2}\right)^2+\frac{\left(s_a-1\right)^2}{t_b} , \nonumber\\
    V_{\text{within}} &:= 
     \sum_{(a,b,c) \in \mathscr{S}_3} \frac{1}{r_a} \mathbb{E}\left[\frac{\sigma\left(Z_0\right)}{\sigma\left(Z_b\right) \sigma\left(Z_c\right)}\right],  \nonumber\\
    V_{\text{between}} &: =
     \sum_{(a,b,c) \in \mathscr{S}_3}
     \frac{4\left(s_a-0.5\right)}{r_b}   \mathbb{E}\left[\frac{\sigma\left(Z_0\right) Z_a}{\sigma\left(Z_a\right)}\right] 
      - \frac{4}{r_b} \left(\mathbb{E}\left[\frac{\sigma^{\prime}\left(Z_0\right)}{\sigma\left(Z_a\right)}\right] + \mathbb{E}\left[\frac{\sigma^{\prime}\left(Z_0\right)}{\sigma\left(Z_c\right)}\right] \right) h_a \nonumber
    \\
    & + \frac{\left(s_c-4 s_a-1\right)\left(s_a-1\right)}{t_b}
    + \frac{4e_{\gamma, 0} h_a (2s_a-s_c+1)}{\gamma t_b}
    -\frac{2\left(f_{b, a}+f_{b, c}\right)\left[\frac{2 e_{\gamma, 0} h_a}{\gamma}-s_a+1\right]}{r_b t_b}  \nonumber\\
    & +\frac{4\left(s_c-0.5\right)}{r_b}\left(\mathbb{E}\left[\frac{\sigma\left(Z_0\right)Z_c}{\sigma\left(Z_a\right)} \right]+\lambda_c^* g_{a, c}\right)  + \frac{4 \sqrt{r_a r_c} h_a h_c e_{\gamma, 0}^2}{r_b \gamma^2 t_b} -\frac{4 \lambda_c^*}{r_b-2 \kappa}\left(s_c-0.5\right) g_{a, c} \nonumber\\
    & + \left. \frac{2 \gamma^2}{r_b-2 \kappa}\left(\mathbb{E}\left[\frac{\sigma^{\prime}\left(Z_0\right)}{\sigma\left(Z_a\right)}\right]-\alpha_a^*\left(s_a-0.5\right)\right)\left(\mathbb{E}\left[\frac{\sigma^{\prime}\left(Z_0\right)}{\sigma\left(Z_c\right)}\right]-\alpha_c^*\left(s_c-0.5\right)\right)\big] \right\}. \nonumber
\end{align}
\end{lemma}

 \begin{remark}
  As a reminder, in equation \eqref{eq:variance}, we split the asymptotic
variance of the cross-fitted estimator into three parts: (i) sum of variance of each pre-cross-fit estimator, (ii) sum of within-pair covariance, and (iii) sum of between-pair covariance. Here, $V_{\text{var}}$, $V_{\text{within}}$, and $V_{\text{between}}$ are the terms in $V_{T_1}$ that contribute to part (i), (ii), and (iii), respectively.
 \end{remark}
  
 We complete the proof of Theorem \ref{thm:dr_distribution}, given Lemma \ref{lemma:representation}  and  \ref{lemma:variance_stabilization}. 
 \begin{proof}[Proof of Theorem \ref{thm:dr_distribution}]
  Lemma \ref{lemma:representation} implies 
  \begin{align}
      \sqrt{n}(\hat{\Delta}_{cf} - \Delta) = T_1 + T_2. \nonumber
  \end{align}
  Lemma \ref{lemma:variance_stabilization} immediately implies that $\sqrt{n}(\hat{\Delta}_{cf} - \Delta)$ has an asymptotic mean-zero Gaussian limit, with variance $\sigma_{cf}^2 = V_{T_1} + V_{T_2}$. 
 \end{proof}

\subsubsection{Comparison with classical formula}\label{eq:formalcompare}
We here present a further detailed comparison of our high-dimensional formula with the classical variance formula for the AIPW (recall from \eqref{eq:classical}). Recall from \eqref{eq:clt} that we expressed our formula as
 \[\sigma^2_{\text{cf}}=\left[ \left(\sigma^{(0)}\right)^2 + \left(\sigma^{(1)}\right)^2 \right] f(\kappa,\gamma^2)+     \frac{\kappa}{9} \left( \sigma_{0\beta}^2 +  \sigma_{1\beta}^2 - 2 \rho_{01} \sigma_{0 \beta} \sigma_{1 \beta}  \right) \Big( \frac{1}{r_{1}} + \frac{1}{r_2} + \frac{1}{r_3} \Big),\]
so formally 
\begin{equation}\label{eq:fkgamma}
    f(\kappa,\gamma^2) = \frac{V_{T_1}}{\Big(\sigma^{(0)}\Big)^2 + \Big(\sigma^{(1)}\Big)^2 },
\end{equation}
where $V_{T_1}$ is as defined in \eqref{eq:VT1}. We will compare this with the classical variance when the noise variances $(\sigma^{(0)})^2=(\sigma^{(1)})^2=\sigma_{\varepsilon}$. Recall from  \eqref{eq:classical} that the classical variance formula is given by 
\[ \sigma^2_{\text{classical}} = 2\sigma^2_{\varepsilon}\mathbb{E}\Big[\frac{1}{\sigma(x_i^{\top}\beta)}\Big]+\text{Var}\{x_i^{\top}(\beta^{(1)}-\beta^{(0)})\}.  \]

In Section \ref{sec:main_results}, we argued that the difference among these lies in  their respective first terms. 
To compare these, we first recall that the first term in $\sigma^2_{\text{cf}}$ comes from the random variable $T_1$ defined in \eqref{cf_estimator}. Note that in a low-dimensional (fixed $p$, large $n$) setting, 
\begin{equation}
\sum_{(a,b,c) \in \mathscr{S}_3} \frac{1}{\sqrt{r_c}}  l_{S_c,S_a}^{\top} f(\mathcal{E}_{S_b},X_{S_b})  \xrightarrow{p} 0 \label{high_d_t1},
\end{equation}
and similarly for $\tilde{l}_{S_c,S_a}^{\top} \tilde{f}(\mathcal{E}_{S_b},X_{S_b}) $. So, the first set of terms in $T_1$ does not contribute in low dimensions. Now we turn to the terms $ \mathcal{E}_{S_c}^{\top}V_{S_c,S_a}
      - \mathcal{E}_{S_c}^{\top}\tilde{V}_{S_c,S_a} $ in \eqref{cf_estimator}. By definition, $\mathcal{E}_{S_c}^{\top}V_{S_c,S_a}$ and $\mathcal{E}_{S_c}^{\top}\tilde{V}_{S_c,S_a} $ are independent, which implies that the total contribution from these in  
      $V_{T_1}$ equals 
      
\begin{align}
    & \frac{1}{18} \operatorname{Var}\left(\sum_{{(a,b,c) \in \mathscr{S}_3}} \frac{1}{\sqrt{r_{c}}} \mathcal{E}_{S_{c}}^{\top} V_{S_{c}, S_{a}}\right)=  \text{Var}\left( \mathcal{E}_{S_{3}}^{\top} V_{S_{3}, S_{1}} \right) + \text{Cov}\left( \mathcal{E}_{S_{3}}^{\top} V_{S_{3}, S_{1}}, \mathcal{E}_{S_{3}}^{\top} V_{S_{3}, S_{2}} \right) \nonumber\\
    =&  \sigma_{\varepsilon}^{2} \mathbb{E}\left[\frac{\sigma\left(Z_{\beta}\right)}{\sigma^{2}\left(Z_{\hat{\beta}_{S_{1}}}\right)}\right] + \sigma_{\varepsilon}^{2} \mathbb{E}\left[\frac{\sigma\left(Z_{\beta}\right)}{\sigma\left(Z_{\hat{\beta}_{S_{1}}}\right) \sigma\left(Z_{\hat{\beta}_{S_{3}}}\right)}\right] + o_p(1) \quad \text{in our setting}, \label{high_d_t2} \\
    =& 2 \sigma_{\varepsilon}^{2} \mathbb{E}\left[\frac{1}{\sigma\left(x^{\top} \beta\right)}\right]+ o_p(1) \quad \text{in low dimensions}. \nonumber
    \end{align}
  Thus, our formula recovers the classical variance in low dimensions.
However, under our high-dimensional setting,  \eqref{high_d_t1} no longer holds, nor does the last step in (\ref{high_d_t2}), thus our regime differs significantly from its classical counterpart.

 \subsubsection{Proof of Lemma \ref{lemma:representation}}
We finally turn to a proof of Lemma \ref{lemma:representation}. 

\begin{proof}[Proof of Lemma \ref{lemma:representation}]
For any $(a,b,c)$ which is a permutation of $ (1,2,3)$, note that \begin{align*}
  &  \hat{\Delta}_{AIPW, 1} - \alpha_1 =  \frac{1}{n_{c}} \sum_{i \in S_{c}}\left\{\frac{A_{i} y_{i}}{\sigma\left(x_{i}^{\top} \hat{\beta}_{S_a}\right)}-\frac{A_{i}-\sigma\left(x_{i}^{\top} \hat{\beta}_{S_a}\right)}{\sigma\left(x_{i}^{\top} \hat{\beta}_{S_a}\right)}\left(\hat{\alpha}^{(1)}_{S_b}+x_{i}^{\top} \hat{\beta}^{(1)}_{S_b}\right)\right\} - \alpha_1\\
  =& \frac{1}{n_{c}} \sum_{i \in S_{c}}\left\{\frac{A_{i} \left(\alpha^{(1)} + x_i{\beta}^{(1)} + \epsilon_i^{(1)} \right)}{\sigma\left(x_{i}^{\top} \hat{\beta}_{S_a}\right)}-\frac{A_{i}-\sigma\left(x_{i}^{\top} \hat{\beta}_{S_a}\right)}{\sigma\left(x_{i}^{\top} \hat{\beta}_{S_a}\right)}\left(\hat{\alpha}^{(1)}_{S_b}+x_{i}^{\top} \hat{\beta}^{(1)}_{S_b}\right)\right\} - \alpha_1\\
  =&\left( \alpha^{(1)} - \hat{\alpha}^{(1)}_{S_b} \right) \left( \frac{1}{n_c}\sum_{i \in S_c}\frac{A_i}{\sigma \left(x_i^{\top} \hat{\beta}_{S_a} \right)}-1\right) + \frac{1}{n_c} \sum_{i \in S_c} \left(\frac{A_{i}x_i}{\sigma(x_{i}^{\top} \hat{\beta}_{S_a})}-x_i \right)^\top \left(\beta^{(1)}-\hat{\beta}^{(1)}_{S_b}\right)   \\
  +& \frac{1}{n_c} \sum_{i \in S_c} \frac{A_{i}\epsilon_i^{(1)}}{\sigma\left(x_{i}^{\top} \hat{\beta}_{S_a}\right)}
  + \frac{1}{n_c}\sum_{i \in S_c} x_i^\top \beta^{(1)}.
\end{align*}

Therefore,  
\begin{align}
    & \sqrt{n} \sum_{(a,b,c) \in \mathscr{S}_3}\left[\hat{\Delta}_{AIPW, 1}-\alpha_{1}\right] =  \sum_{(a,b,c) \in \mathscr{S}_3} \frac{1}{\sqrt{r_c}} \left[ l_{S_c,S_a}^{\top} f(\mathcal{E}_{S_b},X_{S_b}) + \mathcal{E}_{S_c}^{\top}V_{S_c,S_a}+g(X_{S_c}) \right].\label{eq:dr1_rep}
\end{align}

Similarly, we have \begin{align*}
    & \hat{\Delta}_{AIPW, 0}-\alpha_{0} =\frac{1}{n_{c}} \sum_{i \in S_{c}}\left\{\frac{\left(1-A_{i}\right) y_{i}}{1-\sigma\left(x_{i}^{\top} \hat{\beta}_{S_a}\right)}+\frac{A_{i}-\sigma\left(x_{i}^{\top} \hat{\beta}_{S_a}\right)}{1-\sigma\left(x_{i}^{\top} \hat{\beta}\right)}\left(\hat{\alpha}^{(0)}_{S_b}+x_{i}^{\top} \hat{\beta}^{(0)}_{S_b}\right)\right\}  -\alpha_{0} \\
    =& \frac{1}{n_{c}} \sum_{i \in S_{c}}\left\{\frac{\left(1-A_{i}\right) ( \alpha^{(0)} + x_i \beta^{(0)} + \epsilon_i^{(0)})}{1-\sigma\left(x_{i}^{\top} \hat{\beta}_{S_a}\right)}+\frac{A_{i}-\sigma\left(x_{i}^{\top} \hat{\beta}_{S_a}\right)}{1-\sigma\left(x_{i}^{\top} \hat{\beta}_{S_a}\right)}\left(\hat{\alpha}^{(0)}_{S_b}+x_{i}^{\top} \hat{\beta}^{(0)}_{S_b}\right)\right\}  -\alpha_{0} \\
    =& \left(\alpha^{(0)}-\hat{\alpha}^{(0)}_{S_b}\right)\left(\frac{1}{n_{c}} \sum_{i \in S_{c}} \frac{1-A_{i}}{1-\sigma\left(x_{i}^{\top} \hat{\beta}_{S_a}\right)}-1\right) \\
    +& \frac{1}{n_{c}} \sum_{i \in S_{c}}\left(\frac{(1-A_{i}) x_{i}}{1-\sigma\left(x_{i}^{\top} \hat{\beta}_{S_a}\right)}-x_{i}\right)^\top \left(\beta^{(0)}-\hat{\beta}^{(0)}_{S_b}\right) + \frac{1}{n_{c}} \sum_{i \in S_{c}} \frac{(1-A_{i}) \epsilon_{i}^{(0)}}{1-\sigma\left(x_{i}^{\top} \hat{\beta}_{S_a}\right)}+\frac{1}{n_{c}} \sum_{i \in S_{c}}x_{i}^\top \beta^{(0)} .
\end{align*}

In turn, this implies
\begin{align}
    & \sqrt{n} \sum_{(a,b,c) \in \mathscr{S}_3}
    \left[\hat{\Delta}_{AIPW, 0}-\alpha_{0}\right] = \sum_{(a,b,c) \in \mathscr{S}_3} \frac{1}{\sqrt{r_c}} \left[ \tilde{l}_{S_c,S_a}^{\top} \tilde{f}(\mathcal{E}_{S_b},X_{S_b}) + \mathcal{E}_{S_c}^{\top}\tilde{V}_{S_c,S_a}+\tilde{g}(X_{S_c}) \right]. \label{eq:dr0_rep} 
\end{align}
The desired conclusion follows upon combining \eqref{eq:dr1_rep}
and \eqref{eq:dr0_rep}.   
\end{proof}

We prove Lemma \ref{lemma:variance_stabilization} next. We will utilize the following lemma, which is deferred to Section \ref{sec:proof_lem:variance_limit}. 
\begin{lemma}
\label{lem:variance_limit} 
As $n \to \infty$, $\mathrm{Var}(T_1|X,A) \stackrel{p}{\to} V_{T_1}$. 
\end{lemma}
Armed with this result, we complete the proof of Lemma \ref{lemma:variance_stabilization} as follows. 

\begin{proof}[Proof of Lemma \ref{lemma:variance_stabilization}]
First, recall that 
\begin{align}
    T_1 = \frac{1}{6  }  \sum_{(a,b,c) \in \mathscr{S}_3} \frac{1}{\sqrt{r_c}}\left[ l_{S_c,S_a}^{\top} f(\mathcal{E}_{S_b},X_{S_b}) - \tilde{l}_{S_c,S_a}^{\top} \tilde{f}(\mathcal{E}_{S_b},X_{S_b}) + \mathcal{E}_{S_c}^{\top}V_{S_c,S_a}
      - \mathcal{E}_{S_c}^{\top}\tilde{V}_{S_c,S_a} \right]. \nonumber 
\end{align}
Recalling the functional forms of $f,\tilde{f}$, we note that conditioned on everything but $\{\mathcal{E}_{S_1}, \mathcal{E}_{S_2}, \mathcal{E}_{S_3}\}$, $T_1 \sim \mathcal{N}(0, \mathrm{Var}(T_1|X,A))$. Lemma \ref{lem:variance_limit} implies that $\mathrm{Var}(T_1|X,A) \to V_{T_1}$ a.s. On the other hand, 
\begin{align*}
    & T_{2}=\frac{1}{6} \sum_{(a,b,c) \in \mathscr{S}_3} \frac{1}{\sqrt{r_{c}}}\left[g\left(X_{S_{c}}\right)-\tilde{g}\left(X_{S_{c}}\right)\right] = \frac{1}{3} \sum_{i=1}^3 \frac{1}{\sqrt{r_{i}}}\left[g\left(X_{S_{i}}\right)-\tilde{g}\left(X_{S_{i}}\right)\right]\\
    =& \frac{1}{3}  \sum_{i=1}^3 \frac{1}{r_i\sqrt{n}} \sum_{j \in S_{i}} x_j^\top \left( \beta^{(1)} - \beta^{(0)} \right).
\end{align*}
Thus $T_2$ (as a function of $X$) is a mean zero gaussian. Further, 
\begin{align*}
    & \text{Var}\left(  T_2  \right) = \frac{1}{9} \text{Var}\left(  \sum_{i=1}^{3} \frac{1}{r_{i} \sqrt{n}} \sum_{j \in S_{i}} x_{j}^{\top}\left(\beta^{(1)}-\beta^{(0)}\right)  \right) =  \frac{1}{9}   \sum_{i=1}^{3} \frac{1}{r_i^2 n} \text{Var}\left( \sum_{j \in S_{i}} x_{j}^{\top}\left(\beta^{(1)}-\beta^{(0)}\right)  \right) \\
    =& \frac{1}{9} \sum_{i=1}^{3} \frac{1}{r_{i}^{2} n} \sum_{j \in S_{i}}  \text{Var}\left(  x_{j}^{\top}\left(\beta^{(1)}-\beta^{(0)}\right)  \right) =  \frac{1}{9} \sum_{i=1}^{3} \frac{1}{r_{i}^{2} n} n_i \frac{\| \beta^{(1)}-\beta^{(0)}\|^2}{n} \nonumber \\
    =& \frac{1}{9} \sum_{i=1}^{3} \frac{1}{r_i} \frac{\|\beta^{(0)} \|^2 + \|\beta^{(1)} \|^2  - 2 \left(\beta^{(0)}\right)^\top \beta^{(1)} }{n}
    =\kappa \left( \sigma_{0\beta}^2 +  \sigma_{1\beta}^2 - 2 \rho_{01} \sigma_{0 \beta} \sigma_{1 \beta}  \right) * \frac{1}{9} \sum_{i=1}^{3} \frac{1}{r_{i}} + o(1).
\end{align*}
Finally, let $h_1,h_2 : \mathbb{R} \to \mathbb{R}$ be bounded continuous functions, and $\zeta_1, \zeta_2$ i.i.d. $\mathcal{N}(0,1)$ random variables. Then we have, 
\begin{align}
    \mathbb{E}[h_1(T_1) h_2(T_2)] &= \mathbb{E}\Big[ \mathbb{E} \Big[ h_1(T_1) h_2(T_2) |X,A\Big] \Big] \nonumber \\
    &= \mathbb{E}\Big[ h_2(T_2) \mathbb{E}\Big[ h_1(\zeta_1 \sqrt{\mathrm{Var}(T_1|X,A)}) |X,A\Big] \Big]. \nonumber  
\end{align}
By Dominated Convergence, 
\begin{align}
    \mathbb{E}\Big[h_1( \zeta_1 \sqrt{\mathrm{Var}(T_1|X,A)} ) | X,A \Big] \stackrel{p}{\to} \mathbb{E}[h_1(\zeta_1 \sqrt{V_{T_1}})]. \nonumber 
\end{align}
Another application of Dominated Convergence 
\begin{align}
    \mathbb{E}[h_1(T_1) h_2(T_2)] \to \mathbb{E}\Big[ h_1(\zeta_1 \sqrt{V_{T_1}})  \Big] \mathbb{E}\Big[ h_2(\zeta_2 \sqrt{V_{T_2}}) \Big]. \nonumber
\end{align}
This completes the proof. 
\end{proof}

\section{Proof of Lemma \texorpdfstring{\MakeLowercase{\ref{lem:variance_limit}}}{}}
\label{sec:proof_lem:variance_limit} 
We prove Lemma \ref{lem:variance_limit} in this section. This is one of our main technical contributions. First we show that the variance of $T_1$ conditioned on $A, X$ can be expressed as the sum of two conditional variances:

\begin{lemma}\label{lemma_t1_split}
\begin{align}
      \mathrm{Var}(T_1|A,X)=&   \mathrm{Var}\left( \frac{1}{6} \sum_{(a, b, c) \in \mathscr{S}_{3}} \frac{1}{\sqrt{r_{c}}}\left[l_{S_{c}, S_{a}}^{\top} f\left(\mathcal{E}_{S_{b}}, X_{S_{b}}\right)+\mathcal{E}_{S_{c}}^{\top} V_{S_{c}, S_{a}}\right]|A,X \right) \label{lemma_t1_split_res}\\
    +&   \mathrm{Var}\left( \frac{1}{6} \sum_{(a, b, c) \in \mathscr{S}_{3}} \frac{1}{\sqrt{r_{c}}}\left[\tilde{l}_{S_{c}, S_{a}}^{\top} \tilde{f}\left(\mathcal{E}_{S_{b}}, X_{S_{b}}\right)+\mathcal{E}_{S_{c}}^{\top} \tilde{V}_{S_{c}, S_{a}}\right] |A,X \right). \nonumber
\end{align} 

\end{lemma}

\begin{proof}[Proof of Lemma \ref{lemma_t1_split}]
Note that conditioned on $A$ and $X$, \begin{align*}
     & \frac{1}{6} \sum_{(a, b, c) \in \mathscr{S}_{3}} \frac{1}{\sqrt{r_{c}}} \left[ l_{S_{c}, S_{a}}^{\top} f\left(\mathcal{E}_{S_{b}}, X_{S_{b}}\right) + \mathcal{E}_{S_{c}}^{\top} V_{S_{c}, S_{a}} \right] \\
    =&  \frac{1}{6} \sum_{(a, b, c) \in \mathscr{S}_{3}} \frac{1}{\sqrt{r_{c}}} \left[ \frac{1}{\sqrt{n_{c}}}\left(\begin{array}{c}
\sum_{i \in S_{c}}\left(\frac{A_{i}}{\sigma\left(x_{i}^{\top} \hat{\beta}_{S_{a}}\right)}-1\right) \\
\sum_{i \in S_{c}}\left(\frac{A_{i} x_{i}}{\sigma\left(x_{i}^{\top} \hat{\beta}_{S_{a}}\right)}-x_{i}\right)
\end{array}\right)^\top  \left(\begin{array}{c}
\alpha^{(1)}-\hat{\alpha}_{S_{b}}^{(1)} \\
\beta^{(1)}-\hat{\beta}_{S_{b}}^{(1)}
\end{array}\right)  + \frac{1}{\sqrt{n_{c}}} \sum_{i \in S_{c}} \frac{A_{i} \epsilon_{i}^{(1)}}{\sigma\left(x_{i}^{\top} \hat{\beta}_{S_{a}}\right)}  \right]\\
 =&  \frac{1}{6} \sum_{(a, b, c) \in \mathscr{S}_{3}} \frac{1}{\sqrt{r_{c}}} \left[ \frac{1}{\sqrt{n_{c}}}\left(\begin{array}{c}
\sum_{i \in S_{c}}\left(\frac{A_{i}}{\sigma\left(x_{i}^{\top} \hat{\beta}_{S_{a}}\right)}-1\right) \\
\sum_{i \in S_{c}}\left(\frac{A_{i} x_{i}}{\sigma\left(x_{i}^{\top} \hat{\beta}_{S_{a}}\right)}-x_{i}\right)
\end{array}\right)^\top  \left(X_{S_b, 1}^\top X_{S_b, 1}\right)^{-1} X_{S_b, 1}^\top \mathcal{E}_{S_{b},1} \right.\\
+& \left. \frac{1}{\sqrt{n_{c}}} \sum_{i \in S_{c}} \frac{A_{i} \epsilon_{i}^{(1)}}{\sigma\left(x_{i}^{\top} \hat{\beta}_{S_{a}}\right)}  \right]
\end{align*} only depends on $ \mathcal{E}_{S_{1}, 1}, \mathcal{E}_{S_{2}, 1}, $ and $\mathcal{E}_{S_{3}, 1}$. Similarly, conditioned on $A$ and $X$, \begin{align*}
    \frac{1}{6} \sum_{(a, b, c) \in \mathscr{S}_{3}} \frac{1}{\sqrt{r_{c}}}\left[\tilde{l}_{S_{c}, S_{a}}^{\top} \tilde{f}\left(\mathcal{E}_{S_{b}}, X_{S_{b}}\right)+\mathcal{E}_{S_{c}}^{\top} \tilde{V}_{S_{c}, S_{a}}\right]
\end{align*} only depends on $\mathcal{E}_{S_{1}, 0}, \mathcal{E}_{S_{2}, 0}, \text { and } \mathcal{E}_{S_{3}, 0}$.

Since $$\left( \mathcal{E}_{S_{1}, 0}, \mathcal{E}_{S_{2}, 0},\mathcal{E}_{S_{3}, 0} \right) \indep \left(\mathcal{E}_{S_{1}, 1}, \mathcal{E}_{S_{2}, 1}, \mathcal{E}_{S_{3}, 1}\right) \mid A, X,  $$ we know \begin{align*}
    & \frac{1}{6} \sum_{(a, b, c) \in \mathscr{S}_{3}} \frac{1}{\sqrt{r_{c}}}\left[l_{S_{c}, S_{a}}^{\top} f\left(\mathcal{E}_{S_{b}}, X_{S_{b}}\right)+\mathcal{E}_{S_{c}}^{\top} V_{S_{c}, S_{a}}\right]\\ \indep \;\; &  \frac{1}{6} \sum_{(a, b, c) \in \mathscr{S}_{3}} \frac{1}{\sqrt{r_{c}}}\left[\tilde{l}_{S_{c}, S_{a}}^{\top} \tilde{f}\left(\mathcal{E}_{S_{b}}, X_{S_{b}}\right)+\mathcal{E}_{S_{c}}^{\top} \tilde{V}_{S_{c}, S_{a}}\right] \mid A,X.
\end{align*} 

Thus \begin{align*}
    & \operatorname{Var}\left(T_{1} \mid A, X\right) \\
    =& \operatorname{Var}\left( \frac{1}{6} \sum_{(a, b, c) \in \mathscr{S}_{3}} \frac{1}{\sqrt{r_{c}}}\left[l_{S_{c}, S_{a}}^{\top} f\left(\mathcal{E}_{S_{b}}, X_{S_{b}}\right)-\tilde{l}_{S_{c}, S_{a}}^{\top} \tilde{f}\left(\mathcal{E}_{S_{b}}, X_{S_{b}}\right)+\mathcal{E}_{S_{c}}^{\top} V_{S_{c}, S_{a}}-\mathcal{E}_{S_{c}}^{\top} \tilde{V}_{S_{c}, S_{a}}\right]  \mid A, X\right)\\
    =& \operatorname{Var}\left( \frac{1}{6} \sum_{(a, b, c) \in \mathscr{S}_{3}} \frac{1}{\sqrt{r_{c}}}\left[l_{S_{c}, S_{a}}^{\top} f\left(\mathcal{E}_{S_{b}}, X_{S_{b}}\right)+\mathcal{E}_{S_{c}}^{\top} V_{S_{c}, S_{a}}\right]  \mid A, X\right) \\
    +& \text{Var}\left( \frac{1}{6} \sum_{(a, b, c) \in \mathscr{S}_{3}} \frac{1}{\sqrt{r_{c}}}\left[\tilde{l}_{S_{c}, S_{a}}^{\top} \tilde{f}\left(\mathcal{E}_{S_{b}}, X_{S_{b}}\right)+\mathcal{E}_{S_{c}}^{\top} \tilde{V}_{S_{c}, S_{a}}\right]  \mid A, X\right).
\end{align*}
\end{proof}

We evaluate the first term in equation (\ref{lemma_t1_split_res}), and then show that the second term can be analyzed similarly. We express the first term as a sum of terms whose limits we can evaluate:

\begin{lemma}
\label{lem:var_T1_exp}
We have, almost surely, that
\begin{align}
  &  \operatorname{Var}\left(\frac{1}{6} \sum_{(a, b, c) \in \mathscr{S}_{3}} \frac{1}{\sqrt{r_{c}}}\left[l_{S_{c}, S_{a}}^{\top} f\left(\mathcal{E}_{S_{b}}, X_{S_{b}}\right)+\mathcal{E}_{S_{c}}^{\top} V_{S_{c}, S_{a}}\right] \mid A, X\right) \label{eq:var_T1_exp}\\
  =&     \frac{1}{36}  \sum_{(a,b,c) \in \mathscr{S}_3} \Big[ \frac{1}{r_{c}}\left[\operatorname{Var}\left(l_{S_{c}, S_{a}}^{\top} f\left(\mathcal{E}_{S_{b}}, X_{S_{b}}\right) | A,X\right) \right.  \nonumber\\
  +&\left. \operatorname{Var}\left(\mathcal{E}_{S_{c}}^{\top} V_{S_{c}, S_{a}}|A,X\right)\right]  +  \frac{2}{\sqrt{r_{b} r_{c}}} \operatorname{Cov}\left(l_{S_{c}, S_{a}}^{\top} f\left(\mathcal{E}_{S_{b}}, X_{S_{b}}\right), \mathcal{E}_{S_{b}}^{\top} V_{S_{b}, S_{a}} |A,X\right)  \nonumber \\
    +&  \frac{2}{\sqrt{r_{b} r_{c}}} \operatorname{Cov}\left(l_{S_{c}, S_{a}}^{\top} f\left(\mathcal{E}_{S_{b}}, X_{S_{b}}\right), \mathcal{E}_{S_{b}}^{\top} V_{S_{b}, S_{c}} |A,X\right) +  \frac{1}{r_a} \operatorname{Cov}\left(\mathcal{E}_{S_{a}}^{\top} V_{S_{a}, S_{b}}, \mathcal{E}_{S_{a}}^{\top} V_{S_{a}, S_{c}} |A,X\right)\nonumber \\
    +&    \frac{1}{\sqrt{r_a r_c}} \operatorname{Cov}\left(l_{S_{c}, S_{a}} f\left(\mathcal{E}_{S_{b}}, X_{S_{b}}\right), l_{S_{a}, S_{c}}^{\top} f\left(\mathcal{E}_{S_{b}}, X_{S_{b}}\right)|A,X\right) \Big].\nonumber
\end{align}
\end{lemma}

\begin{proof}[Proof of Lemma \ref{lem:var_T1_exp}]

We expand the LHS of equation $(\ref{eq:var_T1_exp})$ and match terms on RHS of the same equation. More specifically, note that \begin{align}
    & \operatorname{Var}\left(\ \sum_{(a, b, c) \in \mathscr{S}_{3}} \frac{1}{\sqrt{r_{c}}}\left[l_{S_{c}, S_{a}}^{\top} f\left(\mathcal{E}_{S_{b}}, X_{S_{b}}\right)+\mathcal{E}_{S_{c}}^{\top} V_{S_{c}, S_{a}}\right] \mid A, X\right) \nonumber\\
    =& \operatorname{Cov}\left( \sum_{(a, b, c) \in \mathscr{S}_{3}} \frac{1}{\sqrt{r_{c}}}\left[l_{S_{c}, S_{a}}^{\top} f\left(\mathcal{E}_{S_{b}}, X_{S_{b}}\right)+\mathcal{E}_{S_{c}}^{\top} V_{S_{c}, S_{a}}\right], \right. \nonumber \\
    &\left. \sum_{(a^{'}, b^{'}, c^{'}) \in \mathscr{S}_{3}} \frac{1}{\sqrt{r_{c^{'}}}}\left[l_{S_{c^{'}}, S_{a^{'}}}^{\top} f\left(\mathcal{E}_{S_{b^{'}}}, X_{S_{b^{'}}}\right)+\mathcal{E}_{S_{c^{'}}}^{\top} V_{S_{c^{'}}, S_{a^{'}}}\right] \mid A, X\right)\nonumber\\
    =&  \sum_{(a, b, c) \in \mathscr{S}_{3}} \sum_{\left(a^{\prime}, b^{\prime}, c^{\prime}\right) \in \mathscr{S}_{3}} \operatorname{Cov}\left( \frac{1}{\sqrt{r_{c}}}l_{S_{c}, S_{a}}^{\top} f\left(\mathcal{E}_{S_{b}}, X_{S_{b}}\right), \frac{1}{\sqrt{r_{c^{\prime}}}}l_{S_{c^{\prime}}, S_{a^{\prime}}}^{\top} f\left(\mathcal{E}_{S_{b^{\prime}}}, X_{S_{b^{\prime}}}\right)   \mid A,X   \right) \label{main_lemma5_res0}  \\
     +& 2 \sum_{(a, b, c) \in \mathscr{S}_{3}} \sum_{\left(a^{\prime}, b^{\prime}, c^{\prime}\right) \in \mathscr{S}_{3}} \operatorname{Cov}\left( \frac{1}{\sqrt{r_{c}}}l_{S_{c}, S_{a}}^{\top} f\left(\mathcal{E}_{S_{b}}, X_{S_{b}}\right), \frac{1}{\sqrt{r_{c^{\prime}}}}\mathcal{E}_{S_{c^{\prime}}}^{\top} V_{S_{c^{\prime}}, S_{a^{\prime}}}  \mid A,X  \right) \nonumber \\
      +& \sum_{(a, b, c) \in \mathscr{S}_{3}} \sum_{\left(a^{\prime}, b^{\prime}, c^{\prime}\right) \in \mathscr{S}_{3}} \operatorname{Cov}\left( \frac{1}{\sqrt{r_{c}}}\mathcal{E}_{S_{c}}^{\top} V_{S_{c}, S_{a}} , \frac{1}{\sqrt{r_{c^{\prime}}}}\mathcal{E}_{S_{c^{\prime}}}^{\top} V_{S_{c^{\prime}}, S_{a^{\prime}}}   \mid A,X  \right).\nonumber  
\end{align}

By further expanding the terms, we have \begin{align}
    & \sum_{(a, b, c) \in \mathscr{S}_{3}} \sum_{\left(a^{\prime}, b^{\prime}, c^{\prime}\right) \in \mathscr{S}_{3}} \operatorname{Cov}\left(\frac{1}{\sqrt{r_{c}}} l_{S_{c}, S_{a}}^{\top} f\left(\mathcal{E}_{S_{b}}, X_{S_{b}}\right), \frac{1}{\sqrt{r_{c^{\prime}}}} l_{S_{c^{\prime}}, S_{a^{\prime}}}^{\top} f\left(\mathcal{E}_{S_{b^{\prime}}}, X_{S_{b^{\prime}}}\right)  \mid A,X  \right) \nonumber \\
    =& \sum_{(a, b, c) \in \mathscr{S}_{3}}  \operatorname{Cov}\left(\frac{1}{\sqrt{r_{c}}} l_{S_{c}, S_{a}}^{\top} f\left(\mathcal{E}_{S_{b}}, X_{S_{b}}\right), \frac{1}{\sqrt{r_{c}}} l_{S_{c}, S_{a}}^{\top} f\left(\mathcal{E}_{S_{b}}, X_{S_{b}}\right)  \mid A,X  \right) \nonumber \\
    +& \sum_{(a, b, c) \in \mathscr{S}_{3}}  \operatorname{Cov}\left(\frac{1}{\sqrt{r_{c}}} l_{S_{c}, S_{a}}^{\top} f\left(\mathcal{E}_{S_{b}}, X_{S_{b}}\right), \frac{1}{\sqrt{r_{a}}} l_{S_{a}, S_{c}}^{\top} f\left(\mathcal{E}_{S_{b}}, X_{S_{b}}\right)  \mid A,X  \right) \nonumber \\
    =& \sum_{(a, b, c) \in \mathscr{S}_{3}}  \frac{1}{{r_{c}}}\operatorname{Var}\left( l_{S_{c}, S_{a}}^{\top} f\left(\mathcal{E}_{S_{b}}, X_{S_{b}}\right)  \mid A,X  \right)  \label{main_lemma5_res1}\\
    +& \sum_{(a, b, c) \in \mathscr{S}_{3}}  \frac{1}{\sqrt{r_ar_{c}}}\operatorname{Cov}\left( l_{S_{c}, S_{a}}^{\top} f\left(\mathcal{E}_{S_{b}}, X_{S_{b}}\right),  l_{S_{a}, S_{c}}^{\top} f\left(\mathcal{E}_{S_{b}}, X_{S_{b}}\right)  \mid A,X  \right). \nonumber
\end{align}

Similarly, we can obtain the following equations: \begin{align}
   & \sum_{(a, b, c) \in \mathscr{S}_{3}} \sum_{\left(a^{\prime}, b^{\prime}, c^{\prime}\right) \in \mathscr{S}_{3}} \operatorname{Cov}\left(\frac{1}{\sqrt{r_{c}}} l_{S_{c}, S_{a}}^{\top} f\left(\mathcal{E}_{S_{b}}, X_{S_{b}}\right), \frac{1}{\sqrt{r_{c^{\prime}}}} \mathcal{E}_{S_{c^{\prime}}}^{\top} V_{S_{c^{\prime}}, S_{a^{\prime}}} \mid A, X\right) \nonumber \\
    =& \sum_{(a, b, c) \in \mathscr{S}_{3}}\frac{1}{\sqrt{r_{b} r_{c}}} \operatorname{Cov} \left(l_{S_{c}, S_{a}}^{\top} f\left(\mathcal{E}_{S_{b}}, X_{S_{b}}\right), \mathcal{E}_{S_{b}}^{\top} V_{S_{b}, S_{a}} \mid A, X\right) \label{main_lemma5_res2}  \\
    +& \sum_{(a, b, c) \in \mathscr{S}_{3}}\frac{1}{\sqrt{r_{b} r_{c}}} \operatorname{Cov}\left(l_{S_{c}, S_{a}}^{\top} f\left(\mathcal{E}_{S_{b}}, X_{S_{b}}\right), \mathcal{E}_{S_{b}}^{\top} V_{S_{b}, S_{c}} \mid A, X\right), \nonumber
\end{align} \begin{align}
    & \sum_{(a, b, c) \in \mathscr{S}_{3}} \sum_{\left(a^{\prime}, b^{\prime}, c^{\prime}\right) \in \mathscr{S}_{3}} \operatorname{Cov}\left(\frac{1}{\sqrt{r_{c}}} \mathcal{E}_{S_{c}}^{\top} V_{S_{c}, S_{a}}, \frac{1}{\sqrt{r_{c^{\prime}}}} \mathcal{E}_{S_{c^{\prime}}}^{\top} V_{S_{c^{\prime}}, S_{a^{\prime}}} \mid A, X\right) \nonumber \\
    =& \sum_{(a, b, c) \in \mathscr{S}_{3}} \frac{1}{r_c} \operatorname{Var}\left(\mathcal{E}_{S_{c}}^{\top} V_{S_{c}, S_{a}} \mid A, X\right) + \sum_{(a, b, c) \in \mathscr{S}_{3}}  \frac{1}{r_{a}} \operatorname{Cov}\left(\mathcal{E}_{S_{a}}^{\top} V_{S_{a}, S_{b}}, \mathcal{E}_{S_{a}}^{\top} V_{S_{a}, S_{c}} \mid A, X\right). \label{main_lemma5_res3}
\end{align}

Plugging equations (\ref{main_lemma5_res1}), (\ref{main_lemma5_res2}), (\ref{main_lemma5_res3}) into RHS of equation (\ref{main_lemma5_res0}), and then dividing both sides of equation (\ref{main_lemma5_res0}) by 36 completes the proof.

\end{proof}

Next we evaluate the deterministic limits of terms which appear on RHS of equation (\ref{eq:var_T1_exp}).

\begin{lemma}\label{lemma_t1_terms_limit}
We have, as $n \to \infty$, 
\begin{itemize}
\item[(i)] 
\begin{align}
    &\operatorname{Var}\left(l_{S_{c}, S_{a}}^{\top} f\left(\mathcal{E}_{S_{b}}, X_{S_{b}}\right) | A,X\right) \stackrel{p}{\to} \left(\sigma^{(1)}\right)^2 \Big[ \frac{2}{r_b - 2 \kappa} \left\{ \kappa + \kappa   \mathbb{E} \left[ \frac{\sigma(Z_\beta) }{\sigma^2\left(Z_{\hat{\beta}_{S_a}}\right)} - \frac{ 2\sigma(Z_\beta) }{\sigma \left(Z_{\hat{\beta}_{S_a}}\right)} \right]  \right. \nonumber \\
    +&\left. r_c \gamma^2 \mathbb{E}^{2}\left[ \left(1 - \alpha_a^* \right) \frac{\sigma\left(Z_{\beta}\right)}{\sigma\left(Z_{\hat{\beta}_{S_{a}}}\right)} - \frac{\sigma^2\left(Z_{\beta}\right)}{\sigma\left(Z_{\hat{\beta}_{S_{a}}}\right)} + \frac{\alpha_a^* }{2} \right] \right. 
    + \left. r_{c} \kappa_{a}\left(\sigma_{a}^{*}\right)^{2}  \mathbb{E}^{2}\left[\frac{\sigma\left(Z_{\beta}\right) }{\sigma\left(Z_{\hat{\beta}_{S_a}}\right)} - \frac{1}{2}\right] \right\}  \nonumber \\
    +& \frac{  4  r^2_{c} e^2_{\gamma, 0} h^2_{a} }{{\gamma^2}r_{b} [(\frac{r_b}{2}- \kappa)(1 - 4 e_{\gamma, 0}^{2})]} + \frac{r_c }{(\frac{r_b}{2} - \kappa)(1 - 4e_{\gamma, 0}^{2}) } \mathbb{E}^2\left[\frac{\sigma\left(Z_{\beta}\right)}{\sigma\left(Z_{\hat{\beta}_{S_{a}}}\right)} - 1\right] \nonumber \\
    -&  \frac{4 r_c e_{\gamma, 0} h_a }{\gamma \left(\frac{r_{b}}{2}-\kappa\right)\left(1-4 e_{\gamma, 0}^{2}\right)}  \mathbb{E}\left[\frac{\sigma\left(Z_{\beta}\right)}{\sigma\left(Z_{\hat{\beta}_{S_{a}}}\right)}-1\right] \Big]. \nonumber 
\end{align}

\item[(ii)] 
\begin{align}
    \operatorname{Var}\left(\mathcal{E}_{S_{c}}^{\top} V_{S_{c}, S_{a}} |A,X \right) \stackrel{p}{\to} \left(\sigma^{(1)}\right)^2 \mathbb{E}\left[  \frac{\sigma(Z_ \beta)}{\sigma^{2}\left(Z_ {\hat{\beta}_{S_a}}\right)}\right]. \nonumber
\end{align}

\item[(iii)] 
\begin{align}
   & \operatorname{Cov}\left(l_{S_{c}, S_{a}}^{\top} f\left(\mathcal{E}_{S_{b}}, X_{S_{b}}\right), \mathcal{E}_{S_{b}}^{\top} V_{S_{b}, S_{a}} |A,X\right) \stackrel{p}{\to}  \nonumber \\
    -&\left(\sigma^{(1)}\right)^2 \Big[\sqrt{r_{b} r_{c}} \frac{\left(\mathbb{E}\left[\frac{\sigma\left(Z_{\beta}\right)}{\sigma\left(Z_{\hat{\beta}_{S_{a}}}\right)}\right]-1\right) \mathbb{E}\left[\frac{\sigma\left(Z_{\beta}\right)}{\sigma\left(Z_{\hat{\beta}_{S_{a}}}\right)}\right]}{\left(\frac{r_{b}}{2}-\kappa\right)\left(1-4 e_{\gamma, 0}^{2}\right)} -\frac{2 \sqrt{r_{b} r_{c}} e_{\gamma, 0} h_{a}}{\gamma\left(\frac{r_{b}}{2}-\kappa\right)\left(1-4 e_{\gamma, 0}^{2}\right)} \cdot \mathbb{E}\left[\frac{\sigma\left(Z_{\beta}\right)}{\sigma\left(Z_{\hat{\beta}_{S_{a}}}\right)}\right] \nonumber \\
    +&  2 \sqrt{\frac{r_c}{r_b}} \left\{\mathbb{E}\left[\frac{\sigma^{\prime}\left(Z_{\beta}\right)}{\sigma\left(Z_{\hat{\beta}_{S_{a}}}\right)}\right] \mathbb{E}\left[ \frac{\sigma(Z_\beta)Z_\beta}{\sigma(Z_{\hat{\beta}_{S_a}})} \right] - \mathbb{E}\left[\frac{\sigma\left(Z_{\beta}\right)}{\sigma\left(Z_{\hat{\beta}_{S_{a}}}\right)} - \frac{1}{2}\right] \mathbb{E}\left[\frac{\sigma(Z_\beta) Z_{\hat{\beta}_{S_a}}}{\sigma\left(Z_{\hat{\beta}_{S_a}}\right)}\right] \right \} \nonumber \\
    +&  \frac{\sqrt{r_c} \left[ \frac{2 e_{\gamma, 0} h_{a} }{ \gamma }  -  \mathbb{E}\left[\frac{\sigma\left(Z_{\beta}\right)}{\sigma\left(Z_{\hat{\beta}_{S_{a}}}\right)}\right] + 1 \right]}{\sqrt{r_b}\left(\frac{r_{b}}{2}-\kappa\right)\left(1-4 e_{\gamma, 0}^{2}\right)}  \nonumber \\
    \cdot & \left[\frac{2 r_{b} h_{a} e_{\gamma, 0}}{\gamma}-4 \kappa e_{\gamma, 0}\left(e_{\gamma, 0}+e^{\frac{\left(\alpha_{a}^{*} \gamma\right)^{2}+\kappa_{a}\left(\sigma_{a}^{*}\right)^{2}}{2}} e_{\gamma,-\alpha_{a}^{*} \gamma}\right)+2 \kappa\left(\frac{1}{2}+e^{\frac{\left(\alpha_{a}^{*} \gamma\right)^{2}+\kappa_{a}\left(\sigma_{a}^{*}\right)^{2}}{2}} q_{\gamma,-\alpha_{a}^{*} \gamma}\right)\right] \Big]. \nonumber 
\end{align}

\item[(iv)]
\begin{align}
    &\operatorname{Cov}\left(l_{S_{c}, S_{a}}^{\top} f\left(\mathcal{E}_{S_{b}}, X_{S_{b}}\right), \mathcal{E}_{S_{b}}^{\top} V_{S_{b}, S_{c}} |A,X\right)  \stackrel{p}{\to} \nonumber \\
    -&\left(\sigma^{(1)}\right)^2\Big[\sqrt{r_{b} r_{c}} \frac{\left(\mathbb{E}\left[\frac{\sigma\left(Z_{\beta}\right)}{\sigma\left(Z_{\hat{\beta}_{S_{a}}}\right)}\right]-1\right) \mathbb{E}\left[\frac{\sigma\left(Z_{\beta}\right)}{\sigma\left(Z_{\hat{\beta}_{S_{a}}}\right)}\right]}{\left(\frac{r_{b}}{2}-\kappa\right)\left(1-4 e_{\gamma, 0}^{2}\right)} \nonumber  \\
    -&\frac{2 \sqrt{r_{b} r_{c}} e_{\gamma, 0} h_{a}}{\gamma\left(\frac{r_{b}}{2}-\kappa\right)\left(1-4 e_{\gamma, 0}^{2}\right)} \cdot \mathbb{E}\left[\frac{\sigma\left(Z_{\beta}\right)}{\sigma\left(Z_{\hat{\beta}_{S_{a}}}\right)}\right] +   \frac{\sqrt{r_c} \left[ \frac{2 e_{\gamma, 0} h_{a} }{ \gamma }  -  \mathbb{E}\left[\frac{\sigma\left(Z_{\beta}\right)}{\sigma\left(Z_{\hat{\beta}_{S_{a}}}\right)}\right] + 1 \right]}{\sqrt{r_b}\left(\frac{r_{b}}{2}-\kappa\right)\left(1-4 e_{\gamma, 0}^{2}\right)}  \nonumber \\
    \cdot & \left[\frac{2 r_{b} h_{c} e_{\gamma, 0}}{\gamma}-4 \kappa e_{\gamma, 0}\left(e_{\gamma, 0}+e^{\frac{\left(\alpha_{c}^{*} \gamma\right)^{2}+\kappa_{c}\left(\sigma_{c}^{*}\right)^{2}}{2}} e_{\gamma,-\alpha_{c}^{*} \gamma}\right)+2 \kappa\left(\frac{1}{2}+e^{\frac{\left(\alpha_{c}^{*} \gamma\right)^{2}+\kappa_{c}\left(\sigma_{c}^{*}\right)^{2}}{2}} q_{\gamma,-\alpha_{c}^{*} \gamma}\right)\right] \nonumber \\
    +& 2 \sqrt{\frac{r_c}{r_b}} \mathbb{E}\left[\frac{\sigma^{\prime}\left(Z_\beta\right)}{\sigma\left(Z_{\hat{\beta}_{S_{c}}}\right)}\right] \mathbb{E}\left[\frac{\sigma(Z_\beta)}{\sigma\left(Z_{\hat{\beta}_{S_{a}}}\right)} Z_\beta\right]-   2 \sqrt{\frac{r_c}{r_b}} \mathbb{E}\left[\frac{\sigma\left(Z_\beta\right) }{\sigma\left(Z_{\hat{\beta}_{S_{c}}}\right)} - \frac{1}{2}\right] \left\{ \mathbb{E}\left[\frac{\sigma\left(Z_{\beta}\right)}{\sigma\left(Z_{\hat{\beta}_{S_{a}}}\right)} Z_{\hat{\beta}_{S_{c}}}\right] \right.  \nonumber \\
      +& \left. \mathbb{E}\left[\sigma\left(Z_{\beta}\right)\left(\frac{1}{\sigma\left(Z_{\hat{\beta}_{S_{a}}}\right)}-1\right) q_{c_{1}}\left(1-\sigma\left(\operatorname{prox}_{q_{c_{1}} \rho}\left(Z_{\hat{\beta}_{S_{c}}+q_{c_{1}}}\right)\right)\right)\right] \right. \nonumber \\ 
     +&\left.  \mathbb{E}\left[\left(1-\sigma\left(Z_{\beta}\right)\right) q_{c_{1}} \sigma\left(\operatorname{prox}_{q_{c_{1}} \rho}\left(Z_{\hat{\beta}_{S_{c}}}\right)\right)\right] \right\}\Big].  \nonumber 
\end{align}

\item[(v)] 
\begin{align}
  \operatorname{Cov}\left(\mathcal{E}_{S_{a}}^{\top} V_{S_{a}, S_{b}}, \mathcal{E}_{S_{a}}^{\top} V_{S_{a}, S_{c}} |A,X\right) \stackrel{p}{\to}  \left(\sigma^{(1)}\right)^2  \mathbb{E} \left[  \frac{\sigma\left( Z_\beta \right)}{\sigma\left(Z_{ \hat{\beta}_{S_a}}\right) \sigma\left(Z_{ \hat{\beta}_{S_c}}\right)}  \right], \nonumber 
\end{align}

\item[(vi)]
\begin{align}
    &\operatorname{Cov}\left(l_{S_{c}, S_{a}} f\left(\mathcal{E}_{S_{b}}, X_{S_{b}}\right), l_{S_{a}, S_{c}}^{\top} f\left(\mathcal{E}_{S_{b}}, X_{S_{b}}\right)|A,X\right) \stackrel{p}{\to} \nonumber \\
    & \left(\sigma^{(1)}\right)^2[ \frac{\sqrt{n_an_c} \left\{ \mathbb{E}\left[ \frac{\sigma(Z_\beta)}{\sigma \left(Z_{\hat{\beta}_{S_a}} \right)} \right] -1 \right\} \left\{ \mathbb{E}\left[ \frac{\sigma(Z_\beta)}{\sigma \left(Z_{\hat{\beta}_{S_c}} \right)} \right] -1 \right\}}{\left(\frac{n_b}{2} - p\right)\left(1-4 e_{\gamma, 0}^{2}\right) } -  \frac{2\sqrt{n_an_c} h_c e_{\gamma,0} \left\{ \mathbb{E}\left[ \frac{\sigma(Z_\beta)}{\sigma \left(Z_{\hat{\beta}_{S_a}} \right)} \right] -1 \right\} }{\gamma \left(\frac{n_b}{2} - p\right)\left(1-4 e_{\gamma, 0}^{2}\right)} \nonumber \\
    -&    \frac{2\sqrt{n_an_c} h_a e_{\gamma,0} \left\{ \mathbb{E}\left[ \frac{\sigma(Z_\beta)}{\sigma \left(Z_{\hat{\beta}_{S_c}} \right)} \right] -1 \right\} }{\gamma \left(\frac{n_b}{2} - p\right)\left(1-4 e_{\gamma, 0}^{2}\right)} + \frac{4 n_{a} n_{c} h_{a} h_{c} e_{\gamma, 0}^{2}}{n_{b} \gamma^{2}\left(\frac{n_{b}}{2}-p\right)\left(1-4 e_{\gamma, 0}^{2}\right)} \nonumber \\
    +& \frac{2n \sqrt{r_{a} r_{c}} \gamma^{2}}{n_b - 2p} \left\{\left(\mathbb{E}\left[\frac{\sigma\left(Z_{\beta}\right)\left(1-\sigma\left(Z_{\beta}\right)\right)}{\sigma\left(Z_{\hat{\beta}_{S_{a}}}\right)}\right]-\alpha_{a}^{*} \mathbb{E}\left[\frac{\sigma\left(Z_{\beta}\right)}{\sigma\left(Z_{\hat{\beta}_{S_{a}}}\right)}-\frac{1}{2}\right]\right) \right. \nonumber \\
   \cdot  &\left. \left(\mathbb{E}\left[\frac{\sigma\left(Z_{\beta}\right)\left(1-\sigma\left(Z_{\beta}\right)\right)}{\sigma\left(Z_{\hat{\beta}_{S_{c}}}\right)}\right]-\alpha_{c}^{*} \mathbb{E}\left[\frac{\sigma\left(Z_{\beta}\right)}{\sigma\left(Z_{\hat{\beta}_{S_{c}}}\right)}-\frac{1}{2}\right]\right)\right\} \nonumber \\
    -& \frac{2n \lambda_{c}^{*} \sqrt{r_{a} r_{c}} }{n_b - 2p} \mathbb{E}\left[\frac{\sigma\left(Z_{\beta}\right) }{\sigma\left(Z_{\beta_{S_{c}}}\right)} - \frac{1}{2}\right]  \mathbb{E}\left[\sigma\left(Z_{\beta}\right)\left(\frac{1}{\sigma\left(Z_{\hat{\beta}_{S_ a}}\right)}-1\right)\left(1-\sigma\left(\operatorname{prox}_{\lambda_{c}^{*} \rho}\left(Z_{\hat{\beta}_{S _c}}+\lambda_{c}^{*}\right)\right)\right)\right] \nonumber \\
    -& \frac{2n \lambda_{a}^{*} \sqrt{r_{a} r_{c}} }{n_b - 2p} \mathbb{E}\left[\frac{\sigma\left(Z_{\beta}\right) }{\sigma\left(Z_{\beta_{S_{a}}}\right)} - \frac{1}{2}\right] \mathbb{E}\left[\sigma\left(Z_{\beta}\right)\left(\frac{1}{\sigma\left(Z_{\hat{\beta}_{S_ c}}\right)}-1\right)\left(1-\sigma\left(\operatorname{prox}_{\lambda_{a}^{*} \rho}\left(Z_{\hat{\beta}_{S_ a}}+\lambda_{a}^{*}\right)\right)\right)\right] \nonumber \\
    -& \frac{2 n \lambda_{c}^{*} \sqrt{r_{a} r_{c}}}{n_{b}-2 p} \mathbb{E}\left[\frac{\sigma\left(Z_{\beta}\right) }{\sigma\left(Z_{\beta_{S_{c}}}\right)} - \frac{1}{2}\right] \mathbb{E}\left[\left(1-\sigma\left(Z_{\beta}\right)\right)  \sigma\left(\operatorname{prox}_{\lambda_{c}^{*} \rho}\left(Z_{\hat{\beta}_{S_ c}}\right)\right)\right] \nonumber \\
    -& \frac{2 n \lambda_{a}^{*} \sqrt{r_{a} r_{c}}}{n_{b}-2 p} \mathbb{E}\left[\frac{\sigma\left(Z_{\beta}\right) }{\sigma\left(Z_{\beta_{S_{a}}}\right)} - \frac{1}{2}\right] \mathbb{E}\left[\left(1-\sigma\left(Z_{\beta}\right)\right)  \sigma\left(\operatorname{prox}_{\lambda_{a}^{*} \rho}\left(Z_{\hat{\beta}_{S_ a}}\right)\right)\right] ] . \nonumber  \nonumber 
\end{align}

\end{itemize} 
\end{lemma}

Now we are able to determine the limiting conditional variance of $T_1$ by collecting terms and exploiting symmetry.


\begin{proof}[Proof of Lemma \ref{lem:variance_limit}]
Combining Lemma \ref{lem:var_T1_exp} and Lemma \ref{lemma_t1_terms_limit} yields an expression for the first term in equation (\ref{lemma_t1_split_res}). 

Note that for all $ i \in S_c$, define $ \tilde{A}_{i}=1-A_{i}, \;\; \tilde{x}_{i}=-x_{i}$. Then $\tilde{x}_{i} \stackrel{\text { i.i.d. }}{\sim} N\left(0, \frac{1}{n} I_{p}\right) $. 

Moreover, since $\sigma(x)=1-\sigma(-x)  \;\; \forall x \in \mathbb{R}$, we have $$\mathbb{P}\left(\tilde{A}_{i}=1 \mid x_{i}\right)=\mathbb{P}\left(A_{i}=0 \mid x_{i}\right)=1-\sigma\left(x_{i}^{\top} \beta\right)=\sigma\left(-x_{i}^{\top} \beta\right)=\sigma\left(\tilde{x}_{i}^{\top} \beta\right) \quad \forall i \in S_c. $$ 
By symmetry we know the expression for the second term in equation (\ref{lemma_t1_split_res}) is the same as that for the first term, expect that $ \sigma^{(1)}$ is replaced with $ \sigma^{(0)}$. Adding the two expressions together yields the desired result.

\end{proof}

\section{Proof of Lemma \texorpdfstring{\MakeLowercase{\ref{lemma_t1_terms_limit} }}{} }
We prove Lemma \ref{lemma_t1_terms_limit} in this section. This is one of the main technical contributions of this paper. To this end, we will need some preliminary technical results. We group these preliminary results for the convenience of the reader.

  \begin{lemma} \label{diff lemma 2}
Let $ p, n \in \mathbb{N}^{+}$ be such that $\lim_{n \rightarrow\infty} \frac{p}{n} = \kappa \in (0, \frac{1}{2})$. Let $\{\beta\}_{n \in \mathbb{N}^+}$ be a sequence of deterministic vectors in $\mathbb{R}^p$ such that $\lim_{n \rightarrow \infty} \frac{\|\beta\|}{\sqrt{n}} = \gamma \text{ for some fixed positive constant $\gamma$.  } $ Let $X \in \mathbb{R}^{n \times p}$ consist of i.i.d. entries $\sim N(0, \frac{1}{n})$. Let $\{A_i\}_{i=1,2, \ldots,n}$ be independent Bernoulli random variables with success probability $\sigma(x_i^{\top} \beta) $. Define $$O = \left[\sum_{i=1}^n A_{i} x_{i} x_{i}^{\top}-\left(\sum_{i=1}^n A_{i}\right)^{-1}\left(\sum_{i=1}^n A_{i} x_{i}\right)\left(\sum_{i=1}^n A_{i} x_{i}^{\top}\right)\right]^{-1}, \quad  Q = \left( \sum_{i=1}^{n} A_{i} x_{i} x_{i}^{\top} \right)^{-1}. $$ Then for any vectors $a, b \in \mathbb{R}^n$, \begin{align*}
& a^{\top} Ob - a^{\top} Qb \\
=& \frac{\left(\frac{1}{n} \sum_{i=1}^{n} A_{i}\right)^{-1} \cdot \left[a^{\top} \left(\sum_{i=1}^{n} A_{i} x_{i} x_{i}^{\top}\right)^{-1} \cdot \frac{1}{\sqrt{n}} \left(\sum_{i=1}^{n} A_{i} x_{i}\right) \right] }{1-\left(\frac{1}{n} \sum_{i=1}^{n} A_{i}\right)^{-1} \cdot \frac{1}{n}\left(\sum_{i=1}^{n} A_{i} x_{i}\right)^{\top}\left(\sum_{i=1}^{n} A_{i} x_{i} x_{i}^{\top}\right)^{-1}\left(\sum_{i=1}^{n} A_{i}  x_{i}\right)}  \\
\cdot & \left[b^{\top} \left(\sum_{i=1}^{n} A_{i} x_{i} x_{i}^{\top}\right)^{-1} \cdot \frac{1}{\sqrt{n}} \left(\sum_{i=1}^{n} A_{i} x_{i}\right) \right].
\end{align*}    
  \end{lemma}

 \begin{lemma} \label{lemma 10}
Let $ p, n \in \mathbb{N}^{+}$ be such that $\lim_{n \rightarrow\infty} \frac{p}{n} = \kappa \in (0, \frac{1}{2})$. Let $\{\beta\}_{n \in \mathbb{N}^+}$ be a sequence of deterministic vectors in $\mathbb{R}^p$ such that $\lim_{n \rightarrow \infty} \frac{\|\beta\|}{\sqrt{n}} = \gamma $ for some fixed positive constant $\gamma$.   Let $X \in \mathbb{R}^{n \times p}$ consist of i.i.d. entries $\sim N(0, \frac{1}{n})$. Let $\{A_i\}_{i=1,2, \ldots,n}$ be independent Bernoulli random variables with success probability $\sigma(x_i^{\top} \beta) $. Then  $$\frac{1}{n}\left(\sum_{i=1}^{n} A_{i} x_{i}\right)^{\top}\left(\sum_{i=1}^{n} A_{i} x_{i} x_{i}^{\top}\right)^{-1}\left(\sum_{i=1}^{n} A_{i} x_{i}\right) \xrightarrow{p} \kappa+2e_{\gamma, 0}^2 \left(1-2 \kappa\right) . $$
  \end{lemma}

   \begin{lemma} \label{lemma 11}
  For any $(a,b,c)$ which is a permutation of $(1,2,3)$, we have \begin{align*}
    & \frac{1}{n}\left(\sum_{i \in S_{b}} A_{i} x_{i}\right)^{\top}\left(\sum_{i \in S_{b}} A_{i} x_{i} x_{i}^{\top}\right)^{-1} \sum_{j \in S_{a}}\left(\frac{A_{j} x_{j}}{\sigma\left(x_{j}^{\top} \hat{\beta}_{S_c}\right)}-x_{j}\right) =   \frac{2r_a e_{\gamma, 0} h_c}{\gamma}  +o_p(1).
\end{align*} 
  \end{lemma}
  
    \begin{lemma} \label{lemma 12}
  For any $(a,b,c)$ which is a permutation of $(1,2,3)$, we have \begin{align*}
   & \frac{1}{n} \left(\sum_{i \in S_{a}} A_{i} x_{i}\right)^{\top}\left(\sum_{i \in S_{a}} A_{i} x_{i} x_{i}^{\top}\right)^{-1} \sum_{i \in S_{a}} \frac{A_{i} x_{i}}{\sigma\left(x_{i}^{\top} \hat{\beta}_{S_b}\right)}\\
   =& \frac{2r_a h_b e_{\gamma, 0}}{\gamma} - 4 \kappa e_{\gamma, 0}\left(e_{\gamma, 0}+e^{\frac{\left(\alpha_{b}^{*} \gamma\right)^{2}+\kappa_{b}\left(\sigma_{b}^{*}\right)^{2}}{2}} e_{\gamma,-\alpha_{b}^{*} \gamma}\right) + 2 \kappa\left(\frac{1}{2}+e^{\frac{\left(\alpha_{b}^{*} \gamma\right)^{2}+\kappa_{b}\left(\sigma_{b}^{*}\right)^{2}}{2}} q_{\gamma,-\alpha_{b}^{*} \gamma}\right) + o_p(1).
\end{align*}
  \end{lemma}
  
  \begin{lemma} \label{vec_lemma} 
  For any $(a,b,c)$ which is a permutation of $(1,2,3)$, we have \begin{align*}
      & \frac{1}{n_c} \left[\sum_{i \in S_{c}}\left(\frac{A_{i} x_{i}}{\sigma\left(x_{i}^{\top} \hat{\beta}_{S_a}\right)}-x_{i}\right)^{\top}\right]\left[\sum_{i \in S_{c}}\left(\frac{A_{i} x_{i}}{\sigma\left(x_{i}^{\top} \hat{\beta}_{S_a}\right)}-x_{i}\right)\right] \\
      =&\kappa + \kappa   \mathbb{E} \left[ \frac{\sigma(Z_\beta) }{\sigma^2\left(Z_{\hat{\beta}_{S_a}}\right)} - \frac{ 2\sigma(Z_\beta) }{\sigma \left(Z_{\hat{\beta}_{S_a}}\right)} \right] + r_c \gamma^2 \mathbb{E}^{2}\left[ \left(1 - \alpha_a^* \right) \frac{\sigma\left(Z_{\beta}\right)}{\sigma\left(Z_{\hat{\beta}_{S_{a}}}\right)} - \frac{\sigma^2\left(Z_{\beta}\right)}{\sigma\left(Z_{\hat{\beta}_{S_{a}}}\right)} + \frac{\alpha_a^* }{2} \right]  \\
      +& r_{c} \kappa_{a}\left(\sigma_{a}^{*}\right)^{2}  \mathbb{E}^{2}\left[\frac{\sigma\left(Z_{\beta}\right) }{\sigma\left(Z_{\hat{\beta}_{S_a}}\right)} - \frac{1}{2}\right]
      + o_p(1).
  \end{align*}
  \end{lemma}
  
  \begin{lemma} \label{vec_lemma_2} For any $(a,b,c)$ which is a permutation of $(1,2,3)$, we have \begin{align*}
      & \frac{1}{\sqrt{n_{a} n_{c}}}\left[\sum_{i \in S_{a}}\left(\frac{A_{i} x_{i}}{\sigma\left(x_{i}^{\top} \hat{\beta}_{S_{c}}\right)}-x_{i}\right)^{\top}\right]\left[\sum_{j \in S_{c}}\left(\frac{A_{j} x_{j}}{\sigma\left(x_{j}^{\top} \hat{\beta}_{S_{a}}\right)}-x_{j}\right)\right] \\
      =& \sqrt{r_ar_c} \gamma^ 2\left\{ \left( \mathbb{E}\left[\frac{\sigma\left(Z_{\beta}\right)\left(1-\sigma\left(Z_{\beta}\right)\right)}{\sigma\left(Z_{\hat{\beta}_{S_{a}}}\right)}\right]  - \alpha_a^* \mathbb{E}\left[\frac{\sigma\left(Z_{\beta}\right)}{\sigma\left(Z_{\hat{\beta}_{S_{a}}}\right)}-\frac{1}{2}\right] \right ) \right. \\ 
      &\left. \left( \mathbb{E}\left[\frac{\sigma\left(Z_{\beta}\right)\left(1-\sigma\left(Z_{\beta}\right)\right)}{\sigma\left(Z_{\hat{\beta}_{S_{c}}}\right)}\right]  - \alpha_c^* \mathbb{E}\left[\frac{\sigma\left(Z_{\beta}\right)}{\sigma\left(Z_{\hat{\beta}_{S_{c}}}\right)}-\frac{1}{2}\right] \right ) \right\}\\
    +&\sqrt{r_{a} r_{c}} \mathbb{E}\left[ \sigma(Z_ \beta) \left(\frac{1}{\sigma\left(Z_{\hat{\beta}_{S_{a}}}\right)}-1\right)C_{c} \lambda_{c}^{*}\left(1-\sigma\left(\operatorname{prox}_{\lambda_{c}^{*} \rho}\left(Z_{\hat{\beta}_{S_{c}}}+\lambda_{c}^{*} \right)\right)\right)\right]  \\
    +& \sqrt{r_{a} r_{c}} \mathbb{E}\left[ (1 - \sigma(Z_\beta)) C_c \lambda_{c}^{*} \sigma\left(\operatorname{prox}_{\lambda_{c}^{*} \rho}\left(Z_{\hat{\beta}_{S_{c}}}\right)\right) \right]\\
    +& \sqrt{r_{a} r_{c}} \mathbb{E}\left[ \sigma(Z_\beta) \left(\frac{1}{\sigma\left(Z _{\hat{\beta}_{S_{c}}}\right)}-1\right)C_{a} \lambda_{a}^{*}\left(1-\sigma\left(\operatorname{prox}_{\lambda_{a}^{*} \rho}\left(Z _{\hat{\beta}_{S_{a}}}+\lambda_{a}^{*} \right)\right)\right)\right] \\ 
    +& \sqrt{r_{a} r_{c}} \mathbb{E}\left[ (1 - \sigma(Z_\beta)) C_a \lambda_{a}^{*} \sigma\left(\operatorname{prox}_{\lambda_{a}^{*} \rho}\left(Z _{\hat{\beta}_{S_{a}}}\right)\right) \right] 
    +o_p(1),
  \end{align*} 
  where $$ C_{i}=-\mathbb{E}\left[\frac{\sigma\left(Z_{\beta}\right) \sigma^{\prime}\left(Z_{\hat{\beta}_{S_{i}}}\right)}{\sigma^{2}\left(Z_{\hat{\beta}_{S_{i}}}\right)}\right] \;\; \forall i = 1,2,3. $$
  
  \end{lemma}

   \begin{lemma} \label{lemma 25 second}
  For any $(a,b,c)$ which is a permutation of $(1,2,3)$, we have \begin{align*}
       & \frac{1}{n}\left(\sum_{j \in S_{c}} \frac{A_{j} x_{j}}{\sigma\left(x_{j}^{T} \hat{\beta}_{S_a}\right)}\right)^{T}\left(\sum_{i \in S_{b}} A_{i} x_{i} x_{i}^{\top}\right)^{-1}\left(\sum_{i \in S_{b}} \frac{A_{i} x_{i}}{\sigma\left(x_{i}^{T} \hat{\beta}_{S_a}\right)}\right)\\
       =& 2r_c\mathbb{E}\left[\frac{\sigma^{\prime}\left(Z_{\beta}\right)}{\sigma\left(Z_{\hat{\beta}_{S_{a}}}\right)}\right] \mathbb{E}\left[ \frac{\sigma(Z_\beta)Z_\beta}{\sigma(Z_{\hat{\beta}_{S_a}})} \right] - 2\mathbb{E}\left[\frac{\sigma\left(Z_{\beta}\right)}{\sigma\left(Z_{\hat{\beta}_{S_{a}}}\right)} - \frac{1}{2}\right] \mathbb{E}\left[\frac{\sigma(Z_\beta) Z_{\hat{\beta}_{S_a}}}{\sigma\left(Z_{\hat{\beta}_{S_a}}\right)}\right]   + o_p(1).
  \end{align*} 
  \end{lemma}

   \begin{lemma} \label{lemma 26}
    For any $(a,b,c)$ which is a permutation of $(1,2,3)$, we have \begin{align*}
    & \frac{1}{n}\left(\sum_{j \in S_{c}} \frac{A_{j} x_{j}}{\sigma\left(x_{j}^{T} \hat{\beta}_{S_{a}}\right)} - x_i\right)^{T}\left(\sum_{i \in S_{b}} A_{i} x_{i} x_{i}^{\top}\right)^{-1}\left(\sum_{i \in S_{b}} \frac{A_{i} x_{i}}{\sigma\left(x_{i}^{T} \hat{\beta}_{S_{c}}\right)}\right)\\
       = & 2 r_c \mathbb{E}\left[\frac{\sigma^{\prime}\left(Z_\beta\right)}{\sigma\left(Z_{\hat{\beta}_{S_{c}}}\right)}\right] \mathbb{E}\left[\frac{\sigma(Z_\beta)}{\sigma\left(Z_{\hat{\beta}_{S_{a}}}\right)} Z_\beta\right]-   2 r_c \mathbb{E}\left[\frac{\sigma\left(Z_\beta\right) \sigma^{\prime}\left(   Z_{\hat{\beta}_{S_{c}}}\right)}{\sigma^{2}\left(Z_{\hat{\beta}_{S_{c}}}\right)}\right] \left\{ \mathbb{E}\left[\frac{\sigma\left(Z_{\beta}\right)}{\sigma\left(Z_{\hat{\beta}_{S_{a}}}\right)} Z_{\hat{\beta}_{S_{c}}}\right] \right. \\
      +& \left. \mathbb{E}\left[\sigma\left(Z_{\beta}\right)\left(\frac{1}{\sigma\left(Z_{\hat{\beta}_{S_{a}}}\right)}-1\right) \lambda_c^* \left(1-\sigma\left(\operatorname{prox}_{\lambda_c^* \rho}\left(Z_{\hat{\beta}_{S_{c}}+\lambda_c^*}\right)\right)\right)\right] \right. \\
      &\left. \mathbb{E}\left[\left(1-\sigma\left(Z_{\beta}\right)\right) q_{c_{1}} \sigma\left(\operatorname{prox}_{\lambda_c^* \rho}\left(Z_{\hat{\beta}_{S_{c}}}\right)\right)\right] \right\} + o_p(1).
    \end{align*}
  
  \end{lemma}

Armed with these preliminary results, we turn to a proof of Lemma \ref{lemma_t1_terms_limit}.

\begin{proof}[Proof of Lemma \ref{lemma_t1_terms_limit}]
We prove each part in turn. For notational convenience, we sometimes keep the conditioning on $A,X$ implicit in our expressions.

\noindent 
 \textbf{Proof of (i):}
 We then find the asymptotic conditional variance of \begin{align*}
     & l_{S_{c}, S_{a}}^{T} f\left(\mathcal{E}_{S_{b}}, X_{S_{b}}\right)\\
     =& \left( \alpha^{(1)} - \hat{\alpha}^{(1)}_{S_b} \right)\left( \frac{1}{\sqrt{n_c}}\sum_{i \in S_c}\left(\frac{A_i}{\sigma(x_i^{\top} \hat{\beta}_{S_a})}-1\right)\right) + \left\langle\beta^{(1)}-\hat{\beta}^{(1)}_{S_b}, \frac{1}{\sqrt{n_c}} \sum_{i \in S_c} \left(\frac{A_{i}x_i}{\sigma(x_{i}^{T} \hat{\beta}_{S_a})}-x_i\right)\right\rangle.
 \end{align*}

By ordinary least squares properties we have $$ \left(\begin{array}{c}
\hat{\alpha}^{(1)}_{S_b}-\alpha^{(1)} \\
\hat{\beta}^{(1)}_{S_b}-\beta^{(1)}
\end{array}\right) \sim N\left(0, \left(\sigma^{(1)}\right)^{2}\left(\tilde{X}_{S_{b,1}}^{T} \tilde{X}_{S_{b,1}}\right)^{-1}\right),$$ which is equivalent to $$\left(\begin{array}{c}
\hat{\alpha}^{(1)}_{S_b}-\alpha^{(1)} \\
\hat{\beta}^{(1)}_{S_b}-\beta^{(1)}
\end{array}\right) \sim N\left(0, \left(\sigma^{(1)}\right)^{2}\left( \sum_{i \in S_b} A_i \left(\begin{array}{c}
1 \\
x_i
\end{array} \right)(1 \;\; x_i)
\right)^{-1}\right) .$$

As a result, conditioned on everything other than noises, \begin{align*}
    & \left(\alpha^{(1)}-\hat{\alpha}_{S_{b}}^{(1)}\right)\left(\frac{1}{\sqrt{n_{c}}} \sum_{i \in S_{c}}\left(\frac{A_{i}}{\sigma\left(x_{i}^{\top} \hat{\beta}_{S_{a}}\right)}-1\right)\right)+\left\langle\beta^{(1)}-\hat{\beta}_{S_{b}}^{(1)}, \frac{1}{\sqrt{n_{c}}} \sum_{i \in S_{c}}\left(\frac{A_{i} x_{i}}{\sigma\left(x_{i}^{T} \hat{\beta}_{S_{a}}\right)}-x_{i}\right)\right\rangle \\
     \sim &  N\left(0, \left(\sigma^{(1)}\right)^{2}\cdot l_{S_{c},S_a}^{T}\left( \sum_{i \in S_b} A_i \left(\begin{array}{c}
1 \\
x_i
\end{array} \right)(1 \;\; x_i)
\right)^{-1}l_{S_{c},S_a}\right).
\end{align*}

We then find the limit of $$ l_{S_{c},S_a}^{T}\left( \sum_{i \in S_b} A_i \left(\begin{array}{c}
1 \\
x_i
\end{array} \right)(1 \;\; x_i)
\right)^{-1}l_{S_{c},S_a}. $$

Define $$B = \left[\sum_{i \in S_b} A_i -(\sum_{i \in S_b} A_i x_i^{\top})(\sum_{i \in S_b} A_i x_ix_i^{\top})^{-1} (\sum_{i \in S_b} A_i x_i) \right]^{-1}, $$ $$ C = -B \left( \sum_{i \in S_b} A_ix_i x_i^\top \right)^{-1}\left(\sum_{i \in S_b} A_ix_i \right) ,  $$ $$D = \left[ \sum_{i \in S_b} A_i x_ix_i^{\top}  - (\sum_{i \in S_b} A_i)^{-1} (\sum_{i \in S_b} A_i x_i) (\sum_{i \in S_b} A_i x_i^{\top}) \right]^{-1}.$$

Then by block matrix inversion formula we have \begin{align*}
    & l_{S_{c},S_a}^{T}\left( \sum_{i \in S_b} A_i \left(\begin{array}{c}
1 \\
x_i
\end{array} \right)(1 \;\; x_i)
\right)^{-1}l_{S_{c},S_a} = l_{S_{c},S_a}^{T} \left[\begin{array}{ll}
B & C^{\top} \\
C & D
\end{array}\right] l_{S_{c},S_a} \\
=& \frac{1}{n_c} \left[ \sum_{i \in S_c} \left(\frac{A_i}{\sigma(x_i^{\top} \hat{\beta}_{S_a})}-1 \right) \right]^2 B + \frac{2}{n_c} \left[ \sum_{i \in S_c} \left(\frac{A_i}{\sigma(x_i^{\top} \hat{\beta}_{S_a})}-1 \right) \right] C^{\top} \left[\sum_{i \in S_c} \left(\frac{A_{i}x_i}{\sigma(x_{i}^{\top} \hat{\beta}_{S_a})}-x_i \right)\right] \\
+& \frac{1}{n_c} \left[\sum_{i \in S_c} \left(\frac{A_{i}x_i}{\sigma(x_{i}^{\top} \hat{\beta}_{S_a})}-x_i \right)^{\top}\right] D \left[\sum_{i \in S_c} \left(\frac{A_{i}x_i}{\sigma(x_{i}^{\top} \hat{\beta}_{S_a})}-x_i \right)\right] .
\end{align*}

By Lemma \ref{lemma_wlln}, we have $$\frac{1}{n_c} \sum_{i \in S_c} \left(\frac{A_i}{\sigma(x_i^{\top} \hat{\beta}_{S_a})} \right) = \mathbb{E}\left[\frac{\sigma\left(Z_{\beta}\right)}{\sigma\left(Z_{\hat{\beta}_{S_{a}}}\right)}\right]+o_{p}(1).$$


Moreover, by Lemma \ref{lemma 10}, we have \begin{align*}
    & \frac{1}{n}\left(\sum_{i \in S_{b}} A_{i} x_{i}^{\top}\right)\left(\sum_{i \in S_{b}} A_{i} x_{i} x_{i}^{\top}\right)^{-1}\left(\sum_{i \in S_{b}} A_{i} x_{i}\right)\\
    =& r_b \cdot \frac{1}{n_b}\left(\sum_{i \in S_{b}} A_{i} \sqrt{\frac{n}{n_b}} x_{i}^{\top}\right)\left(\sum_{i \in S_{b}} A_{i} \sqrt{\frac{n}{n_b}}x_{i} \sqrt{\frac{n}{n_b}} x_{i}^{\top}\right)^{-1}\left(\sum_{i \in S_{b}} A_{i} \sqrt{\frac{n}{n_b}} x_{i}\right)\\
    =& r_b \cdot \left[ \frac{p}{n_b}+2\left(1-\frac{2 p}{n_b}\right) \mathbb{E}^{2}\left[\frac{z}{1+e^{- \gamma z}}\right] \right] + o_p(1) \\
    =& \frac{p}{n} + 2r_b \left(1-\frac{2 p}{n_b}\right) \mathbb{E}^{2}\left[\frac{z}{1+e^{- \gamma z}}\right] + o_p(1)=    \kappa + 2r_be_{\gamma, 0}^{2}(1 - 2 \kappa_b)   + o_p(1).
\end{align*}


Hence \begin{align*}
    &  \frac{1}{n_c} \left[ \sum_{i \in S_c} \left(\frac{A_i}{\sigma(x_i^{\top} \hat{\beta}_{S_a})}-1 \right) \right]^2 B  = \frac{n_c}{n} \left[ \frac{1}{n_c} \sum_{i \in S_c} \left(\frac{A_i}{\sigma(x_i^{\top} \hat{\beta}_{S_a})}-1 \right) \right]^2 \cdot nB\\
    =& r_c \mathbb{E}^2\left[\frac{\sigma\left(Z_{\beta}\right)}{\sigma\left(Z_{\hat{\beta}_{S_{a}}}\right)} - 1\right] \cdot \left[\frac{1}{n}\sum_{i \in S_b} A_i -\frac{1}{n}(\sum_{i \in S_b} A_i x_i^{\top})(\sum_{i \in S_b} A_i x_ix_i^{\top})^{-1} (\sum_{i \in S_b} A_i x_i) \right]^{-1} + o_p(1)\\
    =&   \frac{r_c }{\frac{r_b}{2} - \kappa - 2r_be_{\gamma, 0}^{2}(1 - 2 \kappa_b) } \mathbb{E}^2\left[\frac{\sigma\left(Z_{\beta}\right)}{\sigma\left(Z_{\hat{\beta}_{S_{a}}}\right)} - 1\right]+  o_p(1)\\
    =&   \frac{r_c }{(\frac{r_b}{2} - \kappa)(1 - 4e_{\gamma, 0}^{2}) } \mathbb{E}^2\left[\frac{\sigma\left(Z_{\beta}\right)}{\sigma\left(Z_{\hat{\beta}_{S_{a}}}\right)} - 1\right]+ o_p(1) .
\end{align*}


We then find the limit of $$ \frac{2}{n_c} \left[ \sum_{i \in S_c} \left(\frac{A_i}{\sigma(x_i^{\top} \hat{\beta})}-1 \right) \right] C^{\top} \left[\sum_{i \in S_c} \left(\frac{A_{i}x_i}{\sigma(x_{i}^{\top} \hat{\beta})}-x_i \right)\right]  .$$

By Lemma \ref{lemma 11} we have \begin{align*}
    & \frac{2}{n_c} \left[ \sum_{i \in S_c} \left(\frac{A_i}{\sigma(x_i^{\top} \hat{\beta})}-1 \right) \right] C^{\top} \left[\sum_{i \in S_c} \left(\frac{A_{i}x_i}{\sigma(x_{i}^{\top} \hat{\beta})}-x_i \right)\right] \\
    =& -2nB\left[ \frac{1}{n_c}  \sum_{i \in S_c} \left(\frac{A_i}{\sigma(x_i^{\top} \hat{\beta})}-1 \right) \right] \cdot \frac{1}{n} \left(\sum_{i \in S_{b}} A_{i} x_{i}^{\top}\right) \left(\sum_{i \in S_{b}} A_{i} x_{i} x_{i}^{\top}\right)^{-1} \left[\sum_{i \in S_c} \left(\frac{A_{i}x_i}{\sigma(x_{i}^{\top} \hat{\beta})}-x_i \right)\right]\\
    =&  \frac{-2 \mathbb{E}\left[\frac{\sigma\left(Z_{\beta}\right)}{\sigma\left(Z_{\hat{\beta}_{S_{a}}}\right)}-1\right]}{\frac{r_{b}}{2}-\kappa-2 r_{b} e_{\gamma, 0}^{2}\left(1-2 \kappa_{b}\right)} \cdot \frac{2 r_{c} e_{\gamma, 0} h_{a}}{\gamma} + o_p(1) \\
    =&   \frac{-4 r_c e_{\gamma, 0} h_a }{\gamma \left(\frac{r_{b}}{2}-\kappa\right)\left(1-4 e_{\gamma, 0}^{2}\right)} \cdot \mathbb{E}\left[\frac{\sigma\left(Z_{\beta}\right)}{\sigma\left(Z_{\hat{\beta}_{S_{a}}}\right)}-1\right] + o_p(1).
\end{align*}

Next we find the limit of $$  \frac{1}{n_c} \left[\sum_{i \in S_c} \left(\frac{A_{i}x_i}{\sigma(x_{i}^{\top} \hat{\beta}_{S_a})}-x_i \right)^{\top}\right] D \left[\sum_{i \in S_c} \left(\frac{A_{i}x_i}{\sigma(x_{i}^{\top} \hat{\beta}_{S_a})}-x_i \right)\right]   .$$

By Lemma \ref{diff lemma 2}, \begin{align*}
    &  \frac{1}{n_c} \left[\sum_{i \in S_c} \left(\frac{A_{i}x_i}{\sigma(x_{i}^{\top} \hat{\beta}_{S_a})}-x_i \right)^{\top}\right] D \left[\sum_{i \in S_c} \left(\frac{A_{i}x_i}{\sigma(x_{i}^{\top} \hat{\beta}_{S_a})}-x_i \right)\right] \\
    =&  \frac{1}{n_c} \left[\sum_{i \in S_c} \left(\frac{A_{i}x_i}{\sigma(x_{i}^{\top} \hat{\beta}_{S_a})}-x_i \right)^{\top}\right] \left(\sum_{j \in S_b} A_{j} x_{j} x_{j}^{\top}\right)^{-1} \left[\sum_{i \in S_c} \left(\frac{A_{i}x_i}{\sigma(x_{i}^{\top} \hat{\beta}_{S_a})}-x_i \right)\right] \\
    +&  \frac{\left(\frac{1}{n_b} \sum_{j \in S_b} A_{j}\right)^{-1} \cdot \left[\left[\sum_{i \in S_c} \left(\frac{A_{i}x_i}{\sigma(x_{i}^{\top} \hat{\beta}_{S_a})}-x_i \right)^{\top}\right] \left(\sum_{j \in S_b} A_{j} x_{j} x_{j}^{\top}\right)^{-1} \cdot \frac{1}{n_b} \left(\sum_{j \in S_b} A_{j} x_{j}\right) \right]^2}{1-\left(\frac{1}{n_b} \sum_{j \in S_b} A_{j}\right)^{-1} \cdot \frac{1}{n_b}\left(\sum_{j \in S_b} A_{j} x_{j}\right)^{\top}\left(\sum_{j \in S_b} A_{j} x_{j} x_{j}^{\top}\right)^{-1}\left(\sum_{j \in S_b} A_{j}  x_{j}\right)} .
\end{align*}

Since we have shown that the deterministic equivalent of $\left(\sum_{j \in S_{b}} A_{j} x_{j} x_{j}^{\top}\right)^{-1}$ is $\frac{2 n}{n_{b}-2 p} I_{p}$, by Lemma \ref{vec_lemma} we know 
\begin{align*}
    & \frac{1}{n_c} \left[\sum_{i \in S_c} \left(\frac{A_{i}x_i}{\sigma(x_{i}^{\top} \hat{\beta}_{S_a})}-x_i \right)^{\top}\right] \left(\sum_{j \in S_b} A_{j} x_{j} x_{j}^{\top}\right)^{-1} \left[\sum_{i \in S_c} \left(\frac{A_{i}x_i}{\sigma(x_{i}^{\top} \hat{\beta}_{S_a})}-x_i \right)\right] \\
    =& \frac{2 n}{n_{b}-2 p} \cdot \frac{1}{n_c} \left[\sum_{i \in S_c} \left(\frac{A_{i}x_i}{\sigma(x_{i}^{\top} \hat{\beta}_{S_a})}-x_i \right)^{\top}\right] \left[\sum_{i \in S_c} \left(\frac{A_{i}x_i}{\sigma(x_{i}^{\top} \hat{\beta}_{S_a})}-x_i \right)\right] + o_p(1)\\
   =& \frac{2}{r_b - 2 \kappa} \left\{ \kappa + \kappa   \mathbb{E} \left[ \frac{\sigma(Z_\beta) }{\sigma^2\left(Z_{\hat{\beta}_{S_a}}\right)} - \frac{ 2\sigma(Z_\beta) }{\sigma \left(Z_{\hat{\beta}_{S_a}}\right)} \right] + r_c \gamma^2 \mathbb{E}^{2}\left[ \left(1 - \alpha_a^* \right) \frac{\sigma\left(Z_{\beta}\right)}{\sigma\left(Z_{\hat{\beta}_{S_{a}}}\right)} - \frac{\sigma^2\left(Z_{\beta}\right)}{\sigma\left(Z_{\hat{\beta}_{S_{a}}}\right)} + \frac{\alpha_a^* }{2} \right] \right.\\
    +& \left. r_{c} \kappa_{a}\left(\sigma_{a}^{*}\right)^{2}  \mathbb{E}^{2}\left[\frac{\sigma\left(Z_{\beta}\right) }{\sigma\left(Z_{\hat{\beta}_{S_a}}\right)} - \frac{1}{2}\right] \right\} + o_p(1).
\end{align*}

Further, by Lemma \ref{lemma 10} and Lemma \ref{lemma 11}, we have \begin{align*}
    & \frac{\left(\frac{1}{n_{b}} \sum_{j \in S_{b}} A_{j}\right)^{-1} \cdot\left[\left[\sum_{i \in S_{c}}\left(\frac{A_{i} x_{i}}{\sigma\left(x_{i}^{\top} \hat{\beta}_{S_a}\right)}-x_{i}\right)^{\top}\right]\left(\sum_{j \in S_{b}} A_{j} x_{j} x_{j}^{\top}\right)^{-1} \cdot \frac{1}{n_{b}}\left(\sum_{j \in S_{b}} A_{j} x_{j}\right)\right]^{2}}{1-\left(\frac{1}{n_{b}} \sum_{j \in S_{b}} A_{j}\right)^{-1} \cdot \frac{1}{n_{b}}\left(\sum_{j \in S_{b}} A_{j} x_{j}\right)^{\top}\left(\sum_{j \in S_{b}} A_{j} x_{j} x_{j}^{\top}\right)^{-1}\left(\sum_{j \in S_{b}} A_{j} x_{j}\right)}\\
    =& \frac{  8  r^2_{c} e^2_{\gamma, 0} h^2_{a} }{{\gamma^2}r^2_{b} [1-2 \kappa_b - 4 e_{\gamma, 0}^{2}\left(1-2 \kappa_{b}\right)]} + o_p(1) = \frac{  4  r^2_{c} e^2_{\gamma, 0} h^2_{a} }{{\gamma^2}r_{b} [(\frac{r_b}{2}- \kappa)(1 - 4 e_{\gamma, 0}^{2})]} + o_p(1).
\end{align*} 

Thus \begin{align*}
    & \frac{1}{n_{c}}\left[\sum_{i \in S_{c}}\left(\frac{A_{i} x_{i}}{\sigma\left(x_{i}^{\top} \hat{\beta}_{S_a}\right)}-x_{i}\right)^{\top}\right] D\left[\sum_{i \in S_{c}}\left(\frac{A_{i} x_{i}}{\sigma\left(x_{i}^{\top} \hat{\beta}_{S_a}\right)}-x_{i}\right)\right] \\
    =&  \frac{2}{r_b - 2 \kappa} \left\{ \kappa + \kappa   \mathbb{E} \left[ \frac{\sigma(Z_\beta) }{\sigma^2\left(Z_{\hat{\beta}_{S_a}}\right)} - \frac{ 2\sigma(Z_\beta) }{\sigma \left(Z_{\hat{\beta}_{S_a}}\right)} \right] + r_c \gamma^2 \mathbb{E}^{2}\left[ \left(1 - \alpha_a^* \right) \frac{\sigma\left(Z_{\beta}\right)}{\sigma\left(Z_{\hat{\beta}_{S_{a}}}\right)} - \frac{\sigma^2\left(Z_{\beta}\right)}{\sigma\left(Z_{\hat{\beta}_{S_{a}}}\right)} + \frac{\alpha_a^* }{2} \right] \right.\\
    +& \left. r_{c} \kappa_{a}\left(\sigma_{a}^{*}\right)^{2}  \mathbb{E}^{2}\left[\frac{\sigma\left(Z_{\beta}\right) }{\sigma\left(Z_{\hat{\beta}_{S_a}}\right)} - \frac{1}{2}\right] \right\} + \frac{  4  r^2_{c} e^2_{\gamma, 0} h^2_{a} }{{\gamma^2}r_{b} [(\frac{r_b}{2}- \kappa)(1 - 4 e_{\gamma, 0}^{2})]}+ o_p(1).
\end{align*}

Therefore \begin{align*}
    & l_{S_{c}, S_{a}}^{T}\left(\sum_{i \in S_{b}} A_{i}\left(\begin{array}{c}
1 \\
x_{i}
\end{array}\right)\left(\begin{array}{ll}
1 & x_{i}
\end{array}\right)\right)^{-1} l_{S_{c}, S_{a}}    \\
=&    \frac{2}{r_b - 2 \kappa} \left\{ \kappa + \kappa   \mathbb{E} \left[ \frac{\sigma(Z_\beta) }{\sigma^2\left(Z_{\hat{\beta}_{S_a}}\right)} - \frac{ 2\sigma(Z_\beta) }{\sigma \left(Z_{\hat{\beta}_{S_a}}\right)} \right] + r_c \gamma^2 \mathbb{E}^{2}\left[ \left(1 - \alpha_a^* \right) \frac{\sigma\left(Z_{\beta}\right)}{\sigma\left(Z_{\hat{\beta}_{S_{a}}}\right)} - \frac{\sigma^2\left(Z_{\beta}\right)}{\sigma\left(Z_{\hat{\beta}_{S_{a}}}\right)} + \frac{\alpha_a^* }{2} \right] \right.\\
    +& \left. r_{c} \kappa_{a}\left(\sigma_{a}^{*}\right)^{2}  \mathbb{E}^{2}\left[\frac{\sigma\left(Z_{\beta}\right) }{\sigma\left(Z_{\hat{\beta}_{S_a}}\right)} - \frac{1}{2}\right] \right\} + \frac{  4  r^2_{c} e^2_{\gamma, 0} h^2_{a} }{{\gamma^2}r_{b} [(\frac{r_b}{2}- \kappa)(1 - 4 e_{\gamma, 0}^{2})]}\\
    +& \frac{r_c }{(\frac{r_b}{2} - \kappa)(1 - 4e_{\gamma, 0}^{2}) } \mathbb{E}^2\left[\frac{\sigma\left(Z_{\beta}\right)}{\sigma\left(Z_{\hat{\beta}_{S_{a}}}\right)} - 1\right] -  \frac{4 r_c e_{\gamma, 0} h_a }{\gamma \left(\frac{r_{b}}{2}-\kappa\right)\left(1-4 e_{\gamma, 0}^{2}\right)}  \mathbb{E}\left[\frac{\sigma\left(Z_{\beta}\right)}{\sigma\left(Z_{\hat{\beta}_{S_{a}}}\right)}-1\right].
\end{align*}

Moreover, the asymptotic conditional variance is equal to the term above multiplied by $\left(\sigma^{(1)}\right)^{2}.$

 \textbf{Proof of (ii):} 
 For any permutation $(a,b,c)$ of $(1,2,3)$, first we find the conditional asymptotic variance of 
    $$\mathcal{E}_{S_{c}}^{\top} V_{S_{c}, S_{a}} =  \frac{1}{\sqrt{n_{c}}} \sum_{i \in S_{c}} \frac{A_{i} \epsilon_{i}^{(1)}}{\sigma\left(x_{i}^{\top} \hat{\beta}_{S_a}\right)}.$$
    
    Note that $$\left[\frac{1}{\sqrt{n_{c}}} \sum_{i \in S_{c}} \frac{A_{i} \epsilon_{i}^{(1)}}{\sigma\left(x_{i}^{T} \hat{\beta}_{S_a}\right)} \mid \hat{\beta}_{S_a}, A_{i}, x_{i}\;\; \forall i \in S_{c}\right] \sim N\left(0,   \frac{\left(\sigma^{(1)}\right)^{2}}{n_{c}} \sum_{i \in S_{c}} \frac{A_{i}}{\sigma^{2}\left(x_{i}^{T} \hat{\beta}_{S_a}\right)}\right). $$
    
    Further, we have 
    \begin{align*}
    & \frac{1}{n_{c}} \sum_{i \in S_c} \frac{A_i}{\sigma^2 (x_i^{T} \hat{\beta}_{S_a})} = 
    \mathbb{E}\left[ \frac{A_{1}}{\sigma^{2}\left(x_{1}^{T} \hat{\beta}_{S_a}\right)}\right] + o_p(1)\\
    =& \mathbb{E}\left[ \frac{\sigma(x_{1}^{\top} \beta)}{\sigma^{2}\left(x_{1}^{T} \hat{\beta}_{S_a}\right)}\right] + o_p(1) =  \mathbb{E}\left[  \frac{\sigma(Z_ \beta)}{\sigma^{2}(Z_{ \hat{\beta}_{S_a}})}\right] + o_p(1),
    \end{align*} 
where $$\left[Z_{\beta}, Z_{\hat{\beta}_{S_a}}\right]^{\top} \sim N\left(0,\left[\begin{array}{cc}
\gamma^{2} & \alpha_{a}^{*} \gamma^{2} \\
\alpha_{a}^{*} \gamma^{2} & \kappa_{a}\left(\sigma_{a}^{*}\right)^{2}+\left(\alpha_{a}^{*}\right)^{2} \gamma^{2}
\end{array}\right]\right).$$

Hence the asymptotic conditional variance of $\mathcal{E}_{S_{c}}^{\top} V_{S_{c}, S_{a}}$ is $\left(\sigma^{(1)}\right)^{2} \mathbb{E}\left[  \frac{\sigma(Z_ \beta)}{\sigma^{2}\left(Z_ {\hat{\beta}_{S_a}}\right)}\right].$

 \textbf{Proof of (iii):} 
 Conditioned on everything other than the noises, we have \begin{align*}
  & \mathbb{E} \left[  l_{S_c,S_a}^T f(\mathcal{E}_{S_b},X_{S_b}) \mathcal{E}_{S_b}^TV_{S_b,S_a} \right]\\
  =&-\frac{1}{\sqrt{n_b n_c}}  \mathbb{E} \left[ \left(\begin{array}{c}
\sum_{j \in S_{c}} \left( \frac{A_{j}}{\sigma\left(x_{j}^{T} \hat{\beta}_{S_a}\right)}-1\right) \\
\sum_{j \in S_{c}}\left(\frac{A_{j} x_{j}}{\sigma\left(x_{j}^{T} \hat{\beta}_{S_a}\right)}-x_{j}\right)
\end{array}\right)^T \left(\tilde{X}_{S_{b,1}}^{T} \tilde{X}_{S_{b,1}}\right)^{-1} \tilde{X}_{S_{b,1}}^{T} \mathcal{E}_{S_{b,1}} \cdot \sum_{i \in S_{b,1}} \frac{ \epsilon_{i}^{(1)}}{\sigma\left(x_{i}^{T} \hat{\beta}_{S_a}\right)} \right ] \\
=& -\frac{\left(\sigma^{(1)}\right)^2 }{\sqrt{n_bn_c}}   \left(\begin{array}{c}
\sum_{j \in S_{c}} \left( \frac{A_{j}}{\sigma\left(x_{j}^{T} \hat{\beta}_{S_a}\right)}-1\right) \\
\sum_{j \in S_{c}}\left(\frac{A_{j} x_{j}}{\sigma\left(x_{j}^{T} \hat{\beta}_{S_a}\right)}-x_{j}\right)
\end{array}\right)^T \left(\tilde{X}_{S_{b,1}}^{T} \tilde{X}_{S_{b,1}}\right)^{-1} \tilde{X}_{S_{b,1}}^{T} w_{S_{b,1}},
\end{align*} where $w_{S_{b,1}} $ is the vector that contains $ \frac{1}{\sigma(x_i^T \hat{\beta}_{S_a})} \;\;\forall i \in S_{b} \text{ s.t. } A_i = 1 .$

Note that $$ \left(\tilde{X}_{S_{b,1}}^{T} \tilde{X}_{S_{b,1}}\right)^{-1}  = \left(\sum_{i \in S_{b}} A_{i}\left(\begin{array}{c}
1 \\
x_{i}
\end{array}\right)\left(\begin{array}{ll}
1 & x_{i}
\end{array}\right)\right)^{-1}= \left[\begin{array}{cc}
B & C^{\top} \\
C & D
\end{array}\right].$$

Define $$P_1 = \sum_{j \in S_{c}}\left(\frac{A_{j}}{\sigma\left(x_{j}^{T} \hat{\beta}_{S_a}\right)}-1\right),\quad P_2 = \sum_{j \in S_{c}}\left(\frac{A_{j} x_{j}}{\sigma\left(x_{j}^{T} \hat{\beta}_{S_a}\right)}-x_{j}\right). $$

We have \begin{align*}
     & \frac{1}{\sqrt{n_bn_c}}  \left(\begin{array}{c}
\sum_{j \in S_{c}} \left( \frac{A_{j}}{\sigma\left(x_{j}^{T} \hat{\beta}_{S_a}\right)}-1\right) \\
\sum_{j \in S_{c}}\left(\frac{A_{j} x_{j}}{\sigma\left(x_{j}^{T} \hat{\beta}_{S_a}\right)}-x_{j}\right)
\end{array}\right)^T \left(\tilde{X}_{S_{b,1}}^{T} \tilde{X}_{S_{b,1}}\right)^{-1} \tilde{X}_{S_{b,1}}^{T} w_{S_{b,1}}\\
=& \frac{1}{\sqrt{n_bn_c}} \left[ (D P_1+P_2^TB) \cdot \mathbf{1} + (P_1 B^T+P_2^TC)X_{S_{b,1}}^T \right] w_{S_{b,1}}\\
=& \frac{1}{\sqrt{n_bn_c}} B P_1 \cdot \sum_{i \in S_{b}} \frac{A_i}{\sigma\left(x_{i}^{T} \hat{\beta}_{S_a}\right)} + \frac{1}{\sqrt{n_bn_c}} P_2^T C \cdot \sum_{i \in S_{b}} \frac{A_i}{\sigma\left(x_{i}^{T} \hat{\beta}_{S_a}\right)} \\
&+ \frac{1}{\sqrt{n_bn_c}} P_1 C^T  X_{S_{b,1}}^T w_{S_{b,1}} + \frac{1}{\sqrt{n_bn_c}} P_2^T D X_{S_{b,1}}^T w_{S_{b,1}}.
\end{align*}

Since $$ \frac{1}{n_c} P_1 = \mathbb{E} \left[   \frac{A_{1}}{\sigma\left(x_{1}^{T} \hat{\beta}_{S_a}\right)}-1 \right] + o_p(1) =  \mathbb{E} \left[   \frac{\sigma(Z_\beta)}{\sigma\left(Z_ {\hat{\beta}_{S_a}}\right)} \right] -1 + o(1), $$

We know \begin{align*}
   & \frac{1}{\sqrt{n_bn_c}} B P_1 \cdot \sum_{i \in S_{b}} \frac{A_i}{\sigma\left(x_{i}^{T} \hat{\beta}_{S_a}\right)} \\
   &= \sqrt{r_br_c} \frac{1}{n_c} P_1 \cdot nB \cdot \frac{1}{n_b} \sum_{i \in S_{b}} \frac{A_{i}}{\sigma\left(x_{i}^{T} \hat{\beta}_{S_a}\right)} = \sqrt{r_br_c}\frac{\left( \mathbb{E}\left[\frac{\sigma\left(Z_{\beta}\right)}{\sigma\left(Z_{\hat{\beta}_{S_a}}\right)}\right]-1\right)  \mathbb{E}\left[\frac{\sigma\left(Z_{\beta}\right)}{\sigma\left(Z_{\hat{\beta}_{S_a}}\right)}\right]}{\left(\frac{r_{b}}{2}-\kappa\right)\left(1-4 e_{\gamma, 0}^{2}\right)}.
\end{align*}

Further, \begin{align*}
    & \frac{1}{\sqrt{n_bn_c}} P_{2}^{T} C \cdot \sum_{i \in S_{b}} \frac{A_{i}}{\sigma\left(x_{i}^{T} \hat{\beta}_{S_a}\right)} \\
    =& -\frac{1}{\sqrt{n_bn_c}} B P_{2}^{T} \left(\sum_{i \in S_{b}} A_{i} x_{i} x_{i}^{\top}\right)^{-1}\left(\sum_{i \in S_{b}} A_{i} x_{i}\right) \cdot \sum_{i \in S_{b}} \frac{A_{i}}{\sigma\left(x_{i}^{T} \hat{\beta}_{S_a}\right)}\\
    =& - \sqrt{\frac{r_b}{r_c}} \cdot nB \cdot \frac{1}{n} \left( \sum_{i \in S_{c}}\frac{A_{i} x_{i}}{\sigma\left(x_{i}^{T} \hat{\beta}_{S_{a}}\right)}-x_{i}\right)^{T}\left(\sum_{i \in S_{b}} A_{i} x_{i} x_{i}^{\top}\right)^{-1}\left(\sum_{i \in S_{b}} A_{i} x_{i}\right) \cdot \frac{1}{n_b} \sum_{i \in S_{b}} \frac{A_{i}}{\sigma\left(x_{i}^{T} \hat{\beta}_{S_a}\right)} \\
    =& - \sqrt{\frac{r_b}{r_c}} \frac{1}{\left(\frac{r_{b}}{2}-\kappa\right)\left(1-4 e_{\gamma, 0}^{2}\right)} \cdot \frac{2 r_{c} e_{\gamma, 0} h_{a}}{\gamma} \cdot \mathbb{E}\left[\frac{\sigma\left(Z_{\beta}\right)}{\sigma\left(Z_{\hat{\beta}_{S_{a}}}\right)}\right] + o_p(1)  \\
    =&  -  \frac{2 \sqrt{r_br_c} e_{\gamma, 0} h_{a}}{\gamma \left(\frac{r_{b}}{2}-\kappa\right)\left(1-4 e_{\gamma, 0}^{2}\right)}  \cdot \mathbb{E}\left[\frac{\sigma\left(Z_{\beta}\right)}{\sigma\left(Z_{\hat{\beta}_{S_{a}}}\right)}\right] + o_p(1),
\end{align*} where the second last step is due to Lemma \ref{lemma 11}.

Moreover, due to Lemma \ref{lemma 12}, \begin{align*}
    & \frac{1}{\sqrt{n_bn_c}} P_{1} C^{T} X_{S_{b,1}}^{T} w_{S_{b,1}} = - \frac{1}{\sqrt{n_bn_c}} B P_{1} \left(\sum_{i \in S_{b}} A_{i} x_{i}\right)^T \left(\sum_{i \in S_{b}} A_{i} x_{i} x_{i}^{\top}\right)^{-1} \left(\sum_{i \in S_b} \frac{A_ix_i}{\sigma\left(x_{i}^{T} \hat{\beta}_{S_a}\right)}  \right)  \\
    =&  - \sqrt{\frac{r_c}{r_b}} nB \cdot \frac{1}{n_c}P_1 \cdot \frac{1}{n}  \left(\sum_{i \in S_{b}} A_{i} x_{i}\right)^T \left(\sum_{i \in S_{b}} A_{i} x_{i} x_{i}^{\top}\right)^{-1} \left(\sum_{i \in S_b} \frac{A_ix_i}{\sigma\left(x_{i}^{T} \hat{\beta}_{S_a}\right)}  \right) \\
    =&  - \frac{\sqrt{\frac{r_c}{r_b}} \left(\mathbb{E}\left[\frac{\sigma\left(Z_{\beta}\right)}{\sigma\left(Z_{\hat{\beta}_{S_a}}\right)}\right]-1\right)}{ \left(\frac{r_{b}}{2}-\kappa\right)\left(1-4 e_{\gamma, 0}^{2}\right)}  \cdot \left[ \frac{2 r_{b} h_{a} e_{\gamma, 0}}{\gamma}-4 \kappa e_{\gamma, 0}\left(e_{\gamma, 0}+e^{\frac{\left(\alpha_{a}^{*} \gamma\right)^{2}+\kappa_{a}\left(\sigma_{a}^{*}\right)^{2}}{2}} e_{\gamma,-\alpha_{a}^{*} \gamma}\right)\right.\\
    +& \left.2 \kappa\left(\frac{1}{2}+e^{\frac{\left(\alpha_{a}^{*} \gamma\right)^{2}+\kappa_{a}\left(\sigma_{a}^{*}\right)^{2}}{2}} q_{\gamma,-\alpha_{a}^{*} \gamma}\right)\right] +  o_p(1).
\end{align*} 


Finally, \begin{align*}
    & P_{2}^{T} D X_{S_{b,1}}^{T} w_{S_{b,1}} = \left(\sum_{j \in S_{c}}\frac{A_{j} x_{j}}{\sigma\left(x_{j}^{T} \hat{\beta}_{S_a}\right)}-x_{j}\right)^{T} D \left(\sum_{i \in S_{b}} \frac{A_{i} x_{i}}{\sigma\left(x_{i}^{T} \hat{\beta}_{S_a}\right)}\right) \\
    =&   \left(\sum_{j \in S_{c}}\frac{A_{j} x_{j}}{\sigma\left(x_{j}^{T} \hat{\beta}_{S_a}\right)}\right)^{T} D \left(\sum_{i \in S_{b}} \frac{A_{i} x_{i}}{\sigma\left(x_{i}^{T} \hat{\beta}_{S_a}\right)}\right) + o(1)\\
    =&  \left(\sum_{j \in S_{c}}\frac{A_{j} x_{j}}{\sigma\left(x_{j}^{T} \hat{\beta}_{S_a}\right)}\right)^{T} \left(\sum_{i \in S_{b}} A_{i} x_{i} x_{i}^{\top}\right)^{-1} \left(\sum_{i \in S_{b}} \frac{A_{i} x_{i}}{\sigma\left(x_{i}^{T} \hat{\beta}_{S_a}\right)}\right) \\
    +& 2\frac{\left(\sum_{j \in S_{c}} \frac{A_{j} x_{j}}{\sigma\left(x_{j}^{T} \hat{\beta}_{S_a}\right)}\right)^{\top}\left(\sum_{i \in S_b} A_{i} x_{i} x_{i}^{\top}\right)^{-1} \cdot  \frac{1}{\sqrt{n_b}}\left(\sum_{i \in S_b} A_{i} x_{i}\right) }{1-\frac{2}{n_b}\left(\sum_{i \in S_b} A_{i} x_{i}\right)^{\top}\left(\sum_{i \in S_b} A_{i} x_{i} x_{i}^{\top}\right)^{-1}\left(\sum_{i \in S_b} A_{i} x_{i}\right)} \\
    \cdot &  \left(\sum_{i \in S_{b}} \frac{A_{i} x_{i}}{\sigma\left(x_{i}^{T} \hat{\beta}_{S_a}\right)}\right)^{\top}\left(\sum_{i \in S_b} A_{i} x_{i} x_{i}^{\top}\right)^{-1} \cdot \frac{1}{\sqrt{n_b}}\left(\sum_{i \in S_b} A_{i} x_{i}\right) + o(1).
\end{align*}

By Lemma \ref{lemma 25 second}, we have \begin{align*}
    &  \frac{1}{\sqrt{n_{b} n_{c}}}\left(\sum_{j \in S_{c}} \frac{A_{j} x_{j}}{\sigma\left(x_{j}^{T} \hat{\beta}_{S_a}\right)}\right)^{T}\left(\sum_{i \in S_{b}} A_{i} x_{i} x_{i}^{\top}\right)^{-1}\left(\sum_{i \in S_{b}} \frac{A_{i} x_{i}}{\sigma\left(x_{i}^{T} \hat{\beta}_{S_a}\right)}\right)\\
       =& 2 \sqrt{\frac{r_c}{r_b}} \left\{\mathbb{E}\left[\frac{\sigma^{\prime}\left(Z_{\beta}\right)}{\sigma\left(Z_{\hat{\beta}_{S_{a}}}\right)}\right] \mathbb{E}\left[ \frac{\sigma(Z_\beta)Z_\beta}{\sigma(Z_{\hat{\beta}_{S_a}})} \right] - \mathbb{E}\left[\frac{\sigma\left(Z_{\beta}\right)}{\sigma\left(Z_{\hat{\beta}_{S_{a}}}\right)} - \frac{1}{2}\right] \mathbb{E}\left[\frac{\sigma(Z_\beta) Z_{\hat{\beta}_{S_a}}}{\sigma\left(Z_{\hat{\beta}_{S_a}}\right)}\right] \right \}  + o_p(1).
\end{align*}

Moreover, by Lemma \ref{lemma 11}, we know \begin{align*}
    & \frac{1}{n} \left(\sum_{j \in S_{c}} \frac{A_{j} x_{j}}{\sigma\left(x_{j}^{T} \hat{\beta}_{S_{a}}\right)}\right)^{\top}\left(\sum_{i \in S_{b}} A_{i} x_{i} x_{i}^{\top}\right)^{-1} \left(\sum_{i \in S_{b}} A_{i} x_{i}\right) = \frac{2 r_{c} e_{\gamma, 0} h_{a}}{\gamma} + o_p(1).
\end{align*}

By Lemma \ref{lemma 12}, \begin{align*}
    & \frac{1}{n} \left(\sum_{i \in S_{b}} \frac{A_{i} x_{i}}{\sigma\left(x_{i}^{T} \hat{\beta}_{S_{a}}\right)}\right)^{\top}\left(\sum_{i \in S_{b}} A_{i} x_{i} x_{i}^{\top}\right)^{-1}\left(\sum_{i \in S_{b}} A_{i} x_{i}\right)\\
    =& \frac{2r_b h_a e_{\gamma, 0}}{\gamma} - 4 \kappa e_{\gamma, 0}\left(e_{\gamma, 0}+e^{\frac{\left(\alpha_{a}^{*} \gamma\right)^{2}+\kappa_{a}\left(\sigma_{a}^{*}\right)^{2}}{2}} e_{\gamma,-\alpha_{a}^{*} \gamma}\right)\\
    +& 2 \kappa\left(\frac{1}{2}+e^{\frac{\left(\alpha_{a}^{*} \gamma\right)^{2}+\kappa_{a}\left(\sigma_{a}^{*}\right)^{2}}{2}} q_{\gamma,-\alpha_{a}^{*} \gamma}\right)+ o_p(1).
\end{align*}

Further, similar to Lemma \ref{lemma 10}, we have 
\begin{align*}
    & 1 - \frac{2}{n_b}\left(\sum_{i \in S_{b}} A_{i} x_{i}\right)^{\top}\left(\sum_{i \in S_{b}} A_{i} x_{i} x_{i}^{\top}\right)^{-1}\left(\sum_{i \in S_{b}} A_{i} x_{i}\right)\\
    =& 1 -2\kappa_b - 4  e_{\gamma, 0}^{2}\left(1-2 \kappa_{b}\right)+o_{p}(1) =  (1-2 \kappa_b)(1 - 4 e_{\gamma, 0}^2) +o_{p}(1).
\end{align*}

Therefore \begin{align*}
    & \frac{1}{\sqrt{n_bn_c}}P_{2}^{T} D X_{S_{b, 1}}^{T} w_{S_{b, 1}}  \\
    &= 2 \sqrt{\frac{r_c}{r_b}} \left\{\mathbb{E}\left[\frac{\sigma^{\prime}\left(Z_{\beta}\right)}{\sigma\left(Z_{\hat{\beta}_{S_{a}}}\right)}\right] \mathbb{E}\left[ \frac{\sigma(Z_\beta)Z_\beta}{\sigma(Z_{\hat{\beta}_{S_a}})} \right] - \mathbb{E}\left[\frac{\sigma\left(Z_{\beta}\right)}{\sigma\left(Z_{\hat{\beta}_{S_{a}}}\right)} - \frac{1}{2}\right] \mathbb{E}\left[\frac{\sigma(Z_\beta) Z_{\hat{\beta}_{S_a}}}{\sigma\left(Z_{\hat{\beta}_{S_a}}\right)}\right] \right \} \\
    +& \frac{4 r_c^{0.5} e_{\gamma, 0} h_{a}}{r_b^{1.5}\gamma  \left(1-2 \kappa_{b}\right)\left(1-4 e_{\gamma, 0}^{2}\right)} \cdot \left[ \frac{2r_b h_a e_{\gamma, 0}}{\gamma} - 4 \kappa e_{\gamma, 0}\left(e_{\gamma, 0}+e^{\frac{\left(\alpha_{a}^{*} \gamma\right)^{2}+\kappa_{a}\left(\sigma_{a}^{*}\right)^{2}}{2}} e_{\gamma,-\alpha_{a}^{*} \gamma}\right)\right.\\
    +& \left. 2 \kappa\left(\frac{1}{2}+e^{\frac{\left(\alpha_{a}^{*} \gamma\right)^{2}+\kappa_{a}\left(\sigma_{a}^{*}\right)^{2}}{2}} q_{\gamma,-\alpha_{a}^{*} \gamma}\right) \right] +  o_p(1).
\end{align*}

In conclusion, 
\begin{align*}
    & \frac{1}{\sqrt{n_{b} n_{c}}}\left(\begin{array}{c}
\sum_{j \in S_{c}}\left(\frac{A_{j}}{\sigma\left(x_{j}^{T} \hat{\beta}_{S_{a}}\right)}-1\right) \\
\sum_{j \in S_{c}}\left(\frac{A_{j} x_{j}}{\sigma\left(x_{j}^{T} \hat{\beta}_{S_{a}}\right)}-x_{j}\right)
\end{array}\right)^{T}\left(\tilde{X}_{S_{b, 1}}^{T} \tilde{X}_{S_{b, 1}}\right)^{-1} \tilde{X}_{S_{b, 1}}^{T} w_{S_{b, 1}}\\
=& \sqrt{r_{b} r_{c}} \frac{\left(\mathbb{E}\left[\frac{\sigma\left(Z_{\beta}\right)}{\sigma\left(Z_{\hat{\beta}_{S_{a}}}\right)}\right]-1\right) \mathbb{E}\left[\frac{\sigma\left(Z_{\beta}\right)}{\sigma\left(Z_{\hat{\beta}_{S_{a}}}\right)}\right]}{\left(\frac{r_{b}}{2}-\kappa\right)\left(1-4 e_{\gamma, 0}^{2}\right)} -\frac{2 \sqrt{r_{b} r_{c}} e_{\gamma, 0} h_{a}}{\gamma\left(\frac{r_{b}}{2}-\kappa\right)\left(1-4 e_{\gamma, 0}^{2}\right)} \cdot \mathbb{E}\left[\frac{\sigma\left(Z_{\beta}\right)}{\sigma\left(Z_{\hat{\beta}_{S_{a}}}\right)}\right] \\
    -& \frac{\sqrt{\frac{r_{c}}{r_{b}}}\left(\mathbb{E}\left[\frac{\sigma\left(Z_{\beta}\right)}{\sigma\left(Z_{\hat{\beta}_{S_{a}}}\right)}\right]-1\right)}{\left(\frac{r_{b}}{2}-\kappa\right)\left(1-4 e_{\gamma, 0}^{2}\right)}\left[\frac{2 r_{b} h_{a} e_{\gamma, 0}}{\gamma}-4 \kappa e_{\gamma, 0}\left(e_{\gamma, 0}+e^{\frac{\left(\alpha_{a}^{*} \gamma\right)^{2}+\kappa_{a}\left(\sigma_{a}^{*}\right)^{2}}{2}} e_{\gamma,-\alpha_{a}^{*} \gamma}\right) \right. \\
    +&\left. 2 \kappa\left(\frac{1}{2}+e^{\frac{\left(\alpha_{a}^{*} \gamma\right)^{2}+\kappa_{a}\left(\sigma_{a}^{*}\right)^{2}}{2}} q_{\gamma,-\alpha_{a}^{*} \gamma}\right)\right]\\
    +& 2 \sqrt{\frac{r_c}{r_b}} \left\{\mathbb{E}\left[\frac{\sigma^{\prime}\left(Z_{\beta}\right)}{\sigma\left(Z_{\hat{\beta}_{S_{a}}}\right)}\right] \mathbb{E}\left[ \frac{\sigma(Z_\beta)Z_\beta}{\sigma(Z_{\hat{\beta}_{S_a}})} \right] - \mathbb{E}\left[\frac{\sigma\left(Z_{\beta}\right)}{\sigma\left(Z_{\hat{\beta}_{S_{a}}}\right)}- \frac{1}{2}\right] \mathbb{E}\left[\frac{\sigma(Z_\beta) Z_{\hat{\beta}_{S_a}}}{\sigma\left(Z_{\hat{\beta}_{S_a}}\right)}\right] \right \} \\
    +& \frac{4 r_c^{0.5} e_{\gamma, 0} h_{a}}{r_b^{1.5}\gamma \left(1-2 \kappa_{b}\right)\left(1-4 e_{\gamma, 0}^{2}\right)}\left[ \frac{2r_b h_a e_{\gamma, 0}}{\gamma} - 4 \kappa e_{\gamma, 0}\left(e_{\gamma, 0}+e^{\frac{\left(\alpha_{a}^{*} \gamma\right)^{2}+\kappa_{a}\left(\sigma_{a}^{*}\right)^{2}}{2}} e_{\gamma,-\alpha_{a}^{*} \gamma}\right)\right.\\
    +& \left. 2 \kappa\left(\frac{1}{2}+e^{\frac{\left(\alpha_{a}^{*} \gamma\right)^{2}+\kappa_{a}\left(\sigma_{a}^{*}\right)^{2}}{2}} q_{\gamma,-\alpha_{a}^{*} \gamma}\right) \right] + o_p(1)\\
    =& \sqrt{r_{b} r_{c}} \frac{\left(\mathbb{E}\left[\frac{\sigma\left(Z_{\beta}\right)}{\sigma\left(Z_{\hat{\beta}_{S_{a}}}\right)}\right]-1\right) \mathbb{E}\left[\frac{\sigma\left(Z_{\beta}\right)}{\sigma\left(Z_{\hat{\beta}_{S_{a}}}\right)}\right]}{\left(\frac{r_{b}}{2}-\kappa\right)\left(1-4 e_{\gamma, 0}^{2}\right)} -\frac{2 \sqrt{r_{b} r_{c}} e_{\gamma, 0} h_{a}}{\gamma\left(\frac{r_{b}}{2}-\kappa\right)\left(1-4 e_{\gamma, 0}^{2}\right)} \cdot \mathbb{E}\left[\frac{\sigma\left(Z_{\beta}\right)}{\sigma\left(Z_{\hat{\beta}_{S_{a}}}\right)}\right]\\
    +&  2 \sqrt{\frac{r_c}{r_b}} \left\{\mathbb{E}\left[\frac{\sigma^{\prime}\left(Z_{\beta}\right)}{\sigma\left(Z_{\hat{\beta}_{S_{a}}}\right)}\right] \mathbb{E}\left[ \frac{\sigma(Z_\beta)Z_\beta}{\sigma(Z_{\hat{\beta}_{S_a}})} \right] - \mathbb{E}\left[\frac{\sigma\left(Z_{\beta}\right)}{\sigma\left(Z_{\hat{\beta}_{S_{a}}}\right)} - \frac{1}{2}\right] \mathbb{E}\left[\frac{\sigma(Z_\beta) Z_{\hat{\beta}_{S_a}}}{\sigma\left(Z_{\hat{\beta}_{S_a}}\right)}\right] \right \} \\
    +&  \frac{\sqrt{r_c} \left[ \frac{2 e_{\gamma, 0} h_{a} }{ \gamma }  -  \mathbb{E}\left[\frac{\sigma\left(Z_{\beta}\right)}{\sigma\left(Z_{\hat{\beta}_{S_{a}}}\right)}\right] + 1 \right]}{\sqrt{r_b}\left(\frac{r_{b}}{2}-\kappa\right)\left(1-4 e_{\gamma, 0}^{2}\right)}  \cdot  \left[\frac{2 r_{b} h_{a} e_{\gamma, 0}}{\gamma}\right.\\
    -& \left. 4 \kappa e_{\gamma, 0}\left(e_{\gamma, 0}+e^{\frac{\left(\alpha_{a}^{*} \gamma\right)^{2}+\kappa_{a}\left(\sigma_{a}^{*}\right)^{2}}{2}} e_{\gamma,-\alpha_{a}^{*} \gamma}\right) +2 \kappa\left(\frac{1}{2}+e^{\frac{\left(\alpha_{a}^{*} \gamma\right)^{2}+\kappa_{a}\left(\sigma_{a}^{*}\right)^{2}}{2}} q_{\gamma,-\alpha_{a}^{*} \gamma}\right)\right] + o_p(1),
\end{align*} 
and the asymptotic conditional covariance between $l_{S_{c}, S_{a}}^{T} f\left(\mathcal{E}_{S_{b}}, X_{S_{b}}\right)$ and $ \mathcal{E}_{S_{b}}^{T} V_{S_{b}, S_{a}}$ is equal to the quantity above multiplied by $-\left(\sigma^{(1)}\right)^2$.

\noindent 
\textbf{Proof of (iv):} 
Condition on everything other than the noises, we have \begin{align*}
    & \mathbb{E}\left[ l_{S_{c}, S_{a}}^{T} f\left(\mathcal{E}_{S_{b}}, X_{S_{b}}\right) \mathcal{E}_{S_{b}}^{T} V_{S_{b}, S_{c}} \right]\\
    =& -\mathbb{E}\left[ \frac{1}{\sqrt{n_bn_c}} \left(\begin{array}{l}
\sum_{j \in S_{c}}\left(\frac{A_{j}}{\sigma\left(x_{j}^{T} \hat{\beta}_{S_a}\right)}-1\right) \\
\sum_{j \in S_{c}}\left(\frac{A_{j} x_{j}}{\sigma\left(x_{j}^{T} \hat{\beta}_{S_a}\right)}-x_{j}\right)
\end{array}\right)^{T}\left(\tilde{X}_{S_{b, 1}}^{T} \tilde{X}_{S_{b, 1}}\right)^{-1} \tilde{X}_{S_{b, 1}}^{T} \mathcal{E}_{S_{b, 1}} \sum_{i \in S_{b, 1}} \frac{\epsilon_{i}^{(1)}}{\sigma\left(x_{i}^{T} \hat{\beta}_{S_c}\right)}  \right]\\
=& -\frac{\left(\sigma^{(1)}\right)^2}{\sqrt{n_b n_c}} \left(\begin{array}{c}
\sum_{j \in S_{c}}\left(\frac{A_{j}}{\sigma\left(x_{j}^{T} \hat{\beta}_{S_a}\right)}-1\right) \\
\sum_{j \in S_{c}}\left(\frac{A_{j} x_{j}}{\sigma\left(x_{j}^{T} \hat{\beta}_{S_a}\right)}-x_{j}\right)
\end{array}\right)^{T}\left(\tilde{X}_{S_{b, 1}}^{T} \tilde{X}_{S_{b, 1}}\right)^{-1} \tilde{X}_{S_{b, 1}}^{T} w_{S_{b, 1}},
\end{align*} where $w_{S_{b,1}} $ is the vector that contains $ \frac{1}{\sigma(x_i^T \hat{\beta}_{S_c})} \;\;\forall i \in S_{b} \text{ s.t. } A_i = 1.$

We only need to consider terms that are different from the previous ones: $$ P_{1} C^{T} X_{S_{b, 1}}^{T} w_{S_{b, 1}}, \quad P_{2}^{T} D X_{S_{b, 1}}^{T} w_{S_{b, 1}} .$$

Similar to the previous argument, we have \begin{align*}
    & \frac{1}{\sqrt{n_bn_c}} P_{1} C^{T} X_{S_{b, 1}}^{T} w_{S_{b, 1}}\\ 
    =& -\sqrt{\frac{r_{c}}{r_{b}}} n B \cdot \frac{1}{n_{c}} P_{1} \cdot \frac{1}{n}\left(\sum_{i \in S_{b}} A_{i} x_{i}\right)^{T}\left(\sum_{i \in S_{b}} A_{i} x_{i} x_{i}^{\top}\right)^{-1}\left(\sum_{i \in S_{b}} \frac{A_{i} x_{i}}{\sigma\left(x_{i}^{T} \hat{\beta}_{S_{c}}\right)}\right)\\
    =& -\frac{\sqrt{\frac{r_{c}}{r_{b}}}\left(\mathbb{E}\left[\frac{\sigma\left(Z_{\beta}\right)}{\sigma\left(Z_{\hat{\beta}_{S_{a}}}\right)}\right]-1\right)}{\left(\frac{r_{b}}{2}-\kappa\right)\left(1-4 e_{\gamma, 0}^{2}\right)}\left[\frac{2 r_{b} h_{c} e_{\gamma, 0}}{\gamma}-4 \kappa e_{\gamma, 0}\left(e_{\gamma, 0}+e^{\frac{\left(\alpha_{c}^{*} \gamma\right)^{2}+\kappa_{c}\left(\sigma_{c}^{*}\right)^{2}}{2}} e_{\gamma,-\alpha_{c}^{*} \gamma}\right)\right.\\
    &\left. +2 \kappa\left(\frac{1}{2}+e^{\frac{\left(\alpha_{c}^{*} \gamma\right)^{2}+\kappa_{c}\left(\sigma_{c}^{*}\right)^{2}}{2}} q_{\gamma,-\alpha_{c}^{*} \gamma}\right)\right]
    + o_p(1).
\end{align*}

Next we derive the limit of $P_{2}^{T} D X_{S_{b, 1}}^{T} w_{S_{b, 1}} $. We have \begin{align*}
    & P_{2}^{T} D X_{S_{b, 1}}^{T} w_{S_{b, 1}} =  \left(\sum_{j \in S_{c}} \frac{A_{j} x_{j}}{\sigma\left(x_{j}^{T} \hat{\beta}_{S_a}\right)}-x_{j}\right)^{T} D\left(\sum_{i \in S_{b}} \frac{A_{i} x_{i}}{\sigma\left(x_{i}^{T} \hat{\beta}_{S_c}\right)}\right) \\
    =&  \left(\sum_{j \in S_{c}} \frac{A_{j} x_{j}}{\sigma\left(x_{j}^{T} \hat{\beta}_{S_a}\right)}-x_{j}\right)^{T} \left(\sum_{i \in S_{b}} A_{i} x_{i} x_{i}^{\top}\right)^{-1} \left(\sum_{i \in S_{b}} \frac{A_{i} x_{i}}{\sigma\left(x_{i}^{T} \hat{\beta}_{S_c}\right)}\right)\\
    +& 2 \frac{\left(\sum_{j \in S_{c}} \frac{A_{j} x_{j}}{\sigma\left(x_{j}^{T} \hat{\beta}_{S_a} \right)}- x_j\right)^{\top}\left(\sum_{i \in S_{b}} A_{i} x_{i} x_{i}^{\top}\right)^{-1} \cdot \frac{1}{\sqrt{n_{b}}}\left(\sum_{i \in S_{b}} A_{i} x_{i}\right)}{1-\frac{2}{n}\left(\sum_{i \in S_{b}} A_{i} x_{i}\right)^{\top}\left(\sum_{i \in S_{b}} A_{i} x_{i} x_{i}^{\top}\right)^{-1}\left(\sum_{i \in S_{b}} A_{i} x_{i}\right)}\\
    \cdot & \left(\sum_{i \in S_{b}} \frac{A_{i} x_{i}}{\sigma\left(x_{i}^{T} \hat{\beta}_{S_c}\right)}\right)^{\top}\left(\sum_{i \in S_{b}} A_{i} x_{i} x_{i}^{\top}\right)^{-1} \cdot \frac{1}{\sqrt{n_{b}}}\left(\sum_{i \in S_{b}} A_{i} x_{i}\right).
\end{align*}

By Lemma \ref{lemma 26}, we have \begin{align*}
    & \frac{1}{
    \sqrt{n_bn_c}}\left(\sum_{j \in S_{c}} \frac{A_{j} x_{j}}{\sigma\left(x_{j}^{T} \hat{\beta}_{S_{a}}\right)}\right)^{T}\left(\sum_{i \in S_{b}} A_{i} x_{i} x_{i}^{\top}\right)^{-1}\left(\sum_{i \in S_{b}} \frac{A_{i} x_{i}}{\sigma\left(x_{i}^{T} \hat{\beta}_{S_{c}}\right)}\right)\\
    =& 2 \sqrt{\frac{r_c}{r_b}} \mathbb{E}\left[\frac{\sigma^{\prime}\left(Z_\beta\right)}{\sigma\left(Z_{\hat{\beta}_{S_{c}}}\right)}\right] \mathbb{E}\left[\frac{\sigma(Z_\beta)}{\sigma\left(Z_{\hat{\beta}_{S_{a}}}\right)} Z_\beta\right]-   2 \sqrt{\frac{r_c}{r_b}} \mathbb{E}\left[\frac{\sigma\left(Z_\beta\right) \sigma^{\prime}\left(   Z_{\hat{\beta}_{S_{c}}}\right)}{\sigma^{2}\left(Z_{\hat{\beta}_{S_{c}}}\right)}\right] \left\{ \mathbb{E}\left[\frac{\sigma\left(Z_{\beta}\right)}{\sigma\left(Z_{\hat{\beta}_{S_{a}}}\right)} Z_{\hat{\beta}_{S_{c}}}\right] \right. \\
      &+ \left. \mathbb{E}\left[\sigma\left(Z_{\beta}\right)\left(\frac{1}{\sigma\left(Z_{\hat{\beta}_{S_{a}}}\right)}-1\right) q_{c_{1}}\left(1-\sigma\left(\operatorname{prox}_{q_{c_{1}} \rho}\left(Z_{\hat{\beta}_{S_{c}}+q_{c_{1}}}\right)\right)\right)\right] \right.\\
      &\left. + \mathbb{E}\left[\left(1-\sigma\left(Z_{\beta}\right)\right) q_{c_{1}} \sigma\left(\operatorname{prox}_{q_{c_{1}} \rho}\left(Z_{\hat{\beta}_{S_{c}}}\right)\right)\right] \right\} + o_p(1).
\end{align*}

Further, note that \begin{align*}
    & \mathbb{E}\left[\frac{\sigma\left(Z_{\beta}\right) \sigma^{\prime}\left(Z_{\hat{\beta}_{S_{c}}}\right)}{\sigma^{2}\left(Z_{\hat{\beta}_{S_{c}}}\right)}\right] = \mathbb{E}\left[\frac{\sigma\left(Z_{\beta}\right) }{\sigma\left(Z_{\hat{\beta}_{S_{c}}}\right)} -\frac{1}{2} \right].
\end{align*}

Therefore \begin{align*}
    & \frac{1}{\sqrt{n_{b} n_{c}}} P_{2}^{T} D X_{S_{b, 1}}^{T} w_{S_{b, 1}} \\
    =& 2 \sqrt{\frac{r_c}{r_b}} \mathbb{E}\left[\frac{\sigma^{\prime}\left(Z_\beta\right)}{\sigma\left(Z_{\hat{\beta}_{S_{c}}}\right)}\right] \mathbb{E}\left[\frac{\sigma(Z_\beta)}{\sigma\left(Z_{\hat{\beta}_{S_{a}}}\right)} Z_\beta\right]-   2 \sqrt{\frac{r_c}{r_b}} \mathbb{E}\left[\frac{\sigma\left(Z_\beta\right) }{\sigma\left(Z_{\hat{\beta}_{S_{c}}}\right)} - \frac{1}{2}\right] \left\{ \mathbb{E}\left[\frac{\sigma\left(Z_{\beta}\right)}{\sigma\left(Z_{\hat{\beta}_{S_{a}}}\right)} Z_{\hat{\beta}_{S_{c}}}\right] \right. \\
      &+ \left. \mathbb{E}\left[\sigma\left(Z_{\beta}\right)\left(\frac{1}{\sigma\left(Z_{\hat{\beta}_{S_{a}}}\right)}-1\right) q_{c_{1}}\left(1-\sigma\left(\operatorname{prox}_{q_{c_{1}} \rho}\left(Z_{\hat{\beta}_{S_{c}}+q_{c_{1}}}\right)\right)\right)\right] \right.\\
      &\left.+ \mathbb{E}\left[\left(1-\sigma\left(Z_{\beta}\right)\right) q_{c_{1}} \sigma\left(\operatorname{prox}_{q_{c_{1}} \rho}\left(Z_{\hat{\beta}_{S_{c}}}\right)\right)\right] \right\} \\
      &+ \frac{4 r_c^{0.5} e_{\gamma, 0} h_{a}}{r_b^{1.5}\gamma  \left(1-2 \kappa_{b}\right)\left(1-4 e_{\gamma, 0}^{2}\right)} \left[ \frac{2r_b h_c e_{\gamma, 0}}{\gamma} - 4 \kappa e_{\gamma, 0}\left(e_{\gamma, 0}+e^{\frac{\left(\alpha_{c}^{*} \gamma\right)^{2}+\kappa_{c}\left(\sigma_{c}^{*}\right)^{2}}{2}} e_{\gamma,-\alpha_{c}^{*} \gamma}\right) \right. \\
       &+ 2\left. \kappa\left(\frac{1}{2}+e^{\frac{\left(\alpha_{c}^{*} \gamma\right)^{2}+\kappa_{c}\left(\sigma_{c}^{*}\right)^{2}}{2}} q_{\gamma,-\alpha_{c}^{*} \gamma}\right) \right].
\end{align*}

The limit of other terms above can be derived similarly as before, so we have \begin{align*}
    & \frac{1}{\sqrt{n_{b} n_{c}}}\left(\begin{array}{c}
\sum_{j \in S_{c}}\left(\frac{A_{j}}{\sigma\left(x_{j}^{T} \hat{\beta}_{S_{a}}\right)}-1\right) \\
\sum_{j \in S_{c}}\left(\frac{A_{j} x_{j}}{\sigma\left(x_{j}^{T} \hat{\beta}_{S_{a}}\right)}-x_{j}\right)
\end{array}\right)^{T}\left(\tilde{X}_{S_{b, 1}}^{T} \tilde{X}_{S_{b, 1}}\right)^{-1} \tilde{X}_{S_{b, 1}}^{T} w_{S_{b, 1}}\\
    =& \sqrt{r_{b} r_{c}} \frac{\left(\mathbb{E}\left[\frac{\sigma\left(Z_{\beta}\right)}{\sigma\left(Z_{\hat{\beta}_{S_{a}}}\right)}\right]-1\right) \mathbb{E}\left[\frac{\sigma\left(Z_{\beta}\right)}{\sigma\left(Z_{\hat{\beta}_{S_{a}}}\right)}\right]}{\left(\frac{r_{b}}{2}-\kappa\right)\left(1-4 e_{\gamma, 0}^{2}\right)} -\frac{2 \sqrt{r_{b} r_{c}} e_{\gamma, 0} h_{a}}{\gamma\left(\frac{r_{b}}{2}-\kappa\right)\left(1-4 e_{\gamma, 0}^{2}\right)} \cdot \mathbb{E}\left[\frac{\sigma\left(Z_{\beta}\right)}{\sigma\left(Z_{\hat{\beta}_{S_{a}}}\right)}\right] \\
    &+   \frac{\sqrt{r_c} \left[ \frac{2 e_{\gamma, 0} h_{a} }{ \gamma }  -  \mathbb{E}\left[\frac{\sigma\left(Z_{\beta}\right)}{\sigma\left(Z_{\hat{\beta}_{S_{a}}}\right)}\right] + 1 \right]}{\sqrt{r_b}\left(\frac{r_{b}}{2}-\kappa\right)\left(1-4 e_{\gamma, 0}^{2}\right)}  \\
     &\cdot \left[\frac{2 r_{b} h_{c} e_{\gamma, 0}}{\gamma}-4 \kappa e_{\gamma, 0}\left(e_{\gamma, 0}+e^{\frac{\left(\alpha_{c}^{*} \gamma\right)^{2}+\kappa_{c}\left(\sigma_{c}^{*}\right)^{2}}{2}} e_{\gamma,-\alpha_{c}^{*} \gamma}\right)+2 \kappa\left(\frac{1}{2}+e^{\frac{\left(\alpha_{c}^{*} \gamma\right)^{2}+\kappa_{c}\left(\sigma_{c}^{*}\right)^{2}}{2}} q_{\gamma,-\alpha_{c}^{*} \gamma}\right)\right]\\
    &+ 2 \sqrt{\frac{r_c}{r_b}} \mathbb{E}\left[\frac{\sigma^{\prime}\left(Z_\beta\right)}{\sigma\left(Z_{\hat{\beta}_{S_{c}}}\right)}\right] \mathbb{E}\left[\frac{\sigma(Z_\beta)}{\sigma\left(Z_{\hat{\beta}_{S_{a}}}\right)} Z_\beta\right]-   2 \sqrt{\frac{r_c}{r_b}} \mathbb{E}\left[\frac{\sigma\left(Z_\beta\right) }{\sigma\left(Z_{\hat{\beta}_{S_{c}}}\right)} - \frac{1}{2}\right] \left\{ \mathbb{E}\left[\frac{\sigma\left(Z_{\beta}\right)}{\sigma\left(Z_{\hat{\beta}_{S_{a}}}\right)} Z_{\hat{\beta}_{S_{c}}}\right] \right. \\
      &+ \left. \mathbb{E}\left[\sigma\left(Z_{\beta}\right)\left(\frac{1}{\sigma\left(Z_{\hat{\beta}_{S_{a}}}\right)}-1\right) q_{c_{1}}\left(1-\sigma\left(\operatorname{prox}_{q_{c_{1}} \rho}\left(Z_{\hat{\beta}_{S_{c}}+q_{c_{1}}}\right)\right)\right)\right] \right.\\
      &\left. +  \mathbb{E}\left[\left(1-\sigma\left(Z_{\beta}\right)\right) q_{c_{1}} \sigma\left(\operatorname{prox}_{q_{c_{1}} \rho}\left(Z_{\hat{\beta}_{S_{c}}}\right)\right)\right] \right\}+ o_p(1), 
\end{align*} and the asymptotic conditional covariance between $l_{S_{c}, S_{a}}^{T} f\left(\mathcal{E}_{S_{b}}, X_{S_{b}}\right)$ and $ \mathcal{E}_{S_{b}}^{T} V_{S_{b}, S_{c}}$ is equal to the quantity
above multiplied by $-\left(\sigma^{(1)}\right)^2$.

 \noindent 
 \textbf{Proof of (v):} 
 Conditioned on everything other than noises, we have that, almost surely, \begin{align*}
  &  \text{Cov}\left(\mathcal{E}_{S_{b}}^{T} V_{S_{b}, S_{a}}, \mathcal{E}_{S_{b}}^{T} V_{S_{b}, S_{c}}  \right)  = \mathbb{E} \left[ V_{S_{b}, S_{a}}^{T} \mathcal{E}_{S_{b}}  \mathcal{E}_{S_{b}}^{T} V_{S_{b}, S_{c}}  \right] =  \left(\sigma^{(1)}\right)^2 \mathbb{E} \left[ V_{S_{b}, S_{a}}^{T}  V_{S_{b}, S_{c}}  \right] \\
  =&  \left(\sigma^{(1)}\right)^2 \cdot \frac{1}{n_b} \sum_{i \in S_b} \frac{A_{i}}{\sigma\left(x_{i}^{\top} \hat{\beta}_{S_a}\right) \sigma\left(x_{i}^{\top} \hat{\beta}_{S_c}\right)}   \\
  =& \left(\sigma^{(1)}\right)^2  \mathbb{E} \left[  \frac{\sigma\left( x_1^T\beta \right)}{\sigma\left(x_{1}^{\top} \hat{\beta}_{S_a}\right) \sigma\left(x_{1}^{\top} \hat{\beta}_{S_c}\right)}  \right]  + o(1)= \left(\sigma^{(1)}\right)^2  \mathbb{E} \left[  \frac{\sigma\left( Z_\beta \right)}{\sigma\left(Z_{ \hat{\beta}_{S_a}}\right) \sigma\left(Z_{ \hat{\beta}_{S_c}}\right)}  \right] + o(1), 
\end{align*} where \begin{align*}
    & \text{Cov}\left(Z_{ \hat{\beta}_{S_a}}, Z_{ \hat{\beta}_{S_c}}\right) = \lim_{n \rightarrow \infty} \frac{1}{n} \hat{\beta}_{S_{a}}^T \hat{\beta}_{S_{c}} .
\end{align*}

Further, by an extension of Theorem 4 in \cite{Sur14516}, we have $$\frac{1}{n} \hat{\beta}_{S_{a}}^{T} \hat{\beta}_{S_{c}}=\kappa \cdot \frac{1}{p} \hat{\beta}_{S_{a}}^{T} \hat{\beta}_{S_{c}}=\kappa \mathbb{E}\left[\left(\sigma_{a}^{\star} Z_{a}+\alpha_{a}^{*} \beta\right)\left(\sigma_{c}^{\star} Z_{c}+\alpha_{c}^{*} \beta\right)\right]+o(1)=\alpha_{a}^{*} \alpha_{c}^{*} \gamma^{2}+o(1).$$

 \noindent 
 \textbf{Proof of (vi):} Conditioning on everything other than the noises, we have \begin{align*}
    & \text{Cov}\left( l_{S_{c}, S_{a}}^{T} f\left(\mathcal{E}_{S_{b}}, X_{S_{b}}\right), \quad l_{S_{a}, S_{c}}^{T} f\left(\mathcal{E}_{S_{b}}, X_{S_{b}}\right)\right) =  \mathbb{E}\left[     l_{S_{c}, S_{a}}^{T} f\left(\mathcal{E}_{S_{b}}, X_{S_{b}}\right) \cdot l_{S_{a}, S_{c}}^{T} f\left(\mathcal{E}_{S_{b}}, X_{S_{b}}\right)      \right]\\
    =& \mathbb{E}\left[      l_{S_{c}, S_{a}}^{T} \left(\tilde{X}_{S_{b, 1}}^{T} \tilde{X}_{S_{b, 1}}\right)^{-1}\tilde{X}_{S_{b, 1}}^{T} \mathcal{E}_{S_{b,1}} \mathcal{E}_{S_{b,1}}^{T} \tilde{X}_{S_{b, 1}} \left(\tilde{X}_{S_{b, 1}}^{T} \tilde{X}_{S_{b, 1}}\right)^{-1}  l_{S_{a}, S_{c}}   \right] \\
    =&  \left(\sigma^{(1)}\right)^2     l_{S_{c}, S_{a}}^{T} \left(\tilde{X}_{S_{b, 1}}^{T} \tilde{X}_{S_{b, 1}}\right)^{-1} l_{S_{a}, S_{c}}  .
\end{align*}

Note that \begin{align*}
    &  l_{S_{c}, S_{a}}^{T} \left(\tilde{X}_{S_{b, 1}}^{T} \tilde{X}_{S_{b, 1}}\right)^{-1} l_{S_{a}, S_{c}} = l_{S_{c}, S_{a}}^{T} \left(\sum_{i \in S_{b}} A_{i}\left(\begin{array}{c}
1 \\
x_{i}
\end{array}\right)\left(\begin{array}{ll}
1 & x_{i}
\end{array}\right)\right)^{-1} l_{S_{a}, S_{c}}=  l_{S_{c}, S_{a}}^{T} \left[\begin{array}{ll}
B & C^{\top} \\
C & D
\end{array}\right] l_{S_{a}, S_{c}} ,
\end{align*} where $$
B=\left[\sum_{i \in S_{b}} A_{i}-\left(\sum_{i \in S_{b}} A_{i} x_{i}^{\top}\right)\left(\sum_{i \in S_{b}} A_{i} x_{i} x_{i}^{\top}\right)^{-1}\left(\sum_{i \in S_{b}} A_{i} x_{i}\right)\right]^{-1},$$ $$ C=-B\left(\sum_{i \in S_{b}} A_{i} x_{i} x_{i}^{\top}\right)^{-1}\left(\sum_{i \in S_{b}} A_{i} x_{i}\right) ,$$ $$
D=\left[\sum_{i \in S_{b}} A_{i} x_{i} x_{i}^{\top}-\left(\sum_{i \in S_{b}} A_{i}\right)^{-1}\left(\sum_{i \in S_{b}} A_{i} x_{i}\right)\left(\sum_{i \in S_{b}} A_{i} x_{i}^{\top}\right)\right]^{-1}.
$$

Apply Lemma \ref{lemma 11}, we have \begin{align*}
    &l_{S_{c}, S_{a}}^{T}\left(\tilde{X}_{S_{b,1}}^{T} \tilde{X}_{S_{b,1}}\right)^{-1} l_{S_{a}, S_{c}} = \frac{1}{\sqrt{n_an_c}} \left[ \sum_{i \in S_a} \left(\frac{A_i}{\sigma(x_i^{\top} \hat{\beta}_{S_c})}-1 \right) \right] \left[ \sum_{j \in S_c} \left(\frac{A_j}{\sigma(x_j^{\top} \hat{\beta}_{S_a})}-1 \right) \right] B \\
    &+ \frac{1}{\sqrt{n_an_c}} \left[ \sum_{i \in S_a} \left(\frac{A_i}{\sigma(x_i^{\top} \hat{\beta}_{S_c})}-1 \right) \right] C^{\top} \left[\sum_{j \in S_c} \left(\frac{A_{j}x_j}{\sigma(x_{j}^{\top} \hat{\beta}_{S_a})}-x_j \right)\right] \\
    &+ \frac{1}{\sqrt{n_an_c}} \left[ \sum_{j \in S_c} \left(\frac{A_j}{\sigma(x_j^{\top} \hat{\beta}_{S_a})}-1 \right) \right] C^{\top} \left[\sum_{i \in S_a} \left(\frac{A_{i}x_i}{\sigma(x_{i}^{\top} \hat{\beta}_{S_c})}-x_i \right)\right]\\
    &+  \frac{1}{\sqrt{n_a n_c}} \left[\sum_{i \in S_a} \left(\frac{A_{i}x_i}{\sigma(x_{i}^{\top} \hat{\beta}_{S_c})}-x_i \right)^{\top}\right] D \left[\sum_{j \in S_c} \left(\frac{A_{j}x_j}{\sigma(x_{j}^{\top} \hat{\beta}_{S_a})}-x_j \right)\right] \\
    =&  \frac{\sqrt{n_an_c} \left\{ \mathbb{E}\left[ \frac{\sigma(Z_\beta)}{\sigma \left(Z_{\hat{\beta}_{S_a}} \right)} \right] -1 \right\} \left\{ \mathbb{E}\left[ \frac{\sigma(Z_\beta)}{\sigma \left(Z_{\hat{\beta}_{S_c}} \right)} \right] -1 \right\}}{\left(\frac{n_b}{2} - p\right)\left(1-4 e_{\gamma, 0}^{2}\right) } -  \frac{2\sqrt{n_an_c} h_c e_{\gamma,0} \left\{ \mathbb{E}\left[ \frac{\sigma(Z_\beta)}{\sigma \left(Z_{\hat{\beta}_{S_a}} \right)} \right] -1 \right\} }{\gamma \left(\frac{n_b}{2} - p\right)\left(1-4 e_{\gamma, 0}^{2}\right)}\\
    &-    \frac{2\sqrt{n_an_c} h_a e_{\gamma,0} \left\{ \mathbb{E}\left[ \frac{\sigma(Z_\beta)}{\sigma \left(Z_{\hat{\beta}_{S_c}} \right)} \right] -1 \right\} }{\gamma \left(\frac{n_b}{2} - p\right)\left(1-4 e_{\gamma, 0}^{2}\right)} \\
    &+   \frac{1}{\sqrt{n_a n_c}} \left[\sum_{i \in S_a} \left(\frac{A_{i}x_i}{\sigma(x_{i}^{\top} \hat{\beta}_{S_c})}-x_i \right)^{\top}\right] D \left[\sum_{j \in S_c} \left(\frac{A_{j}x_j}{\sigma(x_{j}^{\top} \hat{\beta}_{S_a})}-x_j \right)\right].
\end{align*}

By Lemma \ref{diff lemma 2}, \begin{align*}
    &  \frac{1}{\sqrt{n_a n_c}} \left[\sum_{i \in S_a} \left(\frac{A_{i}x_i}{\sigma(x_{i}^{\top} \hat{\beta}_{S_c})}-x_i \right)^{\top}\right] D \left[\sum_{j \in S_c} \left(\frac{A_{j}x_j}{\sigma(x_{j}^{\top} \hat{\beta}_{S_a})}-x_j \right)\right]\\
    =& \frac{1}{\sqrt{n_a n_c}} \left[\sum_{i \in S_a} \left(\frac{A_{i}x_i}{\sigma(x_{i}^{\top} \hat{\beta}_{S_c})}-x_i \right)^{\top}\right] \left(\sum_{k \in S_b} A_{k} x_{k} x_{k}^{\top}\right)^{-1} \left[\sum_{j \in S_c} \left(\frac{A_{j}x_j}{\sigma(x_{j}^{\top} \hat{\beta}_{S_a})}-x_j \right)\right]  \\
    &+  \frac{\left(\frac{1}{n_b} \sum_{k \in S_{b}} A_{k}\right)^{-1} \cdot \left[\sum_{i \in S_a} \left(\frac{A_{i}x_i}{\sigma(x_{i}^{\top} \hat{\beta}_{S_c})}-x_i \right)^{\top}\right] \left(\sum_{k \in S_{b}} A_{k} x_{k} x_{k}^{\top}\right)^{-1} \cdot \frac{1}{n_b} \left(\sum_{k \in S_{b}} A_{k} x_{k}\right) }{1-\left(\frac{1}{n_b} \sum_{k \in S_{b}} A_{k}\right)^{-1} \cdot \frac{1}{n_b}\left(\sum_{k \in S_{b}} A_{k} x_{k}\right)^{\top}\left(\sum_{k \in S_{b}} A_{k} x_{k} x_{k}^{\top}\right)^{-1}\left(\sum_{k \in S_{b}} A_{k}  x_{k}\right)} \\
    \cdot &  \left[\sum_{j \in S_c} \left(\frac{A_{j}x_j}{\sigma(x_{j}^{\top} \hat{\beta}_{S_a})}-x_j \right)^{\top}\right] \left(\sum_{k \in S_{b}} A_{k} x_{k} x_{k}^{\top}\right)^{-1} \cdot \frac{1}{n_b} \left(\sum_{k \in S_{b}} A_{k} x_{k}\right)  \\
    =&   \frac{2 n}{n_b-2 p} \cdot \frac{1}{\sqrt{n_a n_c}} \left[\sum_{i \in S_a} \left(\frac{A_{i}x_i}{\sigma(x_{i}^{\top} \hat{\beta}_{S_c})}-x_i \right)^{\top}\right]  \left[\sum_{j \in S_c} \left(\frac{A_{j}x_j}{\sigma(x_{j}^{\top} \hat{\beta}_{S_a})}-x_j \right)\right]     \\
    &+ \frac{4 n_a n_c h_a h_c e^2_{\gamma,0} }{n_b \gamma^2 \left(\frac{n_b}{2}-p\right)\left(1-4 e^2_{\gamma, 0}\right)} + o_p(1).
\end{align*}

By Lemma \ref{vec_lemma_2}, we know the expression above is equal to 
\begin{align*}
& \frac{2n \sqrt{r_{a} r_{c}} \gamma^{2}}{n_b - 2p} \left\{\left(\mathbb{E}\left[\frac{\sigma\left(Z_{\beta}\right)\left(1-\sigma\left(Z_{\beta}\right)\right)}{\sigma\left(Z_{\hat{\beta}_{S_{a}}}\right)}\right]-\alpha_{a}^{*} \mathbb{E}\left[\frac{\sigma\left(Z_{\beta}\right)}{\sigma\left(Z_{\hat{\beta}_{S_{a}}}\right)}-\frac{1}{2}\right]\right)\right.\\
&\left. 
\left(\mathbb{E}\left[\frac{\sigma\left(Z_{\beta}\right)\left(1-\sigma\left(Z_{\beta}\right)\right)}{\sigma\left(Z_{\hat{\beta}_{S_{c}}}\right)}\right]-\alpha_{c}^{*} \mathbb{E}\left[\frac{\sigma\left(Z_{\beta}\right)}{\sigma\left(Z_{\hat{\beta}_{S_{c}}}\right)}-\frac{1}{2}\right]\right)\right\} \\
    &- \frac{2n \lambda_{c}^{*} \sqrt{r_{a} r_{c}} }{n_b - 2p} \mathbb{E}\left[\frac{\sigma\left(Z_{\beta}\right) \sigma^{\prime}\left(Z_{\beta_{S_{c}}}\right)}{\sigma^{2}\left(Z_{\beta_{S_{c}}}\right)}\right]  \mathbb{E}\left[\sigma\left(Z_{\beta}\right)\left(\frac{1}{\sigma\left(Z_{\hat{\beta}_{S_ a}}\right)}-1\right)\left(1-\sigma\left(\operatorname{prox}_{\lambda_{c}^{*} \rho}\left(Z_{\hat{\beta}_{S _c}}+\lambda_{c}^{*}\right)\right)\right)\right] \\
    &- \frac{2n \lambda_{a}^{*} \sqrt{r_{a} r_{c}} }{n_b - 2p} \mathbb{E}\left[\frac{\sigma\left(Z_{\beta}\right) \sigma^{\prime}\left(Z_{\beta_{S_{a}}}\right)}{\sigma^{2}\left(Z_{\beta_{S_{a}}}\right)}\right]  \mathbb{E}\left[\sigma\left(Z_{\beta}\right)\left(\frac{1}{\sigma\left(Z_{\hat{\beta}_{S_ c}}\right)}-1\right)\left(1-\sigma\left(\operatorname{prox}_{\lambda_{a}^{*} \rho}\left(Z_{\hat{\beta}_{S_ a}}+\lambda_{a}^{*}\right)\right)\right)\right] \\
    &- \frac{2 n \lambda_{c}^{*} \sqrt{r_{a} r_{c}}}{n_{b}-2 p} \mathbb{E}\left[\frac{\sigma\left(Z_{\beta}\right) \sigma^{\prime}\left(Z_{\beta_{S_{c}}}\right)}{\sigma^{2}\left(Z_{\beta_{S_{c}}}\right)}\right] \mathbb{E}\left[\left(1-\sigma\left(Z_{\beta}\right)\right)  \sigma\left(\operatorname{prox}_{\lambda_{c}^{*} \rho}\left(Z_{\hat{\beta}_{S_ c}}\right)\right)\right] \\
    &- \frac{2 n \lambda_{a}^{*} \sqrt{r_{a} r_{c}}}{n_{b}-2 p} \mathbb{E}\left[\frac{\sigma\left(Z_{\beta}\right) \sigma^{\prime}\left(Z_{\beta_{S_{a}}}\right)}{\sigma^{2}\left(Z_{\beta_{S_{a}}}\right)}\right] \mathbb{E}\left[\left(1-\sigma\left(Z_{\beta}\right)\right)  \sigma\left(\operatorname{prox}_{\lambda_{a}^{*} \rho}\left(Z_{\hat{\beta}_{S_ a}}\right)\right)\right]\\
    &+ \frac{4 n_a n_c h_a h_c e^2_{\gamma,0} }{n_b \gamma^2 \left(\frac{n_b}{2}-p\right)\left(1-4 e^2_{\gamma, 0}\right)} + o_p(1).
\end{align*}

Further, we know \begin{align*}
    & \mathbb{E}\left[\frac{\sigma\left(Z_{\beta}\right) \sigma^{\prime}\left(Z_{\beta_{S_{c}}}\right)}{\sigma^{2}\left(Z_{\beta_{S_{c}}}\right)}\right] =  \mathbb{E}\left[\frac{\sigma\left(Z_{\beta}\right)\left(1- \sigma\left(Z_{\beta_{S_{c}}}\right)\right)}{\sigma\left(Z_{\beta_{S_{c}}}\right)}\right] =  \mathbb{E}\left[\frac{\sigma\left(Z_{\beta}\right)}{\sigma\left(Z_{\beta_{S_{c}}}\right)}\right] - \frac{1}{2},
\end{align*} and similarly $$ \mathbb{E}\left[\frac{\sigma\left(Z_{\beta}\right) \sigma^{\prime}\left(Z_{\beta_{S_{a}}}\right)}{\sigma^{2}\left(Z_{\beta_{S_{a}}}\right)}\right] = \mathbb{E}\left[\frac{\sigma\left(Z_{\beta}\right)}{\sigma\left(Z_{\beta_{S_{a}}}\right)}\right]-\frac{1}{2}. $$

Thus \begin{align*}
    & l_{S_{c}, S_{a}}^{T}\left(\tilde{X}_{S_{b, 1}}^{T} \tilde{X}_{S_{b, 1}}\right)^{-1} l_{S_{a}, S_{c}} \\
    =& \frac{\sqrt{n_an_c} \left\{ \mathbb{E}\left[ \frac{\sigma(Z_\beta)}{\sigma \left(Z_{\hat{\beta}_{S_a}} \right)} \right] -1 \right\} \left\{ \mathbb{E}\left[ \frac{\sigma(Z_\beta)}{\sigma \left(Z_{\hat{\beta}_{S_c}} \right)} \right] -1 \right\}}{\left(\frac{n_b}{2} - p\right)\left(1-4 e_{\gamma, 0}^{2}\right) } -  \frac{2\sqrt{n_an_c} h_c e_{\gamma,0} \left\{ \mathbb{E}\left[ \frac{\sigma(Z_\beta)}{\sigma \left(Z_{\hat{\beta}_{S_a}} \right)} \right] -1 \right\} }{\gamma \left(\frac{n_b}{2} - p\right)\left(1-4 e_{\gamma, 0}^{2}\right)}\\
    &-    \frac{2\sqrt{n_an_c} h_a e_{\gamma,0} \left\{ \mathbb{E}\left[ \frac{\sigma(Z_\beta)}{\sigma \left(Z_{\hat{\beta}_{S_c}} \right)} \right] -1 \right\} }{\gamma \left(\frac{n_b}{2} - p\right)\left(1-4 e_{\gamma, 0}^{2}\right)} + \frac{4 n_{a} n_{c} h_{a} h_{c} e_{\gamma, 0}^{2}}{n_{b} \gamma^{2}\left(\frac{n_{b}}{2}-p\right)\left(1-4 e_{\gamma, 0}^{2}\right)} \\
    &+ \frac{2n \sqrt{r_{a} r_{c}} \gamma^{2}}{n_b - 2p} \left\{\left(\mathbb{E}\left[\frac{\sigma\left(Z_{\beta}\right)\left(1-\sigma\left(Z_{\beta}\right)\right)}{\sigma\left(Z_{\hat{\beta}_{S_{a}}}\right)}\right]-\alpha_{a}^{*} \mathbb{E}\left[\frac{\sigma\left(Z_{\beta}\right)}{\sigma\left(Z_{\hat{\beta}_{S_{a}}}\right)}-\frac{1}{2}\right]\right) \right.\\
   &\left. \left(\mathbb{E}\left[\frac{\sigma\left(Z_{\beta}\right)\left(1-\sigma\left(Z_{\beta}\right)\right)}{\sigma\left(Z_{\hat{\beta}_{S_{c}}}\right)}\right]-\alpha_{c}^{*} \mathbb{E}\left[\frac{\sigma\left(Z_{\beta}\right)}{\sigma\left(Z_{\hat{\beta}_{S_{c}}}\right)}-\frac{1}{2}\right]\right)\right\} \\
    &- \frac{2n \lambda_{c}^{*} \sqrt{r_{a} r_{c}} }{n_b - 2p} \mathbb{E}\left[\frac{\sigma\left(Z_{\beta}\right) }{\sigma\left(Z_{\beta_{S_{c}}}\right)} - \frac{1}{2}\right]  \mathbb{E}\left[\sigma\left(Z_{\beta}\right)\left(\frac{1}{\sigma\left(Z_{\hat{\beta}_{S_ a}}\right)}-1\right)\left(1-\sigma\left(\operatorname{prox}_{\lambda_{c}^{*} \rho}\left(Z_{\hat{\beta}_{S _c}}+\lambda_{c}^{*}\right)\right)\right)\right] \\
    &- \frac{2n \lambda_{a}^{*} \sqrt{r_{a} r_{c}} }{n_b - 2p} \mathbb{E}\left[\frac{\sigma\left(Z_{\beta}\right) }{\sigma\left(Z_{\beta_{S_{a}}}\right)} - \frac{1}{2}\right] \mathbb{E}\left[\sigma\left(Z_{\beta}\right)\left(\frac{1}{\sigma\left(Z_{\hat{\beta}_{S_ c}}\right)}-1\right)\left(1-\sigma\left(\operatorname{prox}_{\lambda_{a}^{*} \rho}\left(Z_{\hat{\beta}_{S_ a}}+\lambda_{a}^{*}\right)\right)\right)\right] \\
    &- \frac{2 n \lambda_{c}^{*} \sqrt{r_{a} r_{c}}}{n_{b}-2 p} \mathbb{E}\left[\frac{\sigma\left(Z_{\beta}\right) }{\sigma\left(Z_{\beta_{S_{c}}}\right)} - \frac{1}{2}\right] \mathbb{E}\left[\left(1-\sigma\left(Z_{\beta}\right)\right)  \sigma\left(\operatorname{prox}_{\lambda_{c}^{*} \rho}\left(Z_{\hat{\beta}_{S_ c}}\right)\right)\right] \\
    &- \frac{2 n \lambda_{a}^{*} \sqrt{r_{a} r_{c}}}{n_{b}-2 p} \mathbb{E}\left[\frac{\sigma\left(Z_{\beta}\right) }{\sigma\left(Z_{\beta_{S_{a}}}\right)} - \frac{1}{2}\right] \mathbb{E}\left[\left(1-\sigma\left(Z_{\beta}\right)\right)  \sigma\left(\operatorname{prox}_{\lambda_{a}^{*} \rho}\left(Z_{\hat{\beta}_{S_ a}}\right)\right)\right] + o_p(1).
\end{align*}

The asymptotic conditional covariance between $l_{S_{c}, S_{a}}^{T} f\left(\mathcal{E}_{S_{b}}, X_{S_{b}}\right)$ and $ l_{S_{a}, S_{c}}^{T} f\left(\mathcal{E}_{S_{b}}, X_{S_{b}}\right)$ is equal to the quantity above multiplied by $\left(\sigma^{(1)}\right)^2 $. 
\end{proof}

\section{Distributional Identities and Probabilistic asymptotics} 
We collect some elementary probabilistic identities and asymptotic results in this section. These results will be used crucially in the subsequent discussion. 

 \begin{lemma}\label{lemma 0}
   Fix any $\gamma > 0$. Let $z$ be a standard normal random variable. Then $\mathbb{E}\left[\frac{z^2}{1+e^{-\gamma z}}\right] = \frac{1}{2}.$
  \end{lemma}
  \begin{proof}
  \begin{align*}
     &  \mathbb{E}\left[\frac{z^2}{1+e^{-\gamma z}}\right] =  \int_{-\infty}^{\infty} \frac{1}{\sqrt{2 \pi}}e^{-\frac{z^2}{2}}\frac{z^2}{1+e^{-\gamma z}} dz \\
     =& \int_{0}^{\infty} \frac{1}{\sqrt{2 \pi}}e^{-\frac{z^2}{2}}\frac{z^2}{1+e^{-\gamma z}} dz + \int_{-\infty}^{0} \frac{1}{\sqrt{2 \pi}}e^{-\frac{z^2}{2}}\frac{z^2}{1+e^{-\gamma z}} dz\\
       =& \int_{0}^{\infty} \frac{1}{\sqrt{2 \pi}}e^{-\frac{z^2}{2}}\frac{z^2}{1+e^{-\gamma z}} dz + \int_{-\infty}^{0} \frac{1}{\sqrt{2 \pi}}e^{-\frac{z^2}{2}}z^2  \left(1-\frac{1}{1+e^{\gamma z}} \right) dz\\
       =& \int_{0}^{\infty} \frac{1}{\sqrt{2 \pi}}e^{-\frac{z^2}{2}}\frac{z^2}{1+e^{-\gamma z}} dz +\int_{0}^{\infty} \frac{1}{\sqrt{2 \pi}}e^{-\frac{z^2}{2}}z^2  \left(1-\frac{1}{1+e^{-\gamma z}} \right) dz\\ 
       =&  \int_{0}^{\infty} \frac{1}{\sqrt{2 \pi}} e^{-\frac{z^{2}}{2}} z^{2} dz = \frac{1}{2}.
  \end{align*}
  \end{proof}

  \begin{lemma}\label{lemma 1} Fix any $p,n \in \mathbb{N}^{+}$.
   Let $x \sim N(0, \frac{1}{n}I_p)$, and $\beta \in \mathbb{R}^p$ be any fixed vector. Let $ A$ be a Bernoulli random variable with success probability  $\sigma\left(x^{\top} \beta\right)$. Then $ \mathbb{P}( A = 1) = \frac{1}{2}$.
  \end{lemma}
  
  \begin{proof}
  Since $x^{\top} \beta \sim N(0, \frac{\| \beta \|_2^2}{n}) $, we have 
\begin{align*}
& \mathbb{P}( A = 1) 
    =\mathbb{E}[ \sigma(x^{\top}\beta) ]=  \int_{-\infty}^{\infty} \sigma(x) \frac{1}{\sqrt{2\pi} \frac{\| \beta \|}{\sqrt{n}}} e^{-\frac{x^2  }{2 \frac{\| \beta \|_2^2}{n} }} dx\\
    =& \int_{0}^{\infty} \sigma(x) \frac{1}{\sqrt{2\pi} \frac{\| \beta \|}{\sqrt{n}}} e^{-\frac{x^2  }{2 \frac{\| \beta \|_2^2}{n} }} dx + \int_{-\infty}^{0} (1-\sigma(-x)) \frac{1}{\sqrt{2\pi} \frac{\| \beta \|}{\sqrt{n}}} e^{-\frac{x^2  }{2 \frac{\| \beta \|_2^2}{n} }} dx\\
    =& \int_{0}^{\infty} \sigma(x) \frac{1}{\sqrt{2\pi} \frac{\| \beta \|}{\sqrt{n}}} e^{-\frac{x^2  }{2 \frac{\| \beta \|_2^2}{n} }} dx +  \int_{0}^{\infty}  (1-\sigma(y)) \frac{1}{\sqrt{2\pi} \frac{\| \beta \|}{\sqrt{n}}} e^{-\frac{y^2  }{2 \frac{\| \beta \|_2^2}{n} }} dy = \frac{1}{2}.
\end{align*}
  \end{proof}

  \begin{lemma} \label{conditional_expectation_lemma}
  For any $n \in \mathbb{N}^+$,
   let $x_1 \sim N(0, \frac{1}{n})$. Let $\{\beta_1\}_{n \in \mathbb{N}^+} $ be a sequence of deterministic scalars such that $ \lim_{n \rightarrow \infty}  \frac{\beta_1}{\sqrt{n}} = \gamma$ for some constant $\gamma > 0$. Let $ A$ be a Bernoulli random variable with success probability  $\sigma\left(x_1 \beta_1\right)$. Then $$\lim_{n \rightarrow \infty} \mathbb{E}\left[\sqrt{n} x_{1} \mid A=1\right] =2e_{\gamma, 0}. $$
  \end{lemma}
  \begin{proof} For any $n \in \mathbb{N}^+$, we have
    \begin{align*}
        & \mathbb{E} \left[ \sqrt{n} x_{1} \mid A = 1 \right] = \int_{-\infty}^{\infty} \sqrt{n}x p(x_{1}=x \mid A = 1) dx \\
        =& 2\int_{-\infty}^{\infty} x \cdot p(A =  1 \mid x_{1}=x ) p( x_{1}=x ) dx  =  \int_{-\infty}^{\infty} 2\sqrt{n}x \sigma(x\beta_1) \frac{1}{\sqrt{2\pi \frac{1}{n}}}e^{-2n x^2} dx \\
        =&  \int_{-\infty}^{\infty} \frac{1}{\sqrt{2 \pi \frac{1}{n}}} \frac{2\sqrt{n}x}{1+e^{-x\beta_1}} e^{-2n x^2} dx = \int_{-\infty}^{\infty} \frac{1}{\sqrt{2 \pi }} \frac{2z}{1+e^{-z\beta_1/\sqrt{n}}} e^{-2 z^2} dz = \mathbb{E}\left[\frac{2z}{1+e^{-z \beta_1/\sqrt{n}}} \right].
    \end{align*} 
    Since $$ \left|\frac{2 z}{1+e^{-z \beta_{1} / \sqrt{n}}} \right| \leq 2|z| ,$$ by uniform integrability of the sequence $\left\{ \frac{2 z}{1+e^{-z \beta_{1} / \sqrt{n}}}\right\}_{n \in \mathbb{N}^+}$ we complete the proof of the lemma.
  \end{proof}

   \begin{lemma} \label{LLN}
  For every $n \in \mathbb{N}^{+}$, let $z_1,z_2, \ldots, z_n \overset{\text{i.i.d.}}{\sim} N(0, \frac{1}{n}) $. Let $\{\beta_1\}_{n \in \mathbb{N}^+}$ be a sequence of deterministic scalars such that $$\lim_{n \rightarrow \infty} \frac{\beta_1}{\sqrt{n}} = \gamma \text{ for some fixed positive constant $\gamma$.  } $$ Let $\{A_i\}_{i=1,2, \ldots,n}$ be independent Bernoulli random variables with success probability $\sigma(z_i \beta_1) $. Then
  $$\sum_{i=1}^{n} A_{i} z_{i }^{2} = \frac{1}{2}+o_p(1).$$
  \end{lemma}
  \begin{proof}
   Since for any $n \in \mathbb{N}^+$, $\sqrt{n}z_1 \sim N(0,1) $, we have $$ \sum_{i=1}^{n} \mathbb{P}\left(\left|nA_iz_{i}^2\right|>n\right) \leq \sum_{i=1}^{n} \mathbb{P}\left(\left|nz_{i}^2\right|>n\right)  = 2n \mathbb{P}\left(\sqrt{n}z_1>\sqrt{n}\right) \leq 2n \frac{\frac{1}{\sqrt{2 \pi}}e^{-\frac{n}{2}}}{\sqrt{n}}  \rightarrow 0, $$ where the second last step uses the Mills ratio bound.
   
   Further, we have $$\frac{1}{n^2} \sum_{i=1}^n \mathbb{E}\left[n^2A_i^2 z_i^4 \mathbf{1}_{\left(\left|nA_i z_i^2\right| \leq n\right)} \right] \leq \frac{1}{n^2}\sum_{i=1}^n \mathbb{E}\left[n^2 z_i^4  \right] = \frac{3}{n} \rightarrow 0. $$ Hence by weak law for triangular arrays we can obtain that $$ \frac{1}{n} \sum_{i=1}^{n} \left(nA_{i} z_{i}^{2} -\mathbb{E}\left[nA_i z_i^2 \mathbf{1}_{\left(\left|nA_i z_i^2\right| \leq n\right)} \right] \right) =  o_p(1), $$ which implies \begin{align*}
       & \sum_{i=1}^{n} A_{i} z_{i}^{2} = \mathbb{E}\left[n A_{1} z_{1}^{2} \mathbf{1}_{\left(\left|n A_1 z_{1}^{2}\right| \leq n\right)}\right] + o_p(1) = \mathbb{E}\left[n A_{1} z_{1}^{2}\right] + o_p(1).
   \end{align*}
   
   Finally, let $z \sim N(0,1)$, then by Lemma \ref{lemma 0} we know \begin{align*}
      & \mathbb{E}\left[n A_{1} z_{1}^{2}\right] =\mathbb{E}\left[ \sigma(z_1 \beta_1) nz_{1}^{2}\right] = \mathbb{E}\left[ \sigma(z \gamma) z^{2}\right] + o(1) = \mathbb{E}\left[ \frac{z^2}{1+e^{-z \gamma}} \right] + o(1) = \frac{1}{2} + o(1).
   \end{align*} 
   
   Thus we have completed the proof of the lemma.
  \end{proof}

  \begin{lemma} \label{lemma 2}
  Suppose $\lim_{n \rightarrow \infty} \frac{p}{n} = \kappa $ for some fixed $\kappa > 0$. For any $n \in \mathbb{N}^{+}$, let $x \sim N(0, \frac{1}{n}I_p) $. Let $ \{\beta\}_{n \in \mathbb{N}^{+}}, \{\xi\}_{n \in \mathbb{N}^{+}}\in \mathbb{R}^{p}$ be two sequences of deterministic vectors such that $$\frac{\|\beta\|^2}{n} \xrightarrow{} \gamma^2, \quad \frac{\|\xi\|^2}{n} \xrightarrow{} \hat{\gamma}^2, \quad \text{ where }\gamma, \hat{\gamma} \text{ are constants}. $$
   Then the sequence $ \left \{\frac{\left(1+e^{-x^{\top} \xi}\right)^{2}}{1+e^{-x^{\top} \beta}} \right\}_{n \in \mathbb{N}^+}$ is uniformly integrable. 
  \end{lemma}
  \begin{proof}
  Note that \begin{align*}
    \frac{\left(1+e^{-x^{\top} \xi}\right)^{2}}{1+e^{-x_{i 0}^{\top} \beta}} \leq  \left(1+e^{-x^{\top} \xi}\right)^{2} = 1 + 2e^{-x^{\top} \xi} + e^{-2x^{\top} \xi}.
\end{align*}

 Since $ x^{\top} \xi \xrightarrow{d} N\left(0,  \hat{\gamma}^{2}\right) $, and for any $n \in \mathbb{N}^{+}$, $x^{\top} \xi$ follows a normal distribution, we know that for large $n$, the variance of $x^{\top} \xi$ is close to $ \hat{\gamma}^{2}$. Hence $ \left \{\left(1+e^{-x^{\top} \xi}\right)^{2}\right \}_{n \in \mathbb{N}^{+} }$ is uniform integrable, which implies $ \left \{\frac{\left(1+e^{-x^{\top} \xi}\right)^{2}}{1+e^{-x^{\top} \beta}} \right\}_{n \in \mathbb{N}^+}$ is uniform integrable.
  \end{proof}

  \begin{lemma} \label{lemma 3}
 Fix any $(a,b,c)$, a permutation of $(1,2,3)$.  Assume the setting described in section \ref{sec:setup}. Let $x \in \mathbb{R}^{p}$ be any sample in $S_c$. Then the random vector $\left(\begin{array}{c}
x^{\top} \beta \\
x^{\top} \hat{\beta}_{S_a}
\end{array}\right)$ converges in distribution to a multivariate normal distribution with mean zero and variance $\Sigma$, where $$\Sigma = \begin{bmatrix}
\gamma^2 &  \alpha^{*}_{a}\gamma^2  \\
\alpha^{*}_{a}\gamma^2 &  \kappa_a (\sigma^*_a)^2 + (\alpha_a^*)^2 \gamma^2
\end{bmatrix} .$$
  \end{lemma}
  \begin{proof}
  Conditioned on $S_a, S_b$, \begin{align*}
    \left(\begin{array}{c}
x^{\top} \beta \\
x^{\top} \hat{\beta}_{S_a} 
\end{array}\right) \sim N \left( \mathbf{0}, \frac{1}{n} \begin{bmatrix}
\|\beta\|^2 & \beta^{\top} \hat{\beta}_{S_a} \\
\beta^{\top} \hat{\beta}_{S_a} & \|\hat{\beta}_{S_a}\|^2 
\end{bmatrix}  \right).
\end{align*}

By assumption $$\frac{\|\beta\|^{2}}{n}  \rightarrow \gamma^2 . $$

Note that $\hat{\beta}_{S_a}$ is estimated from $S_a$, where $X_a \in \mathbb{R}^{n_a \times p}$ with i.i.d. entries $\sim N(0, \frac{1}{n})$. Define $$\beta^{'} = \sqrt{r_a} \beta, \quad \hat{\beta}_{S_a}^{'} =\sqrt{r_a} \hat{\beta}_{S_a} , \quad X_a^{'} = \frac{1}{\sqrt{r_a}} X_a $$

Then $ X_a^{'} \in  \mathbb{R}^{n_a \times p}$ with i.i.d. entries $\sim N(0, \frac{1}{n_a})$, and $ X_a \beta =X_a^{'} \beta^{'},  X_a \hat{\beta }_{S_a}=X_a^{'} \hat{\beta}_{S_a}^{'}$. Moreover, $$   \frac{\|\beta^{'}\|^{2}}{n_a} = \frac{r_a\|\beta\|^2}{n_a} = \frac{\|\beta\|^{2}}{n}  \rightarrow   \gamma^2.$$


According to Theorem 4 in \cite{Sur14516}, $$ \frac{1}{p} (\beta^{'})^{\top} \hat{\beta}_{S_a}^{'} - \frac{\alpha_a^*}{p} (\beta^{'})^{\top} \beta^{'} \xrightarrow{a.s.} 0 , \quad \text{which is equivalent to} \quad \frac{r_a}{p} \beta^{\top} \hat{\beta}_{S_a} - \frac{\alpha_a^* r_a}{p} \beta^{\top} \beta \xrightarrow{a.s.} 0 .$$

This implies $$ \lim_{n \rightarrow \infty} \frac{1}{n} \beta^{\top} \hat{\beta}_{S_a}  =  \lim_{n \rightarrow \infty} \frac{\alpha_{a}^{*} }{n}\beta^{\top} \beta = \alpha_{a}^{*} \gamma^2 \quad \text{almost surely.} $$

Moreover, applying Theorem 4 in \cite{Sur14516} with $ \psi(t,u) = (t+\alpha_a^* u)^2 $ yields $$      \lim_{n \rightarrow \infty}  \frac{\|\hat{\beta}_{S_a}^{'}\|^2}{p} =  (\sigma_a^*)^2 +  (\alpha_a^*)^2 \lim_{n \rightarrow \infty}  \frac{\| {\beta^{'}}\|^2}{p} =(\sigma_a^*)^2 +  (\alpha_a^*)^2 \frac{r_a \gamma^2}{\kappa} \quad \text{almost surely.}  $$

Thus
$$ \frac{\|\hat{\beta}_{S_a}\|^2}{n} = \frac{\|\hat{\beta}_{S_a}^{\prime}\|^{2}}{p}  \cdot \frac{p}{n} \cdot \frac{1}{r_a} \xrightarrow{a.s.}  \kappa_a (\sigma_a^*)^2 +  (\alpha_a^*)^2\gamma^2 .$$

 Hence we know 
\begin{align*}
    \left(\begin{array}{c}
x^{\top} \beta \\
x^{\top} \hat{\beta}_{S_a}
\end{array}\right)  \xrightarrow{d} N\left( \mathbf{0}, \begin{bmatrix}
\gamma^2 &  \alpha^{*}_{a}\gamma^2  \\
\alpha^{*}_{a}\gamma^2 &  \kappa_a (\sigma^*_a)^2 + (\alpha_a^*)^2 \gamma^2
\end{bmatrix} \right),
\end{align*} which does not depend on $S_a,S_b$. This completes the proof of the lemma.
  \end{proof}

\begin{lemma} \label{lemma_wlln}
Assume the setting described in section \ref{sec:setup}. Then $$ \frac{1}{n_c} \sum_{i \in S_c} \left( \frac{A_i}{\sigma\left( x_i^\top \hat{\beta}_{S_a}\right)} \right) = \mathbb{E}\left[\frac{\sigma\left( Z_\beta\right)}{\sigma\left( Z_{\hat{\beta}_{S_a}}\right)} \right] + o_p(1). $$
\end{lemma}

\begin{proof}
Let $c_1 $ be the index of any sample in $S_c$. Then \begin{align*}
    & \mathbb{E}\left[\frac{A_{c_1}}{\sigma\left( x_{c_1}^\top \hat{\beta}_{S_{a}}\right)} \right] = \mathbb{E}\left[\frac{ \sigma\left( x_{c_1}^\top {\beta}\right)}{\sigma\left( x_{c_1}^\top \hat{\beta}_{S_{a}}\right)} \right] =  \mathbb{E}\left[\frac{\sigma\left(Z_{\beta}\right)}{\sigma\left(Z_{\hat{\beta}_{S_{a}}}\right)}\right]+o(1).
\end{align*}

Define $$S_n = \sum_{i \in S_{c}}\left(\frac{A_{i}}{\sigma\left(x_{i}^{\top} \hat{\beta}_{S_{a}}\right)}\right). $$ 

We have \begin{align*}
    & \mathbb{E}\left[S_n\right] = n_c \mathbb{E}\left[\frac{\sigma\left(x_{c_{1}}^{\top} \beta\right)}{\sigma\left(x_{c_{1}}^{\top} \hat{\beta}_{S_{a}}\right)}\right],
\end{align*} and  \begin{align}
    & \mathbb{E}\left[ \left( \frac{S_n - \mathbb{E}\left[S_n\right]}{n_c} \right)^2 \right] = \frac{1}{n_c^2} \mathbb{E}^2\left[S_n\right] - \frac{1}{n_c^2}\mathbb{E}\left[S_n^2\right] = \frac{1}{n_c} \mathbb{E}\left[\frac{\sigma\left(x_{c_{1}}^{\top} \beta\right)}{\sigma\left(x_{c_{1}}^{\top} \hat{\beta}_{S_{a}}\right)}\right] - \frac{1}{n_c^2} \mathbb{E}\left[S_{n}^{2}\right]. \label{wlln_eq1}
\end{align}

Note that $$ \frac{1}{n_{c}} \mathbb{E}\left[\frac{\sigma\left(x_{c_{1}}^{\top} \beta\right)}{\sigma\left(x_{c_{1}}^{\top} \hat{\beta}_{S_{a}}\right)}\right] = \frac{1}{n_{c}} \mathbb{E}\left[\frac{\sigma\left(Z_{\beta}\right)}{\sigma\left(Z_{\hat{\beta}_{S_{a}}}\right)}\right] + o(n^{-1}) = o(1). $$ and \begin{align*}
    & \frac{1}{n_{c}^{2}} \mathbb{E}\left[S_{n}^{2}\right] = \frac{1}{n_{c}^{2}} \mathbb{E}\left[\mathbb{E}\left[S_{n}^{2} \mid \hat{\beta}_{S_{a}} \right] \right] = \frac{1}{n_c^2} \mathbb{E}\left[\sum_{i \in S_{c}} \mathbb{E}\left[ \frac{A_{i}}{\sigma^2\left(x_{i}^{\top} \hat{\beta}_{S_{a}}\right)}  \mid \hat{\beta}_{S_{a}} \right]\right]\\
    =& \frac{1}{n_c} \mathbb{E}\left[ \mathbb{E}\left[ \frac{\sigma\left(x_{c_1}^\top \beta\right)}{\sigma^2\left(x_{c_1}^{\top} \hat{\beta}_{S_{a}}\right)}  \mid \hat{\beta}_{S_{a}} \right]\right] =\frac{1}{n_c}\mathbb{E}\left[ \frac{\sigma\left(x_{c_1}^\top \beta\right)}{\sigma^2\left(x_{c_1}^{\top} \hat{\beta}_{S_{a}}\right)}  \right] \\
    =& \frac{1}{n_c} \mathbb{E}\left[\frac{\sigma\left(Z_{\beta}\right)}{\sigma^2\left(Z_{\hat{\beta}_{S_{a}}}\right)}\right] +o\left(n^{-1}\right)=o(1).
\end{align*}

By equation (\ref{wlln_eq1}) we know $$ \mathbb{E}\left[\left(\frac{S_{n}-\mathbb{E}\left[S_{n}\right]}{n_{c}}\right)^{2}\right] = o(1), $$ which implies that $$ \frac{S_{n}-\mathbb{E}\left[S_{n}\right]}{n_{c}} = o_p(1).$$

Finally, since $$ \frac{\mathbb{E}\left[S_{n}\right]}{n_c} - \mathbb{E}\left[\frac{\sigma\left(Z_{\beta}\right)}{\sigma\left(Z_{\hat{\beta}_{S_{a}}}\right)}\right]= \mathbb{E}\left[\frac{\sigma\left(x_{c_{1}}^{\top} \beta\right)}{\sigma\left(x_{c_{1}}^{\top} \hat{\beta}_{S_{a}}\right)}\right] - \mathbb{E}\left[\frac{\sigma\left(Z_{\beta}\right)}{\sigma\left(Z_{\hat{\beta}_{S_{a}}}\right)}\right] = o(1),
$$ we have $$ \frac{ S_{n} }{n_c} = \mathbb{E}\left[\frac{\sigma\left(Z_{\beta}\right)}{\sigma\left(Z_{\hat{\beta}_{S_{a}}}\right)}\right] + o_p(1), $$ which completes the proof.

\end{proof}

\section{Linear Algebra and Random Matrix Results}

\begin{lemma}
Let $ p, n \in \mathbb{N}^{+}$ be such that $\lim_{n \rightarrow\infty} \frac{p}{n} = \kappa \in (0,1)$. Let $X \in \mathbb{R}^{n \times p}$ consist of i.i.d. entries $\sim N(0, \frac{1}{n})$. Then $\frac{1}{1- \kappa}I_{p}$ is a deterministic equivalent for $\left(\sum_{i=1}^{n}  x_{i} x_{i}^{\top}\right)^{-1}$. \label{original_deterministic_equiv}
\end{lemma}

\begin{proof}
The Lemma is a special case of Theorem 2.3 in \cite{phdthesis}, where $z = 0$.
\end{proof}

\begin{lemma}\label{lemma a1}
  Fix any $ p, n \in \mathbb{N}^{+}$. Let $\beta \in \mathbb{R}^p $ be any deterministic vector. Suppose $x \in \mathbb{R}^{p}$ consists of i.i.d. entries $\sim N(0, \frac{1}{n})$, and $A \sim \text{Bernoulli}(\sigma(x^{\top} \beta))$. Then $x x^{\top} \sim\left(x x^{\top} \mid A=a\right) \;\; \forall a = 0, 1.$
  \end{lemma}

 \begin{proof}
  Without loss of generality, assume $a = 1$. Let $P$ be an orthogonal matrix whose first row is $\frac{\beta^{\top}}{\|\beta\|} $. Then we have $$P \frac{\beta}{\|\beta\|} = e_1, \quad A \sim \text{Bernoulli}(\sigma((Px)^{\top}(P\beta))) = \text{Bernoulli}(\sigma(\|\beta\|z_{1})). $$
  
  Define $z = Px $. Then $$xx^{\top} = P^{\top} zz^{\top} P, \quad \left(x x^{\top} \mid A=a\right) =P^{\top} \left(z z^{\top} \mid A=a\right)P . $$
  
  Thus we only need to show $ zz^{\top} \sim (zz^{\top} \mid A = 1).$ Define $w := (z \mid A=1)$. We want to prove $ zz^{\top} \sim ww^{\top}.$ Since $z_{j} \indep A \;\; \forall j \geq 2$, we only need to show $$(w_1^2, w_1w_2, w_1w_3 , \ldots, w_1w_p) \sim (z_1^2, z_1z_2, z_1z_3 , \ldots, z_1z_p).$$
  
  We need to show that for any $c_1, c_2, \ldots, c_p \in \mathbb{R}$, \begin{equation}
       \mathbb{P}
(w_1^2 \leq c_1, w_1w_2 \leq c_2, \ldots, w_1w_p \leq c_p) = \mathbb{P}
(z_1^2 \leq c_1, z_1z_2 \leq c_2, \ldots, z_1z_p \leq c_p) . \label{a1_eq1}
  \end{equation} 

Note that \begin{align*}
    & \mathbb{P}
(z_1^2 \leq c_1, z_1z_2 \leq c_2, \ldots, z_1z_p \leq c_p \mid A = 1)\\
=& \mathbb{P}
(A = 1 \mid z_1^2 \leq c_1, z_1z_2 \leq c_2, \ldots, z_1z_p \leq c_p ) \cdot 2 \mathbb{P}(z_1^2 \leq c_1, z_1z_2 \leq c_2, \ldots, z_1z_p \leq c_p  )
\end{align*} and \begin{align*}
    & \mathbb{P}
(z_1^2 \leq c_1, z_1z_2 \leq c_2, \ldots, z_1z_p \leq c_p \mid A = 0)\\
=& \mathbb{P}
(A = 0 \mid z_1^2 \leq c_1, z_1z_2 \leq c_2, \ldots, z_1z_p \leq c_p ) \cdot 2 \mathbb{P}(z_1^2 \leq c_1, z_1z_2 \leq c_2, \ldots, z_1z_p \leq c_p  )
\end{align*}

Therefore, in order to show equation (\ref{a1_eq1}), we only need to show $$\mathbb{P}
(A = 1 \mid z_1^2 \leq c_1, z_1z_2 \leq c_2, \ldots, z_1z_p \leq c_p ) = \mathbb{P}
(A = 0 \mid z_1^2 \leq c_1, z_1z_2 \leq c_2, \ldots, z_1z_p \leq c_p ) = \frac{1}{2}. $$

We know \begin{align*}
    & \mathbb{P}
(A = 1 \mid z_1^2 \leq c_1, z_1z_2 \leq c_2, \ldots, z_1z_p \leq c_p )\\
=& \int_{-\infty}^{\infty} \mathbb{P}
(A = 1 \mid z_1 = z ) p(z_1 = z \mid z_1^2 \leq c_1, z_1z_2 \leq c_2, \ldots, z_1z_p \leq c_p) dz\\
=& \int_{-\infty}^{\infty}\sigma(\|\beta\| z) p(z_1 = z \mid z_1^2 \leq c_1, z_1z_2 \leq c_2, \ldots, z_1z_p \leq c_p) dz\\ 
=& \int_{0}^{\infty}\sigma(\|\beta\| z) p(z_1 = z \mid z_1^2 \leq c_1, z_1z_2 \leq c_2, \ldots, z_1z_p \leq c_p) dz\\
+& \int_{-\infty}^{0}\sigma(\|\beta\| z) p(z_1 = z \mid z_1^2 \leq c_1, z_1z_2 \leq c_2, \ldots, z_1z_p \leq c_p) dz\\
=& \int_{0}^{\infty}\sigma(\|\beta\| z) p(z_1 = z \mid z_1^2 \leq c_1, z_1z_2 \leq c_2, \ldots, z_1z_p \leq c_p) dz\\
+& \int_{-\infty}^{0} \left[1-\sigma(-\|\beta\| z) \right] p(z_1 = z \mid z_1^2 \leq c_1, z_1z_2 \leq c_2, \ldots, z_1z_p \leq c_p) dz.
\end{align*} By a change of variable, we have \begin{align*}
   & \int_{-\infty}^{0} \left[1-\sigma(-\|\beta\| z) \right] p(z_1 = z \mid z_1^2 \leq c_1, z_1z_2 \leq c_2, \ldots, z_1z_p \leq c_p) dz \\
   =& \int_{0}^{\infty} \left[1-\sigma(\|\beta\| z) \right] p(z_1 = -z \mid z_1^2 \leq c_1, z_1z_2 \leq c_2, \ldots, z_1z_p \leq c_p) dz.
\end{align*} 

Further, by symmetry, we have that for all $ z \in \mathbb{R},$ $$p(z_1 = -z \mid z_1^2 \leq c_1, z_1z_2 \leq c_2, \ldots, z_1z_p \leq c_p) = p(z_1 = z \mid z_1^2 \leq c_1, z_1z_2 \leq c_2, \ldots, z_1z_p \leq c_p) . $$

Hence \begin{align*}
  &  \int_{0}^{\infty}\sigma(\|\beta\| z) p(z_1 = z \mid z_1^2 \leq c_1, z_1z_2 \leq c_2, \ldots, z_1z_p \leq c_p) dz\\
+& \int_{-\infty}^{0} \left[1-\sigma(-\|\beta\| z) \right] p(z_1 = z \mid z_1^2 \leq c_1, z_1z_2 \leq c_2, \ldots, z_1z_p \leq c_p) dz \\
=& \int_{0}^{\infty}\sigma(\|\beta\| z) p(z_1 = z \mid z_1^2 \leq c_1, z_1z_2 \leq c_2, \ldots, z_1z_p \leq c_p) dz\\
+& \int_{0}^{\infty} \left[1-\sigma(\|\beta\| z) \right] p(z_1 = z \mid z_1^2 \leq c_1, z_1z_2 \leq c_2, \ldots, z_1z_p \leq c_p) dz\\
=&  \int_{0}^{\infty} p(z_1 = z \mid z_1^2 \leq c_1, z_1z_2 \leq c_2, \ldots, z_1z_p \leq c_p) dz \\
=&   \mathbb{P}(z_1 \geq 0 \mid z_1^2 \leq c_1, z_1z_2 \leq c_2, \ldots, z_1z_p \leq c_p).
\end{align*} 

Therefore we only need to prove $$ \mathbb{P}(z_1 \geq 0 \mid z_1^2 \leq c_1, z_1z_2 \leq c_2, \ldots, z_1z_p \leq c_p) = \frac{1}{2}.$$

Since $$\mathbb{P}(z_1 = 0 \mid z_1^2 \leq c_1, z_1z_2 \leq c_2, \ldots, z_1z_p \leq c_p) = 0, $$ we only need to show $$ \mathbb{P}(z_1 > 0 \mid z_1^2 \leq c_1, z_1z_2 \leq c_2, \ldots, z_1z_p \leq c_p)  =  \mathbb{P}(z_1 < 0 \mid z_1^2 \leq c_1, z_1z_2 \leq c_2, \ldots, z_1z_p \leq c_p) ,$$ which is equivalent to show $$ \mathbb{P}(z_1 > 0, z_1^2 \leq c_1, z_1z_2 \leq c_2, \ldots, z_1z_p \leq c_p) = \mathbb{P}(z_1 < 0, z_1^2 \leq c_1, z_1z_2 \leq c_2, \ldots, z_1z_p \leq c_p). $$

We can assume $c_1  > 0$, otherwise both sides are equal to 0. Note that \begin{align*}
    & \mathbb{P}(z_1 > 0, z_1^2 \leq c_1, z_1z_2 \leq c_2, \ldots, z_1z_p \leq c_p) = \mathbb{P}\left(0 < z_1 \leq \sqrt{c_1}, z_2 \leq \frac{c_2}{z_1}, \ldots, z_p \leq \frac{c_p}{z_1}\right)\\
    =& \int_{0}^{\sqrt{c_1}} \mathbb{P}\left( z_2 \leq \frac{c_2}{z_1}, \ldots, z_p \leq \frac{c_p}{z_1} \mid z_1 = z \right) p(z_1 = z) dz\\
    =& \int_{0}^{\sqrt{c_1}} \mathbb{P}\left( z_2 \leq \frac{c_2}{z_1}\right) \cdot \ldots \cdot \mathbb{P}\left( z_p \leq \frac{c_p}{z_1}\right) p(z_1 = z) dz\\ 
    =& \int_{-\sqrt{c_1}}^{0} \mathbb{P}\left( z_2 \geq \frac{c_2}{z_1}\right) \cdot \ldots \cdot \mathbb{P}\left( z_p \geq \frac{c_p}{z_1}\right) p(z_1 = z) dz\\
    =& \mathbb{P}(z_1 < 0, z_1^2 \leq c_1, z_1z_2 \leq c_2, \ldots, z_1z_p \leq c_p).
\end{align*}

This completes the proof.

  \end{proof}

  \begin{lemma}  \label{deterministic equiv}
Let $ p, n \in \mathbb{N}^{+}$ be such that $\lim_{n \rightarrow\infty} \frac{p}{n} = \kappa \in (0, \frac{1}{2})$. Let $\{\beta\}_{n \in \mathbb{N}^+}$ be a sequence of deterministic vectors in $\mathbb{R}^p$ such that $\lim_{n \rightarrow \infty} \frac{\|\beta\|}{\sqrt{n}} = \gamma $ for some fixed positive constant $\gamma$. Let $X \in \mathbb{R}^{n \times p}$ consist of i.i.d. entries $\sim N(0, \frac{1}{n})$. Let $\{A_i\}_{i=1,2, \ldots,n}$ be independent Bernoulli random variables with success probability $\sigma(x_i^{\top} \beta) $. Then $\frac{2}{1-2 \kappa} I_{p}$ is a deterministic equivalent for $\left(\sum_{i=1}^{n} A_{i} x_{i} x_{i}^{\top}\right)^{-1}$. 
\end{lemma}

\begin{proof}

Since \begin{align*}
    & \frac{1}{n}\sum_{i=1}^{n} A_i \xrightarrow{a.s.} \mathbb{E}\left[ A_1 \right] = \mathbb{E}\left[ \sigma\left(x_1^\top \beta\right)\right] =\int_{-\infty}^{\infty} \sigma(x) \frac{1}{\sqrt{2 \pi} \frac{\|\beta\|}{\sqrt{n}}} e^{-\frac{x^{2}}{2 \frac{\|\beta\|_{2}^{2}}{n}}} d x \\
    =& \int_{0}^{\infty} \sigma(x) \frac{1}{\sqrt{2 \pi} \frac{\|\beta\|}{\sqrt{n}}} e^{-\frac{x^{2}}{2 \frac{\|\beta\|_{2}^{2}}{n}}} d x+\int_{-\infty}^{0}(1-\sigma(-x)) \frac{1}{\sqrt{2 \pi} \frac{\|\beta\|}{\sqrt{n}}} e^{-\frac{x^{2}}{2 \frac{\|\beta\|_{2}^{2}}{n}}} d x\\
    =& \int_{0}^{\infty} \sigma(x) \frac{1}{\sqrt{2 \pi} \frac{\|\beta\|}{\sqrt{n}}} e^{-\frac{x^{2}}{2 \frac{\|\beta\|_{2}^{2}}{n}}} d x+\int_{0}^{\infty}(1-\sigma(y)) \frac{1}{\sqrt{2 \pi} \frac{\|\beta\|}{\sqrt{n}}} e^{-\frac{y^{2}}{2 \frac{\|\beta\|_{2}^{2}}{n}}} d y=\frac{1}{2},
\end{align*} we know that $$ \sum_{i=1}^{n} A_{i} = \frac{n}{2} + o(n) \quad \text{almost surely}. $$

For any sequences of deterministic vectors $a, b \in \mathbb{R}^{p}$ of unit Euclidean norm, for any $a_1, a_2, \ldots, a_n \in \{0,1\}^n$ such that $\sum_{i=1}^n a_i = \frac{n}{2} + o(n)$, by Lemma \ref{lemma a1} we know $$ \left(a^{\top}\left(\sum_{i=1}^{n} A_{i} x_{i} \cdot x_{i}^{\top}\right)^{-1} b \mid A_i = a_i \quad \forall i = 1,2, \ldots, n \right) \sim a^{\top}\left(\sum_{i=1}^{\sum_{i=1}^{n} a_{i}} x_{i} x_{i}^{\top}\right)^{-1} b $$

Since the trace and the magnitude of the largest eigenvalue of $ \sum_{i=1}^{o(n)} x_{i} x_{i}^{\top}$ both converge to 0, we know that, almost surely, \begin{align*}
    & a^{\top}\left(\sum_{i=1}^{\sum_{i=1}^{n} a_{i}} x_{i} x_{i}^{\top}\right)^{-1} b = a^{\top}\left(\sum_{i=1}^{\frac{n}{2} + o(n)} x_{i} x_{i}^{\top}\right)^{-1} b = a^{\top}\left(\sum_{i=1}^{\frac{n}{2} } x_{i} x_{i}^{\top}\right)^{-1} b + o(1)\\
    =& 2 a^{\top}\left(\sum_{i=1}^{\frac{n}{2}} \sqrt{2} x_{i} \cdot \sqrt{2} x_{i}^{\top}\right)^{1} \quad b + o(1) =2 a^{\top} \frac{1}{1-\frac{p}{n / 2}} b + o(1), 
\end{align*} where the last step is obtained by applying Lemma \ref{original_deterministic_equiv}. Thus  $$\left(a^{\top}\left(\sum_{i=1}^{n} A_{i} x_{i} \cdot x_{i}^{\top}\right)^{-1} b \mid A_{i}=a_{i} \quad \forall i=1,2, \ldots, n\right) - 2 a^{\top} \frac{1}{1-\frac{p}{n / 2}} b \xrightarrow{a.s.} 0. $$

This implies that $$ a^\top  \left[ \left(\sum_{i=1}^{n} A_{i} x_{i} x_{i}^{\top}\right)^{-1}-\frac{2}{1-\frac{p}{n / 2}} I_{p} \right]b \xrightarrow{a.s.} 0. $$

Similarly, we can show that for any sequences of deterministic matrix $A \in \mathbb{R}^{p \times p}$ with unit operator norm, we have $$ \frac{1}{n} \text{tr}A \left( \left(\sum_{i=1}^{n} A_{i} x_{i} x_{i}^{\top}\right)^{-1}-\frac{2}{1-\frac{p}{n / 2}} I_{p} \right) \xrightarrow{a.s.} 0. $$ This completes the proof.

\end{proof}

    \begin{lemma}
Let $ p, n \in \mathbb{N}^{+}$ be such that $\lim_{n \rightarrow\infty} \frac{p}{n} = \kappa \in (0, \frac{1}{2})$. Let $\{\beta\}_{n \in \mathbb{N}^+}$ be a sequence of deterministic vectors in $\mathbb{R}^p$ such that $\lim_{n \rightarrow \infty} \frac{\|\beta\|}{\sqrt{n}} = \gamma$ for some fixed positive constant $\gamma$. Let $X \in \mathbb{R}^{n \times p}$ consist of i.i.d. entries $\sim N(0, \frac{1}{n})$. Let $\{A_i\}_{i=1,2, \ldots,n}$ be independent Bernoulli random variables with success probability $\sigma(x_i^{\top} \beta) $. Define $$ P = \sum_{i=1}^{n} x_{i} x_{i}^{\top}, \quad Q = \sum_{i=1}^{n} A_{i} x_{i} x_{i}^{\top}.$$ Then $\exists c_p, c_q > 0$ such that, almost surely, $$ \lim_{n \rightarrow \infty} \lambda_{\min} (P) \geq c_p , \quad \lim_{n \rightarrow \infty} \lambda_{\min}(Q) \geq c_q.  $$
  \end{lemma} \label{matrix_lower_bd}
  \begin{proof}
  By the main theorem in \cite{wishart_lower}, we have $$ \lim _{n \rightarrow \infty} \lambda_{\min }(P) \geq c_{p} $$ for some constant $c_p > 0$ almost surely.

  Moreover, for any $a_1, a_2, \ldots, a_n \in \{0,1\}^n$ such that $\sum_{i=1}^n a_i = \frac{n}{2} + o(n)$, Lemma \ref{lemma a1} implies that $$ \left(  \lambda_{\min}\left(Q\right) \mid  A_{i}=a_{i} \quad \forall i=1,2, \ldots, n \right) \sim \lambda_{\min} \left(\sum_{i=1}^{\sum_{i=1}^{n} a_{i}} x_{i} x_{i}^{T}\right)$$
  
  Again, the main theorem in \cite{wishart_lower} implies that $$\lambda_{\min }\left(\sum_{i=1}^{\sum_{i=1}^{n} a_{i}} x_i x_{i}^{T}\right) \geq c_q $$ for some positive constant $c_q > 0$ almost surely.
  
  Since $$\sum_{i=1}^{n} A_{i}=\frac{n}{2}+o(n) \quad \text{almost surely, }$$ we know that, almost surely, $$\lim _{n \rightarrow \infty} \lambda_{\min }(Q) \geq c_{q}.$$

This completes the proof.
  \end{proof}

  \begin{lemma}
  \label{lemma 6}
   Let $ p, n \in \mathbb{N}^{+}$ be such that $\lim_{n \rightarrow\infty} \frac{p}{n} = \kappa \in (0,\frac{1}{2})$. Let $X \in \mathbb{R}^{n \times p}$ consist of i.i.d. entries $\sim N(0, \frac{1}{n})$. Then $$ \frac{1}{n}\left(\sum_{i=1}^{n} x_{i}\right)^{\top}\left(\sum_{i=1}^{n} x_{i} x_{i}^{\top}\right)^{-1}\left(\sum_{i=1}^{n} x_{i}\right) \stackrel{p}{\rightarrow} \kappa. $$
  \end{lemma}
  \begin{proof}
  Define $$T_1 = (\sum_{i } x_{i})(\sum_{i } x_{i})^{\top}, \quad T_2 = (\sum_{i } x_{i})(\sum_{i } x_{i})^{\top}- \sum_{i } x_ix_i^{\top} .   $$  Then \begin{align*}
   & \frac{1}{n} \left(\sum_{i } x_{i}\right)^{\top}  \left ( \sum_{i } x_ix_i^{\top}  \right )^{-1}   \left( \sum_{i } x_{i} \right)=  \frac{1}{n} \text{Tr}\left[ (\sum_{i } x_{i})(\sum_{i } x_{i})^{\top} \cdot \left ( \sum_{i } x_ix_i^{\top}  \right )^{-1}  \right]\\
    =& \frac{1}{n} \text{Tr}\left[ T_1 \cdot \left ( T_1-T_2  \right )^{-1}  \right] =  \frac{1}{n} \text{Tr}\left[I_{p} + T_2(T_1-T_2)^{-1} \right] =  \frac{p}{n} + \frac{1}{n} \text{Tr}\left[ T_2(T_1-T_2)^{-1} \right],
\end{align*} where the second last step is due to the fact that $$ \left.(A+B)^{-1}=A^{-1}-A^{-1} B(A+B)^{-1}\right. \text{ for any invertible square matrices } A,B. $$

Hence we only need to show $$  \frac{1}{n} \text{Tr}\left[ T_2(\sum_{i } x_ix_i^{\top} )^{-1} \right] \xrightarrow{p} 0.$$

Define $$x^{-\{i\}} = \sum_{1\leq j \leq n, j \neq i} x_j \quad \forall i \in \{ 1,2, \ldots, n\} , \quad x^{-\{i,j\}} = \sum_{1\leq k \leq n, k \neq i,j} x_j \quad \forall i,j \in \{ 1,2, \ldots, n \}.$$  

Then \begin{align*}
    & \frac{1}{n} \text{Tr}\left[ T_2(\sum_{i } x_ix_i^{\top} )^{-1} \right] = \frac{1}{n} \text{Tr}\left\{\left[ \left(\sum_{i} x_{i}\right)\left(\sum_{i} x_{i}\right)^{\top}-\sum_{i} x_{i} x_{i}^{\top}\right] \left(\sum_{j } x_jx_j^{\top} \right)^{-1} \right\}\\
    =&  \frac{1}{n} \text{Tr}\left\{\left[ \sum_{i} x_{i}\left( x^{-\{i\}}\right)^\top \right] \left(\sum_{j} x_jx_j^{\top} \right)^{-1} \right\} = \frac{1}{n} \sum_{i=1}^n \left(x^{-\{i\}}\right)^{\top} \left(\sum_{j } x_jx_j^{\top} \right)^{-1} x_{i} \\ &=\frac{1}{n} \sum_{i=1}^n \frac{\left(x^{-\{i\}}\right)^{\top} (\sum_{j\neq i } x_jx_j^{\top} )^{-1} x_{i}}{1+ x_{i}^{\top}(\sum_{j\neq i } x_jx_j^{\top} )^{-1} x_{i}}.
\end{align*}

We show the term above has expectation equal to 0, and variance converging to 0.

For each $i\in \{1,2, \ldots, n\}, $ by conditioning on $x_i$ and symmetry we have \begin{align*}
    &\mathbb{E} \left[ \frac{\left(x^{\{-i\}}\right)^{\top}\left(\sum_{j \neq i} x_{j} x_{j}^{\top}\right)^{-1} x_{i}}{1+x_{i}^{\top}\left(\sum_{j \neq i} x_{j} x_{j}^{\top}\right)^{-1} x_{i}} \right]  =\mathbb{E} \left[ \mathbb{E} \left[ \frac{\left(x^{\{-i\}}\right)^{\top}\left(\sum_{j \neq i} x_{j} x_{j}^{\top}\right)^{-1} x_{i}}{1+x_{i}^{\top}\left(\sum_{j \neq i} x_{j} x_{j}^{\top}\right)^{-1} x_{i}} \mid x^{\{-i\}} \right]  \right] \\
    =& \mathbb{E} \left[\left(x^{\{-i\}}\right)^{\top}\left(\sum_{j \neq i} x_{j} x_{j}^{\top}\right)^{-1} \mathbb{E} \left[ \frac{ x_{i}}{1+x_{i}^{\top}\left(\sum_{j \neq i} x_{j} x_{j}^{\top}\right)^{-1} x_{i}} \mid x^{\{-i\}} \right]  \right] = 0.
\end{align*}

Hence $$\mathbb{E} \left[\frac{1}{n} \sum_{i=1}^{n} \frac{\left(x^{\{-i\}}\right)^{\top}\left(\sum_{j \neq i} x_{j} x_{j}^{\top}\right)^{-1} x_{i}}{1+x_{i}^{\top}\left(\sum_{j \neq i} x_{j} x_{j}^{\top}\right)^{-1} x_{i}} \right] = 0.$$

Moreover, its variance is \begin{align}
   & \mathbb{E}\left[ \left( \frac{1}{n} \sum_{i=1}^{n} \frac{\left(x^{\{-i\}}\right)^{\top}\left(\sum_{j \neq i} x_{j} x_{j}^{\top}\right)^{-1} x_{i}}{1+x_{i}^{\top}\left(\sum_{j \neq i} x_{j} x_{j}^{\top}\right)^{-1} x_{i}} \right)^2\right] \nonumber \\
   =&  \frac{1}{n^2} \left[ n \cdot \mathbb{E} \left[ \left( \frac{\left(x^{\{-1\}}\right)^{\top}\left(\sum_{j \neq 1} x_{j} x_{j}^{\top}\right)^{-1} x_{1}}{1+x_{1}^{\top}\left(\sum_{j \neq 1} x_{j} x_{j}^{\top}\right)^{-1} x_{2}} \right )^2\right ]  \right]\nonumber\\
   +& \frac{1}{n^2} \left[ n(n-1) \cdot \mathbb{E} \left[ \left( {\left(x^{\{-1\}}\right)^{\top}\left(\sum_{j =1}^n x_{j} x_{j}^{\top}\right)^{-1} x_{1}} \right ) \left( {\left(x^{\{-2\}}\right)^{\top}\left(\sum_{j =2}^n x_{j} x_{j}^{\top}\right)^{-1} x_{1}} \right )\right ]  \right] \nonumber \\
   \leq & \frac{1}{n^2} \left[ n \cdot \mathbb{E} \left[ \left( {\left(x^{\{-1\}}\right)^{\top}\left(\sum_{j \neq 1} x_{j} x_{j}^{\top}\right)^{-1} x_{1}} \right )^2\right ]  \right]\label{var_12}\\
   +& \frac{1}{n^2} \left[ n(n-1) \cdot \mathbb{E} \left[ \left( {\left(x^{\{-1\}}\right)^{\top}\left(\sum_{j =1}^n x_{j} x_{j}^{\top}\right)^{-1} x_{1}} \right ) \left( {\left(x^{\{-2\}}\right)^{\top}\left(\sum_{j =1}^n x_{j} x_{j}^{\top}\right)^{-1} x_{2}} \right )\right ]   \right]. \nonumber
\end{align}

We then show both terms on the RHS of (\ref{var_12}) converge to 0.

The first term converges to 0 because \begin{align*}
    & \lim_{n \rightarrow \infty} \frac{1}{n^{2}}\left[n \cdot \mathbb{E}\left[\left(\left(x^{\{-1\}}\right)^{\top}\left(\sum_{j \neq 1} x_{j} x_{j}^{\top}\right)^{-1} x_{1}\right)^{2}\right]\right] \\
    &=\lim_{n \to \infty} \frac{1}{n} \mathbb{E}\left[x_{1}^{\top}\left(\sum_{j \neq 1} x_{j} x_{j}^{\top}\right)^{-1} x^{\{-1\}}\cdot \left(x^{\{-1\}}\right)^{\top}\left(\sum_{j \neq 1} x_{j} x_{j}^{\top}\right)^{-1} x_{1}\right]\\
    =& \lim_{n \rightarrow \infty}  \frac{1}{n} \mathbb{E}\left[x_{1}^{\top}\mathbb{E}\left[\left(\sum_{j \neq 1} x_{j} x_{j}^{\top}\right)^{-1} x^{\{-1\}} \cdot \left(x^{\{-1\}}\right)^{\top} \left(\sum_{j \neq 1} x_{j} x_{j}^{\top}\right)^{-1}\right] x_{1}\right]\\
    =& \lim_{n \rightarrow \infty} \frac{1}{n^2}{\text{Tr}\left(\mathbb{E}\left[\left(\sum_{j \neq 1} x_{j} x_{j}^{\top}\right)^{-1} x^{\{-1\}} \cdot \left(x^{\{-1\}}\right)^{\top} \left(\sum_{j \neq 1} x_{j} x_{j}^{\top}\right)^{-1}\right]\right)}\\
    =& \lim_{n \rightarrow \infty} \frac{1}{n^2}\mathbb{E}\left[\text{Tr} \left( \left(\sum_{j \neq 1} x_{j} x_{j}^{\top}\right)^{-1} x^{\{-1\}}\cdot \left(x^{\{-1\}}\right)^{\top}\left(\sum_{j \neq 1} x_{j} x_{j}^{\top}\right)^{-1} \right)\right] \\
    =& \lim_{n \rightarrow \infty} \frac{1}{n^2}\mathbb{E}\left[\left(x^{\{-1\}}\right)^{\top} \left(\sum_{j \neq 1} x_{j} x_{j}^{\top}\right)^{-2} x^{\{-1\}}\right].
\end{align*}

By Lemma \ref{matrix_lower_bd} we know there exists a constant $C$ that upper bounds $\lim_{n \rightarrow \infty} \left\|\left(\sum_{j \neq 1} x_{j} x_{j}^{\top}\right)^{-1}\right\|$. Then \begin{align*}
    &  \lim_{n \rightarrow \infty} \frac{1}{n^2}\mathbb{E}\left[\left(x^{\{-1\}}\right)^{\top} \left(\sum_{j \neq 1} x_{j} x_{j}^{\top}\right)^{-2} x^{\{-1\}}\right]  \leq \lim_{n \rightarrow \infty}  \frac{1}{n^2}\mathbb{E}\left[\left(x^{\{-1\}}\right)^{\top} x^{\{-1\}} \right] \cdot C^2 = 0.
\end{align*}

Now we show the second term on the RHS of (\ref{var_12}) also converges to 0: \begin{align*}
    & \lim_{n \rightarrow \infty} \frac{1}{n^{2}}\left[n(n-1) \mathbb{E}\left[\left(\left(x^{\{-1\}}\right)^{\top}\left(\sum_{j=1}^{n} x_{j} x_{j}^{\top}\right)^{-1} x_{1}\right)\left(\left(x^{\{-2\}}\right)^{\top}\left(\sum_{j=1}^{n} x_{j} x_{j}^{\top}\right)^{-1} x_{2}\right)\right]\right]\\
    =& \lim_{n \rightarrow \infty} \mathbb{E}\left[\left(\left(x^{\{-1\}}\right)^{\top}\left(\sum_{j=1}^{n} x_{j} x_{j}^{\top}\right)^{-1} x_{1}\right)\left(\left(x^{\{-2\}}\right)^{\top}\left(\sum_{j=1}^{n} x_{j} x_{j}^{\top}\right)^{-1} x_{2}\right)\right]\\
    =& \lim_{n \rightarrow \infty} \mathbb{E}\left[\left(\left(x^{\{-1,-2\} }+x_2\right)^{\top}\left(\sum_{j=1}^{n} x_{j} x_{j}^{\top}\right)^{-1} x_{1}\right)\left(\left(x^{\{-1,-2\}}+x_1\right)^{\top}\left(\sum_{j=1}^{n} x_{j} x_{j}^{\top}\right)^{-1} x_{2}\right)\right].
\end{align*}

Note that since the term inside the expectation below is an odd function of $x_2$, we have $$ \mathbb{E}\left[\left(\left(x^{\{-1,-2\}}\right)^{\top}\left(\sum_{j=1}^{n} x_{j} x_{j}^{\top}\right)^{-1} x_{1}\right)\left(\left(x^{\{-1,-2\}}+x_1\right)^{\top}\left(\sum_{j=1}^{n} x_{j} x_{j}^{\top}\right)^{-1} x_{2}\right)\right] = 0,$$

Similarly, note that the term inside the expectation below is an odd function of $x_1$). Thus $$ \mathbb{E}\left[\left(x_2^{\top}\left(\sum_{j=1}^{n} x_{j} x_{j}^{\top}\right)^{-1} x_{1}\right)\left(\left(x^{\{-1,-2\}}\right)^{\top}\left(\sum_{j=1}^{n} x_{j} x_{j}^{\top}\right)^{-1} x_{2}\right)\right] = 0.$$

Thus \begin{align*}
    & \lim_{n \rightarrow \infty} \mathbb{E}\left[\left(\left(x^{\{-1,-2\}}+x_2\right)^{\top}\left(\sum_{j=1}^{n} x_{j} x_{j}^{\top}\right)^{-1} x_{1}\right)\left( \left(x^{\{-1,-2\}}+ x_1\right)^{\top}\left(\sum_{j=1}^{n} x_{j} x_{j}^{\top}\right)^{-1} x_{2}\right)\right] \\
    =& \lim_{n \rightarrow \infty} \mathbb{E}\left[\left(x_2^{\top}\left(\sum_{j=1}^{n} x_{j} x_{j}^{\top}\right)^{-1} x_{1}\right)\left(x_1^{\top}\left(\sum_{j=1}^{n} x_{j} x_{j}^{\top}\right)^{-1} x_{2}\right)\right].
\end{align*}

Hence in order to show the second term in (\ref{var_12}) converges to 0, we only need to show $$ x_{2}^{\top}\left(\sum_{j=1}^{n} x_{j} x_{j}^{\top}\right)^{-1} x_{1} \xrightarrow{a.s.}  0 .$$

By the Sherman–Morrison formula we have $$ x_{2}^{\top}\left(\sum_{j=1}^{n} x_{j} x_{j}^{\top}\right)^{-1} x_{1}  = \frac{x_{2}^{\top}\left(\sum_{j \neq 1} x_{j} x_{j}^{\top}\right)^{-1} x_{1}}{1+x_{1}^{\top}\left(\sum_{j \neq 1} x_{j} x_{j}^{\top}\right)^{-1} x_{1}}.$$ 

By Lemma \ref{original_deterministic_equiv}, we know the deterministic equivalent of $\left(\sum_{j \neq 1} x_{j} x_{j}^{\top}\right)^{-1}$ is $\frac{1}{1- \frac{p}{n}}I$. Thus the denominator converges almost surely to some constant. Hence we only need to show $$x_{2}^{\top}\left(\sum_{j \neq 1} x_{j} x_{j}^{\top}\right)^{-1} x_{1} \xrightarrow{a.s.} 0. $$

Apply the Sherman–Morrison formula again, We have \begin{align*}
    & x_{2}^{\top}\left(\sum_{j \neq 1} x_{j} x_{j}^{\top}\right)^{-1} x_{1} = \frac{x_1^{\top}  \left(\sum_{j \neq 1,2} x_{j} x_{j}^{\top}\right)^{-1}  x_2}{1+x_2^{\top}  \left(\sum_{j \neq 1,2} x_{j} x_{j}^{\top}\right)^{-1}  x_2}.
\end{align*}

Since the deterministic equivalent of $\left(\sum_{j \neq 1,2} x_{j} x_{j}^{\top}\right)^{-1}$ is also $\frac{1}{1- \frac{p}{n}}I$, we have $$ x_{2}^{\top}\left(\sum_{j \neq 1,2} x_{j} x_{j}^{\top}\right)^{-1} x_{1} =\frac{1}{1-\frac{p}{n}} x_{2}^{\top} x_{1}+o(1)=o(1) \quad \text{a.s.} $$

Thus we have shown $$x_{2}^{\top}\left(\sum_{j=1}^{n} x_{j} x_{j}^{\top}\right)^{-1} x_{1} \stackrel{\text { a.s. }}{\longrightarrow} 0,$$ and this completes the proof.
  \end{proof}

   \begin{lemma}
   \label{lemma 7}
Let $ p, n \in \mathbb{N}^{+}$ be such that $\lim_{n \rightarrow\infty} \frac{p}{n} = \kappa \in (0,\frac{1}{2})$. Let $X \in \mathbb{R}^{n \times p}$ consist of i.i.d. entries $\sim N(0, \frac{1}{n})$. Also, let $z \in \mathbb{R}^{n}$ consist of i.i.d. entries $\sim N(0, \frac{1}{n})$ such that $z \indep X$. Then  $$\left(\sum_{i=1}^n z_i x_{i}\right)^{\top}\left(\sum_{i=1}^n x_{i} x_{i}^{\top}\right)^{-1} \left(\sum_{i=1}^n z_ix_{i}\right) \xrightarrow{p} \kappa, $$ $$ \left(\sum_{i=1}^n \frac{1}{\sqrt{n}} x_{i}\right)^{\top}\left(\sum_{i=1}^n x_{i} x_{i}^{\top}\right)^{-1} \left(\sum_{i=1}^n z_ix_{i}\right) \xrightarrow{p} 0.$$
  \end{lemma}
  
  \begin{proof}
  First we show \begin{equation}
  \label{lemma 6 res 1}
      \left(\sum_{i=1}^{n} z_{i} x_{i}\right)^{\top}\left(\sum_{i=1}^{n} x_{i} x_{i}^{\top}\right)^{-1}\left(\sum_{i=1}^{n} z_{i} x_{i}\right) \stackrel{p}{\rightarrow} \kappa.
  \end{equation}
  
  We have 
  \begin{align*}
    &\mathbb{E} \left[ \left(\sum_{i=1}^{n} z_{i} x_{i}\right)^{\top}\left(\sum_{i=1}^{n} x_{i} x_{i}^{\top}\right)^{-1} \left(\sum_{i=1}^{n} z_{i} x_{i} \right)\right] \\
    =&  \mathbb{E} \left[ \mathbb{E} \left[ \left(\sum_{i=1}^{n} z_{i} x_{i}\right)^{\top}\left(\sum_{i=1}^{n} x_{i} x_{i}^{\top}\right)^{-1} \left(\sum_{i=1}^{n} z_{i} x_{i}\right)  \mid x_i \;\; \forall i\right]  \right]\\
    =&\frac{1}{n} \mathbb{E} \left[ \sum_{i=1}^n \left(x_i^{\top} \left(\sum_{i=1}^{n} x_{i} x_{i}^{\top}\right)^{-1} x_i \right) \right] =  \frac{1}{n}  \mathbb{E} \left[\mathrm{Tr} \left( \left(\sum_{i=1}^n x_ix_i^{\top}\right) \left(\sum_{i=1}^{n} x_{i} x_{i}^{\top}\right)^{-1}   \right) \right] \\
    =&  \frac{p}{n} = \kappa + o(1).
\end{align*}

Hence we only need to show the variance converges to 0, which is equivalent to show $$  \mathbb{E} \left[\left[ \left(\sum_{i=1}^{n} z_{i} x_{i}\right)^{\top}\left(\sum_{i=1}^{n} x_{i} x_{i}^{\top}\right)^{-1}\left( \sum_{i=1}^{n} z_{i} x_{i}\right) \right]^2\right]  \rightarrow \kappa^2  .$$

Note that \begin{align*}
    & \mathbb{E} \left[\left[ \left(\sum_{i=1}^{n} z_{i} x_{i}\right)^{\top}\left(\sum_{i=1}^{n} x_{i} x_{i}^{\top}\right)^{-1}\left( \sum_{i=1}^{n} z_{i} x_{i}\right) \right]^2\right] =  \\
    &\mathbb{E} \left[  \sum_{1 \leq i < j \leq n} z_i^2 z_j^2 \left[  x_{i}^{\top}\left(\sum_{k=1}^{n} x_{k} x_{k}^{\top}\right)^{-1} x_{i}\right] \left[  x_{j}^{\top}\left(\sum_{k=1}^{n} x_{k} x_{k}^{\top}\right)^{-1} x_{j}\right]  \right] \\
    +&  \mathbb{E} \left[  \sum_{1 \leq i < j \leq n} z_i^2 z_j^2 \left[  x_{i}^{\top}\left(\sum_{k=1}^{n} x_{k} x_{k}^{\top}\right)^{-1} x_{j}\right]^2  \right] + \mathbb{E} \left[  \sum_{i=1}^n z_i^4 \left[  x_{i}^{\top}\left(\sum_{k=1}^{n} x_{k} x_{k}^{\top}\right)^{-1} x_{i}\right]^2  \right]  \\
    =& \frac{1}{n^2} \mathbb{E} \left[  \sum_{1 \leq i < j \leq n}  \left[  x_{i}^{\top}\left(\sum_{k=1}^{n} x_{k} x_{k}^{\top}\right)^{-1} x_{i}\right] \left[  x_{j}^{\top}\left(\sum_{k=1}^{n} x_{k} x_{k}^{\top}\right)^{-1} x_{j}\right]  \right] \\
    +&  \frac{1}{n^2} \mathbb{E} \left[  \sum_{1 \leq i < j \leq n}  \left[  x_{i}^{\top}\left(\sum_{k=1}^{n} x_{k} x_{k}^{\top}\right)^{-1} x_{j}\right]^2  \right] + \frac{3}{n^2}  \mathbb{E}\left[\sum_{i=1}^{n} \left[x_{i}^{\top}\left(\sum_{k=1}^{n} x_{k} x_{k}^{\top}\right)^{-1} x_{i}\right]^{2}\right] \\
    =& \mathbb{E}\left[\left[ \frac{1}{n} \left(\sum_{i=1}^{n}  x_{i}\right)^{\top}\left(\sum_{i=1}^{n} x_{i} x_{i}^{\top}\right)^{-1}\left(\sum_{i=1}^{n}  x_{i}\right)\right]^{2}\right] + o(1) = \kappa^2 + o(1),
\end{align*} where the second last step is due to the fact that $$ \lim_{n \rightarrow \infty}  \mathbb{E} \left[ \left( x_{1}^{\top}\left(\sum_{i=1}^{n} x_{i} x_{i}^{\top}\right)^{-1}  x_{2} \right)^2 \right] = 0, \quad \lim_{n \rightarrow \infty}  \mathbb{E} \left[ \left( x_{1}^{\top}\left(\sum_{i=1}^{n} x_{i} x_{i}^{\top}\right)^{-1}  x_{1} \right)^2 \right] = O(1). $$



Hence $$  \mathbb{E} \left[\left[ \left(\sum_{i=1}^{n} z_{i} x_{i}\right)^{\top}\left(\sum_{i=1}^{n} x_{i} x_{i}^{\top}\right)^{-1}\left( \sum_{i=1}^{n} z_{i} x_{i}\right) \right]^2\right]  \rightarrow \kappa^2 ,$$ and thus we have proved equation (\ref{lemma 6 res 1}). We then show \begin{equation}
 \label{lemma 6 res 2}
    \left(\sum_{i=1}^{n} \frac{1}{\sqrt{n}} x_{i}\right)^{\top}\left(\sum_{i=1}^{n} x_{i} x_{i}^{\top}\right)^{-1}\left(\sum_{i=1}^{n} z_{i} x_{i}\right) \stackrel{p}{\rightarrow} 0.
\end{equation}

Note that by symmetry, $$\mathbb{E}\left[\left(\sum_{i=1}^{n} z_i x_{i}^{\top}\right)\left(\sum_{i=1}^{n} x_{i} x_{i}^{\top}\right)^{-1}\left(\sum_{i=1}^{n} \frac{1}{\sqrt{n}} x_{i}\right)\right] = 0.$$

Hence we only need to show that the variance converges to 0.

Since for any $j \neq k ,\; j,k \in \{1,2, \ldots, n\} $, \begin{align*}
    &\mathbb{E}\left[ z_jz_k \cdot  x_{j}^{\top}\left(\sum_{i=1}^{n} x_{i} x_{i}^{\top}\right)^{-1}\left(\sum_{i=1}^{n}\frac{1}{\sqrt{n}} x_{i}\right) \cdot \left(\sum_{i=1}^{n}\frac{1}{\sqrt{n}} x_{i}\right)^{\top} \left(\sum_{i=1}^{n} x_{i} x_{i}^{\top}\right)^{-1}   x_{k}\right] = 0,
\end{align*} we have \begin{align*}
    & \mathbb{E}\left[\left(\left(\sum_{i=1}^{n} z_i x_{i}^{\top}\right)\left(\sum_{i=1}^{n} x_{i} x_{i}^{\top}\right)^{-1}\left(\sum_{i=1}^{n} \frac{1}{\sqrt{n}} x_{i}\right) \right)^2\right]\\
    =& \sum_{j=1}^n \mathbb{E}\left[ z_j^2 \cdot  x_{j}^{\top}\left(\sum_{i=1}^{n} x_{i} x_{i}^{\top}\right)^{-1}\left(\sum_{i=1}^{n}\frac{1}{\sqrt{n}} x_{i}\right) \cdot \left(\sum_{i=1}^{n} \frac{1}{\sqrt{n}}x_{i}\right)^{\top} \left(\sum_{i=1}^{n} x_{i} x_{i}^{\top}\right)^{-1}   x_{j}\right]\\
    =& \mathbb{E}\left[ \left( x_{1}^{\top}\left(\sum_{i=1}^{n} x_{i} x_{i}^{\top}\right)^{-1}\left(\sum_{i=1}^{n}\frac{1}{\sqrt{n}} x_{i}\right) \right)^2 \right].
\end{align*}

Similar to the previous arguments, we can apply leave-one-out argument to obtain \begin{align*}
    & x_{1}^{\top}\left(\sum_{i=1}^{n} x_{i} x_{i}^{\top}\right)^{-1}\left(\sum_{i=1}^{n} \frac{1}{\sqrt{n}} x_{i}\right) \\
    &= \frac{x_{1}^{\top}\left(\sum_{i=2}^{n} x_{i} x_{i}^{\top}\right)^{-1}\left(\sum_{i=1}^{n} \frac{1}{\sqrt{n}} x_{i}\right)}{1 + x_1 ^\top \left(\sum_{i=2}^{n} x_{i} x_{i}^{\top}\right)^{-1} x_1} = O(1) \cdot x_{1}^{\top}\left(\sum_{i=2}^{n} x_{i} x_{i}^{\top}\right)^{-1}\left(\sum_{i=1}^{n} \frac{1}{\sqrt{n}} x_{i}\right) \\
    =&  O(1) \cdot x_{1}^{\top}\left(\sum_{i=2}^{n} x_{i} x_{i}^{\top}\right)^{-1}\left(\sum_{i=2}^{n} \frac{1}{\sqrt{n}} x_{i}\right) + O(1) \cdot x_{1}^{\top}\left(\sum_{i=2}^{n} x_{i} x_{i}^{\top}\right)^{-1} \frac{1}{\sqrt{n}}x_1 = o(1).
\end{align*}




Further, since \begin{align*}
   & \lim_{n \rightarrow \infty} \left | x_{j}^{\top}\left(\sum_{i=1}^{n} x_{i} x_{i}^{\top}\right)^{-1}\left(\sum_{i=1}^{n} \frac{1}{\sqrt{n}} x_{i}\right) \right |
   \leq  \lim_{n \rightarrow \infty} \sqrt{ \|x_j\|_2^2 \cdot \left \|\sum_{i=1}^{n} \frac{1}{\sqrt{n}} x_{i} \right \|_2^2 } \cdot \left \| \left(\sum_{i=1}^{n} x_{i} x_{i}^{\top}\right)^{-1} \right \|,
\end{align*} which is upper bounded by some constant, by bounded convergence theorem we have $$\mathbb{E}\left[\left(x_{j}^{\top}\left(\sum_{i=1}^{n} x_{i} x_{i}^{\top}\right)^{-1}\left(\sum_{i=1}^{n} \frac{1}{\sqrt{n}} x_{i}\right)\right)^{2}\right] = 0,$$ which completes the proof equation (\ref{lemma 6 res 2}).
  \end{proof}
  
 \begin{lemma}
  \label{lemma 9.5}
Let $ p, n \in \mathbb{N}^{+}$ be such that $\lim_{n \rightarrow\infty} \frac{p}{n} = \kappa \in (0, \frac{1}{2})$. Let $\{\beta\}_{n \in \mathbb{N}^+}$ be a sequence of deterministic vectors in $\mathbb{R}^p$ such that $\lim_{n \rightarrow \infty} \frac{\beta_1}{\sqrt{n}} = \gamma$ for some fixed positive constant $\gamma$, and $ \beta_i = 0 \;\; \forall i = 2,3, \ldots p.$ Let $X \in \mathbb{R}^{n \times p}$ consist of i.i.d. entries $\sim N(0, \frac{1}{n})$. Let $\{A_i\}_{i=1,2, \ldots,n}$ be independent Bernoulli random variables with success probability $\sigma(x_i^{\top} \beta) =\sigma(x_{i1} \beta_1) $. Also, let $$y_{i}=\left[\begin{array}{llll}
x_{i 2} & x_{i 3} & \ldots & x_{i n}
\end{array}\right]^{\top} \quad \forall i=1,2, \ldots, n. $$ Then $$ \left(\sum_{i=1}^n A_{i} x_{i 1} y_{i}^{\top}\right)\left(\sum_{i=1}^n A_{i} y_{i} y_{i}^{\top}\right)^{-1}\left(\sum_{i=1}^n\frac{1}{\sqrt{n}} A_{i} y_{i}\right)=2\kappa e_{\gamma, 0}  + o_{p}(1).$$
  \end{lemma}
  \begin{proof}
   Fix any $n \in \mathbb{N}^+$. For any outcome $(a_1, a_2 , \ldots, a_n) \in \{0,1\}^n$ of $(A_1, A_2, \ldots, A_n)$, define $s_n = \sum_{i=1}^n a_i$. Then \begin{align*}
   & \left(\sum_{i=1}^n A_{i} x_{i 1} y_{i}^{\top}\right)\left(\sum_{i=1}^n A_{i} y_{i} y_{i}^{\top}\right)^{-1}\left(\sum_{i=1}^n\frac{1}{\sqrt{n}} A_{i} y_{i}\right)\mid \left(A_{i}=a_{i} \;\; \forall i \right) \\
   \sim & \left(\sum_{i =1 }^{s_n} \frac{1}{\sqrt{n}} y_{i}\right)^{\top}\left(\sum_{i =1 }^{s_n} y_{i} y_{i}^{\top}\right)^{-1}\left(\sum_{i =1 }^{s_n} w_{i 1} y_{i}\right) ,
\end{align*} where $w_{i1} \overset{\text{i.i.d.}}{\sim} (x_{i1} \mid A_1 = 1) \;\; \forall i = 1,2, \ldots, n$.

Further, from the proof of lemma \ref{conditional_expectation_lemma} we can obtain that, \begin{align*}
        & \mathbb{E} \left[  w_{i1} \right] = \frac{1}{\sqrt{n}}
    \mathbb{E}\left[\frac{2z}{1+e^{-z  \beta_1/\sqrt{n}}} \right] :=  c_n, \quad  \forall i=1,2, \ldots, n, \quad \text{where } z \sim N(0, 1).
    \end{align*}

    Since $$\frac{1}{n} \sum_{i=1}^n A_i = \mathbb{E}[A_i] + o(1) = \frac{1}{2} + o(1), \quad \text{we have} \quad \mathbb{P}\left(\sum_{i=1}^{n} A_{i} \geq (\frac{1}{2} - \delta)n \right) \rightarrow 1 \quad \forall \delta > 0.$$ 
    
    Choose any $\delta \in (0, \frac{1}{2} - \kappa)$. There exists an $N$ such that $\sum_{i=1}^{ (\frac{1}{2} - \delta)n   } y_{i} y_{i}^{\top}  $ is invertible for any $n \geq N$. Fix any $\epsilon > 0$, we have \begin{align*}
        & \mathbb{P}\left( \left | \left(\sum_{i=1}^{n} A_{i} (x_{i 1} - c_n) y_{i}^{\top}\right)\left(\sum_{i=1}^{n} A_{i} y_{i} y_{i}^{\top}\right)^{-1}\left(\sum_{i=1}^{n} \frac{1}{\sqrt{n}} A_{i} y_{i}\right) \right| \geq \epsilon  \right)\\
        =& \sum_{(a_1,a_2, \ldots, a_n) \in \{0,1\}^n} \mathbb{P}\left( \left | \left(\sum_{i=1}^{n} A_{i} (x_{i 1} - c_n) y_{i}^{\top}\right)\left(\sum_{i=1}^{n} A_{i} y_{i} y_{i}^{\top}\right)^{-1}\left(\sum_{i=1}^{n} \frac{1}{\sqrt{n}} A_{i} y_{i}\right) \right| \geq \epsilon \left| A_i = a_i \;\; \forall i  \right.  \right) \\
        \cdot & \mathbb{P} \left(A_{i}=a_{i} \;\; \forall i \right)\\
        \leq & \sum_{\substack{(a_1,a_2, \ldots, a_n) \in \{0,1\}^n \\ \text{ s.t. } s_n \geq (\frac{1}{2} - \delta)n} } \mathbb{P}\left( \left | \left(\sum_{i=1}^{n} A_{i} (x_{i 1} - c_n) y_{i}^{\top}\right)\left(\sum_{i=1}^{n} A_{i} y_{i} y_{i}^{\top}\right)^{-1}\left(\sum_{i=1}^{n} \frac{1}{\sqrt{n}} A_{i} y_{i}\right) \right| \geq \epsilon \left| A_i = a_i \;\; \forall i  \right.  \right) \\
        \cdot & \mathbb{P} \left(A_{i}=a_{i} \;\; \forall i \right) + \mathbb{P} \left( \sum_{i=1}^n A_i < (\frac{1}{2} - \delta)n\right) \\
        = & \sum_{\substack{(a_1,a_2, \ldots, a_n) \in \{0,1\}^n \\ \text{ s.t. } s_n \geq (\frac{1}{2} - \delta)n} } \mathbb{P}\left( \left | \left(\sum_{i=1}^{s_n} (w_{i 1} - c_n) y_{i}^{\top}\right)\left(\sum_{i=1}^{s_n}  y_{i} y_{i}^{\top}\right)^{-1}\left(\sum_{i=1}^{s_n} \frac{1}{\sqrt{n}}  y_{i}\right) \right| \geq \epsilon  \;\; \forall i    \right)\\
        \cdot&  \mathbb{P} \left(A_{i}=a_{i} \;\; \forall i \right) + o(1).
    \end{align*}
    
    By Lemma \ref{lemma 0} and Lemma \ref{conditional_expectation_lemma}, for any $i=1,2, \ldots, n$, \begin{align*}
         \text{Var}( \sqrt{n} w_{i1}) = & \mathbb{E}\left[n w_{i1}^2 \right] -nc_n^2 = \mathbb{E}\left[\frac{2 z^2}{1+e^{-z \beta_{1} / \sqrt{n}}}\right] - nc_n^2 = 1 - nc_n^2 \leq 1.
    \end{align*}
    
    Thus similar to the proof of Lemma \ref{lemma 7}, we can show that \begin{equation*} 
        \left(\sum_{i=1}^{{n}}\left(w_{i 1}-c_{n}\right) y_{i}^{\top}\right)\left(\sum_{i=1}^{{n}} y_{i} y_{i}^{\top}\right)^{-1}\left(\sum_{i=1}^{{n}} \frac{1}{\sqrt{n}} y_{i}\right) = o_p(1),
    \end{equation*} which implies \begin{align*}
        & \sum_{\substack{(a_1,a_2, \ldots, a_n) \in \{0,1\}^n \\ \text{ s.t. } s_n \geq (\frac{1}{2} - \delta)n} } \mathbb{P}\left( \left | \left(\sum_{i=1}^{s_n} (w_{i 1} - c_n) y_{i}^{\top}\right)\left(\sum_{i=1}^{s_n}  y_{i} y_{i}^{\top}\right)^{-1}\left(\sum_{i=1}^{s_n} \frac{1}{\sqrt{n}}  y_{i}\right) \right| \geq \epsilon  \;\; \forall i    \right)\\
        \cdot&  \mathbb{P} \left(A_{i}=a_{i} \;\; \forall i \right) + o(1) \\
        \leq &  \mathbb{P}\left( \left | \left(\sum_{i=1}^{(\frac{1}{2} - \delta)n} (w_{i 1} - c_n) y_{i}^{\top}\right)\left(\sum_{i=1}^{(\frac{1}{2} - \delta)n}  y_{i} y_{i}^{\top}\right)^{-1}\left(\sum_{i=1}^{(\frac{1}{2} - \delta)n} \frac{1}{\sqrt{n}}  y_{i}\right) \right| \geq \epsilon  \;\; \forall i  \right) \\
        \cdot& \mathbb{P}\left(\sum_{i=1}^n A_i \geq (\frac{1}{2} - \delta)n \right)+o(1) =  o(1).
    \end{align*}
    
    Hence we have proved that $$\left(\sum_{i=1}^{n} A_{i}\left(x_{i 1}-c_{n}\right) y_{i}^{\top}\right)\left(\sum_{i=1}^{n} A_{i} y_{i} y_{i}^{\top}\right)^{-1}\left(\sum_{i=1}^{n} \frac{1}{\sqrt{n}} A_{i} y_{i}\right) =o_p(1),$$ which implies \begin{align*}
        & \left(\sum_{i=1}^{n} A_{i}x_{i 1} y_{i}^{\top}\right)\left(\sum_{i=1}^{n} A_{i} y_{i} y_{i}^{\top}\right)^{-1}\left(\sum_{i=1}^{n} \frac{1}{\sqrt{n}} A_{i} y_{i}\right) \\ 
        &= \sqrt{n}c_n \left(\sum_{i=1}^{n} \frac{1}{\sqrt{n}} A_{i} y_{i}^{\top}\right)\left(\sum_{i=1}^{n} A_{i} y_{i} y_{i}^{\top}\right)^{-1}\left(\sum_{i=1}^{n} \frac{1}{\sqrt{n}} A_{i} y_{i}\right)  + o_p(1).
    \end{align*}
    
    Since $$\sqrt{n}c_n = 2e_{\gamma, 0} + o_p(1), $$ and by Lemma \ref{lemma 6}, $$\left(\sum_{i=1}^{n} \frac{1}{\sqrt{n}} A_{i} y_{i}^{\top}\right)\left(\sum_{i=1}^{n} A_{i} y_{i} y_{i}^{\top}\right)^{-1}\left(\sum_{i=1}^{n} \frac{1}{\sqrt{n}} A_{i} y_{i}\right) = \kappa + o_p(1), $$ we know \begin{align*}
        & \sqrt{n} c_{n}\left(\sum_{i=1}^{n} \frac{1}{\sqrt{n}} A_{i} y_{i}^{\top}\right)\left(\sum_{i=1}^{n} A_{i} y_{i} y_{i}^{\top}\right)^{-1}\left(\sum_{i=1}^{n} \frac{1}{\sqrt{n}} A_{i} y_{i}\right) =  2\kappa e_{\gamma, 0} + o_p(1).
    \end{align*}
    
    Hence we have completed the proof of the lemma.
    
  \end{proof}

   \begin{lemma} \label{diff_lemma_1}
Let $ p, n \in \mathbb{N}^{+}$ be such that $\lim_{n \rightarrow\infty} \frac{p}{n} = \kappa \in (0, \frac{1}{2})$. Let $\{\beta\}_{n \in \mathbb{N}^+}$ be a sequence of deterministic vectors in $\mathbb{R}^p$ such that $\lim_{n \rightarrow \infty} \frac{\beta_1}{\sqrt{n}} = \gamma $ for some fixed positive constant $\gamma$, and $ \beta_i = 0 \;\; \forall i = 2,3, \ldots p.$ Let $X \in \mathbb{R}^{n \times p}$ consist of i.i.d. entries $\sim N(0, \frac{1}{n})$. Let $\{A_i\}_{i=1,2, \ldots,n}$ be independent Bernoulli random variables with success probability $\sigma(x_i^{\top} \beta) =\sigma(x_{i1} \beta_1) $. Also, let $$y_{i}=\left[\begin{array}{llll}
x_{i 2} & x_{i 3} & \ldots & x_{i n}
\end{array}\right]^{\top} \quad \forall i=1,2, \ldots, n. $$ For any two sequences of $p$-dimensional vectors $\{a\},\{b\}$, define $$O_{a,b} = a^{\top}\left(\sum_{i =1}^nA_i y_{i} y_{i}^{\top}-\left(\sum_{i=1}^nA_i x_{i 1}^{2}\right)^{-1}\left(\sum_{i =1}^nA_i x_{i 1} y_{i}\right)\left(\sum_{i =1}^nA_i x_{i 1} y_{i}\right)^{\top}\right)^{-1} b,$$ $$  Q_{a,b} = a^{\top}\left(\sum_{i =1}^nA_i y_{i} y_{i}^{\top}\right)^{-1}b.$$ Then \begin{align}
    &O_{a,b} - Q_{a,b} \nonumber \\
    =& \frac{\left(\sum_{i=1}^nA_i x_{i 1}^{2}\right)^{-1}\left[a^{\top}\left(\sum_{i=1}^nA_i y_{i} y_{i}^{\top}\right)^{-1}\left(\sum_{i=1}^nA_i x_{i 1} y_{i}\right)\right]}{1-\left(\sum_{i=1}^nA_i x_{i 1}^{2}\right)^{-1}\left(\sum_{i=1}^nA_i x_{i 1} y_{i}\right)^{\top}\left(\sum_{i=1}^nA_i y_{i} y_{i}^{\top}\right)^{-1}\left(\sum_{i=1}^nA_i x_{i 1} y_{i}\right)}  \\
    \cdot & \left[b^{\top}\left(\sum_{i=1}^nA_i y_{i} y_{i}^{\top}\right)^{-1}\left(\sum_{i=1}^nA_i x_{i 1} y_{i}\right)\right]. \nonumber
\end{align}
In particular, if $a = b = \sum_{i=1}^{n} \frac{1}{\sqrt{n}} A_{i} y_{i}$, then $$ O_{a,b} - Q_{a,b} = \frac{  8  \kappa^2 e_{\gamma, 0}^2   }{ 1-2 \kappa  } + o_p(1).$$
  \end{lemma}
  \begin{proof}
  Define $$U = \sum_{i=1}^nA_i y_{i} y_{i}^{\top}, \quad V = \left(\sum_{i=1}^nA_i x_{i 1}^{2}\right)^{-1}\left(\sum_{i=1}^nA_i x_{i 1} y_{i}\right)\left(\sum_{i=1}^nA_i x_{i 1} y_{i}\right)^{\top} .$$

Then we have \begin{align*}
    O_{a,b} =& a^{\top} \left(U-V \right)^{-1} b =  a^{\top}  \left[U^{-1}+\frac{1}{1-\operatorname{Tr}\left(V U^{-1}\right)} U^{-1} V U^{-1}\right]b\\
    =& Q_{a,b} +  \frac{1}{\left(1-\operatorname{Tr}\left(V U^{-1}\right)\right)} a^{\top}   U^{-1} V U^{-1}b\\
    =& Q_{a,b} + \frac{\left(\sum_{i=1}^nA_i x_{i 1}^{2}\right)^{-1}\left[a^{\top}\left(\sum_{i=1}^nA_i y_{i} y_{i}^{\top}\right)^{-1}\left(\sum_{i=1}^nA_i x_{i 1} y_{i}\right)\right]}{1-\left(\sum_{i=1}^nA_i x_{i 1}^{2}\right)^{-1}\left(\sum_{i=1}^nA_i x_{i 1} y_{i}\right)^{\top}\left(\sum_{i=1}^nA_i y_{i} y_{i}^{\top}\right)^{-1}\left(\sum_{i=1}^nA_i x_{i 1} y_{i}\right)}\\
    \cdot & \left[b^{\top}\left(\sum_{i=1}^nA_i y_{i} y_{i}^{\top}\right)^{-1}\left(\sum_{i=1}^nA_i x_{i 1} y_{i}\right)\right]
\end{align*} where the second equality is due to the lemma in \cite{10.2307/2690437}, which also implies that the denominator is non-zero.

In particular, if $$a = b = \sum_{i=1}^{n} \frac{1}{\sqrt{n}} A_{i} y_{i},$$ then \begin{align*}
    & \frac{\left(\sum_{i=1}^{n} A_{i} x_{i 1}^{2}\right)^{-1}\left[a^{\top}\left(\sum_{i=1}^{n} A_{i} y_{i} y_{i}^{\top}\right)^{-1}\left(\sum_{i=1}^{n} A_{i} x_{i 1} y_{i}\right)\right]}{1-\left(\sum_{i=1}^{n} A_{i} x_{i 1}^{2}\right)^{-1}\left(\sum_{i=1}^{n} A_{i} x_{i 1} y_{i}\right)^{\top}\left(\sum_{i=1}^{n} A_{i} y_{i} y_{i}^{\top}\right)^{-1}\left(\sum_{i=1}^{n} A_{i} x_{i 1} y_{i}\right)}\\
    \cdot& \left[b^{\top}\left(\sum_{i=1}^{n} A_{i} y_{i} y_{i}^{\top}\right)^{-1}\left(\sum_{i=1}^{n} A_{i} x_{i 1} y_{i}\right)\right]\\
    =& \frac{\left(\sum_{i=1}^{n} A_{i} x_{i 1}^{2}\right)^{-1}\left[\left( \sum_{i=1}^{n} \frac{1}{\sqrt{n}} A_{i} y_{i}\right)^{\top}\left(\sum_{i=1}^{n} A_{i} y_{i} y_{i}^{\top}\right)^{-1}\left(\sum_{i=1}^{n} A_{i} x_{i 1} y_{i}\right)\right]^2}{1-\left(\sum_{i=1}^{n} A_{i} x_{i 1}^{2}\right)^{-1}\left(\sum_{i=1}^{n} A_{i} x_{i 1} y_{i}\right)^{\top}\left(\sum_{i=1}^{n} A_{i} y_{i} y_{i}^{\top}\right)^{-1}\left(\sum_{i=1}^{n} A_{i} x_{i 1} y_{i}\right)}.
\end{align*}
By Lemma \ref{LLN} and Lemma \ref{lemma 9.5}, we know $$\left(\sum_{i=1}^{n} A_{i} x_{i 1}^{2}\right)^{-1} \xrightarrow{p} \frac{1}{2}, \quad \left(\sum_{i=1}^{n} \frac{1}{\sqrt{n}} A_{i} y_{i}\right)^{\top}\left(\sum_{i=1}^{n} A_{i} y_{i} y_{i}^{\top}\right)^{-1}\left(\sum_{i=1}^{n} A_{i} x_{i 1} y_{i}\right) \xrightarrow{p} 2\kappa e_{\gamma, 0}. $$

Moreover, since by Lemma \ref{lemma 0} and Lemma \ref{conditional_expectation_lemma}, we have \begin{align*}
    & \text{Var}( \sqrt{n} w_{i 1}) = \mathbb{E}\left[nw_{i 1}^2 \right] \mathbb{E}^2\left[\sqrt{n} w_{i 1} \right] = 1 - \mathbb{E}^2\left[\frac{2 z}{1+e^{-\gamma z}}\right] = 1-4 e_{\gamma, 0}^2, \;\; \\
    & \text{where } w_{i1} = (x_{i1} \mid A_i = 1), \;\; i = 1,2, \ldots, n,
\end{align*} similar to the proof of Lemma \ref{lemma 7}, we can show that \begin{equation*} 
        \left(\sum_{i=1}^{{n}}\left(w_{i 1}-c_{n}\right) y_{i}^{\top}\right)\left(\sum_{i=1}^{{n}} y_{i} y_{i}^{\top}\right)^{-1}\left(\sum_{i=1}^{{n}}\left(w_{i 1}-c_{n}\right) y_{i}\right) = \kappa  \left( 1-4 e_{\gamma, 0}^2 \right) + o_p(1).
    \end{equation*}
    
    This implies \begin{align*}
        & \left(\sum_{i=1}^{n}w_{i 1} y_{i}^{\top}\right)\left(\sum_{i=1}^{n} y_{i} y_{i}^{\top}\right)^{-1}\left(\sum_{i=1}^{n}w_{i 1} y_{i}\right) \\
        =&  \left(\sum_{i=1}^{{n}}\left(w_{i 1}-c_{n}\right) y_{i}^{\top}\right)\left(\sum_{i=1}^{{n}} y_{i} y_{i}^{\top}\right)^{-1}\left(\sum_{i=1}^{{n}}\left(w_{i 1}-c_{n}\right) y_{i}\right) \\
        +& 2\sqrt{n}c_n \left(\sum_{i=1}^{n}\frac{1}{\sqrt{n}} y_{i}^{\top}\right)\left(\sum_{i=1}^{n} y_{i} y_{i}^{\top}\right)^{-1}\left(\sum_{i=1}^{n}w_{i 1} y_{i}\right)\\
        -&  nc_n^2 \cdot \frac{1}{n} \left(\sum_{i=1}^{n} y_{i}^{\top}\right)\left(\sum_{i=1}^{n} y_{i} y_{i}^{\top}\right)^{-1}\left(\sum_{i=1}^{n} y_{i}\right)\\
        =& \kappa\left(1-4 e_{\gamma}^{2}\right) +  2 \cdot 2 e_{\gamma, 0} \cdot 2 \kappa e_{\gamma, 0} - 4 e_{\gamma, 0}^2 \kappa  + o_p(1)= \kappa + o_p(1).
    \end{align*}
    
    Then similar to the proof of lemma \ref{lemma 9.5}, we have \begin{equation} \label{eq 9.6}
     \left(\sum_{i=1}^{n} A_{i} x_{i 1} y_{i}\right)^{\top}\left(\sum_{i=1}^{n} A_{i} y_{i} y_{i}^{\top}\right)^{-1}\left(\sum_{i=1}^{n} A_{i} x_{i 1} y_{i}\right) =  \kappa  + o_p(1)  .     \end{equation} As a result, \begin{align*}
        & \frac{\left(\sum_{i=1}^{n} A_{i} x_{i 1}^{2}\right)^{-1}\left[\left(\sum_{i=1}^{n} \frac{1}{\sqrt{n}} A_{i} y_{i}\right)^{\top}\left(\sum_{i=1}^{n} A_{i} y_{i} y_{i}^{\top}\right)^{-1}\left(\sum_{i=1}^{n} A_{i} x_{i 1} y_{i}\right)\right]^{2}}{1-\left(\sum_{i=1}^{n} A_{i} x_{i 1}^{2}\right)^{-1}\left(\sum_{i=1}^{n} A_{i} x_{i 1} y_{i}\right)^{\top}\left(\sum_{i=1}^{n} A_{i} y_{i} y_{i}^{\top}\right)^{-1}\left(\sum_{i=1}^{n} A_{i} x_{i 1} y_{i}\right)} \\
        =&  \frac{  8 \kappa^2  e_{\gamma, 0}^2   }{ 1- 2\kappa  } + o_p(1), 
    \end{align*} which completes the proof of the lemma.
  \end{proof}
  
 \begin{lemma} \label{lemma 11 pre}
   Let $P$ be an orthogonal matrix such that $P \beta = \|\beta\|e_1 $. Define $$z_{i}=P x_{i}, \quad y_{i}=\left[\begin{array}{llll}
z_{i 2} z_{i 3} & \ldots & z_{i p}
\end{array}\right]^{\top}, \quad \forall i=1,2, \ldots, n.  $$ Then for any $(a,b,c)$ which is a permutation of $(1,2,3)$, we have $$\frac{1}{n}\left(\sum_{i \in S_{b}} A_{i} y_{i}^{\top}\right)\left(\sum_{i \in S_{b}} A_{i} y_{i} y_{i}^{\top}\right)^{-1}\left(\sum_{j \in S_{a}} \frac{A_{j} y_{j}}{\sigma\left(x_{j}^{\top} \hat{\beta}_{S_c}\right)}\right)=o_{p}(1).$$
  \end{lemma} 
  \begin{proof}
  Since \begin{align*}
    & \frac{1}{n}\left(\sum_{i \in S_{b}} A_{i} y_{i}^{\top}\right)\left(\sum_{i \in S_{b}} A_{i} y_{i} y_{i}^{\top}\right)^{-1}\left(\sum_{j \in S_{a}} \frac{A_{j} y_{j}}{\sigma\left(x_{j}^{\top} \hat{\beta}_{S_c}\right)}\right) \\
    =& \frac{1}{n}\left(\sum_{i \in S_{b}} A_{i} y_{i}^{\top}\right)\left(\sum_{i \in S_{b}} A_{i} y_{i} y_{i}^{\top}\right)^{-1}\left(\sum_{j \in S_{a}} A_{j} y_{j}\right) \\
    &+ \frac{1}{n}\left(\sum_{i \in S_{b}} A_{i} y_{i}^{\top}\right)\left(\sum_{i \in S_{b}} A_{i} y_{i} y_{i}^{\top}\right)^{-1}\left(\sum_{j \in S_{a}} A_{j} y_{j} e^{-x_j^{\top} \hat{\beta}_{S_c}} \right)\\
    =& \frac{1}{n}\left(\sum_{i \in S_{b}} A_{i} y_{i}^{\top}\right)\left(\sum_{i \in S_{b}} A_{i} y_{i} y_{i}^{\top}\right)^{-1}\left(\sum_{j \in S_{a}} A_{j} y_{j} e^{-x_j^{\top} \hat{\beta}_{S_c}} \right) + o_p(1),
\end{align*} we only need to show $$ \frac{1}{n}\left(\sum_{i \in S_{b}} A_{i} y_{i}^{\top}\right)\left(\sum_{i \in S_{b}} A_{i} y_{i} y_{i}^{\top}\right)^{-1}\left(\sum_{j \in S_{a}} A_{j} y_{j} e^{-x_{j}^{\top} \hat{\beta}_{S_c}}\right) = o_{p}(1).$$

Similar to the proof of Lemma \ref{lemma 9.5}, we only need to prove $$ \frac{1}{n}\left(\sum_{i \in S_{b}}  y_{i}^{\top}\right)\left(\sum_{i \in S_{b}}  y_{i} y_{i}^{\top}\right)^{-1}\left(\sum_{j \in S_{a}}  y_{j} e^{-w_{j}^{\top} \hat{\beta}_{S_c}}\right) = o_{p}(1), \;\; \text{where} \;\; w_j = (x_j \mid A_j = 1) \;\; \forall j \in S_a. $$

Note that by symmetry, the expectation of the left hand side is 0. Thus we only need to show its variance converges to 0.

Define $$T = \frac{1}{n}\left(\sum_{i \in S_{b}} y_{i}^{\top}\right)\left(\sum_{i \in S_{b}} y_{i} y_{i}^{\top}\right)^{-1}\left(\sum_{j \in S_{a}} y_{j} e^{-w_{j}^{\top} \hat{\beta}_{S_{c}}}\right),$$ 
$$T^{\{-b_1\}} = \frac{1}{n}\left(\sum_{i =2}^{n_b} y_{i}^{\top}\right)\left(\sum_{i =2}^{n_b} y_{i} y_{i}^{\top}\right)^{-1}\left(\sum_{j \in S_{a}} y_{j} e^{-w_{j}^{\top} \hat{\beta}_{S_{c}}}\right) ,$$ 
$$T^{\{-a_1\}} = \frac{1}{n}\left(\sum_{i \in S_{b}} y_{i}^{\top}\right)\left(\sum_{i \in S_{b}} y_{i} y_{i}^{\top}\right)^{-1}\left(\sum_{j =2}^{ n_{a}} y_{j} e^{-w_{j}^{\top} \hat{\beta}_{S_{c}}}\right). $$

Then by Efron-Stein inequality, we have \begin{align*}
    & \text{Var}\left(T \right) \leq O(1) \mathbb{E}\left[n_b \left(T- T^{\left\{-b_{1}\right\}}\right)^2  + n_a\left(T- T^{\left\{-a_{1}\right\}}\right)^2 \right].
\end{align*}

Thus we only need to show $$ \mathbb{E}\left[\left(T-T^{\left\{-b_{1}\right\}}\right)^{2} \right] = o(n^{-1}), \quad  \mathbb{E}\left[\left(T-T^{\left\{-a_{1}\right\}}\right)^{2} \right] = o(n^{-1}).$$

Note that \begin{align*}
    & T = \frac{1}{n}\left(\sum_{i \in S_{b}} y_{i}^{\top}\right)\left(\sum_{i \in S_{b}} y_{i} y_{i}^{\top}\right)^{-1}\left(\sum_{j \in S_{a}} y_{j} e^{-w_{j}^{\top} \hat{\beta}_{S_{c}}}\right) \\
    =& \frac{1}{n}\left(\sum_{i =2}^{n_b} y_{i}^{\top}\right)\left(\sum_{i \in S_{b}} y_{i} y_{i}^{\top}\right)^{-1}\left(\sum_{j \in S_{a}} y_{j} e^{-w_{j}^{\top} \hat{\beta}_{S_{c}}}\right) + \frac{1}{n} y_1^\top \left(\sum_{i \in S_{b}} y_{i} y_{i}^{\top}\right)^{-1}\left(\sum_{j \in S_{a}} y_{j} e^{-w_{j}^{\top} \hat{\beta}_{S_{c}}}\right)\\
    =& \frac{1}{n}\left(\sum_{i=2}^{n_{b}} y_{i}^{\top}\right) \left[ \left(\sum_{i =2}^{n_b} y_{i} y_{i}^{\top}\right)^{-1} - \frac{\left(\sum_{i =2}^{n_b} y_{i} y_{i}^{\top}\right)^{-1}y_1^\top y_1 \left(\sum_{i =2}^{n_b} y_{i} y_{i}^{\top}\right)^{-1}}{1 + y_1^\top \left(\sum_{i =2}^{n_b} y_{i} y_{i}^{\top}\right)^{-1} y_1} \right] \left(\sum_{j \in S_{a}} y_{j} e^{-w_{j}^{\top} \hat{\beta}_{S_{c}}}\right) \\
    +& \frac{1}{n}\left(\sum_{j \in S_{a}} y_{j}^\top e^{-w_{j}^{\top} \hat{\beta}_{S_{c}}}\right) \frac{ \left(\sum_{i=2}^{n_{b}} y_{i} y_{i}^{\top}\right)^{-1} y_{1} }{ 1+y_{1}^{\top}\left(\sum_{i=2}^{n_{b}} y_{i} y_{i}^{\top}\right)^{-1} y_{1} }\\
    =& T^{\left(-b_{1}\right)} + \frac{1}{n}\left(\sum_{i=2}^{n_{b}} y_{i}^{\top}\right) \frac{\left(\sum_{i=2}^{n_{b}} y_{i} y_{i}^{\top}\right)^{-1} y_{1} y_{1}^{\top} \left(\sum_{i=2}^{n_{b}} y_{i} y_{i}^{\top}\right)^{-1}}{1+y_{1}^{\top}\left(\sum_{i=2}^{n_{b}} y_{i} y_{i}^{\top}\right)^{-1} y_{1}} \left(\sum_{j \in S_{a}} y_{j} e^{-w_{j}^{\top} \hat{\beta}_{S_{c}}}\right) \\
    +& \frac{1}{n}\left(\sum_{j \in S_{a}} y_{j}^\top e^{-w_{j}^{\top} \hat{\beta}_{S_{c}}}\right) \frac{ \left(\sum_{i=2}^{n_{b}} y_{i} y_{i}^{\top}\right)^{-1} y_{1} }{ 1+y_{1}^{\top}\left(\sum_{i=2}^{n_{b}} y_{i} y_{i}^{\top}\right)^{-1} y_{1} }.
\end{align*}

Denote the deterministic equivalent of $\left(\sum_{i=2}^{n_{b}} y_{i} y_{i}^{\top}\right)^{-1}$ as $c I$ for some constant $c$. Then \begin{align*}
    & \frac{1}{n}\left(\sum_{i=2}^{n_{b}} y_{i}^{\top}\right) \frac{\left(\sum_{i=2}^{n_{b}} y_{i} y_{i}^{\top}\right)^{-1} y_{1} y_{1}^{\top} \left(\sum_{i=2}^{n_{b}} y_{i} y_{i}^{\top}\right)^{-1}}{1+y_{1}^{\top}\left(\sum_{i=2}^{n_{b}} y_{i} y_{i}^{\top}\right)^{-1} y_{1}}\left(\sum_{j \in S_{a}} y_{j} e^{-w_{j}^{\top} \hat{\beta}_{S_{c}}}\right) \\
    =& O(1) \frac{1}{n} \left(\sum_{i=2}^{n_{b}} y_{i}^{\top}\right) \left(\sum_{i=2}^{n_{b}} y_{i} y_{i}^{\top}\right)^{-1} y_{1} y_{1}^{\top}  \left(\sum_{j \in S_{a}} y_{j} e^{-w_{j}^{\top} \hat{\beta}_{S_{c}}}\right) + o(1) \quad \text{a.s.,}
\end{align*} and \begin{align*}
    & \frac{1}{n}\left(\sum_{j \in S_{a}} y_{j}^{\top} e^{-w_{j}^{\top} \hat{\beta}_{S_{c}}}\right) \frac{\left(\sum_{i=2}^{n_{b}} y_{i} y_{i}^{\top}\right)^{-1} y_{1}}{1+y_{1}^{\top}\left(\sum_{i=2}^{n_{b}} y_{i} y_{i}^{\top}\right)^{-1} y_{1}} =  O(1) \frac{1}{n}\left(\sum_{j \in S_{a}} y_{j}^{\top} e^{-w_{j}^{\top} \hat{\beta}_{S_{c}}}\right)y_1 +o(1) .
\end{align*}

Thus \begin{align*}
    & \mathbb{E}\left[n\left(T-T^{\left(-b_{1}\right)}\right)^{2}\right] \\
    =& \mathbb{E}\left[n\left( O(1) \frac{1}{n}\left(\sum_{i=2}^{n_{b}} y_{i}^{\top}\right)\left(\sum_{i=2}^{n_{b}} y_{i} y_{i}^{\top}\right)^{-1} y_{1} y_{1}^{\top} \left(\sum_{j \in S_{a}} y_{j} e^{-w_{j}^{\top} \hat{\beta}_{S_{c}}}\right) \right.\right. \\
    -& \left.\left. O(1) \frac{1}{n}\left(\sum_{j \in S_{a}} y_{j}^{\top} e^{-w_{j}^{\top} \hat{\beta}_{S_{c}}}\right) y_{1} \right)^{2}\right] + o(1)\\
    \leq &  O(1) \mathbb{E} \left[  \frac{1}{n} \left[ \left(\sum_{i=2}^{n_{b}} y_{i}^{\top}\right)\left(\sum_{i=2}^{n_{b}} y_{i} y_{i}^{\top}\right)^{-1} y_{1} y_{1}^{\top}\left(\sum_{j \in S_{a}} y_{j} e^{-w_{j}^{\top} \hat{\beta}_{S_{c}}}\right) \right]^2 \right] \\
    &+ O(1) \mathbb{E}\left[\frac{1}{n} \left[ \left(\sum_{j \in S_{a}} y_{j}^{\top} e^{-w_{j}^{\top} \hat{\beta}_{S_{c}}}\right) y_{1} \right]^2 \right] + o(1).
\end{align*}

Since \begin{align*}
    \left(\sum_{i=2}^{n_{b}} y_{i}^{\top}\right)\left(\sum_{i=2}^{n_{b}} y_{i} y_{i}^{\top}\right)^{-1} y_{1}^{\top} = o(1), \quad \frac{1}{\sqrt{n}}  y_{1}^{\top}\left(\sum_{j \in S_{a}} y_{j} e^{-w_{j}^{\top} \hat{\beta}_{S_{c}}}\right) = o(1)
\end{align*} (because the left hand side has zero mean and variance converging to zero), we have $$\mathbb{E}\left[n\left(T-T^{\left\{-b_{1}\right\}}\right)^{2}\right] = o(1).$$

We then show $$\mathbb{E}\left[n\left(T-T^{\left\{-a_{1}\right\}}\right)^{2}\right] = o(1).$$

Note that \begin{align*}
    & \mathbb{E}\left[n\left(T-T^{\left\{-a_{1}\right\}}\right)^{2}\right] = \frac{1}{n} \mathbb{E}\left[\left( \left(\sum_{i \in S_{b}} y_{i}^{\top}\right)\left(\sum_{i \in S_{b}} y_{i} y_{i}^{\top}\right)^{-1} y_{a_1} e^{-w_{a_1}^{\top} \hat{\beta}_{S_{c}}} \right)^{2}\right]\\
    \leq &  \sqrt{\mathbb{E}\left[ \left( \frac{1}{\sqrt{n}} \left(\sum_{i \in S_{b}} y_{i}^{\top}\right)\left(\sum_{i \in S_{b}} y_{i} y_{i}^{\top}\right)^{-1} y_{a_1}\right)^4 \right] \mathbb{E}\left[ e^{-4w_{a_1}^{\top} \hat{\beta}_{S_{c}}} \right]}
\end{align*}

Further, $$ \frac{1}{\sqrt{n}}\left(\sum_{i \in S_{b}} y_{i}^{\top}\right)\left(\sum_{i \in S_{b}} y_{i} y_{i}^{\top}\right)^{-1} y_{a_1} = o(1), $$ since the left hand side has zero mean and variance converging to zero. Thus we only need to show $$ \mathbb{E}\left[e^{-4 w_{a_1}^{\top} \hat{\beta}_{S_{c}}}\right] = O(1). $$ 

Let $P$ be an orthogonal matrix such that $P \beta = \|\beta\|e_1 $. Define $$\tilde{\beta}_{S_c} = P\hat{\beta}_{S_c},\quad \tilde{\beta}_{2:p} = [\tilde{\beta}_2\;\; \tilde{\beta}_3 \;\; \ldots \;\; \tilde{\beta}_p]^{\top} ,   \quad z_i = Px_i, \quad y_i = [z_{i2}\; z_{i3}\; \ldots\; z_{ip}]^{\top}, \quad \forall i = 1,2, \ldots, n.$$ 

We have \begin{align*}
    & \mathbb{E}\left[e^{-4 w_{a_1}^{\top} \hat{\beta}_{S_{c}}}\right] = \mathbb{E}\left[e^{-4 x_{a_1}^{\top} \hat{\beta}_{S_{c}}} \mid A_{a_1} = 1\right] =\mathbb{E}\left[\mathbb{E}\left[e^{-4 x_{a_1}^{\top} \hat{\beta}_{S_{c}}} \mid A_{a_1} = 1, \hat{\beta}_{S_{c}}\right] \mid A_{a_1=1}\right]\\
    =&  \mathbb{E}\left[ \mathbb{E}\left[e^{-4 z_{a_{1},1}\tilde{\beta}_{1, S_c} } \mid A_{a_1}=1, \hat{\beta}_{S_{c}} \right] \mathbb{E}\left[e^{ y_{a_1}^\top \tilde{\beta}_{2:p,S_c} } \mid \hat{\beta}_{S_{c}}\right]  \mid A_{a_1}=1 \right],
\end{align*} where \begin{align*}
    & \mathbb{E}\left[e^{y_{a_1}^{\top} \tilde{\beta}_{2: p, S_{c}}} \mid \hat{\beta}_{S_{c}}\right] = e^{\frac{1}{2n} \| \tilde{\beta}_{2: p, S_{c}} \|^2} \leq e^{\frac{1}{2 n}\left\|\tilde{\beta}_{ S_{c}}\right\|^{2}} = e^{\frac{1}{2 n}\left\|\hat{\beta}_{ S_{c}}\right\|^{2}} = O(1),
\end{align*} 

Moreover, since $$
\tilde{\beta}_{1, S_c}=\frac{\beta^{\top} \hat{\beta}_{S_{c}}}{\|\beta\|} = O(1),
$$ we have \begin{align*}
    & \mathbb{E}\left[\mathbb{E}\left[e^{-4 z_{a_{1}, 1} \tilde{\beta}_{1, S_{c}}} \mid A_{a_{1}}=1, \hat{\beta}_{S_{c}}\right] \mathbb{E}\left[e^{y_{a_{1}}^{\top} \tilde{\beta}_{2: p, S_{c}}} \mid \hat{\beta}_{S_{c}}\right] \mid A_{a_{1}}=1\right]\\
    =& O(1) \mathbb{E}\left[\mathbb{E}\left[e^{-4 z_{a_{1}, 1} \tilde{\beta}_{1, S_{c}}} \mid A_{a_{1}}=1, \hat{\beta}_{S_{c}}\right] \mid A_{a_{1}}=1\right] =  O(1) \mathbb{E}\left[e^{-4 z_{a_{1}, 1} \tilde{\beta}_{1, S_{c}}}  \mid A_{a_{1}}=1\right] = O(1),
\end{align*} which completes the proof of the lemma.




  \end{proof}
  
  \begin{lemma}\label{lemma 12pre}
   Let $P$ be an orthonormal matrix such that $P \beta = \|\beta\|e_1 $. Define $$\tilde{\beta} = P\hat{\beta}_{S_b},\quad \tilde{\beta}_{2:p} = [\tilde{\beta}_2\;\; \tilde{\beta}_3 \;\; \ldots \;\; \tilde{\beta}_p]^{\top} ,   \quad z_i = Px_i, \quad y_i = [z_{i2}\; z_{i3}\; \ldots\; z_{ip}]^{\top}, \quad \forall i = 1,2, \ldots, n.$$ Then for any $(a,b,c)$ which is a permutation of $(1,2,3)$, we have \begin{align*}
       & \frac{1}{\sqrt{n}} \left(\sum_{i \in S_{a}} \frac{A_{i} y_{i}}{\sigma\left(z_{i}^{\top} \tilde{\beta}\right)}\right)^{\top}\left(\sum_{i \in S_{a}} A_{i} y_{i} y_{i}^{\top}\right)^{-1} \sum_{i \in S_{a}} A_{i} z_{i 1} y_{i}\\
       =& 2 \kappa\left(e_{\gamma, 0}+e^{\frac{\left(\alpha_{b}^{*} \gamma\right)^{2}+\kappa_{b}\left(\sigma_{b}^{*}\right)^{2}}{2}} e_{\gamma,-\alpha_{b}^{*} \gamma}\right) +o_p(1),
   \end{align*} \begin{align*}
       &\frac{1}{{n}} \left(\sum_{i \in S_a} A_{i} y_{i}\right)^{\top} \left(\sum_{i  \in S_a} A_iy_iy_i^{\top} \right)^{-1}\left(\sum_{i \in S_a} \frac{A_{i} y_{i}}{\sigma\left(z_{i}^{\top} \tilde{\beta}\right)}\right) \\
       =& 2\kappa\left(\frac{1}{2}+ e^{\frac{\left(\alpha_{b}^{*} \gamma\right)^{2}+\kappa_{b}\left(\sigma_{b}^{*}\right)^{2}}{2}} q_{\gamma,-\alpha_{b}^{*} \gamma}\right) +o_p(1). 
   \end{align*}  
   \end{lemma}

   \begin{proof}
   
   First we show \begin{align*}
       & \frac{1}{\sqrt{n}} \mathbb{E}\left[\left(\sum_{i \in S_{a}} \frac{A_{i} y_{i}}{\sigma\left(z_{i}^{\top} \tilde{\beta}\right)}\right)^{\top}\left(\sum_{i \in S_{a}} A_{i} y_{i} y_{i}^{\top}\right)^{-1} \sum_{i \in S_{a}} A_{i} z_{i 1} y_{i}\right]\\
       =& 2 \kappa\left(e_{\gamma, 0}+e^{\frac{\left(\alpha_{b}^{*} \gamma\right)^{2}+\kappa_{b}\left(\sigma_{b}^{*}\right)^{2}}{2}} e_{\gamma,-\alpha_{b}^{*} \gamma}\right) +o(1).
   \end{align*} 
   
   We only need to show 
   \begin{align*}
   &\frac{1}{\sqrt{n}} \mathbb{E}\left[\left(\sum_{i \in S_{a}} \frac{A_{i} y_{i}}{\sigma\left(z_{i}^{\top} \tilde{\beta}\right)}\right)^{\top}\left(\sum_{i \in S_{a}} A_{i} y_{i} y_{i}^{\top}\right)^{-1} \sum_{i \in S_{a}} A_{i} z_{i 1} y_{i} \mid \hat{\beta}_{S_b} \right] \\
   &= 2 \kappa\left(e_{\gamma, 0}+e^{\frac{\left(\alpha_{b}^{*} \gamma\right)^{2}+\kappa_{b}\left(\sigma_{b}^{*}\right)^{2}}{2}} e_{\gamma,-\alpha_{b}^{*} \gamma}\right) +o(1) ,
   \end{align*}

   We know \begin{align*}
    &\frac{1}{\sqrt{n}} \mathbb{E} \left[  \left(\sum_{i \in S_a} \frac{A_{i} y_{i}}{\sigma\left(z_{i}^{\top} \tilde{\beta}\right)}\right)^{\top}   \left(\sum_{i \in S_a} A_iy_{i} y_{i}^{\top}\right)^{-1} \sum_{i \in S_a} A_i z_{i 1} y_{i}\mid \hat{\beta}_{S_b} \right] \\
    =& \frac{1}{\sqrt{n}} \sum_{i\in S_a} \mathbb{E} \left[  \left( \frac{A_{i} y_{i}}{\sigma\left(z_{i}^{\top} \tilde{\beta}\right)}\right)^{\top}   \left(\sum_{j \in S_a } A_jy_{j} y_{j}^{\top}\right)^{-1}  A_i z_{i 1} y_{i}\mid \hat{\beta}_{S_b} \right]\\
    +& \frac{1}{\sqrt{n}} \sum_{i \neq k ,\; i,k \in S_a} \mathbb{E} \left[  \left( \frac{A_{i} y_{i}}{\sigma\left(z_{i}^{\top} \tilde{\beta}\right)}\right)^{\top}   \left(\sum_{j \in S_a } A_jy_{j} y_{j}^{\top}\right)^{-1}  A_k z_{k 1} y_{k} \mid \hat{\beta}_{S_b}\right].
    \end{align*}
    
    For any $i \neq k, \; i ,k \in S_a $, by symmetry of $y_k$ we have \begin{align*}
        & \mathbb{E} \left[\left( \frac{A_{i} y_{i}}{\sigma\left(z_{i}^{\top} \tilde{\beta}\right)}\right)^{\top}   \left(\sum_{j \in S_a } A_jy_{j} y_{j}^{\top}\right)^{-1}  A_k z_{k 1} y_{k}\mid \hat{\beta}_{S_b} \right] \\
        =& \mathbb{E} \left[\left( \frac{A_{i} y_{i}}{\sigma\left(z_{i}^{\top} \tilde{\beta}\right)}\right)^{\top}   \left(A_ky_{k} y_{k}^{\top} + \sum_{j \in S_a, \; j \neq k } A_jy_{j} y_{j}^{\top}\right)^{-1}  A_k z_{k 1} y_{k} \mid \hat{\beta}_{S_b}\right]= 0 .
    \end{align*} 
    
    Hence \begin{align*}
    &\frac{1}{\sqrt{n}} \mathbb{E} \left[  \left(\sum_{i \in S_a} \frac{A_{i} y_{i}}{\sigma\left(z_{i}^{\top} \tilde{\beta}\right)}\right)^{\top}   \left(\sum_{i \in S_a} A_iy_{i} y_{i}^{\top}\right)^{-1} \sum_{i \in S_a} A_i z_{i 1} y_{i} \mid \hat{\beta}_{S_b} \right] \\
    =& \frac{1}{\sqrt{n}} \sum_{i\in S_a} \mathbb{E} \left[  \left( \frac{A_{i} y_{i}}{\sigma\left(z_{i}^{\top} \tilde{\beta}\right)}\right)^{\top}   \left(\sum_{j \in S_a } A_jy_{j} y_{j}^{\top}\right)^{-1}  A_i z_{i 1} y_{i}\mid \hat{\beta}_{S_b} \right] \\
    &= r_a \mathbb{E} \left[   \frac{\sqrt{n} A_{1} z_{1 1}}{\sigma\left(z_{1}^{\top} \tilde{\beta}\right)} \cdot  y_1^{\top}   \left(\sum_{j \in S_a} A_jy_{j} y_{j}^{\top}\right)^{-1}   y_1 \mid \hat{\beta}_{S_b} \right]\\
    =& r_a \mathbb{E} \left[   \frac{\sqrt{n}  z_{1 1}}{\sigma\left(z_{1}^{\top} \tilde{\beta}\right)} \cdot  y_1^{\top}   \left( y_1^{\top} y_1 + \sum_{j \in S_a, j \neq 1 } A_jy_{j} y_{j}^{\top}\right)^{-1}   y_1 \mid A_1=1 ,  \hat{\beta}_{S_b} \right] \cdot \mathbb{P}(A_1=1)\\
    =& \frac{r_a}{2} \mathbb{E} \left[   \frac{\sqrt{n}  w_{1 1}}{\sigma\left(w_{1 1} \tilde{\beta}_1 + y_1^{\top} \tilde{\beta}_{2:p}\right)} \cdot  y_1^{\top}   \left( y_1^{\top} y_1 + \sum_{j \in S_a, j \neq 1 } A_jy_{j} y_{j}^{\top}\right)^{-1}    y_1  \mid \hat{\beta}_{S_b} \right], \\
    &\text{where }w_{11} = (z_{11} \mid A_1=1)\\
    =& \frac{r_a}{2} \mathbb{E} \left[   \frac{\sqrt{n}  w_{1 1}}{\sigma\left(w_{1 1} \tilde{\beta}_1 + y_1^{\top} \tilde{\beta}_{2:p}\right)} \cdot     \frac{y_1^{\top}\left( \sum_{j \in S_a,j \neq 1 } A_jy_{j} y_{j}^{\top}\right)^{-1}    y_1}{1+y_1^{\top}\left( \sum_{j \in S_a,j \neq 1 } A_jy_{j} y_{j}^{\top}\right)^{-1}    y_1}  \mid \hat{\beta}_{S_b} \right] \quad \text{(Sherman–Morrison)}\\
    =&\frac{r_a}{2} \mathbb{E} \left[  \sqrt{n}  w_{1 1}(1+e^{-w_{1 1} \tilde{\beta}_1 - y_1^{\top} \tilde{\beta}_{2:p}}) \cdot    \frac{y_1^{\top}\left( \sum_{j \in S_a,j \neq 1 } A_jy_{j} y_{j}^{\top}\right)^{-1}    y_1}{1+y_1^{\top}\left( \sum_{j \in S_a,j \neq 1 } A_jy_{j} y_{j}^{\top}\right)^{-1}    y_1}  \mid \hat{\beta}_{S_b} \right] \\
    =& \frac{r_a}{2} \mathbb{E} \left[  \sqrt{n}  w_{1 1} \right] \cdot  \mathbb{E} \left[    \frac{y_1^{\top}\left( \sum_{j \in S_a,j \neq 1 } A_jy_{j} y_{j}^{\top}\right)^{-1}    y_1}{1+y_1^{\top}\left( \sum_{j \in S_a,j \neq 1 } A_jy_{j} y_{j}^{\top}\right)^{-1}    y_1} \right]\\
    +& \frac{r_a}{2} \mathbb{E} \left[  \sqrt{n}  w_{1 1} e^{-w_{11} \tilde{\beta}_{1}} \mid \hat{\beta}_{S_b} \right] \cdot  \mathbb{E} \left[      \frac{y_1^{\top}\left( \sum_{j \in S_a,j \neq 1 } A_jy_{j} y_{j}^{\top}\right)^{-1}    y_1}{1+y_1^{\top}\left( \sum_{j \in S_a,j \neq 1 } A_jy_{j} y_{j}^{\top}\right)^{-1}    y_1} \cdot  e^{-y_{1}^{\top} \tilde{\beta}_{2:p}} \mid \hat{\beta}_{S_b} \right].
\end{align*}
Similar to Lemma \ref{deterministic equiv}, we can show that the deterministic equivalent of $ \left(\sum_{j \in S_a,j \neq 1} A_{j} y_{j} y_{j}^{\top}\right)^{-1} $ is $$\frac{n}{n_{a}} \cdot \frac{2}{1-\frac{2 p}{n_{a}}} I_{p-1}=\frac{2 n}{n_{a}-2 p} I_{p-1} .$$

Since $\lim_{n \rightarrow \infty} \left \| \left(\sum_{j \in S_a,j \neq 1} A_{j} y_{j} y_{j}^{\top}\right)^{-1} \right \| $ is bounded, by uniform integrability we obtain \begin{align*}
  &   \lim_{n \rightarrow \infty} \mathbb{E} \left[     \frac{y_1^{\top}\left( \sum_{j \neq 1 } A_jy_{j} y_{j}^{\top}\right)^{-1}    y_1}{1+y_1^{\top}\left( \sum_{j \neq 1 } A_jy_{j} y_{j}^{\top}\right)^{-1}    y_1}   \right] = \lim_{n \rightarrow \infty} \frac{\frac{2 n}{n_{a}-2 p}\|y_1\|^2 }{1+\frac{2 n}{n_{a}-2 p}\|y_1\|^2 } = \lim_{n \rightarrow \infty} \frac{\frac{2 n}{n_{a}-2 p} \frac{p}{n} }{1+\frac{2 n}{n_{a}-2 p}\frac{p}{n} } \\
  &=  2 \lim_{n \rightarrow \infty}\frac{p}{n_a} =  2 \kappa_a.
\end{align*} 

Further, by Lemma \ref{conditional_expectation_lemma} we know $$\lim _{n \rightarrow \infty} \mathbb{E}\left[\sqrt{n} w_{11}\right] = 2 e_{\gamma, 0}.$$

We then find $ \lim_{n \rightarrow \infty} \mathbb{E}\left[\sqrt{n} w_{11} e^{-w_{11} \tilde{\beta}_{1}} \mid \hat{\beta}_{S_b} \right]$. Since the first row of the projection matrix $P$ is $\frac{\beta^{\top}}{\|\beta\|}$, we have \begin{align*}
   \frac{\tilde{\beta}_1}{\sqrt{n}}  = \frac{1}{\sqrt{n}} (P\hat{\beta}_{S_b})_1 = \frac{\beta^{\top} \hat{\beta}_{S_b}}{\sqrt{n}\|\beta\|} \xrightarrow{a.s.} \alpha_{b}^{*} \gamma.
\end{align*} 

Thus \begin{align*}
    &  \lim_{n \rightarrow \infty} \mathbb{E}\left[\sqrt{n} w_{11} e^{-w_{11} \tilde{\beta}_{1}} \mid \hat{\beta}_{S_b} \right] \\
    =&  \lim_{n \rightarrow \infty} \mathbb{E}\left[\sqrt{n} w_{11} e^{-\sqrt{n}w_{11} \tilde{\beta}_{1}/\sqrt{n}} \mid \hat{\beta}_{S_b} \right]=  \lim_{n \rightarrow \infty} \mathbb{E}\left[\sqrt{n} z_{11} e^{-\sqrt{n}z_{11} \tilde{\beta}_{1}/\sqrt{n}} \mid A_1 = 1 ,\hat{\beta}_{S_b} \right]\\
    =& 2\lim_{n \rightarrow \infty} \int_{-\infty}^{\infty} \sqrt{n} z e^{-\sqrt{n}z \tilde{\beta}_{1}/\sqrt{n}} \cdot p(A_1=1 \mid z_{11}=z) \cdot p\left(z_{11}=z\right) d z\\
    =& 2\lim_{n \rightarrow \infty} \int_{-\infty}^{\infty} \sqrt{n} z e^{-\sqrt{n}z \tilde{\beta}_{1}/\sqrt{n}} \cdot \frac{1}{1+e^{-z\|\beta\|}} \cdot \frac{\sqrt{n}}{\sqrt{2 \pi}} e^{-\frac{n z^{2}}{2}} d z\\
    =&  \frac{2}{\sqrt{2\pi}}\lim_{n \rightarrow \infty} \int_{-\infty}^{\infty}  z e^{-z \tilde{\beta}_{1}/\sqrt{n}} \cdot \frac{1}{1+e^{-z\|\beta\|/\sqrt{n}}}  e^{-\frac{ z^{2}}{2}} d z \\
    =&  2\lim_{n \rightarrow \infty} \mathbb{E}\left[   \frac{ z e^{-z \tilde{\beta}_{1}/\sqrt{n}}}{1+e^{-z\|\beta\|/\sqrt{n}}} \mid \hat{\beta}_{S_b} \right] = 2\lim_{n \rightarrow \infty} \mathbb{E}\left[   \frac{ z e^{-z \tilde{\beta}_{1}/\sqrt{n}}}{1+e^{-z\|\beta\|/\sqrt{n}}}  \right], 
\end{align*} where the last step is because the limit of the conditional expectation is a constant that does not depend on $ \hat{\beta}_{S_b} $.

For large $n$, there exists $c$ such that \begin{align*}
     \left| \frac{z e^{-z \tilde{\beta}_{1} / \sqrt{n}}}{1+e^{-z\|\beta\| / \sqrt{n}}}  \right| \leq  \left| {z e^{-z \tilde{\beta}_{1} / \sqrt{n}}}  \right| \leq ze^{-cz},
\end{align*}

Thus by dominant convergence theorem \begin{align*}
    & 2 \lim _{n \rightarrow \infty} \mathbb{E}\left[\frac{z e^{-z \tilde{\beta}_{1} / \sqrt{n}}}{1+e^{-z\|\beta\| / \sqrt{n}}}  \right] = 2  \mathbb{E}\left[\lim _{n \rightarrow \infty} \frac{z e^{-z \tilde{\beta}_{1} / \sqrt{n}}}{1+e^{-z\|\beta\| / \sqrt{n}}}  \right] = 2  \mathbb{E}\left[\frac{z e^{-\alpha_{b}^{*} \gamma z  }}{1+e^{-\gamma  z}}\right] \\
    =&  2 \mathbb{E} \left[  \frac{z - \alpha_b^* \gamma }{1 + e^{-\gamma(z - \alpha_b^* \gamma)}} \right] e^{\frac{(\alpha_b^* \gamma)^2}{2}} = 2e^{\frac{(\alpha_b^* \gamma)^2}{2}} e_{\gamma, -\alpha_{b}^{*} \gamma} , 
\end{align*} where the second last step is obtained via a change of variable.

Next we find $$\lim_{n \rightarrow \infty} \mathbb{E}\left[\frac{y_{1}^{\top}\left(\sum_{j \in S_{a}, j \neq 1} A_{j} y_{j} y_{j}^{\top}\right)^{-1} y_{1}}{1+y_{1}^{\top}\left(\sum_{j \in S_{a}, j \neq 1} A_{j} y_{j} y_{j}^{\top}\right)^{-1} y_{1}} \cdot e^{-y_{1}^{\top} \tilde{\beta}_{2:p}}  \mid \hat{\beta}_{S_b} \right].$$

We have shown $$ \lim_{n \rightarrow \infty} \frac{y_{1}^{\top}\left(\sum_{j \in S_{a}, j \neq 1} A_{j} y_{j} y_{j}^{\top}\right)^{-1} y_{1}}{1+y_{1}^{\top}\left(\sum_{j \in S_{a}, j \neq 1} A_{j} y_{j} y_{j}^{\top}\right)^{-1} y_{1}} = 2 \kappa_a \quad \text{a.s.}$$

We then show\begin{align*}
    & \lim_{n \rightarrow \infty} \mathbb{E}\left[\frac{y_{1}^{\top}\left(\sum_{j \in S_{a}, j \neq 1} A_{j} y_{j} y_{j}^{\top}\right)^{-1} y_{1}}{1+y_{1}^{\top}\left(\sum_{j \in S_{a}, j \neq 1} A_{j} y_{j} y_{j}^{\top}\right)^{-1} y_{1}} \cdot e^{-y_{1}^{\top} \tilde{\beta}_{2:p}}   \mid \hat{\beta}_{S_b}\right] =  \lim_{n \rightarrow \infty} \mathbb{E}\left[ 2 \kappa_{a}  e^{-y_{1}^{\top} \tilde{\beta}_{2:p}}  \mid \hat{\beta}_{S_b} \right].
\end{align*}

We only need to show $$ \lim_{n \rightarrow \infty} \mathbb{E}\left[\left( \frac{y_{1}^{\top}\left(\sum_{j \in S_{a}, j \neq 1} A_{j} y_{j} y_{j}^{\top}\right)^{-1} y_{1}}{1+y_{1}^{\top}\left(\sum_{j \in S_{a}, j \neq 1} A_{j} y_{j} y_{j}^{\top}\right)^{-1} y_{1}} - 2 \kappa_{a} \right) \cdot e^{-y_{1}^{\top} \tilde{\beta}_{2:p}}  \mid \hat{\beta}_{S_b} \right]  = 0.$$

By Cauchy–Schwarz inequality, we have \begin{align*}
    & \mathbb{E}\left[\left( \frac{y_{1}^{\top}\left(\sum_{j \in S_{a}, j \neq 1} A_{j} y_{j} y_{j}^{\top}\right)^{-1} y_{1}}{1+y_{1}^{\top}\left(\sum_{j \in S_{a}, j \neq 1} A_{j} y_{j} y_{j}^{\top}\right)^{-1} y_{1}} - 2 \kappa_{a} \right) \cdot e^{-y_{1}^{\top} \tilde{\beta}_{2:p}}  \mid \hat{\beta}_{S_b} \right] \\
    \leq & \sqrt{ \mathbb{E}\left[\left( \frac{y_{1}^{\top}\left(\sum_{j \in S_{a}, j \neq 1} A_{j} y_{j} y_{j}^{\top}\right)^{-1} y_{1}}{1+y_{1}^{\top}\left(\sum_{j \in S_{a}, j \neq 1} A_{j} y_{j} y_{j}^{\top}\right)^{-1} y_{1}} - 2 \kappa_{a}\right) ^2 \right]  \mathbb{E}\left[ e^{-2y_{1}^{\top} \tilde{\beta}_{2: p}}  \mid \hat{\beta}_{S_b}\right] }.
\end{align*}

Note that \begin{align*}
    & \mathbb{E}\left[e^{-2 y_{1}^{\top} \tilde{\beta}_{2: p}} \mid \hat{\beta}_{S_b}\right] = \mathbb{E}\left[e^{\frac{2 \| \tilde{\beta}_{2: p} \|^2}{n}  }  \mid \hat{\beta}_{S_b}\right] =\mathbb{E}\left[e^{\frac{2 \| \tilde{\beta}_{2: p} \|^2}{n}  }  \right] .
\end{align*} Further, \begin{align*}
    &\frac{1}{n} \sum_{i=2}^{p} \tilde{\beta}_{i}^{2} =  \frac{1}{n} \| \tilde{\beta}\|^2 - \frac{1}{n} \tilde{\beta}_1^2 = \frac{1}{n} \| \hat{\beta}_{S_b}\|^2 - \frac{1}{n} (\frac{\beta^{\top}}{\| \beta \|} \hat{\beta}_{S_b})^2\\ \xrightarrow{a.s.} & \kappa_{b}\left(\sigma_{b}^{*}\right)^{2}+\left(\alpha_{b}^{*}\right)^{2} \gamma^{2} - \frac{(\alpha_{b}^{*} \gamma^{2})^2}{\gamma^2}=  \kappa_{b}\left(\sigma_{b}^{*}\right)^{2}.
\end{align*}

Thus we have $$ \lim_{n \rightarrow \infty }\mathbb{E}\left[e^{\frac{2\|\tilde{\beta}_{2: p}\|^{2}}{n}}\right] = e^{2 \kappa_{b}\left(\sigma_{b}^{*}\right)^{2} } = O(1) .$$

Moreover, we have $$\mathbb{E}\left[\left(\frac{y_{1}^{\top}\left(\sum_{j \in S_{a}, j \neq 1} A_{j} y_{j} y_{j}^{\top}\right)^{-1} y_{1}}{1+y_{1}^{\top}\left(\sum_{j \in S_{a}, j \neq 1} A_{j} y_{j} y_{j}^{\top}\right)^{-1} y_{1}}-2 \kappa_a\right)^{2}\right]= o(1).$$

Thus we have shown $$\lim _{n \rightarrow \infty} \mathbb{E}\left[\frac{y_{1}^{\top}\left(\sum_{j \in S_{a}, j \neq 1} A_{j} y_{j} y_{j}^{\top}\right)^{-1} y_{1}}{1+y_{1}^{\top}\left(\sum_{j \in S_{a}, j \neq 1} A_{j} y_{j} y_{j}^{\top}\right)^{-1} y_{1}} \cdot e^{-y_{1}^{\top} \tilde{\beta}_{2: p}} \mid \hat{\beta}_{S_b} \right]=\lim _{n \rightarrow \infty} \mathbb{E}\left[2 \kappa_a e^{-y_{1}^{\top} \tilde{\beta}_{2: p}} \mid \hat{\beta}_{S_b} \right].$$ 

As a result, \begin{align*}
    & \lim _{n \rightarrow \infty} \mathbb{E}\left[\frac{y_{1}^{\top}\left(\sum_{j \in S_{a}, j \neq 1} A_{j} y_{j} y_{j}^{\top}\right)^{-1} y_{1}}{1+y_{1}^{\top}\left(\sum_{j \in S_{a}, j \neq 1} A_{j} y_{j} y_{j}^{\top}\right)^{-1} y_{1}} \cdot e^{-y_{1}^{\top} \tilde{\beta}_{2: p}}\mid \hat{\beta}_{S_b}\right]= \lim _{n \rightarrow \infty} \mathbb{E}\left[2 \kappa_a e^{-y_{1}^{\top} \tilde{\beta}_{2: p}}\mid \hat{\beta}_{S_b}\right] \\
    =& 2 \kappa_a \lim _{n \rightarrow \infty} \mathbb{E}\left[e^{\frac{\|\tilde{\beta}_{2: p}\|^{2}}{2n}}\mid \hat{\beta}_{S_b}\right]  = 2 \kappa_a e^{ \frac{1}{2}\kappa_{b}\left(\sigma_{b}^{*}\right)^{2}}.
\end{align*}

Therefore \begin{align*}
    & \frac{1}{\sqrt{n}} \mathbb{E}\left[\left(\sum_{i \in S_{a}} \frac{A_{i} y_{i}}{\sigma\left(z_{i}^{\top} \tilde{\beta}\right)}\right)^{\top}\left(\sum_{i \in S_{a}} A_{i} y_{i} y_{i}^{\top}\right)^{-1} \sum_{i \in S_{a}} A_{i} z_{i 1} y_{i} \mid \hat{\beta}_{S_b} \right] \\
   = &  \frac{r_a}{2} \cdot 2 e_{\gamma, 0} \cdot 2 \kappa_a + \frac{r_a}{2} \cdot 2 e^{\frac{\left(\alpha_{b}^{*} \gamma\right)^{2}}{2}} e_{\gamma,-\alpha_{b}^{*} \gamma} \cdot 2 \kappa_a e^{\frac{1}{2} \kappa_{b}\left(\sigma_{b}^{*}\right)^{2}} + o(1) \\
   &= 2 \kappa\left ( e_{\gamma, 0} +  e^{\frac{\left(\alpha_{b}^{*} \gamma\right)^{2}+\kappa_{b}\left(\sigma_{b}^{*}\right)^{2}}{2}}   e_{\gamma,-\alpha_{b}^{*} \gamma}  \right ) + o(1),
\end{align*} which implies 
\begin{align*} 
&\frac{1}{\sqrt{n}} \mathbb{E}\left[\left(\sum_{i \in S_{a}} \frac{A_{i} y_{i}}{\sigma\left(z_{i}^{\top} \tilde{\beta}\right)}\right)^{\top}\left(\sum_{i \in S_{a}} A_{i} y_{i} y_{i}^{\top}\right)^{-1} \sum_{i \in S_{a}} A_{i} z_{i 1} y_{i} \right] \\
&= 2 \kappa\left(e_{\gamma, 0}+e^{\frac{\left(\alpha_{b}^{*} \gamma\right)^{2}+\kappa_{b}\left(\sigma_{b}^{*}\right)^{2}}{2}} e_{\gamma,-\alpha_{b}^{*} \gamma}\right) + o(1).
\end{align*} 

We then find the limit of $$\frac{1}{n} \mathbb{E}\left[\left(\sum_{i \in S_{a}} A_{i} y_{i}\right)^{\top}\left(\sum_{i \in S_{a}} A_{i} y_{i} y_{i}^{\top}\right)^{-1}\left(\sum_{i \in S_{a}} \frac{A_{i} y_{i}}{\sigma\left(z_{i}^{\top} \tilde{\beta}\right)}\right)  \right].$$

Similar to the previous argument, we can treat $\hat{\beta}_{S_b}$ as a deterministic vector. Note that \begin{align*}
    & \lim_{n \rightarrow \infty} \frac{1}{n} \mathbb{E}\left[\left(\sum_{i \in S_{a}} A_{i} y_{i}\right)^{\top}\left(\sum_{i \in S_{a}} A_{i} y_{i} y_{i}^{\top}\right)^{-1}\left(\sum_{i \in S_{a}} \frac{A_{i} y_{i}}{\sigma\left(z_{i}^{\top} \tilde{\beta}\right)}\right)\right] \\
    =& \lim_{n \rightarrow \infty} \frac{1}{n} \mathbb{E}\left[\left(\sum_{i \in S_{a}} A_{i} y_{i}\right)^{\top}\left(\sum_{i \in S_{a}} A_{i} y_{i} y_{i}^{\top}\right)^{-1}\left(\sum_{i \in S_{a}} {A_{i} y_{i}}\right)\right] \\
    +& \lim_{n \rightarrow \infty} \frac{1}{n} \mathbb{E}\left[\left(\sum_{i \in S_{a}} A_{i} y_{i}\right)^{\top}\left(\sum_{i \in S_{a}} A_{i} y_{i} y_{i}^{\top}\right)^{-1}\left(\sum_{i \in S_{a}} A_{i} y_{i}e^{-z_i^{\top} \tilde{\beta}}\right)\right]\\
    =& \lim_{n \rightarrow \infty} \frac{1}{n} \mathbb{E}\left[\left(\sum_{i \in S_{a}} A_{i} y_{i}\right)^{\top}\left(\sum_{i \in S_{a}} A_{i} y_{i} y_{i}^{\top}\right)^{-1}\left(\sum_{i \in S_{a}} A_{i} y_{i}e^{-z_i^{\top} \tilde{\beta}}\right)\right] + \kappa. 
\end{align*}

Further, for any $i \neq j , \; i,j \in S_a$, by symmetry of $y_i$ we have \begin{align*}
    & \mathbb{E}\left[ A_{i} y_{i}^{\top}\left(\sum_{k \in S_a} A_{k} y_{k} y_{k}^{\top}\right)^{-1} {A_{j} y_{j}}e^{-z_{j}^{\top} \tilde{\beta}}\right]\\
    =&  \frac{1}{2}\mathbb{E}\left[  y_{i}^{\top}\left(y_iy_i^{\top}+\sum_{k \neq i} A_{k} y_{k} y_{k}^{\top}\right)^{-1} {A_{j} y_{j}}e^{-z_{j}^{\top} \tilde{\beta}}\right] = 0.
\end{align*} Therefore \begin{align*}
    & \frac{1}{n} \mathbb{E}\left[\left(\sum_{i \in S_a} A_{i} y_{i}\right)^{\top}\left(\sum_{i\in S_a} A_{i} y_{i} y_{i}^{\top}\right)^{-1}\left(\sum_{i\in S_a} {A_{i} y_{i}}e^{-z_{i}^{\top} \tilde{\beta}}\right)\right] \\
   & = \sum_{i\in S_a} \frac{1}{n} \mathbb{E}\left[ A_{i} y_{i}^{\top}\left(\sum_{j\in S_a} A_{j} y_{j} y_{j}^{\top}\right)^{-1} {A_{i} y_{i}}e^{-z_{i}^{\top} \tilde{\beta}}\right]\\
    =& r_a \mathbb{E}\left[ A_{1} y_{1}^{\top}\left(\sum_{j\in S_a} A_{j} y_{j} y_{j}^{\top}\right)^{-1} {A_{1} y_{1}}e^{-z_{11} \tilde{\beta}_1- y_1 ^{\top} \tilde{\beta}_{2:p}}\right] \\
    =&  \frac{r_a}{2}\mathbb{E}\left[  y_{1}^{\top}\left(y_1y_1^{\top}+\sum_{j \in S_a, j \neq 1} A_{j} y_{j} y_{j}^{\top}\right)^{-1} { y_{1}}e^{-w_{11} \tilde{\beta}_1- y_1 ^{\top} \tilde{\beta}_{2:p}}\right]\\
    =& \frac{r_a}{2} \mathbb{E}\left[ e^{-w_{11} \tilde{\beta}_1} \right] \mathbb{E}\left[  y_{1}^{\top}\left(y_1y_1^{\top}+\sum_{j\in S_a, j \neq 1} A_{j} y_{j} y_{j}^{\top}\right)^{-1} { y_{1}}e^{- y_1 ^{\top} \tilde{\beta}_{2:p}}\right]\\
    =& \frac{r_a}{2} \mathbb{E}\left[ e^{-w_{11} \tilde{\beta}_1} \right] \mathbb{E}\left[ e^{- y_1 ^{\top} \tilde{\beta}_{2:p}} \frac{y_{1}^{\top}\left(\sum_{j \neq 1} A_{j} y_{j} y_{j}^{\top}\right)^{-1} { y_{1}}}{1+y_1^{\top}\left(\sum_{j \neq 1} A_{j} y_{j} y_{j}^{\top}\right)^{-1} { y_{1}} }\right] \quad \text{(Sherman–Morrison)}.
\end{align*}

Since \begin{align*}
    & \lim_{n \rightarrow \infty}\mathbb{E}\left[e^{-w_{11} \tilde{\beta}_{1}}\right] = \lim_{n \rightarrow \infty} \mathbb{E}\left[e^{-z_{11} \tilde{\beta}_{1}} \mid A_1 = 1 \right] \\
    =&  2 \lim_{n \rightarrow \infty} \int_{-\infty}^{\infty}  e^{-\sqrt{n} z \tilde{\beta}_{1} / \sqrt{n}} \cdot p\left(A_{1}=1 \mid z_{11}=z\right) \cdot p\left(z_{11}=z\right) d z\\
    =& 2 \lim _{n \rightarrow \infty} \int_{-\infty}^{\infty}  e^{-\sqrt{n} z \tilde{\beta}_{1} / \sqrt{n}} \cdot \frac{1}{1+e^{-z\|\beta\|}} \cdot \frac{\sqrt{n}}{\sqrt{2 \pi}} e^{-\frac{n z^{2}}{2}} d z \\
    =&  \frac{2}{\sqrt{2 \pi}} \lim _{n \rightarrow \infty} \int_{-\infty}^{\infty}  e^{-z \tilde{\beta}_{1} / \sqrt{n}} \cdot \frac{1}{1+e^{-z\|\beta\| / \sqrt{n}}} e^{-\frac{z^{2}}{2}} d z\\
    =& 2 \lim _{n \rightarrow \infty} \mathbb{E}\left[\frac{ e^{-z \tilde{\beta}_{1} / \sqrt{n}}}{1+e^{-z\|\beta\| / \sqrt{n}}}\right] \\
    =&  2 \mathbb{E}\left[\frac{ e^{-\alpha_{b}^{*} \gamma z}}{1+e^{-\gamma z}}\right] = 2 \mathbb{E}\left[ \frac{1}{1+e^{-\gamma (z - \alpha_b^* \gamma)}} \right] e^{\frac{(\alpha_b^* \gamma)^2}{2}} = 2e^{\frac{\left(\alpha_{b}^{*} \gamma\right)^{2}}{2}} q_{\gamma, -\alpha_{b}^{*} \gamma},
\end{align*}

We have \begin{align*}
        & \lim _{n \rightarrow \infty} \frac{1}{n} \mathbb{E}\left[\left(\sum_{i \in S_{a}} A_{i} y_{i}\right)^{\top}\left(\sum_{i \in S_{a}} A_{i} y_{i} y_{i}^{\top}\right)^{-1}\left(\sum_{i \in S_{a}} \frac{A_{i} y_{i}}{\sigma\left(z_{i}^{\top} \tilde{\beta}\right)}\right)\right]\\
        =& \kappa +  r_a \mathbb{E}\left[\frac{ e^{-\alpha_{b}^{*} \gamma z}}{1+e^{-\gamma z}}\right] \cdot 2 \kappa_a e^{\frac{1}{2} \kappa_{b}\left(\sigma_{b}^{*}\right)^{2}} =  2\kappa\left(\frac{1}{2}+ e^{\frac{\left(\alpha_{b}^{*} \gamma\right)^{2}+\kappa_{b}\left(\sigma_{b}^{*}\right)^{2}}{2}} q_{\gamma,-\alpha_{b}^{*} \gamma}\right).
\end{align*}

We then show $$\text{Var}\left( \frac{1}{n} \left(\sum_{i \in S_{a}} A_{i} y_{i}\right)^{\top}\left(\sum_{i \in S_{a}} A_{i} y_{i} y_{i}^{\top}\right)^{-1}\left(\sum_{i \in S_{a}} \frac{A_{i} y_{i}}{\sigma\left(z_{i}^{\top} \tilde{\beta}\right)}\right) \right)  = o(1). $$

We know $ \frac{1}{n}\left(\sum_{i \in S_{a}} A_{i} y_{i}\right)^{\top}\left(\sum_{i \in S_{a}} A_{i} y_{i} y_{i}^{\top}\right)^{-1}\left(\sum_{i \in S_{a}} {A_{i} y_{i}}\right) $ converges to some constant. Furthermore, uniform integrability holds for the sequence. Therefore we only need to show

$$\text{Var}\left(\frac{1}{n}\left(\sum_{i \in S_{a}} A_{i} y_{i}\right)^{\top}\left(\sum_{i \in S_{a}} A_{i} y_{i} y_{i}^{\top}\right)^{-1}\left(\sum_{i \in S_{a}} {A_{i} y_{i} e^{-z_i^\top \tilde{\beta}}}\right) \right) = o(1). $$

Note that \begin{align*}
    &\frac{1}{n}\left(\sum_{i \in S_{a}} A_{i} y_{i}\right)^{\top}\left(\sum_{i \in S_{a}} A_{i} y_{i} y_{i}^{\top}\right)^{-1}\left(\sum_{i \in S_{a}} {A_{i} y_{i} e^{-z_i^\top \tilde{\beta}}}\right) \\
    &= \frac{1}{n}\left(\sum_{i \in S_{a}} A_{i} y_{i}\right)^{\top}\left(\sum_{i \in S_{a}} A_{i} y_{i} y_{i}^{\top}\right)^{-1}\left(\sum_{i \in S_{a}} {A_{i} y_{i} e^{ y_i^\top \tilde{\beta}_{2:p}}} \cdot e^{-z_{i1}\tilde{\beta}_1}\right)\\
    \sim & \frac{1}{n}\left(\sum_{i=1}^{\lfloor\frac{n_a}{2}\rfloor}  y_{i}\right)^{\top}\left(\sum_{i=1}^{\lfloor\frac{n_a}{2}\rfloor}   y_{i} y_{i}^{\top}\right)^{-1}\left(\sum_{i=1}^{\lfloor\frac{n_a}{2}\rfloor}  { y_{i} e^{ y_i^\top \tilde{\beta}_{2:p}}} \cdot e^{-w_{i1}\tilde{\beta}_1}\right).
\end{align*}

For all $i,j \in \{1,2, \ldots, \left\lfloor\frac{n_{a}}{2}\right\rfloor \}$, define $$l_i = e^{-w_{i 1} \tilde{\beta}_{1}}, \quad T = \frac{1}{n}\left(\sum_{i=1}^{\lfloor\frac{n_a}{2}\rfloor}  y_{i}\right)^{\top}\left(\sum_{i=1}^{\lfloor\frac{n_a}{2}\rfloor}   y_{i} y_{i}^{\top}\right)^{-1}\left(\sum_{i=1}^{\lfloor\frac{n_a}{2}\rfloor}  { y_{i} e^{ y_i^\top \tilde{\beta}_{2:p}}} l_i\right) ,$$ $$T^{(l_j)} = \frac{1}{n}\left(\sum_{i=1}^{\lfloor\frac{n_a}{2}\rfloor}  y_{i}\right)^{\top}\left(\sum_{i=1}^{\lfloor\frac{n_a}{2}\rfloor}   y_{i} y_{i}^{\top}\right)^{-1}\left({ y_{j} e^{ y_j^\top \tilde{\beta}_{2:p}}} l_j^{'} + \sum_{1 \leq i \leq \lfloor\frac{n_a}{2}\rfloor, i \neq j}  { y_{i} e^{ y_i^\top \tilde{\beta}_{2:p}}} l_i \right) , $$
$$ T^{(y_j)} = \frac{1}{n}\left(\sum_{1 \leq i \leq \lfloor\frac{n_a}{2}\rfloor, i \neq j}  y_{i}\right)^{\top}\left(\sum_{1 \leq i \leq \lfloor\frac{n_a}{2}\rfloor, i \neq j}  y_{i} y_{i}^{\top}\right)^{-1}\left( \sum_{1 \leq i \leq \lfloor\frac{n_a}{2}\rfloor, i \neq j}  { y_{i} e^{ y_i^\top \tilde{\beta}_{2:p}}} l_i \right) . $$

Since $y_1, y_2, \ldots, y_{n_a}, l_1, l_2, \ldots, l_{n_a}$ are independent, by Efron–Stein inequality, we have \begin{align*}
    &\operatorname{Var}(T) \leq \frac{1}{2} \mathbb{E}\left[\sum_{i=1}^{\lfloor\frac{n_a}{2}\rfloor}\left(T-T^{(l_i)}\right)^{2}+\sum_{j=1}^{\lfloor\frac{n_a}{2}\rfloor}\left(T-T^{(y_j)}\right)^{2}\right] \\
    &= O(n) \mathbb{E}\left[\left(T-T^{(l_1)}\right)^{2}\right] + O(n) \mathbb{E}\left[\left(T-T^{(y_1)}\right)^{2}\right].
\end{align*}

Define $\tilde{n}_a = \lfloor\frac{n_a}{2}\rfloor.$ Then \begin{align*}
    &\mathbb{E}\left[\left(T-T^{(l_1)}\right)^{2}\right] = \frac{1}{n^2} \mathbb{E}\left[\left(\left(\sum_{i=1}^{\left\lfloor\frac{n_{a}}{2}\right\rfloor} y_{i}\right)^{\top}\left(\sum_{i=1}^{\left\lfloor\frac{n_{a}}{2}\right\rfloor} y_{i} y_{i}^{\top}\right)^{-1}y_{j} e^{y_{j}^{\top} \tilde{\beta}_{2: p}}\right)^{2} \cdot \left(l_j - l_j^{'} \right)^2 \right]\\
    =& \frac{1}{n^2} \mathbb{E}\left[\left(\left(\sum_{i=1}^{\tilde{n}_a} y_{i}\right)^{\top}\left(\sum_{i=1}^{\tilde{n}_a} y_{i} y_{i}^{\top}\right)^{-1}y_{j} \right)^{2} e^{2y_{j}^{\top} \tilde{\beta}_{2: p}} \cdot \left(l_j - l_j^{'} \right)^2 \right]\\
    \leq & \frac{1}{n^2} \sqrt{\mathbb{E}\left[\left(\left(\sum_{i=1}^{\tilde{n}_{a}} y_{i}\right)^{\top}\left(\sum_{i=1}^{\tilde{n}_{a}} y_{i} y_{i}^{\top}\right)^{-1} y_{j}\right)^{4}  \right] \cdot \mathbb{E} \left[ e^{4 y_{j}^{\top} \tilde{\beta}_{2: p}} \cdot\left(l_{j}-l_{j}^{\prime}\right)^{4}  \right]}\\
    =& O\left(\frac{1}{n^2}\right) \sqrt{\mathbb{E}\left[\left(\left(\sum_{i=1}^{\tilde{n}_{a}} y_{i}\right)^{\top}\left(\sum_{i=1}^{\tilde{n}_{a}} y_{i} y_{i}^{\top}\right)^{-1} y_{j}\right)^{4}  \right]}.
\end{align*}

Thus in order to show $$O(n) \mathbb{E}\left[\left(T-T^{\left(l_{1}\right)}\right)^{2}\right] = o(1), $$ we only need to show 
$$\mathbb{E}\left[\left(\left(\sum_{i=1}^{\tilde{n}_{a}} y_{i}\right)^{\top}\left(\sum_{i=1}^{\tilde{n}_{a}} y_{i} y_{i}^{\top}\right)^{-1} y_{j}\right)^{4}  \right] = o(n^2).$$

By Sherman–Morrison formula, we have \begin{align*}
    & \left(\sum_{i=1}^{\tilde{n}_a} y_{i}\right)^{\top}\left(\sum_{i=1}^{\tilde{n}_a} y_{i} y_{i}^{\top}\right)^{-1}y_{j} =  \frac{\left(\sum_{i=1}^{\tilde{n}_a} y_{i}\right)^{\top} \left(\sum_{ 1 \leq i \leq \tilde{n}_a, i \leq j} y_{i} y_{i}^{\top}\right)^{-1} y_j}{1 + y_j^\top \left(\sum_{ 1 \leq i \leq \tilde{n}_a, i \leq j} y_{i} y_{i}^{\top}\right)^{-1} y_j } \\
    &=  O(1) \cdot \left(\sum_{i=1}^{\tilde{n}_a} y_{i}\right)^{\top} \left(\sum_{ 1 \leq i \leq \tilde{n}_a, i \neq j} y_{i} y_{i}^{\top}\right)^{-1} y_j.
\end{align*}

Thus we only need to show $$\mathbb{E}\left[\left(\left(\sum_{i=1}^{\tilde{n}_{a}} y_{i}\right)^{\top}\left(\sum_{ 1 \leq i \leq \tilde{n}_a, i \neq j} y_{i} y_{i}^{\top}\right)^{-1} y_{j}\right)^{4}  \right] = o(n^2).$$

In fact, we only need to show $$ \frac{1}{\sqrt{n}}\left(\sum_{i=1}^{\tilde{n}_{a}} y_{i}\right)^{\top}\left(\sum_{1 \leq i \leq \tilde{n}_{a}, i \neq j} y_{i} y_{i}^{\top}\right)^{-1} y_{j} = o(1) \quad  \text{a.s.} $$

Since \begin{align*}
    &  \left(\sum_{i=1}^{\tilde{n}_{a}} y_{i}\right)^{\top}\left(\sum_{1 \leq i \leq \tilde{n}_a, i \neq j} y_{i} y_{i}^{\top}\right)^{-1} y_{j} \\
    &=  y_j^{\top}\left(\sum_{1 \leq i \leq \tilde{n}_a, i \neq j} y_{i} y_{i}^{\top}\right)^{-1} y_{j} + 
    \left(\sum_{1 \leq i \leq \tilde{n}_a, i \neq j} y_i \right)^{\top}\left(\sum_{1 \leq i \leq \tilde{n}_a, i \neq j} y_{i} y_{i}^{\top}\right)^{-1} y_{j}, 
\end{align*} and $$y_{j}^{\top}\left(\sum_{1 \leq i \leq \tilde{n}_{a}, i \neq j} y_{i} y_{i}^{\top}\right)^{-1} y_{j}=O(1) \quad \text{a.s.},$$ we only need to show $$ \frac{1}{\sqrt{n}}\left(\sum_{1 \leq i \leq \tilde{n}_{a}, i \neq j} y_{i}\right)^{\top}\left(\sum_{1 \leq i \leq \tilde{n}_{a}, i \neq j} y_{i} y_{i}^{\top}\right)^{-1} y_{j}  = o(1) \quad  \text{a.s.} $$

By symmetry the expectation of the left hand side is 0. Hence we only need to show $$\frac{1}{n} \mathbb{E}\left[\left( \left(\sum_{1 \leq i \leq \tilde{n}_{a}, i \neq j} y_{i}\right)^{\top}\left(\sum_{1 \leq i \leq \tilde{n}_{a}, i \neq j} y_{i} y_{i}^{\top}\right)^{-1} y_{j} \right)^2 \right]   = o(1). $$

 We have  \begin{align*}
    &\mathbb{E}\left[ \left( \left(\sum_{1 \leq i \leq \tilde{n}_{a}, i \neq j} y_{i}\right)^{\top}\left(\sum_{1 \leq i \leq \tilde{n}_{a}, i \neq j} y_{i} y_{i}^{\top}\right)^{-1} y_{j} \right)^2 \right]\\
    =& \mathbb{E}\left[ \left(\sum_{1 \leq i \leq \tilde{n}_{a}, i \neq j} y_{i}\right)^{\top}\left(\sum_{1 \leq i \leq \tilde{n}_{a}, i \neq j} y_{i} y_{i}^{\top}\right)^{-1} y_{j} y_j^\top \left(\sum_{1 \leq i \leq \tilde{n}_{a}, i \neq j} y_{i} y_{i}^{\top}\right)^{-1} \left(\sum_{1 \leq i \leq \tilde{n}_{a}, i \neq j} y_{i}\right) \right]\\
    =& \frac{1}{n} \mathbb{E}\left[ \left(\sum_{1 \leq i \leq \tilde{n}_{a}, i \neq j} y_{i}\right)^{\top}\left(\sum_{1 \leq i \leq \tilde{n}_{a}, i \neq j} y_{i} y_{i}^{\top}\right)^{-2}  \left(\sum_{1 \leq i \leq \tilde{n}_{a}, i \neq j} y_{i}\right) \right]  = O(1) .
\end{align*}

Thus we have proved $$O(n) \mathbb{E}\left[\left(T-T^{\left(l_{1}\right)}\right)^{2}\right]  = o(1) .$$

Next we show $$O(n) \mathbb{E}\left[\left(T-T^{\left(y_{1}\right)}\right)^{2}\right]  = o(1) . $$

Note that \begin{align*}
    & T-T^{\left(y_{1}\right)}\\
    =& \frac{1}{n} y_j^{\top}\left(\sum_{i=1}^{\left\lfloor\frac{n_{a}}{2}\right\rfloor} y_{i} y_{i}^{\top}\right)^{-1}\left(\sum_{i=1}^{\left\lfloor\frac{n_{a}}{2}\right\rfloor} y_{i} e^{y_{i}^{\top} \tilde{\beta}_{2: p}} l_{i}\right) + \frac{1}{n}\left(\sum_{i=1}^{\left\lfloor\frac{n_{a}}{2}\right\rfloor} y_{i}\right)^{\top}\left(\sum_{i=1}^{\left\lfloor\frac{n_{a}}{2}\right\rfloor} y_{i} y_{i}^{\top}\right)^{-1} y_{j} e^{y_{j}^{\top} \tilde{\beta}_{2: p}} l_{j}\\
    +& \frac{1}{n}\left(\sum_{1 \leq i \leq \lfloor\frac{n_a}{2}\rfloor, i \neq j}  y_{i}\right)^{\top}\left[ \left(\sum_{1 \leq i \leq \lfloor\frac{n_a}{2}\rfloor}  y_{i} y_{i}^{\top}\right)^{-1} - \left(\sum_{1 \leq i \leq \lfloor\frac{n_a}{2}\rfloor, i \neq j}  y_{i} y_{i}^{\top}\right)^{-1} \right]\\
    \cdot & \left( \sum_{1 \leq i \leq \lfloor\frac{n_a}{2}\rfloor, i \neq j}  { y_{i} e^{ y_i^\top \tilde{\beta}_{2:p}}} l_i \right) .
\end{align*}

Thus we only need to show 
\begin{align*} 
&\mathbb{E}\left[ \left(\frac{1}{n} y_{j}^{\top}\left(\sum_{i=1}^{\left\lfloor\frac{n_{a}}{2}\right\rfloor} y_{i} y_{i}^{\top}\right)^{-1}\left(\sum_{i=1}^{\left\lfloor\frac{n_{a}}{2}\right\rfloor} y_{i} e^{y_{i}^{\top} \tilde{\beta}_{2: p}} l_{i}\right) \right)^2 \right] = o(n^{-1}), \\ 
 &\mathbb{E}\left[ \left(   \frac{1}{n}\left(\sum_{1 \leq i \leq \lfloor\frac{n_a}{2}\rfloor, i \neq j}  y_{i}\right)^{\top}\left[ \left(\sum_{1 \leq i \leq \lfloor\frac{n_a}{2}\rfloor}  y_{i} y_{i}^{\top}\right)^{-1} - \left(\sum_{1 \leq i \leq \lfloor\frac{n_a}{2}\rfloor, i \neq j}  y_{i} y_{i}^{\top}\right)^{-1} \right] \right.\right. \\
 \cdot& \left. \left. \left( \sum_{1 \leq i \leq \lfloor\frac{n_a}{2}\rfloor, i \neq j}  { y_{i} e^{ y_i^\top \tilde{\beta}_{2:p}}} l_i \right) \right)^2 \right] = o(n^{-1}).
\end{align*}

First we show $$\mathbb{E}\left[\left(\frac{1}{\sqrt{n}} y_{j}^{\top}\left(\sum_{i=1}^{\left\lfloor\frac{n_{a}}{2}\right\rfloor} y_{i} y_{i}^{\top}\right)^{-1}\left(\sum_{i=1}^{\left\lfloor\frac{n_{a}}{2}\right\rfloor} y_{i} e^{y_{i}^{\top} \tilde{\beta}_{2: p}} l_{i}\right)\right)^{2}\right]=o\left(1\right).$$

Since $$ \frac{1}{\sqrt{n}} y_{j}^{\top}\left(\sum_{i=1}^{\left\lfloor\frac{n_{a}}{2}\right\rfloor} y_{i} y_{i}^{\top}\right)^{-1}\left( y_{i} e^{y_{j}^{\top} \tilde{\beta}_{2: p}} l_{j}\right) = o(1) \quad \text{a.s.}, $$ we only need to show $$ \mathbb{E}\left[\left(\frac{1}{\sqrt{n}} y_{j}^{\top}\left(\sum_{i=1}^{\left\lfloor\frac{n_{a}}{2}\right\rfloor} y_{i} y_{i}^{\top}\right)^{-1}\left(\sum_{1 \leq i  \leq \left\lfloor\frac{n_{a}}{2}\right\rfloor, i \neq j} y_{i} e^{y_{i}^{\top} \tilde{\beta}_{2: p}} l_{i}\right)\right)^{2}\right]=o(1) .$$

By Sherman–Morrison, \begin{align*}
    & \frac{1}{\sqrt{n}} y_{j}^{\top}\left(\sum_{i=1}^{\left\lfloor\frac{n_{a}}{2}\right\rfloor} y_{i} y_{i}^{\top}\right)^{-1}\left(\sum_{1 \leq i \leq\left\lfloor\frac{n_{a}}{2}\right\rfloor, i \neq j} y_{i} e^{y_{i}^{\top} \tilde{\beta}_{2: p}} l_{i}\right)\\
    =& \frac{1}{\sqrt{n}} y_{j}^{\top}\left(\sum_{1 \leq i \leq \left\lfloor\frac{n_{a}}{2}\right\rfloor, i \neq j} y_{i} y_{i}^{\top}\right)^{-1}\left(\sum_{1 \leq i \leq\left\lfloor\frac{n_{a}}{2}\right\rfloor, i \neq j} y_{i} e^{y_{i}^{\top} \tilde{\beta}_{2: p}} l_{i}\right)  \frac{1}{ y_j^\top \left(\sum_{1 \leq i \leq\left\lfloor\frac{n_{a}}{2}\right\rfloor, i \neq j} y_{i} y_{i}^{\top}\right)^{-1} y_j }.
\end{align*}

Thus \begin{align*}
    & \mathbb{E}\left[\left(\frac{1}{\sqrt{n}} y_{j}^{\top}\left(\sum_{i=1}^{\left\lfloor\frac{n_{a}}{2}\right\rfloor} y_{i} y_{i}^{\top}\right)^{-1}\left(\sum_{1 \leq i \leq\left\lfloor\frac{n_{a}}{2}\right\rfloor, i \neq j} y_{i} e^{y_{i}^{\top} \tilde{\beta}_{2: p}} l_{i}\right)\right)^{2}\right]\\
    =& C \mathbb{E}\left[ \left( \frac{1}{\sqrt{n}} y_{j}^{\top}\left(\sum_{1 \leq i \leq\left\lfloor\frac{n_{a}}{2}\right\rfloor, i \neq j} y_{i} y_{i}^{\top}\right)^{-1}\left(\sum_{1 \leq i \leq\left\lfloor\frac{n_{a}}{2}\right\rfloor, i \neq j} y_{i} e^{y_{i}^{\top} \tilde{\beta}_{2: p}} l_{i}\right) \right)^2 \right] + o(1) \\
    & \text{ for some constant } C \\
    =&  \frac{C}{n^2}  \mathbb{E}\left[ \left(\sum_{1 \leq i \leq\left\lfloor\frac{n_{a}}{2}\right\rfloor, i \neq j} y_{i} e^{y_{i}^{\top} \tilde{\beta}_{2: p}} l_{i}\right)^\top \left(\sum_{1 \leq i \leq\left\lfloor\frac{n_{a}}{2}\right\rfloor, i \neq j} y_{i} y_{i}^{\top}\right)^{-2} \left(\sum_{1 \leq i \leq\left\lfloor\frac{n_{a}}{2}\right\rfloor, i \neq j} y_{i} e^{y_{i}^{\top} \tilde{\beta}_{2: p}} l_{i}\right) \right]\\
    +& o(1) =  o(1).
\end{align*}

Finally we show 
\begin{align*}
   & \mathbb{E}\Big[ \Big(   \frac{1}{\sqrt{n}}\Big(\sum_{1 \leq i \leq \lfloor\frac{n_a}{2}\rfloor, i \neq j}  y_{i}\Big)^{\top}\Big[ \Big(\sum_{1 \leq i \leq \lfloor\frac{n_a}{2}\rfloor}  y_{i} y_{i}^{\top}\Big)^{-1} - \Big(\sum_{1 \leq i \leq \lfloor\frac{n_a}{2}\rfloor, i \neq j}  y_{i} y_{i}^{\top}\Big)^{-1} \Big]\\
&\Big( \sum_{1 \leq i \leq \lfloor\frac{n_a}{2}\rfloor, i \neq j}  { y_{i} e^{ y_i^\top \tilde{\beta}_{2:p}}} l_i \Big) \Big)^2 \Big] = o(1). 
\end{align*}

Note that \begin{align*}
    &    \frac{1}{\sqrt{n}}\left(\sum_{1 \leq i \leq \lfloor\frac{n_a}{2}\rfloor, i \neq j}  y_{i}\right)^{\top}\left[ \left(\sum_{1 \leq i \leq \lfloor\frac{n_a}{2}\rfloor}  y_{i} y_{i}^{\top}\right)^{-1} - \left(\sum_{1 \leq i \leq \lfloor\frac{n_a}{2}\rfloor, i \neq j}  y_{i} y_{i}^{\top}\right)^{-1} \right]\\
    \cdot & \left( \sum_{1 \leq i \leq \lfloor\frac{n_a}{2}\rfloor, i \neq j}  { y_{i} e^{ y_i^\top \tilde{\beta}_{2:p}}} l_i \right)\\
    =& -\frac{1}{\sqrt{n}}\left(\sum_{1 \leq i \leq \lfloor\frac{n_a}{2}\rfloor, i \neq j}  y_{i}\right)^{\top}\left[ \frac{\left(\sum_{1 \leq i \leq \lfloor\frac{n_a}{2}\rfloor, i \neq j}  y_{i} y_{i}^{\top}\right)^{-1} y_j y_j^{\top} \left(\sum_{1 \leq i \leq \lfloor\frac{n_a}{2}\rfloor, i \neq j}  y_{i} y_{i}^{\top}\right)^{-1}}{1 + y_j^{\top} \left(\sum_{1 \leq i \leq \lfloor\frac{n_a}{2}\rfloor, i \neq j}  y_{i} y_{i}^{\top}\right)^{-1} y_j}  \right] \\
    &\left( \sum_{1 \leq i \leq \lfloor\frac{n_a}{2}\rfloor, i \neq j}  { y_{i} e^{ y_i^\top \tilde{\beta}_{2:p}}} l_i \right)\\
    =& C \frac{1}{\sqrt[4]{n}} \left(\sum_{1 \leq i \leq\left\lfloor\frac{n_{a}}{2}\right\rfloor, i \neq j} y_{i}\right)^{\top}\left(\sum_{1 \leq i \leq\left\lfloor\frac{n_{a}}{2}\right\rfloor, i \neq j} y_{i} y_{i}^{\top}\right)^{-1} y_{j} \\
    \cdot & \frac{1}{\sqrt[4]{n}} \left(\sum_{1 \leq i \leq\left\lfloor\frac{n_{a}}{2}\right\rfloor, i \neq j} y_{i}^\top e^{y_{i}^{\top} \tilde{\beta}_{2: p}} l_{i}\right)\left(\sum_{1 \leq i \leq\left\lfloor\frac{n_{a}}{2}\right\rfloor, i \neq j} y_{i} y_{i}^{\top}\right)^{-1} y_{j} + o(1)  \quad \text{ for some constant }C.
\end{align*}

Since \begin{align*}
    & \mathbb{E}\left[\frac{1}{\sqrt[4]{n}}\left(\sum_{1 \leq i \leq\left\lfloor\frac{n_{a}}{2}\right\rfloor, i \neq j} y_{i}\right)^{\top}\left(\sum_{1 \leq i \leq\left\lfloor\frac{n_{a}}{2}\right\rfloor, i \neq j} y_{i} y_{i}^{\top}\right)^{-1} y_{j} \right] = 0,
\end{align*} and \begin{align*}
    & \mathbb{E}\left[\left[ \frac{1}{\sqrt[4]{n}}\left(\sum_{1 \leq i \leq\left\lfloor\frac{n_{a}}{2}\right\rfloor, i \neq j} y_{i}\right)^{\top}\left(\sum_{1 \leq i \leq\left\lfloor\frac{n_{a}}{2}\right\rfloor, i \neq j} y_{i} y_{i}^{\top}\right)^{-1} y_{j} \right]^2 \right]\\
    =& \frac{1}{\sqrt{n}}\mathbb{E}\left[ \left(\sum_{1 \leq i \leq\left\lfloor\frac{n_{a}}{2}\right\rfloor, i \neq j} y_{i}\right)^{\top}\left(\sum_{1 \leq i \leq\left\lfloor\frac{n_{a}}{2}\right\rfloor, i \neq j} y_{i} y_{i}^{\top}\right)^{-1} y_j y_j^\top \left(\sum_{1 \leq i \leq\left\lfloor\frac{n_{a}}{2}\right\rfloor, i \neq j} y_{i} y_{i}^{\top}\right)^{-1}\right.\\
     \cdot& \left. \left(\sum_{1 \leq i \leq\left\lfloor\frac{n_{a}}{2}\right\rfloor, i \neq j} y_{i}\right) \right]\\
    =& \frac{1}{n^{1.5}} \mathbb{E}\left[ \left(\sum_{1 \leq i \leq\left\lfloor\frac{n_{a}}{2}\right\rfloor, i \neq j} y_{i}\right)^{\top}\left(\sum_{1 \leq i \leq\left\lfloor\frac{n_{a}}{2}\right\rfloor, i \neq j} y_{i} y_{i}^{\top}\right)^{-2} \left(\sum_{1 \leq i \leq\left\lfloor\frac{n_{a}}{2}\right\rfloor, i \neq j} y_{i}\right) \right] = o(1),
\end{align*} we know $$ \frac{1}{\sqrt[4]{n}}\left(\sum_{1 \leq i \leq\left\lfloor\frac{n_{a}}{2}\right\rfloor, i \neq j} y_{i}\right)^{\top}\left(\sum_{1 \leq i \leq\left\lfloor\frac{n_{a}}{2}\right\rfloor, i \neq j} y_{i} y_{i}^{\top}\right)^{-1} y_{j} = o(1) \quad \text{a.s.}, $$ which implies \begin{align*}
    & \mathbb{E}\left[ \left( \frac{1}{\sqrt[4]{n}}\left(\sum_{1 \leq i \leq\left\lfloor\frac{n_{a}}{2}\right\rfloor, i \neq j} y_{i}\right)^{\top}\left(\sum_{1 \leq i \leq\left\lfloor\frac{n_{a}}{2}\right\rfloor, i \neq j} y_{i} y_{i}^{\top}\right)^{-1} y_{j} \right)^4 \right] = o(1).
\end{align*} 

Similarly we can show \begin{align*}
    & \mathbb{E}\left[ \left(\frac{1}{\sqrt[4]{n}}\left(\sum_{1 \leq i \leq\left\lfloor\frac{n_{a}}{2}\right\rfloor, i \neq j} y_{i}^{\top} e^{y_{i}^{\top} \tilde{\beta}_{2: p}} l_{i}\right)\left(\sum_{1 \leq i \leq\left\lfloor\frac{n_{a}}{2}\right\rfloor, i \neq j} y_{i} y_{i}^{\top}\right)^{-1} y_{j}\right)^4  \right] = o(1) .
\end{align*}

By Cauchy-Schwarz, we have 
\begin{align*} 
&\mathbb{E}\Big[ \Big(   \frac{1}{\sqrt{n}}\Big(\sum_{1 \leq i \leq \lfloor\frac{n_a}{2}\rfloor, i \neq j}  y_{i}\Big)^{\top}\Big[ \Big(\sum_{1 \leq i \leq \lfloor\frac{n_a}{2}\rfloor}  y_{i} y_{i}^{\top}\Big)^{-1} - \Big(\sum_{1 \leq i \leq \lfloor\frac{n_a}{2}\rfloor, i \neq j}  y_{i} y_{i}^{\top}\Big)^{-1} \Big]\\
\cdot & \Big( \sum_{1 \leq i \leq \lfloor\frac{n_a}{2}\rfloor, i \neq j}  { y_{i} e^{ y_i^\top \tilde{\beta}_{2:p}}} l_i \Big) \Big)^2 \Big] = o(1). 
\end{align*} 

This completes the proof.

   \end{proof}

\section{Proofs of Intermediate Results}

\subsection{Proof of Lemma \ref{diff lemma 2}} 
 Define $$U = \sum_{i=1}^{n} A_{i} x_{i} x_{i}^{\top}, \quad V =\frac{1}{n} \left(\frac{1}{n}\sum_{i=1}^n A_{i}\right)^{-1}\left(\sum_{i=1}^{n} A_{i} x_{i}\right)\left(\sum_{i=1}^{n} A_{i} x_{i}^{\top}    \right).$$ Then \begin{align*}
    & a^{\top}(O-Q) b = a^{\top}(U-V)^{-1}b - a^{\top} U^{-1}b \\
        =& a^{\top} \left[U^{-1} + \frac{1}{1-\text{Tr}(VU^{-1})}U^{-1}VU^{-1} \right] b -  a^{\top} U^{-1}b = \frac{1}{1-\text{Tr}(VU^{-1})}a^{\top}U^{-1}VU^{-1}b,
    \end{align*} where the second equality is due to the lemma in \cite{10.2307/2690437}, which also implies that the denominator is non-zero.
    
    We know \begin{align*}
        & \text{Tr}(VU^{-1}) =\left(\frac{1}{n} \sum_{i=1}^{n} A_{i}\right)^{-1} \cdot \frac{1}{n} \left(\sum_{i=1}^{n} A_{i} x_{i}\right)^{\top} \left (\sum_{i=1}^{n} A_{i} x_{i} x_{i}^{\top} \right)^{-1} \left(\sum_{i=1}^{n} A_{i} x_{i}\right),
    \end{align*} and \begin{align*}
    & a^{\top} U^{-1} V U^{-1} b\\
    =& \left(\frac{1}{n} \sum_{i=1}^{n} A_{i}\right)^{-1}  \left[a^{\top} \left(\sum_{i=1}^{n} A_{i} x_{i} x_{i}^{\top}\right)^{-1}  \frac{1}{\sqrt{n}} \left(\sum_{i=1}^{n} A_{i} x_{i}\right) \right] \\
    \cdot& \left[b^{\top} \left(\sum_{i=1}^{n} A_{i} x_{i} x_{i}^{\top}\right)^{-1} \frac{1}{\sqrt{n}} \left(\sum_{i=1}^{n} A_{i} x_{i}\right) \right].
\end{align*}

Thus \begin{align*}
    & a^{\top}(O-Q) b =  \frac{1}{1-\text{Tr}(VU^{-1})}a^{\top}U^{-1}VU^{-1}b \\
        =& \frac{\left(\frac{1}{n} \sum_{i=1}^{n} A_{i}\right)^{-1} \cdot \left[a^{\top} \left(\sum_{i=1}^{n} A_{i} x_{i} x_{i}^{\top}\right)^{-1} \cdot \frac{1}{\sqrt{n}} \left(\sum_{i=1}^{n} A_{i} x_{i}\right) \right] }{1-\left(\frac{1}{n} \sum_{i=1}^{n} A_{i}\right)^{-1} \cdot \frac{1}{n}\left(\sum_{i=1}^{n} A_{i} x_{i}\right)^{\top}\left(\sum_{i=1}^{n} A_{i} x_{i} x_{i}^{\top}\right)^{-1}\left(\sum_{i=1}^{n} A_{i}  x_{i}\right)} \\
        \cdot & \left[b^{\top} \left(\sum_{i=1}^{n} A_{i} x_{i} x_{i}^{\top}\right)^{-1} \cdot \frac{1}{\sqrt{n}} \left(\sum_{i=1}^{n} A_{i} x_{i}\right) \right].
    \end{align*}

\subsection{Proof of Lemma \ref{lemma 10}} 
Without loss of generality, assume $\beta_i = 0 \;\; \forall i \neq 1$. Then $A_i = \text{Bernoulli}(\sigma(x_{i1 }\beta_1)) \;\; \forall i = 1,2, \ldots, n. $
   Define $$ y_i = [x_{i2} \; x_{i3} \; \ldots \; x_{in}]^{\top} \quad \forall i = 1,2, \ldots, n.$$ Thus 
   \begin{align*}
    & \frac{1}{n}\left(\sum_{i=1}^{n} A_{i} x_{i}\right)^{\top}\left(\sum_{i=1}^{n} A_{i} x_{i} x_{i}^{\top}\right)^{-1}\left(\sum_{i=1}^{n} A_{i} x_{i}\right)  \\
    =&  \frac{1}{n}\begin{bmatrix}
\sum_{i =1}^n A_i x_{i1}\\
\sum_{i =1}^n A_i y_i
\end{bmatrix}^{\top}  
\begin{bmatrix}
\sum_{i =1}^n A_i x^2_{i1} & \sum_{i =1}^n A_i x_{i1} y_i^{\top}\\
\sum_{i =1}^n A_i x_{i1} y_i & \sum_{i =1}^n A_i  y_iy_i^{\top}
\end{bmatrix}^{-1}
\begin{bmatrix}
\sum_{i =1}^n A_i x_{i1}\\
\sum_{i =1}^n A_i y_i
\end{bmatrix} \\
=&  \frac{1}{n} \begin{bmatrix}
\sum_{i =1}^n A_i x_{i1}\\
\sum_{i =1}^n A_i y_i
\end{bmatrix}^{\top}  
\begin{bmatrix}
B & C^{\top}\\
C & D
\end{bmatrix}
\begin{bmatrix}
\sum_{i =1}^n A_i x_{i1}\\
\sum_{i =1}^n A_i y_i
\end{bmatrix} \\
=& \frac{1}{n} B(\sum_{i =1}^n A_i x_{i1})^2 + \frac{2}{n}(\sum_{i =1}^n A_i x_{i1}) C^{\top} \sum_{i =1}^n A_i y_i +\frac{1}{n} (\sum_{i =1}^n A_i y_i)^{\top} D \sum_{i =1}^n A_i y_i,
\end{align*} 
where $$B = \left ( \sum_{i =1}^n A_i x^2_{i1} - \sum_{i =1}^n A_i x_{i1} y_i^{\top} \left(\sum_{i =1}^n A_i  y_iy_i^{\top} \right)^{-1} \sum_{i =1}^n A_i x_{i1} y_i \right )^{-1} ,$$ $$ C = -B  (\sum_{i =1}^n A_i  y_iy_i^{\top})^{-1}  \sum_{i =1}^n A_i x_{i1}y_i , $$ $$D =\left ( \sum_{i =1}^n A_i y_iy_i^{\top} - (\sum_{i =1}^n A_i x^2_{i1})^{-1}(\sum_{i =1}^n A_i x_{i1}y_i )(\sum_{i =1}^n A_i  x_{i1}y_i)^{\top} \right )^{-1} .$$

Hence \begin{align}
\label{eq 2.10}
    & \frac{1}{n}\left(\sum_{i=1}^{n} A_{i} x_{i}\right)^{\top}\left(\sum_{i=1}^{n} A_{i} x_{i} x_{i}^{\top}\right)^{-1}\left(\sum_{i=1}^{n} A_{i} x_{i}\right) = T_1 + T_2 + T_3,
\end{align} where \begin{align*}
    T_1 =&\frac{1}{n}  \left ( \sum_{i =1}^n A_i x^2_{i1} - \sum_{i =1}^n A_i x_{i1} y_i^{\top} (\sum_{i =1}^n A_i  y_iy_i^{\top})^{-1} \sum_{i =1}^n A_i x_{i1} y_i \right )^{-1} \left(\sum_{i =1}^n A_i x_{i1} \right)^2 \\
    =& \frac{\left(\frac{1}{n}\sum_{i} A_i \sqrt{n}x_{i1}\right)^2}{ \frac{1}{n}\sum_{i}A_i \left(\sqrt{n} x_{i1}\right)^2 - (\sum_{i =1}^n A_i x_{i1} y_i^{\top}) \left(\sum_{i =1}^n A_i  y_iy_i^{\top}\right)^{-1} (\sum_{i =1}^n A_i x_{i1} y_i)},
\end{align*} \begin{align*}
    T_2 =& \frac{2}{n} \left(\sum_{i =1}^n A_i x_{i 1}\right) C^{\top} \left(\sum_{i =1}^n A_i y_{i} \right) \\
    =&  -\frac{2B}{n} \left(\sum_{i =1}^n A_i x_{i 1}\right)  \left (\sum_{i =1}^n A_i x_{i 1} y_{i}^{\top} \right ) \left(\sum_{i =1}^n A_i y_{i} y_{i}^{\top}\right)^{-1}  \left(\sum_{i =1}^n A_i y_{i}\right)\\
    =&- 2B \left(\frac{1}{n}\sum_{i =1}^n A_i \sqrt{n}x_{i 1}\right)  \left(\sum_{i =1}^n A_i x_{i1}y_i^{\top} \right) \left(\sum_{i =1}^n A_i  y_iy_i^{\top}\right)^{-1}  \left( \sum_{i =1}^n \frac{1}{\sqrt{n}}  A_i y_{i} \right),
\end{align*} \begin{align*}
    T_3 =& \frac{1}{n} \left(\sum_{i =1}^n A_i y_{i} \right)^{\top}  \left ( \sum_{i =1}^n A_i y_iy_i^{\top} - (\sum_{i =1}^n A_i x^2_{i1})^{-1}(\sum_{i =1}^n A_i x_{i1}y_i )(\sum_{i =1}^n A_i  x_{i1}y_i)^{\top} \right )^{-1}  \left(  \sum_{i =1}^n A_i y_{i} \right)\\
    =& {\frac{1}{n} \left(\sum_{i =1}^n A_i y_{i} \right )^{\top}  \left ( \sum_{i =1}^n A_i y_iy_i^{\top}  \right )^{-1} \left(   \sum_{i =1}^n A_i y_{i}  \right)} + \frac{8\kappa^2 e_{\gamma, 0}^2}{1-2\kappa} \quad \text{(Due to lemma \ref{diff_lemma_1})}.
\end{align*}

By applying the weak law for triangular arrays similarly as in the proof of Lemma \ref{LLN}, we obtain $$ \frac{1}{n} \sum_{i} A_{i} \sqrt{n} x_{i 1} \xrightarrow{p} e_{\gamma, 0} , \quad \frac{1}{n} \sum_{i} A_{i}\left(\sqrt{n} x_{i 1}\right)^{2} \xrightarrow{p} \frac{1}{2}.  $$

Moreover, from the proof of Lemma \ref{diff_lemma_1}, we have $$ \left(\sum_{i=1}^{n} A_{i} x_{i 1} y_{i}^{\top}\right)\left(\sum_{i=1}^{n} A_{i} y_{i} y_{i}^{\top}\right)^{-1}\left(\sum_{i=1}^{n} A_{i} x_{i 1} y_{i}\right)  = \kappa +  o_p(1),$$ which implies $$T_1 = \frac{e_{\gamma, 0}^2 }{ \frac{1}{2} -\kappa } + o_p(1) = \frac{2e_{\gamma, 0}^2 }{1 -2\kappa } + o_p(1).$$

Next we find the limit of $T_2$. We know \begin{align*}
    & T_2 = -2 B\left(\frac{1}{n} \sum_{i=1}^{n} A_{i} \sqrt{n} x_{i 1}\right)\left(\sum_{i=1}^{n} A_{i} x_{i 1} y_{i}^{\top}\right)\left(\sum_{i=1}^{n} A_{i} y_{i} y_{i}^{\top}\right)^{-1}\left(\sum_{i=1}^{n} \frac{1}{\sqrt{n}} A_{i} y_{i}\right)\\
    =& -2B \cdot e_{\gamma, 0} \cdot 2\kappa e_{\gamma, 0}  + o_p(1)=  -B \cdot 4\kappa e^2_\gamma \ o_p(1).
\end{align*}

Further, we have \begin{align*}
     & B = \left(\sum_{i=1}^{n} A_{i} x_{i 1}^{2}-\sum_{i=1}^{n} A_{i} x_{i 1} y_{i}^{\top}\left(\sum_{i=1}^{n} A_{i} y_{i} y_{i}^{\top}\right)^{-1} \sum_{i=1}^{n} A_{i} x_{i 1} y_{i}\right)^{-1}\\
     =&  \frac{1}{\frac{1}{2} - \frac{p}{n}}  +o_p(1) = \frac{2}{1-2\kappa}+o_p(1).
\end{align*}

Hence $$T_2 = - \frac{8 \kappa e_{\gamma, 0}^2}{1 - 2 \kappa}  + o_p(1).$$

Finally, since $$\frac{1}{n}\left(\sum_{i=1}^{n} A_{i} y_{i}\right)^{\top}\left(\sum_{i=1}^{n} A_{i} y_{i} y_{i}^{\top}\right)^{-1}\left(\sum_{i=1}^{n} A_{i} y_{i}\right) = \kappa + o_p(1), $$ we have $$T_3 = \kappa  + \frac{8 \kappa^{2} e_{\gamma}^{2}}{1-2 \kappa} + o_p(1).$$
  Plugging the results into equation (\ref{eq 2.10}), we have \begin{align*}
       \frac{1}{n}\left(\sum_{i=1}^{n} A_{i} x_{i}\right)^{\top}\left(\sum_{i=1}^{n} A_{i} x_{i} x_{i}^{\top}\right)^{-1}\left(\sum_{i=1}^{n} A_{i} x_{i}\right) =&  \frac{2 e_{\gamma}^{2}}{1-2 \kappa} - \frac{8 \kappa e_{\gamma}^{2}}{1-2 \kappa} + \kappa+\frac{8 r^{2} e_{\gamma}^{2}}{1-2 \kappa} + o_p(1)  \\
      &=  \kappa + 2 e_{\gamma, 0}^2 (1-2\kappa) + o_p(1),
  \end{align*} which completes the proof of the lemma.

\subsection{Proof of Lemma \ref{lemma 11}} 
 Note that \begin{align*}
    & \frac{1}{n}  \left(\sum_{i \in S_{b}} A_{i} x_{i}\right)^{\top}\left(\sum_{i \in S_{b}} A_{i} x_{i} x_{i}^{\top}\right)^{-1} \sum_{j \in S_{a}}\left(\frac{A_{j} x_{j}}{\sigma\left(x_{j}^{\top} \hat{\beta}_{S_c}\right)}-x_{j}\right)   \\
    =& \frac{1}{n}  \left(\sum_{i \in S_{b}} A_{i} x_{i}\right)^{\top}\left(\sum_{i \in S_{b}} A_{i} x_{i} x_{i}^{\top}\right)^{-1} \sum_{j \in S_{a}}\left(\frac{A_{j} x_{j}}{\sigma\left(x_{j}^{\top} \hat{\beta}_{S_c}\right)}\right)\\
    -& \frac{1}{n}  \left(\sum_{i \in S_{b}} A_{i} x_{i}\right)^{\top}\left(\sum_{i \in S_{b}} A_{i} x_{i} x_{i}^{\top}\right)^{-1} \left( \sum_{j \in S_{a}}x_{j}\right).
\end{align*}

Since $$ \mathbb{E}\left[\frac{1}{n}\left(\sum_{i \in S_{b}} A_{i} x_{i}\right)^{\top}\left(\sum_{i \in S_{b}} A_{i} x_{i} x_{i}^{\top}\right)^{-1}\left( \sum_{j \in S_{a}} x_{j}\right) \right]  = 0, $$ and \begin{align*}
    & \frac{1}{n^2} \mathbb{E}\left[\left(\left(\sum_{i \in S_{b}} A_{i} x_{i}\right)^{\top}\left(\sum_{i \in S_{b}} A_{i} x_{i} x_{i}^{\top}\right)^{-1} \left( \sum_{j \in S_{a}} x_{j} \right)\right)^2 \right] \\
    =& \frac{1}{n^2} \mathbb{E}\left[\left(\sum_{i \in S_{b}} A_{i} x_{i}\right)^{\top}\left(\sum_{i \in S_{b}} A_{i} x_{i} x_{i}^{\top}\right)^{-1} \left( \sum_{j \in S_{a}} x_{j} \right) \left( \sum_{j \in S_{a}} x_{j}^{\top} \right)\left(\sum_{i \in S_{b}} A_{i} x_{i} x_{i}^{\top}\right)^{-1} \left(\sum_{i \in S_{b}} A_{i} x_{i}\right) \right] \\
    =& \frac{1}{n^2} \mathbb{E}\left[\left(\sum_{i \in S_{b}} A_{i} x_{i}\right)^{\top}\left(\sum_{i \in S_{b}} A_{i} x_{i} x_{i}^{\top}\right)^{-1} \mathbb{E} \left[ \left(\sum_{j \in S_{a}} x_{j} \right) \left(\sum_{j \in S_{a}} x_{i}^{\top} \right) \right ]\left(\sum_{i \in S_{b}} A_{i} x_{i} x_{i}^{\top}\right)^{-1}\right.\\
    \cdot & \left. \left(\sum_{i \in S_{b}} A_{i} x_{i}\right) \right] \\
    =& \frac{n_a}{n^2 }\mathbb{E}\left[ \frac{1}{n} \left(\sum_{i \in S_{b}} A_{i} x_{i}\right)^{\top}\left(\sum_{i \in S_{b}} A_{i} x_{i} x_{i}^{\top}\right)^{-2} \left(\sum_{i \in S_{b}} A_{i} x_{i}\right) \right] \rightarrow  0,
\end{align*} we know $$ \frac{1}{n}\left(\sum_{i \in S_{b}} A_{i} x_{i}\right)^{\top}\left(\sum_{i \in S_{b}} A_{i} x_{i} x_{i}^{\top}\right)^{-1} \left(\sum_{j \in S_{a}} x_{j} \right) \xrightarrow{p} 0 . $$

Hence we only need to find the limit of $$ \frac{1}{n}\left(\sum_{i \in S_{b}} A_{i} x_{i}\right)^{\top}\left(\sum_{i \in S_{b}} A_{i} x_{i} x_{i}^{\top}\right)^{-1} \left( \sum_{j \in S_{a}}\frac{A_{j} x_{j}}{\sigma\left(x_{j}^{\top} \hat{\beta}_{S_c}\right)}\right)  . $$

Let $P$ be an orthogonal matrix such that $P \beta = \|\beta\|e_1 $. Define $$\tilde{\beta}_{S_c} = P\hat{\beta}_{S_c},\quad \tilde{\beta}_{2:p} = [\tilde{\beta}_2\;\; \tilde{\beta}_3 \;\; \ldots \;\; \tilde{\beta}_p]^{\top} ,   \quad z_i = Px_i, \quad y_i = [z_{i2}\; z_{i3}\; \ldots\; z_{ip}]^{\top}, \quad \forall i = 1,2, \ldots, n.$$ 

Then \begin{align}
    & \frac{1}{n}\left(\sum_{i \in S_{b}} A_{i} x_{i}\right)^{\top}\left(\sum_{i \in S_{b}} A_{i} x_{i} x_{i}^{\top}\right)^{-1}\left( \sum_{j \in S_{a}}\frac{A_{j} x_{j}}{\sigma\left(x_{j}^{\top} \hat{\beta}_{S_c}\right)}\right)  \nonumber \\
    &= \frac{1}{n}\left(\sum_{i \in S_{b}} A_{i} z_{i}\right)^{\top}\left(\sum_{i \in S_{b}} A_{i} z_{i} z_{i}^{\top}\right)^{-1} \left( \sum_{j \in S_{a}}\frac{A_{j} z_{j}}{\sigma\left(x_{j}^{\top} \hat{\beta}_{S_c}\right)}\right)\nonumber\\
    =& \frac{1}{n} \left[\begin{array}{c}
\sum_{i \in S_b} A_i z_{i 1}  \\
\sum_{i \in S_b} A_iy_{i}
\end{array}\right]^{\top}\left[\begin{array}{cc}
B & C^{\top} \\
C & D
\end{array}\right] \left[\begin{array}{c}
\sum_{j \in S_{a}}\frac{A_{j} z_{j1}}{\sigma\left(x_{j}^{\top} \hat{\beta}_{S_c}\right)} \nonumber \\
\sum_{j \in S_{a}}\frac{A_{j} y_{j}}{\sigma\left(x_{j}^{\top} \hat{\beta}_{S_c}\right)}
\end{array}\right]\nonumber\\
=& \frac{1}{n}B \left( \sum_{i \in S_{b}} A_{i} z_{i 1} \right) \cdot \left( \sum_{j \in S_{a}} \frac{A_{j} z_{j 1}}{\sigma\left(x_{j}^{\top} \hat{\beta}_{S_c}\right)} \right) + \frac{1}{n}  \left( \sum_{i \in S_{b}} A_{i} z_{i 1} \right) C^{\top} \left( \sum_{j \in S_{a}} \frac{A_{j} y_j}{\sigma\left(x_{j}^{\top} \hat{\beta}_{S_c}\right)} \right) \label{eq_t1} \\
+& \frac{1}{n} \left( \sum_{j \in S_{a}} \frac{A_{j} z_{j 1}}{\sigma\left(x_{j}^{\top} \hat{\beta}_{S_c}\right)} \right) C^{\top} \left( \sum_{i \in S_{b}} A_{i} y_{i } \right)   + \frac{1}{n}  \left( \sum_{i \in S_{b}} A_{i} y_{i }^{\top} \right) D \left( \sum_{j \in S_{a}} \frac{A_{j} y_j}{\sigma\left(x_{j}^{\top} \hat{\beta}_{S_c}\right)} \right), \nonumber 
\end{align} where $$B = \left ( \sum_{i \in S_b}A_i z^2_{i1} - \sum_{i \in S_b}A_i z_{i1} y_i^{\top} (\sum_{i \in S_b}A_i y_iy_i^{\top})^{-1} \sum_{i \in S_b}A_i z_{i1} y_i \right )^{-1},$$ $$C = -B  (\sum_{i \in S_b}A_i y_iy_i^{\top})^{-1}  \sum_{i \in S_b}A_i z_{i1}y_i , $$ $$D =\left ( \sum_{i \in S_b}A_i y_iy_i^{\top} - (\sum_{i \in S_b}A_i z^2_{i1})^{-1}(\sum_{i \in S_b}A_i z_{i1}y_i )(\sum_{i \in S_b}A_i z_{i1}y_i)^{\top} \right )^{-1}. $$

First we find the limit of $$\frac{1}{n} B\left(\sum_{i \in S_{b}} A_{i} z_{i 1}\right) \cdot\left(\sum_{j \in S_{a}} \frac{A_{j} z_{j 1}}{\sigma\left(x_{j}^{\top} \hat{\beta}_{S_c}\right)}\right).$$

By triangular array law of large numbers, we have $$ \frac{1}{\sqrt{n}} \sum_{i \in S_{b}} A_{i} z_{i 1}= r_b e_{\gamma, 0}+o_{p}(1) .$$

Moreover, by Lemma \ref{lemma 3} and Lemma \ref{lemma 0}, \begin{align*}
  &   \sum_{i \in S_{b}} A_{i} z_{i 1}^{2} = r_b \mathbb{E} \left[A_1(\sqrt{n}z_{11})^2 \right] + o_p(1) \\
  &= r_b \mathbb{E} \left[\sigma(x_1^{\top} \beta) \left(\frac{\sqrt{n} x_{1}^{\top} \beta}{\|\beta\|} \right)^2 \right] + o_p(1) =  r_b \mathbb{E} \left[\frac{(Z_ \beta / \gamma)^2}{1+e^{-\gamma * Z_ \beta / \gamma}} \right] + o_p(1) = \frac{r_b}{2} + o_p(1) ,
\end{align*} and \begin{align*}
  &  \frac{1}{\sqrt{n}}\sum_{i \in S_{a}} \frac{A_{i} z_{i 1}}{\sigma\left(x_{i}^{\top} \hat{\beta}_{S_c}\right)} = r_a \mathbb{E}\left[ \frac{A_{1} \sqrt{n} z_{1 1}}{\sigma\left(x_{1}^{\top} \hat{\beta}_{S_c}\right)} \right] +o_p(1) = r_a \mathbb{E}\left[ \frac{\sigma \left(x_1^{\top} \beta \right) \sqrt{n} z_{1 1}}{\sigma\left(x_{1}^{\top} \hat{\beta}_{S_c}\right)} \right] +o_p(1)\\
  =& r_a\mathbb{E}\left[ \frac{(1+e^{-x_1^{\top} \hat{\beta}_{S_c}}) \frac{\sqrt{n}x_1^{\top} \beta}{\|\beta\|} }{1+e^{-x_{1}^{\top} {\beta}}} \right] +o_p(1) =  \frac{r_a h_c }{\gamma} +o_p(1).
\end{align*} 

Further, from the proof of Lemma \ref{lemma 9.5}, we have $$\sum_{i \in S_{b}} A_{i} z_{i 1} y_{i}^{\top}\left(\sum_{i \in S_{b}} A_{i} y_{i} y_{i}^{\top}\right)^{-1} \sum_{i \in S_{b}} A_{i} z_{i 1} y_{i} = \kappa + o_p(1) .$$

{This implies $$B = \frac{1}{\frac{r_b}{2}-\kappa} = \frac{2}{r_b - 2 \kappa}.$$}

Hence \begin{align}
    & \frac{1}{n} B\left(\sum_{i \in S_{b}} A_{i} z_{i 1}\right) \cdot\left(\sum_{j \in S_{a}} \frac{A_{j} z_{j 1}}{\sigma\left(x_{j}^{\top} \hat{\beta}_{S_c}\right)}\right) \nonumber \\
    =&  \frac{2}{r_b - 2 \kappa} \cdot \frac{1}{\sqrt{n}} \left(\sum_{i \in S_{b}} A_{i} z_{i 1}\right) \cdot \frac{1}{\sqrt{n}} \left(\sum_{j \in S_{a}} \frac{A_{j} z_{j 1}}{\sigma\left(x_{j}^{\top} \hat{\beta}_{S_c}\right)}\right) + o_p(1) \nonumber \\
    =& \frac{2r_a r_b}{r_b - 2\kappa} \cdot \frac{ e_{\gamma, 0}   h_c }{\gamma}+ o_p(1) = \frac{2r_a r_be_{\gamma, 0}   h_c}{\gamma(r_b - 2\kappa)} + o_p(1). \label{resa}
\end{align}

We then find the limit of $$ \frac{1}{n}\left(\sum_{i \in S_{b}} A_{i} z_{i 1}\right) C^{\top}\left(\sum_{j \in S_{a}} \frac{A_{j} y_{j}}{\sigma\left(x_{j}^{\top} \hat{\beta}_{S_c}\right)}\right) \quad \text{and} \quad \frac{1}{n}\left(\sum_{j \in S_{a}} \frac{A_{j} z_{j 1}}{\sigma\left(x_{j}^{\top} \hat{\beta}_{S_c}\right)}\right) C^{\top}\left(\sum_{i \in S_{b}} A_{i} y_{i}\right).  $$ 

We have \begin{align}
    & \frac{1}{n}\left(\sum_{j \in S_{a}} \frac{A_{j} z_{j 1}}{\sigma\left(x_{j}^{\top} \hat{\beta}_{S_c}\right)}\right) C^{\top}\left(\sum_{i \in S_{b}} A_{i} y_{i}\right) \nonumber \\
    =& -B\left(\frac{1}{\sqrt{n}}\sum_{j \in S_{a}} \frac{A_{j} z_{j 1}}{\sigma\left(x_{j}^{\top} \hat{\beta}_{S_c}\right)}\right) \cdot \frac{1}{\sqrt{n}} \left( \sum_{i \in S_{b}} A_{i} z_{i 1} y_{i}^{\top} \right) \left(\sum_{i \in S_{b}} A_{i} y_{i} y_{i}^{\top}\right)^{-1} \left(\sum_{i \in S_{b}} A_{i} y_{i}\right) \nonumber\\
    =& -  \frac{2}{r_b - 2 \kappa}  \cdot \frac{r_a h_c}{\gamma}  \cdot 2 \kappa e_{\gamma, 0} +  o_p(1) = - \frac{4 \kappa e_{\gamma, 0} r_a h_c}{\gamma(r_b - 2 \kappa)} +  o_p(1),   \label{resb}
\end{align} Moreover, similar to the proof of Lemma \ref{lemma 11 pre}, we can show \begin{equation} \label{sim res}
    \frac{1}{\sqrt{n}}\left(\sum_{i \in S_{b}} A_{i} z_{i 1} y_{i}^{\top}\right)\left(\sum_{i \in S_{b}} A_{i} y_{i} y_{i}^{\top}\right)^{-1}\left(\sum_{j \in S_{a}} \frac{A_{j} y_{j}}{\sigma\left(x_{j}^{\top} \hat{\beta}_{S_c}\right)}\right) = o_p(1).
\end{equation} Thus \begin{align}
    & \frac{1}{n}\left(\sum_{i \in S_{b}} A_{i} z_{i 1}\right) C^{\top}\left(\sum_{j \in S_{a}} \frac{A_{j} y_{j}}{\sigma\left(x_{j}^{\top} \hat{\beta}_{S_c}\right)}\right) \nonumber \\
    =& -B\left(\frac{1}{\sqrt{n}}\sum_{i \in S_{b}} A_{i} z_{i 1}\right) \cdot \frac{1}{\sqrt{n}} \left( \sum_{i \in S_{b}} A_{i} z_{i 1} y_{i}^{\top} \right) \left(\sum_{i \in S_{b}} A_{i} y_{i} y_{i}^{\top}\right)^{-1} \left(\sum_{j \in S_{a}} \frac{A_{j} y_{j}}{\sigma\left(x_{j}^{\top} \hat{\beta}_{S_c}\right)}\right) \nonumber \\
    =& -\frac{2}{r_b - 2 \kappa} \cdot r_{b} e_{\gamma, 0} \cdot 0 + o_p(1) = o_p(1). \label{resc}
\end{align}

Finally we find the limit of $$\frac{1}{n}\left(\sum_{i \in S_{b}} A_{i} y_{i}^{\top}\right) D\left(\sum_{j \in S_{a}} \frac{A_{j} y_{j}}{\sigma\left(x_{j}^{\top} \hat{\beta}_{S_c}\right)}\right).$$ 

We know \begin{align*}
    & \frac{1}{n}\left(\sum_{i \in S_{b}} A_{i} y_{i}^{\top}\right) D\left(\sum_{j \in S_{a}} \frac{A_{j} y_{j}}{\sigma\left(x_{j}^{\top} \hat{\beta}_{S_c}\right)}\right) \\
    =& \frac{1}{n}\left(\sum_{i \in S_{b}} A_{i} y_{i}^{\top}\right) \left(\sum_{i \in S_{b}} A_{i} y_{i} y_{i}^{\top}-\left(\sum_{i \in S_{b}} A_{i} z_{i 1}^{2}\right)^{-1}\left(\sum_{i \in S_{b}} A_{i} z_{i 1} y_{i}\right)\left(\sum_{i \in S_{b}} A_{i} z_{i 1} y_{i}\right)^{\top}\right)^{-1}\\
    \cdot & \left(\sum_{j \in S_{a}} \frac{A_{j} y_{j}}{\sigma\left(x_{j}^{\top} \hat{\beta}_{S_c}\right)}\right).
\end{align*}
By Lemma \ref{diff_lemma_1} and equation (\ref{sim res}) we know \begin{align*}
    & \frac{1}{n}\left(\sum_{i \in S_{b}} A_{i} y_{i}^{\top}\right) D\left(\sum_{j \in S_{a}} \frac{A_{j} y_{j}}{\sigma\left(x_{j}^{\top} \hat{\beta}_{S_c}\right)}\right)  \\
    &= \frac{1}{n}\left(\sum_{i \in S_{b}} A_{i} y_{i}^{\top}\right) \left(\sum_{i \in S_{b}} A_{i} y_{i} y_{i}^{\top}\right)^{-1}\left(\sum_{j \in S_{a}} \frac{A_{j} y_{j}}{\sigma\left(x_{j}^{\top} \hat{\beta}_{S_c}\right)}\right) + o_p(1) = o_p(1).
\end{align*}
Combined with Lemma \ref{lemma 11 pre}, we have \begin{equation} \label{resd}
    \frac{1}{n}\left(\sum_{i \in S_{b}} A_{i} y_{i}^{\top}\right) D\left(\sum_{j \in S_{a}} \frac{A_{j} y_{j}}{\sigma\left(x_{j}^{\top} \hat{\beta}_{S_c}\right)}\right) = o_p(1). 
\end{equation}

Plugging equations (\ref{resa}), (\ref{resb}), (\ref{resc}), (\ref{resd}) into equation (\ref{eq_t1}), we have \begin{align*}
  & \frac{1}{n}\left(\sum_{i \in S_{b}} A_{i} x_{i}\right)^{\top}\left(\sum_{i \in S_{b}} A_{i} x_{i} x_{i}^{\top}\right)^{-1}\left(\sum_{j \in S_{a}} \frac{A_{j} x_{j}}{\sigma\left(x_{j}^{\top} \hat{\beta}_{S_c}\right)}\right)  \\
  &= \frac{2 r_{a} r_{b} e_{\gamma} h_{c}}{\gamma\left(r_{b}-2 \kappa\right)} - \frac{4 \kappa e_{\gamma} r_{a} h_{c}}{\gamma\left(r_{b}-2 \kappa\right)} +o_p(1) =  \frac{2r_a e_{\gamma, 0} h_c}{\gamma}  +o_p(1) .
\end{align*}

\subsection{Proof of Lemma \ref{lemma 12}} 
Let $P$ be an orthonormal matrix such that $P \beta = \|\beta\|e_1 $. Define $$\tilde{\beta} = P\hat{\beta}_{S_b},\quad \tilde{\beta}_{2:p} = [\tilde{\beta}_2\;\; \tilde{\beta}_3 \;\; \ldots \;\; \tilde{\beta}_p]^{\top} ,   \quad z_i = Px_i, \quad y_i = [z_{i2}\; z_{i3}\; \ldots\; z_{ip}]^{\top}, \quad \forall i = 1,2, \ldots, n.$$ 
  Then \begin{align}
    & \frac{1}{n}\left(\sum_{i \in S_{a}} A_{i} x_{i}\right)^{\top}\left(\sum_{i \in S_{a}} A_{i} x_{i} x_{i}^{\top}\right)^{-1}\left( \sum_{i \in S_{a}}\frac{A_{i} x_{i}}{\sigma\left(x_{i}^{\top} \hat{\beta}_{S_b}\right)}\right) \nonumber \\
    =& \frac{1}{n}\left(\sum_{i \in S_{a}} A_{i} z_{i}\right)^{\top}\left(\sum_{i \in S_{a}} A_{i} z_{i} z_{i}^{\top}\right)^{-1} \left( \sum_{i \in S_{a}}\frac{A_{i} z_{i}}{\sigma\left(x_{i}^{\top} \hat{\beta}_{S_b}\right)}\right)\nonumber\\
    =& \frac{1}{n} \left[\begin{array}{c}
\sum_{i \in S_{a}} A_i z_{i 1}  \\
\sum_{i \in S_{a}} A_iy_{i}
\end{array}\right]^{\top}\left[\begin{array}{cc}
B & C^{\top} \\
C & D
\end{array}\right] \left[\begin{array}{c}
\sum_{i \in S_{a}}\frac{A_{i} z_{i1}}{\sigma\left(x_{i}^{\top} \hat{\beta}_{S_b}\right)} \nonumber \\
\sum_{i \in S_{a}}\frac{A_{i} y_{i}}{\sigma\left(x_{i}^{\top} \hat{\beta}_{S_b}\right)}
\end{array}\right]\nonumber\\
=& \frac{1}{n}B \left( \sum_{i \in S_{a}} A_{i} z_{i 1} \right) \cdot \left( \sum_{i \in S_{a}} \frac{A_{i} z_{i 1}}{\sigma\left(x_{i}^{\top} \hat{\beta}_{S_b}\right)} \right) + \frac{1}{n}  \left( \sum_{i \in S_{a}} A_{i} z_{i 1} \right) C^{\top} \left( \sum_{i \in S_{a}} \frac{A_{i} y_i}{\sigma\left(x_{i}^{\top} \hat{\beta}_{S_b}\right)} \right)  \label{eq_ft1}\\
+& \frac{1}{n} \left( \sum_{i \in S_{a}} \frac{A_{i} z_{i 1}}{\sigma\left(x_{i}^{\top} \hat{\beta}_{S_b}\right)} \right) C^{\top} \left( \sum_{i \in S_{a}} A_{i} y_{i } \right)   + \frac{1}{n}  \left( \sum_{i \in S_{a}} A_{i} y_{i }^{\top} \right) D \left( \sum_{i \in S_{a}} \frac{A_{i} y_i}{\sigma\left(x_{i}^{\top} \hat{\beta}_{S_b}\right)} \right),  \nonumber
\end{align} where $$B = \left ( \sum_{i \in S_{a}}A_i z^2_{i1} - \sum_{i \in S_{a}}A_i z_{i1} y_i^{\top} (\sum_{i \in S_{a}}A_i y_iy_i^{\top})^{-1} \sum_{i \in S_{a}}A_i z_{i1} y_i \right )^{-1},$$ $$ C = -B  (\sum_{i \in S_{a}}A_i y_iy_i^{\top})^{-1}  \sum_{i \in S_{a}}A_i z_{i1}y_i , $$ $$D =\left ( \sum_{i \in S_{a}}A_i y_iy_i^{\top} - (\sum_{i \in S_{a}}A_i z^2_{i1})^{-1}(\sum_{i \in S_{a}}A_i z_{i1}y_i )(\sum_{i \in S_{a}}A_i z_{i1}y_i)^{\top} \right )^{-1}. $$

By triangular array weak law of large numbers, we have $$ \sum_{i \in S_a} A_iz_{i 1}^{2} = \frac{r_a}{2} + o_p(1), \quad \frac{1}{\sqrt{n}} \sum_{i \in S_{a}} A_{i} z_{i 1} = r_a e_{\gamma, 0} + o_p(1), \quad \frac{1}{\sqrt{n}} \sum_{i \in S_{a}} \frac{A_{i} z_{i 1}}{\sigma\left(z_{i}^{\top} \tilde{\beta}\right)}= \frac{r_a h_b}{\gamma}  +o_p(1).$$

Further, similar to Lemma \ref{lemma 11} we can show $$B  =\frac{2}{r_a -2 \kappa} +o_p(1),  $$ and thus \begin{equation}
     \frac{1}{n} B\left(\sum_{i \in S_{a}} A_{i} z_{i 1}\right) \cdot\left(\sum_{i \in S_{a}} \frac{A_{i} z_{i 1}}{\sigma\left(x_{i}^{\top} \hat{\beta}_{S_b}\right)}\right)= \frac{2r_{a}^2 h_{b}  e_{\gamma, 0}}{\gamma(r_a - 2\kappa)}    + o_p(1). \label{resffff} \end{equation}

We then find the limit of $$ \frac{1}{n}\left(\sum_{i \in S_{a}} A_{i} z_{i 1}\right) C^{\top}\left(\sum_{i \in S_{a}} \frac{A_{i} y_{i}}{\sigma\left(x_{i}^{\top} \hat{\beta}_{S_b}\right)}\right) \quad \text{and} \quad \frac{1}{n}\left(\sum_{i \in S_{a}} \frac{A_{i} z_{i 1}}{\sigma\left(x_{i}^{\top} \hat{\beta}_{S_b}\right)}\right) C^{\top}\left(\sum_{i \in S_{a}} A_{i} y_{i}\right).  $$

We have \begin{align}
    & \frac{1}{n}\left(\sum_{i \in S_{a}} \frac{A_{i} z_{i 1}}{\sigma\left(x_{i}^{\top} \hat{\beta}_{S_b}\right)}\right) C^{\top}\left(\sum_{i \in S_{a}} A_{i} y_{i}\right) \nonumber \\
    =& -B\left(\frac{1}{\sqrt{n}}\sum_{i \in S_{a}} \frac{A_{i} z_{i 1}}{\sigma\left(x_{i}^{\top} \hat{\beta}_{S_b}\right)}\right) \cdot \frac{1}{\sqrt{n}} \left( \sum_{i \in S_{a}} A_{i} z_{i 1} y_{i}^{\top} \right) \left(\sum_{i \in S_{a}} A_{i} y_{i} y_{i}^{\top}\right)^{-1} \left(\sum_{i \in S_{a}} A_{i} y_{i}\right) \nonumber\\
    =& - \frac{2}{r_a - 2 \kappa} \cdot \frac{r_a h_b}{\gamma} \cdot 2 \kappa e_{\gamma, 0} + o_p(1) = - \frac{4 \kappa r_a h_b e_{\gamma,0}}{\gamma(r_a - 2 \kappa)} + o_p(1) ,   \label{resfff}
\end{align} and by Lemma \ref{lemma 12pre}, \begin{align} 
    & \frac{1}{n}\left(\sum_{i \in S_{a}} A_{i} z_{i 1}\right) C^{\top}\left(\sum_{i \in S_{a}} \frac{A_{i} y_{i}}{\sigma\left(x_{i}^{\top} \hat{\beta}_{S_b}\right)}\right) \nonumber \\
    =& -B\left(\frac{1}{\sqrt{n}}\sum_{i \in S_{a}} A_{i} z_{i 1}\right) \cdot \frac{1}{\sqrt{n}} \left( \sum_{i \in S_{a}} A_{i} z_{i 1} y_{i}^{\top} \right) \left(\sum_{i \in S_{a}} A_{i} y_{i} y_{i}^{\top}\right)^{-1} \left(\sum_{i \in S_{a}} \frac{A_{i} y_{i}}{\sigma\left(x_{i}^{\top} \hat{\beta}_{S_b}\right)}\right) \nonumber \\
    =& -\frac{2}{r_a - 2 \kappa} \cdot r_{a} e_{\gamma, 0} \cdot 2 \kappa\left(e_{\gamma, 0}+e^{\frac{\left(\alpha_{b}^{*} \gamma\right)^{2}+\kappa_{b}\left(\sigma_{b}^{*}\right)^{2}}{2}} e_{\gamma,-\alpha_{b}^{*} \gamma}\right) + o_p(1)\nonumber\\ 
    &= - \frac{4 \kappa r_a e_{\gamma, 0}}{r_a - 2\kappa}  \left(e_{\gamma, 0}+e^{\frac{\left(\alpha_{b}^{*} \gamma\right)^{2}+\kappa_{b}\left(\sigma_{b}^{*}\right)^{2}}{2}} e_{\gamma,-\alpha_{b}^{*} \gamma}\right) + o_p(1) . \label{resff}
\end{align}

We then find the limit of $$ \frac{1}{n}\left(\sum_{i \in S_{a}} A_{i} y_{i}^{\top}\right) D\left(\sum_{i \in S_{a}} \frac{A_{i} y_{i}}{\sigma\left(x_{i}^{\top} \hat{\beta}_{S_b}\right)}\right) . $$

We know \begin{align*}
    & \frac{1}{n}\left(\sum_{i \in S_{a}} A_{i} y_{i}^{\top}\right) D\left(\sum_{i \in S_{a}} \frac{A_{i} y_{i}}{\sigma\left(x_{i}^{\top} \hat{\beta}_{S_b}\right)}\right) \\
    =& \frac{1}{n}\left(\sum_{i \in S_{a}} A_{i} y_{i}^{\top}\right) \left(\sum_{i \in S_{a}} A_{i} y_{i} y_{i}^{\top}-\left(\sum_{i \in S_{a}} A_{i} z_{i 1}^{2}\right)^{-1}\left(\sum_{i \in S_{a}} A_{i} z_{i 1} y_{i}\right)\left(\sum_{i \in S_{a}} A_{i} z_{i 1} y_{i}\right)^{\top}\right)^{-1}\\
    \cdot & \left(\sum_{i \in S_{a}} \frac{A_{i} y_{i}}{\sigma\left(x_{i}^{\top} \hat{\beta}_{S_b}\right)}\right)
\end{align*}

By Lemma \ref{diff_lemma_1} we know \begin{align*}
    & \frac{1}{n}\left(\sum_{i \in S_{a}} A_{i} y_{i}^{\top}\right) D\left(\sum_{i \in S_{a}} \frac{A_{i} y_{i}}{\sigma\left(x_{i}^{\top} \hat{\beta}_{S_b}\right)}\right)  = \frac{1}{n}\left(\sum_{i \in S_{a}} A_{i} y_{i}^{\top}\right) \left(\sum_{i \in S_{a}} A_{i} y_{i} y_{i}^{\top}\right)^{-1}\left(\sum_{i \in S_{a}} \frac{A_{i} y_{i}}{\sigma\left(x_{i}^{\top} \hat{\beta}_{S_b}\right)}\right) \\
    +& \frac{1}{n} \cdot \frac{\left(\sum_{i \in S_a} A_{i} z_{i 1}^{2}\right)^{-1}\left[\left(\sum_{i \in S_{a}} A_{i} y_{i}^{\top}\right)\left(\sum_{i\in S_a} A_{i} y_{i} y_{i}^{\top}\right)^{-1}\left(\sum_{i\in S_a} A_{i} z_{i 1} y_{i}\right)\right]}{1-\left(\sum_{i\in S_a} A_{i} z_{i 1}^{2}\right)^{-1}\left(\sum_{i\in S_a} A_{i} z_{i 1} y_{i}\right)^{\top}\left(\sum_{i\in S_a} A_{i} y_{i} y_{i}^{\top}\right)^{-1}\left(\sum_{i\in S_a} A_{i} z_{i 1} y_{i}\right)} \\
    \cdot& \left[\left(\sum_{i \in S_{a}} \frac{A_{i} y_{i}^{\top}}{\sigma\left(x_{i}^{\top} \hat{\beta}_{S_b}\right)}\right)\left(\sum_{i\in S_a} A_{i} y_{i} y_{i}^{\top}\right)^{-1}\left(\sum_{i\in S_a} A_{i} z_{i 1} y_{i}\right)\right].
\end{align*}

By Lemma \ref{lemma 9.5} and equation (\ref{eq 9.6}), we can further simplify the above relation: \begin{align*}
    & \frac{1}{n}\left(\sum_{i \in S_{a}} A_{i} y_{i}^{\top}\right) D\left(\sum_{i \in S_{a}} \frac{A_{i} y_{i}}{\sigma\left(x_{i}^{\top} \hat{\beta}_{S_b}\right)}\right)  = \frac{1}{n}\left(\sum_{i \in S_{a}} A_{i} y_{i}^{\top}\right) \left(\sum_{i \in S_{a}} A_{i} y_{i} y_{i}^{\top}\right)^{-1}\left(\sum_{i \in S_{a}} \frac{A_{i} y_{i}}{\sigma\left(x_{i}^{\top} \hat{\beta}_{S_b}\right)}\right) \nonumber \\
    +&  \frac{\frac{2}{r_a} \cdot 2 e_{\gamma,0} \kappa  }{1 - \frac{2}{r_a}  \cdot \kappa } \cdot \frac{1}{\sqrt{n}} \left[\left(\sum_{i \in S_{a}} \frac{A_{i} y_{i}^{\top}}{\sigma\left(x_{i}^{\top} \hat{\beta}_{S_b}\right)}\right)\left(\sum_{i\in S_a} A_{i} y_{i} y_{i}^{\top}\right)^{-1}\left(\sum_{i\in S_a} A_{i} z_{i 1} y_{i}\right)\right] + o_p(1).
\end{align*}

By Lemma \ref{lemma 12pre}, we have \begin{align}
    &  \frac{1}{n}\left(\sum_{i \in S_{a}} A_{i} y_{i}^{\top}\right) \left(\sum_{i \in S_{a}} A_{i} y_{i} y_{i}^{\top}\right)^{-1}\left(\sum_{i \in S_{a}} \frac{A_{i} y_{i}}{\sigma\left(x_{i}^{\top} \hat{\beta}_{S_b}\right)}\right) \nonumber \\
    +&  \frac{\frac{2}{r_a} \cdot 2 e_{\gamma,0} \kappa  }{1 - \frac{2}{r_a}  \cdot \kappa } \cdot \frac{1}{\sqrt{n}} \left[\left(\sum_{i \in S_{a}} \frac{A_{i} y_{i}^{\top}}{\sigma\left(x_{i}^{\top} \hat{\beta}_{S_b}\right)}\right)\left(\sum_{i\in S_a} A_{i} y_{i} y_{i}^{\top}\right)^{-1}\left(\sum_{i\in S_a} A_{i} z_{i 1} y_{i}\right)\right] + o_p(1) \nonumber \\
    =& 2 \kappa\left(\frac{1}{2}+e^{\frac{\left(\alpha_{b}^{*} \gamma\right)^{2}+\kappa_{b}\left(\sigma_{b}^{*}\right)^{2}}{2}} q_{\gamma,-\alpha_{b}^{*} \gamma}\right) + \frac{ 4 e_{\gamma,0} \kappa  }{r_a - 2 \kappa } \cdot 2 \kappa\left(e_{\gamma, 0}+e^{\frac{\left(\alpha_{b}^{*} \gamma\right)^{2}+\kappa_{b}\left(\sigma_{b}^{*}\right)^{2}}{2}} e_{\gamma,-\alpha_{b}^{*} \gamma}\right) . \label{resf}
\end{align}

Plug equations (\ref{resffff}), (\ref{resfff}), (\ref{resff}), and (\ref{resf}) into equation (\ref{eq_ft1}), we obtain \begin{align*}
    & \frac{1}{n}\left(\sum_{i \in S_{a}} A_{i} x_{i}\right)^{\top}\left(\sum_{i \in S_{a}} A_{i} x_{i} x_{i}^{\top}\right)^{-1}\left(\sum_{i \in S_{a}} \frac{A_{i} x_{i}}{\sigma\left(x_{i}^{\top} \hat{\beta}_{S_b}\right)}\right) \\
    =& \frac{2 r_{a}^{2} h_{b} e_{\gamma, 0}}{\gamma\left(r_{a}-2 \kappa\right)}  -\frac{4 \kappa r_{a} h_{b} e_{\gamma, 0}}{\gamma\left(r_{a}-2 \kappa\right)} -\frac{4 \kappa r_{a} e_{\gamma, 0}}{r_{a}-2 \kappa}\left(e_{\gamma, 0}+e^{\frac{\left(\alpha_{b}^{*} \gamma\right)^{2}+\kappa_{b}\left(\sigma_{b}^{*}\right)^{2}}{2}} e_{\gamma,-\alpha_{b}^{*} \gamma}\right) \\
    &+ 2 \kappa\left(\frac{1}{2}+e^{\frac{\left(\alpha_{b}^{*} \gamma\right)^{2}+\kappa_{b}\left(\sigma_{b}^{*}\right)^{2}}{2}} q_{\gamma,-\alpha_{b}^{*} \gamma}\right) + \frac{ 4 e_{\gamma,0} \kappa  }{r_a - 2 \kappa } \cdot 2 \kappa\left(e_{\gamma, 0}+e^{\frac{\left(\alpha_{b}^{*} \gamma\right)^{2}+\kappa_{b}\left(\sigma_{b}^{*}\right)^{2}}{2}} e_{\gamma,-\alpha_{b}^{*} \gamma}\right)+ o_p(1)\\
    =& \frac{2r_a h_b e_{\gamma, 0}}{\gamma} - 4 \kappa e_{\gamma, 0}\left(e_{\gamma, 0}+e^{\frac{\left(\alpha_{b}^{*} \gamma\right)^{2}+\kappa_{b}\left(\sigma_{b}^{*}\right)^{2}}{2}} e_{\gamma,-\alpha_{b}^{*} \gamma}\right) + 2 \kappa\left(\frac{1}{2}+e^{\frac{\left(\alpha_{b}^{*} \gamma\right)^{2}+\kappa_{b}\left(\sigma_{b}^{*}\right)^{2}}{2}} q_{\gamma,-\alpha_{b}^{*} \gamma}\right) + o_p(1).
\end{align*}


\subsection{Proof of Lemma \ref{vec_lemma}}
Note that \begin{align*}
  &  \mathbb{E} \left[ \frac{1}{n_c} \left[\sum_{i \in S_{c}}\left(\frac{A_{i} x_{i}}{\sigma\left(x_{i}^{\top} \hat{\beta}_{S_a}\right)}-x_{i}\right)^{\top}\right]\left[\sum_{i \in S_{c}}\left(\frac{A_{i} x_{i}}{\sigma\left(x_{i}^{\top} \hat{\beta}_{S_a}\right)}-x_{i}\right)\right] \right]\\
    =&  \mathbb{E} \left[ \left(\frac{A_{1} x_{1}}{\sigma\left(x_{1}^{\top} \hat{\beta}_{S_a}\right)}-x_{1}\right)^{\top} \left(\frac{A_{1} x_{1}}{\sigma\left(x_{1}^{\top} \hat{\beta}_{S_a}\right)}-x_{1}\right) \right] \\
    &+ (n_c-1) \mathbb{E} \left[ \left(\frac{A_{1} x_{1}}{\sigma\left(x_{1}^{\top} \hat{\beta}_{S_a}\right)}-x_{1}\right)^{\top} \left(\frac{A_{2} x_{2}}{\sigma\left(x_{2}^{\top} \hat{\beta}_{S_a}\right)}-x_{2}\right) \right]\\
    =& \mathbb{E} \left[ \frac{A_{1} x_{1}^\top x_1}{\sigma^2\left(x_{1}^{\top} \hat{\beta}_{S_a}\right)}+x_{1}^\top x_1 - \frac{ 2A_{1} x_{1}^\top x_1}{\sigma \left(x_{1}^{\top} \hat{\beta}_{S_a}\right)} \right]\\
    &+ (n_c-1) \mathbb{E} \left[ \left(\frac{A_{1} x_{1}}{\sigma\left(x_{1}^{\top} \hat{\beta}_{S_a}\right)}-x_{1}\right)^{\top} \right] \mathbb{E} \left[ \frac{A_{2} x_{2}}{\sigma\left(x_{2}^{\top} \hat{\beta}_{S_a}\right)}-x_{2} \right]\\
    =& \kappa \cdot  \mathbb{E} \left[ 1 + \frac{\sigma(x_1^\top \beta) }{\sigma^2\left(x_{1}^{\top} \hat{\beta}_{S_a}\right)} - \frac{ 2\sigma(x_1^\top \beta) }{\sigma \left(x_{1}^{\top} \hat{\beta}_{S_a}\right)} \right] + (n_c-1) \mathbb{E} \left[ \left(\frac{\sigma(x_1^\top \beta) x_{1}}{\sigma\left(x_{1}^{\top} \hat{\beta}_{S_a}\right)}\right)^{\top} \right] \mathbb{E} \left[ \frac{\sigma(x_1^\top \beta) x_{1}}{\sigma\left(x_{1}^{\top} \hat{\beta}_{S_a}\right)} \right]\\
    =& \kappa + \kappa \cdot  \mathbb{E} \left[ \frac{\sigma(x_1^\top \beta) }{\sigma^2\left(x_{1}^{\top} \hat{\beta}_{S_a}\right)} - \frac{ 2\sigma(x_1^\top \beta) }{\sigma \left(x_{1}^{\top} \hat{\beta}_{S_a}\right)} \right] + (n_c-1) \mathbb{E} \left[ \left(\frac{\sigma(x_1^\top \beta) x_{1}}{\sigma\left(x_{1}^{\top} \hat{\beta}_{S_a}\right)}\right)^{\top} \right] \mathbb{E} \left[ \frac{\sigma(x_1^\top \beta) x_{1}}{\sigma\left(x_{1}^{\top} \hat{\beta}_{S_a}\right)} \right].
\end{align*} 


By Stein's Lemma, we know \begin{align*}
    & \mathbb{E}\left[x_1 \frac{\sigma\left(x_1^{T} \beta\right)}{\sigma\left(x_1^{T} \hat{\beta}_{S_a}\right)} \mid \hat{\beta}_{S_a} \right]  =\frac{1}{n} \mathbb{E}\left[\boldsymbol{\nabla}_{x} \frac{\sigma\left(x_1^{T} \beta\right)}{\sigma\left(x_1^{T} \hat{\beta}_{S_a}\right)} \mid \hat{\beta}_{S_a}\right]  \\
    &= \frac{1}{n} \mathbb{E}\left[\frac{\sigma^{\prime}\left(x_1^{T} \beta\right)}{\sigma\left(x_1^{T} \hat{\beta}_{S_a}\right)} \mid  \hat{\beta}_{S_a} \right] \beta-\frac{1}{n} \mathbb{E}\left[\frac{\sigma\left(x_1^{T} \beta\right) \sigma^{\prime}\left(x_1^{T} \hat{\beta}_{S_a}\right)}{\sigma^{2}\left(x_1^{T} \hat{\beta}_{S_a}\right)} \mid \hat{\beta}_{S_a}\right] \hat{\beta}_{S_a}.
\end{align*}

Since $$ \mathbb{E}\left[\frac{\sigma^{\prime}\left(x_{1}^{T} \beta\right)}{\sigma\left(x_{1}^{T} \hat{\beta}_{S_a}\right)} \mid \hat{\beta}_{S_a}\right]  =  \mathbb{E}\left[\frac{\sigma^{\prime}\left(x_{1}^{T} \beta\right)}{\sigma\left(x_{1}^{T} \hat{\beta}_{S_a}\right)}\right] + o(1)  = \mathbb{E}\left[\frac{\sigma^{\prime}\left(Z_\beta\right)}{\sigma\left(Z_{ \hat{\beta}_{S_a}}\right)}\right]+ o(1) ,$$ \begin{align*}
    & \mathbb{E}\left[\frac{\sigma\left(x_{1}^{T} \beta\right) \sigma^{\prime}\left(x_{1}^{T} \hat{\beta}_{S_a}\right)}{\sigma^{2}\left(x_{1}^{T} \hat{\beta}_{S_a}\right)} \mid \hat{\beta}_{S_a}\right] =  \mathbb{E}\left[\frac{\sigma\left(x_{1}^{T} \beta\right) \sigma^{\prime}\left(x_{1}^{T} \hat{\beta}_{S_a}\right)}{\sigma^{2}\left(x_{1}^{T} \hat{\beta}_{S_a}\right)} \right] + o(1)\\
    =& \mathbb{E}\left[\frac{\sigma\left(Z_\beta\right) \sigma^{\prime}\left(Z_{ \hat{\beta}_{S_a}}\right)}{\sigma^{2}\left(Z_{\hat{\beta}_{S_a}}\right)} \right]   + o(1) ,
\end{align*}   we know \begin{align*}
    & \left(n_{c}-1\right) \mathbb{E}\left[\left(\frac{\sigma\left(x_{1}^{\top} \beta\right) x_{1}}{\sigma\left(x_{1}^{\top} \hat{\beta}_{S_a}\right)}\right)^{\top}\right] \mathbb{E}\left[\frac{\sigma\left(x_{1}^{\top} \beta\right) x_{1}}{\sigma\left(x_{1}^{\top} \hat{\beta}_{S_a}\right)}\right] \\
    =& r_c \left \{ \mathbb{E}^2\left[\frac{\sigma^{\prime}\left(Z_{\beta}\right)}{\sigma\left(Z_{\hat{\beta}_{S_a}}\right)}\right] \cdot \frac{\|\beta\|^2}{n} + \mathbb{E}^2\left[\frac{\sigma\left(Z_{\beta}\right) \sigma^{\prime}\left(Z_{\hat{\beta}_{S_a}}\right)}{\sigma^{2}\left(Z_{\hat{\beta}_{S_a}}\right)}\right]  \cdot \frac{\|\hat{\beta}_{S_a}\|^2}{n}  \right.\\
    &\left. - 2 \mathbb{E}\left[\frac{\sigma^{\prime}\left(Z_{\beta}\right)}{\sigma\left(Z_{\hat{\beta}_{S_a}}\right)}\right] \mathbb{E}\left[\frac{\sigma\left(Z_{\beta}\right) \sigma^{\prime}\left(Z_{\hat{\beta}_{S_a}}\right)}{\sigma^{2}\left(Z_{\hat{\beta}_{S_a}}\right)}\right] \cdot \frac{\beta^\top \hat{\beta}_{S_a}}{n} \right \} + o(1)\\
    =& r_c \left \{ \mathbb{E}^2\left[\frac{\sigma^{\prime}\left(Z_{\beta}\right)}{\sigma\left(Z_{\hat{\beta}_{S_a}}\right)}\right]  \gamma^2 + \mathbb{E}^2\left[\frac{\sigma\left(Z_{\beta}\right) \sigma^{\prime}\left(Z_{\hat{\beta}_{S_a}}\right)}{\sigma^{2}\left(Z_{\hat{\beta}_{S_a}}\right)}\right]   (\kappa_{a}\left(\sigma_{a}^{*}\right)^{2}+\left(\alpha_{a}^{*}\right)^{2} \gamma^{2}) \right.\\
    &\left. - 2 \mathbb{E}\left[\frac{\sigma^{\prime}\left(Z_{\beta}\right)}{\sigma\left(Z_{\hat{\beta}_{S_a}}\right)}\right] \mathbb{E}\left[\frac{\sigma\left(Z_{\beta}\right) \sigma^{\prime}\left(Z_{\hat{\beta}_{S_a}}\right)}{\sigma^{2}\left(Z_{\hat{\beta}_{S_a}}\right)}\right] \alpha_{a}^{*} \gamma^{2} \right\}
    + o(1).
\end{align*}
Hence 
\begin{align*}
    & \mathbb{E} \left[ \frac{1}{n_{c}}\left[\sum_{i \in S_{c}}\left(\frac{A_{i} x_{i}}{\sigma\left(x_{i}^{\top} \hat{\beta}_{S_a}\right)}-x_{i}\right)^{\top}\right]\left[\sum_{i \in S_{c}}\left(\frac{A_{i} x_{i}}{\sigma\left(x_{i}^{\top} \hat{\beta}_{S_a}\right)}-x_{i}\right)\right]  \right]\\
    =& \kappa + \kappa   \mathbb{E} \left[ \frac{\sigma(Z_\beta) }{\sigma^2\left(Z_{\hat{\beta}_{S_a}}\right)} - \frac{ 2\sigma(Z_\beta) }{\sigma \left(Z_{\hat{\beta}_{S_a}}\right)} \right] + r_c \gamma^2 \mathbb{E}^{2}\left[\frac{\sigma^{\prime}\left(Z_{\beta}\right)}{\sigma\left(Z_{\hat{\beta}_{S_a}}\right)}\right] \\
    &+ r_c \left(\kappa_{a}\left(\sigma_{a}^{*}\right)^{2}+\left(\alpha_{a}^{*}\right)^{2} \gamma^{2}\right) \mathbb{E}^{2}\left[\frac{\sigma\left(Z_{\beta}\right) \sigma^{\prime}\left(Z_{\hat{\beta}_{S_a}}\right)}{\sigma^{2}\left(Z_{\hat{\beta}_{S_a}}\right)}\right] \\
    &- 2 r_c \alpha_{a}^{*} \gamma^{2} \mathbb{E}\left[\frac{\sigma^{\prime}\left(Z_{\beta}\right)}{\sigma\left(Z_{\hat{\beta}_{S_a}}\right)}\right] \mathbb{E}\left[\frac{\sigma\left(Z_{\beta}\right) \sigma^{\prime}\left(Z_{\hat{\beta}_{S_a}}\right)}{\sigma^{2}\left(Z_{\hat{\beta}_{S_a}}\right)}\right] \\
   =& \kappa + \kappa   \mathbb{E} \left[ \frac{\sigma(Z_\beta) }{\sigma^2\left(Z_{\hat{\beta}_{S_a}}\right)} - \frac{ 2\sigma(Z_\beta) }{\sigma \left(Z_{\hat{\beta}_{S_a}}\right)} \right]  + r_c \gamma^2 \mathbb{E}^2\left[\frac{\sigma^{\prime}\left(Z_{\beta}\right)}{\sigma\left(Z_{\hat{\beta}_{S_a}}\right)} - \alpha_a^* \frac{\sigma\left(Z_{\beta}\right) \sigma^{\prime}\left(Z_{\hat{\beta}_{S_a}}\right)}{\sigma^{2}\left(Z_{\hat{\beta}_{S_a}}\right)} \right] \\
  & + r_c \kappa_{a}\left(\sigma_{a}^{*}\right)^{2} \mathbb{E}^{2}\left[\frac{\sigma\left(Z_{\beta}\right) \sigma^{\prime}\left(Z_{\hat{\beta}_{S_a}}\right)}{\sigma^{2}\left(Z_{\hat{\beta}_{S_a}}\right)}\right].
\end{align*}

Since $\sigma^{'}(\cdot) = \sigma(1 - \sigma)(\cdot) $, we have \begin{align*}
    & \mathbb{E}^{2}\left[\frac{\sigma\left(Z_{\beta}\right) \sigma^{\prime}\left(Z_{\hat{\beta}_{S_a}}\right)}{\sigma^{2}\left(Z_{\hat{\beta}_{S_a}}\right)}\right] =  \mathbb{E}^{2}\left[\frac{\sigma\left(Z_{\beta}\right) \left [ 1 - \sigma\left(Z_{\hat{\beta}_{S_a}}\right) \right]}{\sigma\left(Z_{\hat{\beta}_{S_a}}\right)}\right] = \mathbb{E}^{2}\left[\frac{\sigma\left(Z_{\beta}\right) }{\sigma\left(Z_{\hat{\beta}_{S_a}}\right)} - \frac{1}{2}\right],
\end{align*} and \begin{align*}
    & \mathbb{E}^{2}\left[\frac{\sigma^{\prime}\left(Z_{\beta}\right)}{\sigma\left(Z_{\hat{\beta}_{S_{a}}}\right)}-\alpha_{a}^{*} \frac{\sigma\left(Z_{\beta}\right) \sigma^{\prime}\left(Z_{\hat{\beta}_{S_{a}}}\right)}{\sigma^{2}\left(Z_{\hat{\beta}_{S_{a}}}\right)}\right] \\
    =& \mathbb{E}^{2}\left[\frac{\sigma\left(Z_{\beta}\right)\left[1-\sigma\left(Z_{\beta}\right) \right]}{\sigma\left(Z_{\hat{\beta}_{S_{a}}}\right)}-\alpha_{a}^{*} \frac{\sigma\left(Z_{\beta}\right) \left [ 1-\sigma\left(Z_{\hat{\beta}_{S_{a}}}\right) \right]}{\sigma\left(Z_{\hat{\beta}_{S_{a}}}\right)}\right]\\
    =&  \mathbb{E}^{2}\left[ \left(1 - \alpha_a^* \right) \frac{\sigma\left(Z_{\beta}\right)}{\sigma\left(Z_{\hat{\beta}_{S_{a}}}\right)} - \frac{\sigma^2\left(Z_{\beta}\right)}{\sigma\left(Z_{\hat{\beta}_{S_{a}}}\right)} + \alpha_a^* \sigma(Z_\beta) \right] \\
    =&  \mathbb{E}^{2}\left[ \left(1 - \alpha_a^* \right) \frac{\sigma\left(Z_{\beta}\right)}{\sigma\left(Z_{\hat{\beta}_{S_{a}}}\right)} - \frac{\sigma^2\left(Z_{\beta}\right)}{\sigma\left(Z_{\hat{\beta}_{S_{a}}}\right)} + \frac{\alpha_a^* }{2} \right].
\end{align*}

Therefore \begin{align*}
     & \mathbb{E} \left[ \frac{1}{n_{c}}\left[\sum_{i \in S_{c}}\left(\frac{A_{i} x_{i}}{\sigma\left(x_{i}^{\top} \hat{\beta}_{S_a}\right)}-x_{i}\right)^{\top}\right]\left[\sum_{i \in S_{c}}\left(\frac{A_{i} x_{i}}{\sigma\left(x_{i}^{\top} \hat{\beta}_{S_a}\right)}-x_{i}\right)\right]  \right]\\
     =& \kappa + \kappa   \mathbb{E} \left[ \frac{\sigma(Z_\beta) }{\sigma^2\left(Z_{\hat{\beta}_{S_a}}\right)} - \frac{ 2\sigma(Z_\beta) }{\sigma \left(Z_{\hat{\beta}_{S_a}}\right)} \right]  + r_c \gamma^2 \mathbb{E}^2\left[\frac{\sigma^{\prime}\left(Z_{\beta}\right)}{\sigma\left(Z_{\hat{\beta}_{S_a}}\right)} - \alpha_a^* \frac{\sigma\left(Z_{\beta}\right) \sigma^{\prime}\left(Z_{\hat{\beta}_{S_a}}\right)}{\sigma^{2}\left(Z_{\hat{\beta}_{S_a}}\right)} \right] \\
     &+ r_c \kappa_{a}\left(\sigma_{a}^{*}\right)^{2} \mathbb{E}^{2}\left[\frac{\sigma\left(Z_{\beta}\right) \sigma^{\prime}\left(Z_{\hat{\beta}_{S_a}}\right)}{\sigma^{2}\left(Z_{\hat{\beta}_{S_a}}\right)}\right] \\
     =& \kappa + \kappa   \mathbb{E} \left[ \frac{\sigma(Z_\beta) }{\sigma^2\left(Z_{\hat{\beta}_{S_a}}\right)} - \frac{ 2\sigma(Z_\beta) }{\sigma \left(Z_{\hat{\beta}_{S_a}}\right)} \right] + r_c \gamma^2 \mathbb{E}^{2}\left[ \left(1 - \alpha_a^* \right) \frac{\sigma\left(Z_{\beta}\right)}{\sigma\left(Z_{\hat{\beta}_{S_{a}}}\right)} - \frac{\sigma^2\left(Z_{\beta}\right)}{\sigma\left(Z_{\hat{\beta}_{S_{a}}}\right)} + \frac{\alpha_a^* }{2} \right] \\
     &+ r_{c} \kappa_{a}\left(\sigma_{a}^{*}\right)^{2}  \mathbb{E}^{2}\left[\frac{\sigma\left(Z_{\beta}\right) }{\sigma\left(Z_{\hat{\beta}_{S_a}}\right)} - \frac{1}{2}\right].
\end{align*}

Next we show $$\text{Var}\left( \frac{1}{n_{c}}\left[\sum_{i \in S_{c}}\left(\frac{A_{i} x_{i}}{\sigma\left(x_{i}^{\top} \hat{\beta}_{S_a}\right)}-x_{i}\right)^{\top}\right]\left[\sum_{i \in S_{c}}\left(\frac{A_{i} x_{i}}{\sigma\left(x_{i}^{\top} \hat{\beta}_{S_a}\right)}-x_{i}\right)\right] \right) \rightarrow 0 . $$

Define $$T = \frac{1}{n_{c}}\left[\sum_{i \in S_{c}}\left(\frac{A_{i} x_{i}}{\sigma\left(x_{i}^{\top} \hat{\beta}_{S_{a}}\right)}-x_{i}\right)^{\top}\right]\left[\sum_{i \in S_{c}}\left(\frac{A_{i} x_{i}}{\sigma\left(x_{i}^{\top} \hat{\beta}_{S_{a}}\right)}-x_{i}\right)\right], $$

$$T^{-\{1\}} = \frac{1}{n_{c}}\left[\sum_{i =2}^{ n_{c}}\left(\frac{A_{i} x_{i}}{\sigma\left(x_{i}^{\top} \hat{\beta}_{S_{a}}\right)}-x_{i}\right)^{\top}\right]\left[\sum_{i =2}^{ n_{c}}\left(\frac{A_{i} x_{i}}{\sigma\left(x_{i}^{\top} \hat{\beta}_{S_{a}}\right)}-x_{i}\right)\right].$$

By Efron-Stein, we only need to show $$ n \mathbb{E}\left[  \left( T - T^{-\{1\}}  \right)^2 \right] = o(1).$$

We have \begin{align*}
    & T - T^{-\{1\}} = \frac{2}{n_c} \left(\frac{A_{1} x_{1}}{\sigma\left(x_{1}^{\top} \hat{\beta}_{S_{a}}\right)}-x_{1}\right)^{\top} \left[\sum_{i =2}^{ n_{c}}\left(\frac{A_{i} x_{i}}{\sigma\left(x_{i}^{\top} \hat{\beta}_{S_{a}}\right)}-x_{i}\right)\right] + \frac{1}{n_c} \left\|\frac{A_{1} x_{1}}{\sigma\left(x_{1}^{\top} \hat{\beta}_{S_{a}}\right)}-x_{1}\right\|^2
\end{align*}

We know $$ n  \left(\frac{1}{n_{c}}\left\|\frac{A_{1} x_{1}}{\sigma\left(x_{1}^{\top} \hat{\beta}_{S_{a}}\right)}-x_{1}\right\|^{2}\right)^2 = o(1) \quad \text{a.s.}$$

Thus we only need to show $$  \mathbb{E}\left[\left( \frac{1}{\sqrt{n_c}}\left(\frac{A_{1} x_{1}}{\sigma\left(x_{1}^{\top} \hat{\beta}_{S_{a}}\right)}-x_{1}\right)^{\top}\left[\sum_{i=2}^{n_{c}}\left(\frac{A_{i} x_{i}}{\sigma\left(x_{i}^{\top} \hat{\beta}_{S_{a}}\right)}-x_{i}\right)\right]\right)^2 \right] = o(1). $$

Note that \begin{align*}
    & \mathbb{E}\left[\left(\frac{1}{\sqrt{n_{c}}}\left(\frac{A_{1} x_{1}}{\sigma\left(x_{1}^{\top} \hat{\beta}_{S_{a}}\right)}-x_{1}\right)^{\top}\left[\sum_{i=2}^{n_{c}}\left(\frac{A_{i} x_{i}}{\sigma\left(x_{i}^{\top} \hat{\beta}_{S_{a}}\right)}-x_{i}\right)\right]\right)^{2}\right]\\
    =& O(n) \mathbb{E}\left[ \left( \frac{1}{\sqrt{n_{c}}}\left(\frac{A_{1} x_{1}}{\sigma\left(x_{1}^{\top} \hat{\beta}_{S_{a}}\right)}-x_{1}\right)^{\top} \left(\frac{A_{2} x_{2}}{\sigma\left(x_{2}^{\top} \hat{\beta}_{S_{a}}\right)}-x_{2}\right) \right)^2 \right] \\
    +& O(n^2) \mathbb{E}\left[ \frac{1}{n_c}\left(\frac{A_{1} x_{1}}{\sigma\left(x_{1}^{\top} \hat{\beta}_{S_{a}}\right)}-x_{1}\right)^{\top}\left(\frac{A_{2} x_{2}}{\sigma\left(x_{2}^{\top} \hat{\beta}_{S_{a}}\right)}-x_{2}\right) \right.\\
    \cdot & \left. \left(\frac{A_{1} x_{1}}{\sigma\left(x_{1}^{\top} \hat{\beta}_{S_{a}}\right)}-x_{1}\right)^{\top}\left(\frac{A_{3} x_{3}}{\sigma\left(x_{3}^{\top} \hat{\beta}_{S_{a}}\right)}-x_{3}\right) \right].
\end{align*}

Further, \begin{align*}
    & O(n) \mathbb{E}\left[\left(\frac{1}{\sqrt{n_{c}}}\left(\frac{A_{1} x_{1}}{\sigma\left(x_{1}^{\top} \hat{\beta}_{S_{a}}\right)}-x_{1}\right)^{\top}\left(\frac{A_{2} x_{2}}{\sigma\left(x_{2}^{\top} \hat{\beta}_{S_{a}}\right)}-x_{2}\right)\right)^{2}\right]\\
    =&  \mathbb{E}\left[\left(\frac{A_{1} }{\sigma\left(x_{1}^{\top} \hat{\beta}_{S_{a}}\right)}-1\right)^2\left(\frac{A_{2} }{\sigma\left(x_{2}^{\top} \hat{\beta}_{S_{a}}\right)}-1\right)^{2} (x_1^\top x_2)^2\right] = o(1),
\end{align*} where the last step is obtained by applying Cauchy-Schwarz.

Thus we only need to show 
\begin{align*} 
&O\left(n^{2}\right) \mathbb{E}\left[\frac{1}{n_{c}}\left(\frac{A_{1} x_{1}}{\sigma\left(x_{1}^{\top} \hat{\beta}_{S_{a}}\right)}-x_{1}\right)^{\top}\left(\frac{A_{2} x_{2}}{\sigma\left(x_{2}^{\top} \hat{\beta}_{S_{a}}\right)}-x_{2}\right)\right.\\
\cdot & \left .\left(\frac{A_{1} x_{1}}{\sigma\left(x_{1}^{\top} \hat{\beta}_{S_{a}}\right)}-x_{1}\right)^{\top}\left(\frac{A_{3} x_{3}}{\sigma\left(x_{3}^{\top} \hat{\beta}_{S_{a}}\right)}-x_{3}\right)\right] = o(1).
\end{align*} 

This is equivalent to show $$ \mathbb{E}\left[\left(\frac{A_{1}}{\sigma\left(x_{1}^{\top} \hat{\beta}_{S_{a}}\right)}-1\right)^{2}\left(\frac{A_{2}}{\sigma\left(x_{2}^{\top} \hat{\beta}_{S_{a}}\right)}-1\right)\left(\frac{A_{3}}{\sigma\left(x_{3}^{\top} \hat{\beta}_{S_{a}}\right)}-1\right) x_{1}^\top x_{2} \cdot x_{1}^\top x_{3}\right] = o(n^{-1}).$$

After expanding the square, we need to show the following: 

$$ \mathbb{E}\left[ \frac{A_{1}}{\sigma^2\left(x_{1}^{\top} \hat{\beta}_{S_{a}}\right)}  \left(\frac{A_{2}}{\sigma\left(x_{2}^{\top} \hat{\beta}_{S_{a}}\right)}-1\right)\left(\frac{A_{3}}{\sigma\left(x_{3}^{\top} \hat{\beta}_{S_{a}}\right)}-1\right) x_{1}^{\top} x_{2} \cdot x_{1}^{\top} x_{3} \right] = o(n^{-1}), $$

$$ \mathbb{E}\left[ \frac{2A_{1}}{\sigma\left(x_{1}^{\top} \hat{\beta}_{S_{a}}\right)}  \left(\frac{A_{2}}{\sigma\left(x_{2}^{\top} \hat{\beta}_{S_{a}}\right)}-1\right)\left(\frac{A_{3}}{\sigma\left(x_{3}^{\top} \hat{\beta}_{S_{a}}\right)}-1\right) x_{1}^{\top} x_{2} \cdot x_{1}^{\top} x_{3} \right] = o(n^{-1}), $$

$$ \mathbb{E}\left[ \left(\frac{A_{2}}{\sigma\left(x_{2}^{\top} \hat{\beta}_{S_{a}}\right)}-1\right)\left(\frac{A_{3}}{\sigma\left(x_{3}^{\top} \hat{\beta}_{S_{a}}\right)}-1\right) x_{1}^{\top} x_{2} \cdot x_{1}^{\top} x_{3} \right] = o(n^{-1}). $$

We first show the first equality: 

\begin{align*}
    & \mathbb{E}\left[\frac{A_{1}}{\sigma^{2}\left(x_{1}^{\top} \hat{\beta}_{S_{a}}\right)}\left(\frac{A_{2}}{\sigma\left(x_{2}^{\top} \hat{\beta}_{S_{a}}\right)}-1\right)\left(\frac{A_{3}}{\sigma\left(x_{3}^{\top} \hat{\beta}_{S_{a}}\right)}-1\right) x_{1}^{\top} x_{2} \cdot x_{1}^{\top} x_{3}\right]\\
    =& \mathbb{E}\left[\mathbb{E}\left[ \frac{\sigma(x_1^\top \beta)}{\sigma^{2}\left(x_{1}^{\top} \hat{\beta}_{S_{a}}\right)}\left(\frac{\sigma(x_2^\top \beta)}{\sigma\left(x_{2}^{\top} \hat{\beta}_{S_{a}}\right)}-1\right)\left(\frac{\sigma(x_3^\top \beta)}{\sigma\left(x_{3}^{\top} \hat{\beta}_{S_{a}}\right)}-1\right) x_{1}^{\top} x_{2} \cdot x_{1}^{\top} x_{3} \mid S_a, x_1 \right]\right]\\
    =& \mathbb{E}\left[ \frac{\sigma\left(x_{1}^{\top} \beta\right)}{\sigma^{2}\left(x_{1}^{\top} \hat{\beta}_{S_{a}}\right)} \mathbb{E}^2\left[  \left(\frac{\sigma\left(x_{2}^{\top} \beta\right)}{\sigma\left(x_{2}^{\top} \hat{\beta}_{S_{a}}\right)}-1\right) x_1^\top x_2 \mid S_a, x_1 \right] \right] = O(n^{-2}),
\end{align*} where the last equality can be obtained by applying Stein's Lemma. 

The other two terms can be proved similarly. This completes the proof of the lemma.

\subsection{Proof of Lemma \ref{vec_lemma_2}} 
First we find the limit of $$ \frac{1}{\sqrt{n_a n_c}} \mathbb{E} \left[\sum_{i \in S_a} \left(\frac{A_{i}x_i}{\sigma(x_{i}^{\top} \hat{\beta}_{S_c})}-x_i \right)^{\top}\right]  \left[\sum_{j \in S_c} \left(\frac{A_{i}x_j}{\sigma(x_{j}^{\top} \hat{\beta}_{S_a})}-x_j \right)\right]. $$

We have \begin{align*}
    &  \frac{1}{\sqrt{n_a n_c}} \mathbb{E} \left[\sum_{i \in S_a} \left(\frac{A_{i}x_i}{\sigma(x_{i}^{\top} \hat{\beta}_{S_c})}-x_i \right)^{\top}\right]  \left[\sum_{j \in S_c} \left(\frac{A_{j}x_j}{\sigma(x_{j}^{\top} \hat{\beta}_{S_a})}-x_j \right)\right]\\
    =&  \sqrt{n_an_c} \mathbb{E}\left[ \left(\frac{A_{a_1} }{\sigma\left(x_{a_1}^{\top} \hat{\beta}_{S_{c}}\right)}-1 \right) \left(\frac{A_{c_1} }{\sigma\left(x_{c_1}^{\top} \hat{\beta}_{S_{a}}\right)}-1 \right)  x_{a_1}^\top x_{c_1} \right]\\
    =& \sqrt{n_an_c} \mathbb{E}\left[ \left(\frac{A_{a_1} }{\sigma\left(x_{a_1}^{\top} \hat{\beta}_{S_{c}}^{\{-c_1\}}\right)}-1 \right) \left(\frac{A_{c_1} }{\sigma\left(x_{c_1}^{\top} \hat{\beta}_{S_{a}}^{\{-a_1\}}\right)}-1 \right)  x_{a_1}^\top x_{c_1} \right] \\
    +& \sqrt{n_an_c} \mathbb{E}\left[ A_{a_1} \left(\frac{1 }{\sigma\left(x_{a_1}^{\top} \hat{\beta}_{S_{c}}\right)}-\frac{1 }{\sigma\left(x_{a_1}^{\top} \hat{\beta}_{S_{c}}^{\{-c_1\}}\right)} \right) \left(\frac{A_{c_1} }{\sigma\left(x_{c_1}^{\top} \hat{\beta}_{S_{a}}^{\{-a_1\}}\right)}-1 \right)  x_{a_1}^\top x_{c_1} \right]\\
    +& \sqrt{n_an_c} \mathbb{E}\left[A_{c_1} \left(\frac{A_{a_1} }{\sigma\left(x_{a_1}^{\top} \hat{\beta}_{S_{c}}^{\{-c_1\}}\right)}-1 \right) \left(\frac{1 }{\sigma\left(x_{c_1}^{\top} \hat{\beta}_{S_{a}}\right)}- \frac{1 }{\sigma\left(x_{c_1}^{\top} \hat{\beta}_{S_{a}}^{\{-a_1\}}\right)} \right)  x_{a_1}^\top x_{c_1} \right]\\
    +& \sqrt{n_an_c} \mathbb{E}\left[A_{a_1}A_{c_1} \left(\frac{1 }{\sigma\left(x_{a_1}^{\top} \hat{\beta}_{S_{c}}\right)}- \frac{1 }{\sigma\left(x_{a_1}^{\top} \hat{\beta}_{S_{c}}^{\{-c_1\}}\right)} \right) \left(\frac{1 }{\sigma\left(x_{c_1}^{\top} \hat{\beta}_{S_{a}}\right)}- \frac{1 }{\sigma\left(x_{c_1}^{\top} \hat{\beta}_{S_{a}}^{\{-a_1\}}\right)} \right)  x_{a_1}^\top x_{c_1} \right].
\end{align*}

We know the last term is $o(1)$. Next we compute the limit of the first term. Similar to Lemma \ref{vec_lemma}, we know \begin{align*}
    &  \mathbb{E}\left[\frac{\sigma\left(x_{a_1}^{T} \beta\right) x_{a_1}}{\sigma\left(x_{a_1}^{\top} \hat{\beta}_{S_{c}}\right)}\right] =\frac{1}{n} \mathbb{E}\left[\frac{\sigma^{\prime}\left(Z_{\beta}\right)}{\sigma\left(Z_{\hat{\beta}_{S_c}}\right)}\right] \beta - \frac{1}{n} \mathbb{E}\left[\frac{\sigma\left(Z_{\beta}\right) \sigma^{\prime}\left(Z_{\hat{\beta}_{S_c}}\right)}{\sigma^{2}\left(Z_{\hat{\beta}_{S_c}}\right)}\right] \hat{\beta}_{S_c} + o(1).
\end{align*} Thus \begin{align*}
    & \sqrt{n_a n_c} \mathbb{E}\left[\left(\frac{\sigma\left(x_{a_1}^{T} \beta\right) x_{a_1}}{\sigma\left(x_{a_1}^{\top} \hat{\beta}_{S_{c}}^{\{-c_1\}}\right)}-x_{a_1}\right)^{T}\left(\frac{\sigma\left(x_{c_1}^{T} \beta\right) x_{c_1}}{\sigma\left(x_{c_1}^{\top} \hat{\beta}_{S_{a}}^{\{-a_1\}}\right)}-x_{c_1}\right)\right]  \\
    =& \sqrt{n_a n_c} \mathbb{E}\left[\frac{\sigma\left(x_{a_1}^{T} \beta\right) x_{a_1}}{\sigma\left(x_{a_1}^{\top} \hat{\beta}_{S_{c}}^{\{-c_1\}}\right)}\right]^{T}\mathbb{E}\left[\frac{\sigma\left(x_{c_1}^{T} \beta\right) x_{c_1}}{\sigma\left(x_{c_1}^{\top} \hat{\beta}_{S_{a}}^{\{-a_1\}}\right)}\right]\\
    =& \sqrt{n_a n_c} \mathbb{E}\left[\frac{\sigma\left(x_{a_1}^{T} \beta\right) x_{a_1}}{\sigma\left(x_{a_1}^{\top} \hat{\beta}_{S_{c}}\right)}\right]^{T}\mathbb{E}\left[\frac{\sigma\left(x_{c_1}^{T} \beta\right) x_{c_1}}{\sigma\left(x_{c_1}^{\top} \hat{\beta}_{S_{a}}\right)}\right] + o(1)\\
    =& \sqrt{r_ar_c}\left\{ \mathbb{E}\left[\frac{\sigma^{\prime}\left(Z_{\beta}\right)}{\sigma\left(Z_{\hat{\beta}_{S_{a}}}\right)}\right] \mathbb{E}\left[\frac{\sigma^{\prime}\left(Z_{\beta}\right)}{\sigma\left(Z_{\hat{\beta}_{S_{c}}}\right)}\right] \frac{\|\beta\|^2}{n} \right.\\
    +& \left. \mathbb{E}\left[\frac{\sigma\left(Z_{\beta}\right) \sigma^{\prime}\left(Z_{\hat{\beta}_{S_{a}}}\right)}{\sigma^{2}\left(Z_{\hat{\beta}_{S_{a}}}\right)}\right] \mathbb{E}\left[\frac{\sigma\left(Z_{\beta}\right) \sigma^{\prime}\left(Z_{\hat{\beta}_{S_{c}}}\right)}{\sigma^{2}\left(Z_{\hat{\beta}_{S_{c}}}\right)}\right] \frac{\hat{\beta}_{S_a}^T \hat{\beta}_{S_c}}{n} \right.\\
    -& \left. \mathbb{E}\left[\frac{\sigma^{\prime}\left(Z_{\beta}\right)}{\sigma\left(Z_{\hat{\beta}_{S_{a}}}\right)}\right] \mathbb{E}\left[\frac{\sigma\left(Z_{\beta}\right) \sigma^{\prime}\left(Z_{\hat{\beta}_{S_{c}}}\right)}{\sigma^{2}\left(Z_{\hat{\beta}_{S_{c}}}\right)}\right] \frac{\beta^T \hat{\beta}_{S_c}}{n} \right.\\
    -& \left. \mathbb{E}\left[\frac{\sigma^{\prime}\left(Z_{\beta}\right)}{\sigma\left(Z_{\hat{\beta}_{S_{c}}}\right)}\right] \mathbb{E}\left[\frac{\sigma\left(Z_{\beta}\right) \sigma^{\prime}\left(Z_{\hat{\beta}_{S_{a}}}\right)}{\sigma^{2}\left(Z_{\hat{\beta}_{S_{a}}}\right)}\right] \frac{\beta^T \hat{\beta}_{S_a}}{n} \right\} + o(1)\\
    =& \sqrt{r_ar_c}\left\{ \mathbb{E}\left[\frac{\sigma\left(Z_{\beta}\right)(1-\sigma\left(Z_{\beta}\right))}{\sigma\left(Z_{\hat{\beta}_{S_{a}}}\right)}\right] \mathbb{E}\left[\frac{\sigma\left(Z_{\beta}\right)(1-\sigma\left(Z_{\beta}\right))}{\sigma\left(Z_{\hat{\beta}_{S_{c}}}\right)}\right] \frac{\|\beta\|^2}{n} \right.\\
    +& \left. \mathbb{E}\left[\frac{\sigma\left(Z_{\beta}\right) \left(1-\sigma\left(Z_{\hat{\beta}_{S_{a}}}\right)\right)}{\sigma\left(Z_{\hat{\beta}_{S_{a}}}\right)}\right] \mathbb{E}\left[\frac{\sigma\left(Z_{\beta}\right)\left(1- \sigma\left(Z_{\hat{\beta}_{S_{c}}}\right)\right)}{\sigma\left(Z_{\hat{\beta}_{S_{c}}}\right)}\right] \frac{\hat{\beta}_{S_a}^T \hat{\beta}_{S_c}}{n} \right.\\
    -& \left. \mathbb{E}\left[\frac{\sigma\left(Z_{\beta}\right)(1-\sigma\left(Z_{\beta}\right))}{\sigma\left(Z_{\hat{\beta}_{S_{a}}}\right)}\right] \mathbb{E}\left[\frac{\sigma\left(Z_{\beta}\right)\left( 1-\sigma\left(Z_{\hat{\beta}_{S_{c}}}\right)\right)}{\sigma\left(Z_{\hat{\beta}_{S_{c}}}\right)}\right] \frac{\beta^T \hat{\beta}_{S_c}}{n} \right.\\
    -& \left. \mathbb{E}\left[\frac{\sigma\left(Z_{\beta}\right)(1-\sigma\left(Z_{\beta}\right))}{\sigma\left(Z_{\hat{\beta}_{S_{c}}}\right)}\right] \mathbb{E}\left[\frac{\sigma\left(Z_{\beta}\right) \left(1-\sigma\left(Z_{\hat{\beta}_{S_{a}}}\right)\right)}{\sigma\left(Z_{\hat{\beta}_{S_{a}}}\right)}\right] \frac{\beta^T \hat{\beta}_{S_a}}{n} \right\} + o(1)
\end{align*}



By an extension of Theorem 4 in \cite{Sur14516} \footnote{$\hat{\beta}_{S_a}$ and $\hat{\beta}_{S_c}$ can be tracked by two independent AMP algorithms on the two splits. Using state evolution, marginally, $\hat{\beta}_{S_a}$ and $\hat{\beta}_{S_c}$ are empirically like $\alpha^* \beta + \sigma^* Z$, where $Z \sim \mathcal{N}(0,1)$. As the two splits are independent, the resulting gaussians are independent in the limit.}, we have \begin{align*}
    &\frac{1}{n} \hat{\beta}_{S_{a}}^{T} \hat{\beta}_{S_{c}} = \kappa \cdot \frac{1}{p} \hat{\beta}_{S_{a}}^{T} \hat{\beta}_{S_{c}} = \kappa \mathbb{E}\left[\left(\sigma_a^{\star} Z_a + \alpha_a^* \beta \right) \left(\sigma_c^{\star} Z_c+ \alpha_c^* \beta\right)\right]+ o(1) = \alpha_a^*\alpha_c^*  \gamma^2+ o(1).
\end{align*}

Thus \begin{align*}
    & \frac{1}{\sqrt{n_{a} n_{c}}}\left[\sum_{i \in S_{a}}\left(\frac{A_{i} x_{i}}{\sigma\left(x_{i}^{\top} \hat{\beta}_{S_{c}}\right)}-x_{i}\right)^{\top}\right]\left[\sum_{j \in S_{c}}\left(\frac{A_{j} x_{j}}{\sigma\left(x_{j}^{\top} \hat{\beta}_{S_{a}}\right)}-x_{j}\right)\right]\\
    =& \sqrt{r_a r_c}\gamma^ 2 \left\{ \mathbb{E}\left[\frac{\sigma\left(Z_{\beta}\right)(1-\sigma\left(Z_{\beta}\right))}{\sigma\left(Z_{\hat{\beta}_{S_{a}}}\right)}\right] \mathbb{E}\left[\frac{\sigma\left(Z_{\beta}\right)(1-\sigma\left(Z_{\beta}\right))}{\sigma\left(Z_{\hat{\beta}_{S_{c}}}\right)}\right]  \right.\\
    +& \left. \mathbb{E}\left[  \frac{\sigma\left(Z_{\beta}\right)}{ \sigma\left(Z_{\hat{\beta}_{S_{a}}}\right) } - \frac{1}{2}\right]  \mathbb{E}\left[  \frac{\sigma\left(Z_{\beta}\right)}{ \sigma\left(Z_{\hat{\beta}_{S_{c}}}\right) } - \frac{1}{2}\right] \alpha_{a}^{*} \alpha_{c}^{*}\right.\\
    -& \left. \mathbb{E}\left[\frac{\sigma\left(Z_{\beta}\right)\left(1-\sigma\left(Z_{\beta}\right)\right)}{\sigma\left(Z_{\hat{\beta}_{S_{a}}}\right)}\right] \mathbb{E}\left[\frac{\sigma\left(Z_{\beta}\right)}{\sigma\left(Z_{\hat{\beta}_{S_{c}}}\right)}-\frac{1}{2}\right] \alpha_{c}^{*} \right.\\
    -& \left. \mathbb{E}\left[\frac{\sigma\left(Z_{\beta}\right)\left(1-\sigma\left(Z_{\beta}\right)\right)}{\sigma\left(Z_{\hat{\beta}_{S_{c}}}\right)}\right] \mathbb{E}\left[\frac{\sigma\left(Z_{\beta}\right)}{\sigma\left(Z_{\hat{\beta}_{S_{a}}}\right)}-\frac{1}{2}\right] \alpha_{a}^{*} \right \} + o(1)\\
    =&  \sqrt{r_ar_c} \gamma^ 2 \left( \mathbb{E}\left[\frac{\sigma\left(Z_{\beta}\right)\left(1-\sigma\left(Z_{\beta}\right)\right)}{\sigma\left(Z_{\hat{\beta}_{S_{a}}}\right)}\right]  - \alpha_a^* \mathbb{E}\left[\frac{\sigma\left(Z_{\beta}\right)}{\sigma\left(Z_{\hat{\beta}_{S_{a}}}\right)}-\frac{1}{2}\right] \right ) \\
    &\left( \mathbb{E}\left[\frac{\sigma\left(Z_{\beta}\right)\left(1-\sigma\left(Z_{\beta}\right)\right)}{\sigma\left(Z_{\hat{\beta}_{S_{c}}}\right)}\right]  - \alpha_c^* \mathbb{E}\left[\frac{\sigma\left(Z_{\beta}\right)}{\sigma\left(Z_{\hat{\beta}_{S_{c}}}\right)}-\frac{1}{2}\right] \right ) + o(1).
\end{align*}

Let's then consider the cross terms: \begin{align*}
    & \sqrt{n_{a} n_{c}} \mathbb{E}\left[A_{a_{1}}\left(\frac{1}{\sigma\left(x_{a_{1}}^{\top} \hat{\beta}_{S_{c}}\right)}-\frac{1}{\sigma\left(x_{a_{1}}^{\top} \hat{\beta}_{S_{c}}^{\{-c_1\}}\right)}\right)\left(\frac{A_{c_{1}}}{\sigma\left(x_{c_{1}}^{\top} \hat{\beta}_{S_{a}}^{\{-a_1\}}\right)}-1\right) x_{a_{1}}^{\top} x_{c_{1}}\right]\\
   = &\sqrt{n_{a} n_{c}} \mathbb{E}\left[\sigma (x_{a_1}^\top \beta) \left(\frac{1}{\sigma\left(x_{a_{1}}^{T} \hat{\beta}_{S_{c}}\right)}-\frac{1}{\sigma\left(x_{a_{1}}^{T} \hat{\beta}_{S_{c}}^{-\{1\}}\right)}\right)\left(\frac{A_{c_{1}}}{\sigma\left(x_{c_{1}}^{\top} \hat{\beta}_{S_{a}}^{\{-a_1\}}\right)}-1\right)x_{a_{1}}^{\top} x_{c_{1}} \right]\\
   =& \sqrt{n_{a} n_{c}} \mathbb{E}\left[\mathbb{E}\left[\sigma (x_{a_1}^\top \beta) \left(\frac{1}{\sigma\left(x_{a_{1}}^{T} \hat{\beta}_{S_{c}}\right)}-\frac{1}{\sigma\left(x_{a_{1}}^{T} \hat{\beta}_{S_{c}}^{-\{1\}}\right)}\right) \right.\right.\\
   \cdot & \left. \left. \left(\frac{A_{c_{1}}}{\sigma\left(x_{c_{1}}^{\top} \hat{\beta}_{S_{a}}^{\{-a_1\}}\right)}-1\right)x_{a_{1}}^{\top} x_{c_{1}} \mid A_{c_1}, x_{c_1}, \hat{\beta}_{S_{a}}^{\{-a_1\}} \right]\right]\\
   =&  \sqrt{n_{a} n_{c}} \mathbb{E}\left[\left(\frac{A_{c_{1}}}{\sigma\left(x_{c_{1}}^{\top} \hat{\beta}_{S_{a}}^{\{-a_1\}}\right)}-1\right) \right.\\
   \cdot & \left. \mathbb{E}\left[\sigma (x_{a_1}^\top \beta) \left(\frac{1}{\sigma\left(x_{a_{1}}^{T} \hat{\beta}_{S_{c}}\right)}-\frac{1}{\sigma\left(x_{a_{1}}^{T} \hat{\beta}_{S_{c}}^{-\{1\}}\right)}\right)x_{a_{1}}^{\top} x_{c_{1}} \mid A_{c_1}, x_{c_1}, \hat{\beta}_{S_{a}}^{\{-a_1\}} \right]\right].
\end{align*}

Define $$C_{i}=-\mathbb{E}\left[\frac{\sigma\left(Z_{\beta}\right) \sigma^{\prime}\left(Z_{\hat{\beta}_{S_{i}}}\right)}{\sigma^{2}\left(Z_{\hat{\beta}_{S_{i}}}\right)}\right] \quad \forall i=1,2,3,  $$ $$q_{i,S_j}:=x_{i}^{\top} \left(\sum_{k \in S_j, k \neq i} x_k \sigma^{\prime}\left(\hat{\beta}_{s_j}^{\top}x_k\right) x_k^{\top} \right)^{-1} x_i  \quad \forall j =1,2,3 ,\quad \forall i \in S_j. $$
      
Apply Stein's Lemma, we obtain \begin{align*}
    & \sqrt{n_{a} n_{c}} \mathbb{E}\left[\left(\frac{A_{c_{1}}}{\sigma\left(x_{c_{1}}^{\top} \hat{\beta}_{S_{a}}^{\left\{-a_{1}\right\}}\right)}-1\right)\right.\\
    \cdot & \left. \mathbb{E}\left[\sigma\left(x_{a_{1}}^{\top} \beta\right)\left(\frac{1}{\sigma\left(x_{a_{1}}^{T} \hat{\beta}_{S_{c}}\right)}-\frac{1}{\sigma\left(x_{a_{1}}^{T} \hat{\beta}_{S_{c}}^{-\{1\}}\right)}\right) x_{a_{1}}^{\top} x_{c_{1}} \mid A_{c_{1}}, x_{c_{1}}, \hat{\beta}_{S_{a}}^{\left\{-a_{1}\right\}}\right]\right]\\
    =& \sqrt{n_{a} n_{c}} \mathbb{E}\left[\left(\frac{A_{c_{1}}}{\sigma\left(x_{c_{1}}^{\top} \hat{\beta}_{S_{a}}^{\left\{-a_{1}\right\}}\right)}-1\right) \mathbb{E}\left[\sigma\left(x_{a_{1}}^{\top} \beta\right)\frac{1}{\sigma\left(x_{a_{1}}^{T} \hat{\beta}_{S_{c}}\right)} x_{a_{1}}^{\top} x_{c_{1}} \mid A_{c_{1}}, x_{c_{1}}, \hat{\beta}_{S_{a}}^{\left\{-a_{1}\right\}}\right]\right]\\
    -& \sqrt{n_{a} n_{c}} \mathbb{E}\left[\left(\frac{A_{c_{1}}}{\sigma\left(x_{c_{1}}^{\top} \hat{\beta}_{S_{a}}^{\left\{-a_{1}\right\}}\right)}-1\right) \mathbb{E}\left[\sigma\left(x_{a_{1}}^{\top} \beta\right)\frac{1}{\sigma\left(x_{a_{1}}^{T} \hat{\beta}_{S_{c}}^{-\{1\}}\right)} x_{a_{1}}^{\top} x_{c_{1}} \mid A_{c_{1}}, x_{c_{1}}, \hat{\beta}_{S_{a}}^{\left\{-a_{1}\right\}}\right]\right]\\
    =& \sqrt{r_{a} r_{c}} \mathbb{E}\left[\left(\frac{A_{c_{1}}}{\sigma\left(x_{c_{1}}^{\top} \hat{\beta}_{S_{a}}^{\left\{-a_{1}\right\}}\right)}-1\right) C_c x_{c_1}^\top \left(  \hat{\beta}_{S_c} -   \hat{\beta}_{S_c}^{\{-c_1\}} \right) \right] + o(1).
\end{align*}
By Lemma 21 in \cite{Sur14516}, we know \begin{align*}
    & \sqrt{r_{a} r_{c}} \mathbb{E}\left[\left(\frac{A_{c_{1}}}{\sigma\left(x_{c_{1}}^{\top} \hat{\beta}_{S_{a}}^{\left\{-a_{1}\right\}}\right)}-1\right) C_c x_{c_1}^\top \left(  \hat{\beta}_{S_c} -   \hat{\beta}_{S_c}^{\{-c_1\}} \right) \right] \\
    =& \sqrt{r_{a} r_{c}} \mathbb{E}\left[\left(\frac{A_{c_{1}}}{\sigma\left(x_{c_{1}}^{\top} \hat{\beta}_{S_{a}}^{\{-a_1\}}\right)}-1\right)C_{c} q_{c_{1}, S_c}\left(A_{c_{1}}-\sigma\left(\operatorname{prox}_{q_{c_{1}, S_c} \cdot \rho}\left(x_{c_{1}}^{\top} \hat{\beta}_{S_{c}}^{-\{c_1\}}+q_{c_{1}, S_c} A_{c_{1}}\right)\right)\right)\right]\\
    +& o\left(1\right).
    \end{align*}
    
    By Theorem 13 in \cite{Sur14516}, we have $$ q_{c_{1}, S_c} = \lambda_c^* + o_p(1).$$ 
    
    Thus \begin{align*}
    & \sqrt{r_{a} r_{c}} \mathbb{E}\left[\left(\frac{A_{c_{1}}}{\sigma\left(x_{c_{1}}^{\top} \hat{\beta}_{S_{a}}^{\{-a_1\}}\right)}-1\right)C_{c} q_{c_{1}, S_c}\left(A_{c_{1}}-\sigma\left(\operatorname{prox}_{q_{c_{1}, S_c} \cdot \rho}\left(x_{c_{1}}^{\top} \hat{\beta}_{S_{c}}^{-\{c_1\}}+q_{c_{1}, S_c} A_{c_{1}}\right)\right)\right)\right]\\
    =& \sqrt{r_{a} r_{c}} \mathbb{E}\left[\left(\frac{A_{c_{1}}}{\sigma\left(x_{c_{1}}^{\top} \hat{\beta}_{S_{a}}^{\{-a_1\}}\right)}-1\right)C_{c} \lambda_c^*\left(A_{c_{1}}-\sigma\left(\operatorname{prox}_{\lambda_c^* \rho}\left(x_{c_{1}}^{\top} \hat{\beta}_{S_{c}}^{-\{c_1\}}+\lambda_c^* A_{c_{1}}\right)\right)\right)\right] + o_p(1)\\
    =& \sqrt{r_{a} r_{c}} \mathbb{E}\left[ \sigma(Z_ \beta) \left(\frac{1}{\sigma\left(Z _{\hat{\beta}_{S_{a}}}\right)}-1\right)C_c \lambda_{c}^{*}\left(1-\sigma\left(\operatorname{prox}_{\lambda_{c}^{*} \rho}\left(Z _{\hat{\beta}_{S_{c}}}+\lambda_{c}^{*} \right)\right)\right)\right] \\
    &+ \sqrt{r_{a} r_{c}} \mathbb{E}\left[ (1 - \sigma(Z_\beta)) C_c \lambda_{c}^{*} \sigma\left(\operatorname{prox}_{\lambda_{c}^{*} \rho}\left(Z _{\hat{\beta}_{S_{c}}}\right)\right) \right] 
    + o_p\left(1\right).
\end{align*} 

Hence we have \begin{align*}
    & \frac{1}{\sqrt{n_{a} n_{c}}} \mathbb{E}\left[\sum_{i \in S_{a}}\left(\frac{A_{i} x_{i}}{\sigma\left(x_{i}^{\top} \hat{\beta}_{S_{c}}\right)}-x_{i}\right)^{\top}\right]\left[\sum_{j \in S_{c}}\left(\frac{A_{j} x_{j}}{\sigma\left(x_{j}^{\top} \hat{\beta}_{S_{a}}\right)}-x_{j}\right)\right] \\
    =& \sqrt{r_ar_c} \gamma^ 2 \left( \mathbb{E}\left[\frac{\sigma\left(Z_{\beta}\right)\left(1-\sigma\left(Z_{\beta}\right)\right)}{\sigma\left(Z_{\hat{\beta}_{S_{a}}}\right)}\right]  - \alpha_a^* \mathbb{E}\left[\frac{\sigma\left(Z_{\beta}\right)}{\sigma\left(Z_{\hat{\beta}_{S_{a}}}\right)}-\frac{1}{2}\right] \right ) \\
    &\left( \mathbb{E}\left[\frac{\sigma\left(Z_{\beta}\right)\left(1-\sigma\left(Z_{\beta}\right)\right)}{\sigma\left(Z_{\hat{\beta}_{S_{c}}}\right)}\right]  - \alpha_c^* \mathbb{E}\left[\frac{\sigma\left(Z_{\beta}\right)}{\sigma\left(Z_{\hat{\beta}_{S_{c}}}\right)}-\frac{1}{2}\right] \right ) \\
    +&\sqrt{r_{a} r_{c}} \mathbb{E}\left[ \sigma(Z_ \beta) \left(\frac{1}{\sigma\left(Z _{\hat{\beta}_{S_{a}}}\right)}-1\right)C_{c} \lambda_{c}^{*}\left(1-\sigma\left(\operatorname{prox}_{\lambda_{c}^{*} \rho}\left(Z _{\hat{\beta}_{S_{c}}}+\lambda_{c}^{*} \right)\right)\right)\right] \\
    &+ \sqrt{r_{a} r_{c}} \mathbb{E}\left[ (1 - \sigma(Z_\beta)) C_c \lambda_{c}^{*} \sigma\left(\operatorname{prox}_{\lambda_{c}^{*} \rho}\left(Z _{\hat{\beta}_{S_{c}}}\right)\right) \right]\\
    +& \sqrt{r_{a} r_{c}} \mathbb{E}\left[ \sigma(Z_ \beta) \left(\frac{1}{\sigma\left(Z _{\hat{\beta}_{S_{c}}}\right)}-1\right)C_{a} \lambda_{a}^{*}\left(1-\sigma\left(\operatorname{prox}_{\lambda_{a}^{*} \rho}\left(Z _{\hat{\beta}_{S_{a}}}+\lambda_{a}^{*} \right)\right)\right)\right] \\
    &+ \sqrt{r_{a} r_{c}} \mathbb{E}\left[ (1 - \sigma(Z_\beta)) C_a \lambda_{a}^{*} \sigma\left(\operatorname{prox}_{\lambda_{a}^{*} \rho}\left(Z _{\hat{\beta}_{S_{a}}}\right)\right) \right]
    + o_p(1).
\end{align*}

We then show \begin{align*}
    \text{Var}\left( \frac{1}{\sqrt{n_{a} n_{c}}}\left[\sum_{i \in S_{a}}\left(\frac{A_{i} x_{i}}{\sigma\left(x_{i}^{\top} \hat{\beta}_{S_{c}}\right)}-x_{i}\right)^{\top}\right]\left[\sum_{j \in S_{c}}\left(\frac{A_{j} x_{j}}{\sigma\left(x_{j}^{\top} \hat{\beta}_{S_{a}}\right)}-x_{j}\right)\right] \right) = o(1).
\end{align*}

This is equivalent to show \begin{align*}
    &\frac{1}{n_an_c} \mathbb{E}\left[\left( \left[\sum_{i \in S_{a}}\left(\frac{A_{i} x_{i}}{\sigma\left(x_{i}^{\top} \hat{\beta}_{S_{c}}\right)}-x_{i}\right)^{\top}\right]\left[\sum_{j \in S_{c}}\left(\frac{A_{j} x_{j}}{\sigma\left(x_{j}^{\top} \hat{\beta}_{S_{a}}\right)}-x_{j}\right)\right]\right)^2 \right]\\
=& \left(\frac{1}{\sqrt{n_{a} n_{c}}} \mathbb{E}\left[\sum_{i \in S_{a}}\left(\frac{A_{i} x_{i}}{\sigma\left(x_{i}^{\top} \hat{\beta}_{S_{c}}\right)}-x_{i}\right)^{\top}\right]\left[\sum_{j \in S_{c}}\left(\frac{A_{j} x_{j}}{\sigma\left(x_{j}^{\top} \hat{\beta}_{S_{a}}\right)}-x_{j}\right)\right]\right)^2 + o(1). \end{align*}

Note that \begin{align}
    & \frac{1}{n_{a} n_{c}} \mathbb{E}\left[\left(\left[\sum_{i \in S_{a}}\left(\frac{A_{i} x_{i}}{\sigma\left(x_{i}^{\top} \hat{\beta}_{S_{c}}\right)}-x_{i}\right)^{\top}\right]\left[\sum_{j \in S_{c}}\left(\frac{A_{j} x_{j}}{\sigma\left(x_{j}^{\top} \hat{\beta}_{S_{a}}\right)}-x_{j}\right)\right]\right)^{2}\right] \nonumber\\
    =& \frac{1}{n_{a} n_{c}} \mathbb{E}\left[\left(\left[\sum_{i \in S_{a}}\left(\frac{A_{i} }{\sigma\left(x_{i}^{\top} \hat{\beta}_{S_{c}}\right)}-1\right)x_i^{\top}\right]\left[\sum_{j \in S_{c}}\left(\frac{A_{j} }{\sigma\left(x_{j}^{\top} \hat{\beta}_{S_{a}}\right)}-1\right)x_{j}\right]\right)^{2}\right] \nonumber\\
    =& \frac{1}{n_{a} n_{c}} \mathbb{E}\left[\sum_{i=1}^{n_{a}} \sum_{j=1}^{n_{c}}\left(\frac{A_{i}}{\sigma\left(x_{i}^{T} \hat{\beta}_{S_{c}}\right)}-1\right)^{2}\left(\frac{A_{j}}{\sigma\left(x_{j}^{T} \hat{\beta}_{S_{a}}\right)}-1\right)^{2}\left( x_{i}^{T} x_{j}\right)^{2}\right] \label{lemma_overall_four_terms} \\
    +& \frac{2}{n_an_c} \mathbb{E}\left[ \sum_{i=1}^{n_{a}} \sum_{1 \leq j_{1}<j_{2} \leq n_{c}}\left(\frac{A_{i}}{\sigma\left(x_{i}^{T} \hat{\beta}_{S_{c}}\right)}-1\right)^{2}\left(\frac{A_{j_1}}{\sigma\left(x_{j_1}^{\top} \hat{\beta}_{S_{a}}\right)}-1\right)\right. \nonumber \\
    \cdot & \left. \left(\frac{A_{j_2}}{\sigma\left(x_{j_2}^{\top} \hat{\beta}_{S_{a}}\right)}-1\right) x_i^\top x_{j_1}x_i^\top x_{j_2} \right] \nonumber \\
    +& \frac{2}{n_an_c} \mathbb{E}\left[ \sum_{1 \leq i_{1}<i_{2} \leq n_{a}} \sum_{j=1}^{n_{c}} \left(\frac{A_{i_1}}{\sigma\left(x_{i_1}^{T} \hat{\beta}_{S_{c}}\right)}-1\right)\left(\frac{A_{i_2}}{\sigma\left(x_{i_2}^{T} \hat{\beta}_{S_{c}}\right)}-1\right)\right. \nonumber \\
    \cdot & \left. \left(\frac{A_j}{\sigma\left(x_{j}^{\top} \hat{\beta}_{S_{a}}\right)}-1\right)^2  x_{i_1}^{\top} x_{j}x_{i_2}^{\top} x_{j}\right]  \nonumber \\
    +& \frac{4}{ n_a n_c } \cdot \mathbb{E}\Big[ \sum_{1 \leq i_{1}<i_{2} \leq n_{a}} \sum_{1 \leq j_{1}<j_{2} \leq n_{c}}  \Big(\frac{A_{i_{1}}}{\sigma\left(x_{i_{1}}^{T} \hat{\beta}_{S_{c}}\right)}-1\Big)\Big(\frac{A_{i_{2}}}{\sigma\left(x_{i_{2}}^{T} \hat{\beta}_{S_{c}}\right)}-1\Big) \nonumber \\
   &\Big(\frac{A_{j_{1}}}{\sigma\left(x_{j_{1}}^{\top} \hat{\beta}_{S_{a}}\right)}-1\Big)\Big(\frac{A_{j_{2}}}{\sigma\left(x_{j_{2}}^{\top} \hat{\beta}_{S_{a}}\right)}-1\Big) x_{i_{1}}^{\top} x_{j_1}x_{i_{2}}^{\top} x_{j_2} \Big].  \nonumber
\end{align}

First we show the first three terms are $o(1)$. We only show that the third term is $o(1)$, and the proof for the first two are similar. 

Since $$ \frac{A_{a_{1}}}{\sigma\left(x_{a_{1}}^{T} \hat{\beta}_{S_{c}}\right)}-\frac{A_{a_{1}}}{\sigma\left(x_{a_{1}}^{T} \hat{\beta}_{S_{c}}^{-\{c_1\}}\right)} = o(1), $$ $$ \left(\frac{A_{a_{2}}}{\sigma\left(x_{a_{2}}^{T} \hat{\beta}_{S_{c}}\right)}-1\right)\left(\frac{A_{c_{1}}}{\sigma\left(x_{c_{1}}^{\top} \hat{\beta}_{S_{a}}\right)}-1\right)^{2} x_{a_{1}}^{\top} x_{c_{a}} x_{a_{2}}^{\top} x_{c_{1}} = O(\frac{1}{n}), $$ we have \begin{align*}
    & n \mathbb{E}\left[\left(\frac{A_{a_{1}}}{\sigma\left(x_{a_{1}}^{T} \hat{\beta}_{S_{c}}\right)}-\frac{A_{a_{1}}}{\sigma\left(x_{a_{1}}^{T} \hat{\beta}_{S_{c}}^{-\{c_1\}}\right)}\right)\left(\frac{A_{a_{2}}}{\sigma\left(x_{a_{2}}^{T} \hat{\beta}_{S_{c}}\right)}-1\right)\right.\\
    \cdot & \left.\left(\frac{A_{c_1}}{\sigma\left(x_{c_1}^{\top} \hat{\beta}_{S_{a}}\right)}-1\right)^{2}x_{a_{1}}^{\top} x_{c_a}x_{a_{2}}^{\top} x_{c_1}\right] = o(1).
\end{align*} 

Therefore \begin{align*}
    & n \mathbb{E}\left[\left(\frac{A_{a_{1}}}{\sigma\left(x_{a_{1}}^{T} \hat{\beta}_{S_{c}}\right)}-1\right)\left(\frac{A_{a_{2}}}{\sigma\left(x_{a_{2}}^{T} \hat{\beta}_{S_{c}}\right)}-1\right)\left(\frac{A_{c_1}}{\sigma\left(x_{c_1}^{\top} \hat{\beta}_{S_{a}}\right)}-1\right)^{2}x_{a_{1}}^{\top} x_{c_a}x_{a_{2}}^{\top} x_{c_1}\right]\\
    =& n \mathbb{E}\left[\left(\frac{A_{a_{1}}}{\sigma\left(x_{a_{1}}^{T} \hat{\beta}_{S_{c}}^{\{c_1\}}\right)}-1\right)\left(\frac{A_{a_{2}}}{\sigma\left(x_{a_{2}}^{T} \hat{\beta}_{S_{c}}\right)}-1\right)\left(\frac{A_{c_1}}{\sigma\left(x_{c_1}^{\top} \hat{\beta}_{S_{a}}\right)}-1\right)^{2}x_{a_{1}}^{\top} x_{c_a}x_{a_{2}}^{\top} x_{c_1}\right] + o(1).
\end{align*}

Similarly, we can obtain that \begin{align*}
    & n \mathbb{E}\left[\left(\frac{A_{a_{1}}}{\sigma\left(x_{a_{1}}^{T} \hat{\beta}_{S_{c}}\right)}-1\right)\left(\frac{A_{a_{2}}}{\sigma\left(x_{a_{2}}^{T} \hat{\beta}_{S_{c}}\right)}-1\right)\left(\frac{A_{c_1}}{\sigma\left(x_{c_1}^{\top} \hat{\beta}_{S_{a}}\right)}-1\right)^{2} x_{a_{1}}^{\top} x_{c_1}x_{a_{2}}^{\top} x_{c_1}\right]\\
    =& n \mathbb{E}\left[\left(\frac{A_{a_{1}}}{\sigma\left(x_{a_{1}}^{T} \hat{\beta}_{S_{c}}^{\{-c_1\}}\right)}-1\right)\left(\frac{A_{a_{2}}}{\sigma\left(x_{a_{2}}^{T} \hat{\beta}_{S_{c}}^{\{-c_1\}}\right)}-1\right)\left(\frac{A_{c_1}}{\sigma\left(x_{c_1}^{\top} \hat{\beta}_{S_{a}}^{\{-a_1,-a_2\}}\right)}-1\right)^{2} \right.\\
    \cdot & \left. x_{a_{1}}^{\top} x_{c_1}x_{a_{2}}^{\top} x_{c_1}\right] + o(1)\\
    =& n \mathbb{E}\left[\left(\frac{A_{c_1}}{\sigma\left(x_{c_1}^{\top} \hat{\beta}_{S_{a}}^{\{-a_1,-a_2\}}\right)}-1\right)^{2} \mathbb{E}\left[\left(\frac{A_{a_{1}}}{\sigma\left(x_{a_{1}}^{T} \hat{\beta}_{S_{c}}^{\{-c_1\}}\right)}-1\right)\left(\frac{A_{a_{2}}}{\sigma\left(x_{a_{2}}^{T} \hat{\beta}_{S_{c}}^{\{-c_1\}}\right)}-1\right)  \right. \right.  \\
    &\left. \left. 
    \cdot x_{a_{1}}^{\top}x_{c_1}x_{a_{2}}^{\top} x_{c_1}\mid S_c, \hat{\beta}_{S_{a}}^{\{-a_1,-a_2\}} \right]\right] + o(1)\\
    =& n \mathbb{E}\left[\left(\frac{A_{c_1}}{\sigma\left(x_{c_1}^{\top} \hat{\beta}_{S_{a}}^{\{-a_1,-a_2\}}\right)}-1\right)^{2} \mathbb{E}^2\left[\left(\frac{A_{a_{1}}}{\sigma\left(x_{a_{1}}^{T} \hat{\beta}_{S_{c}}^{\{-c_1\}}\right)}-1\right)x_{a_{1}}^{\top} x_{c_1}\mid S_c \right]\right] + o(1).
\end{align*}

By Stein's Lemma, we know $$\mathbb{E}^2\left[\left(\frac{A_{a_{1}}}{\sigma\left(x_{a_{1}}^{T} \hat{\beta}_{S_{c}}^{\{-c_1\}}\right)}-1\right)x_{a_{1}}^{\top} x_{c_1}\mid S_c \right] = O(\frac{1}{n^2}).$$

Thus the third term in (\ref{lemma_overall_four_terms}) is $o(1)$. Next we consider the last term. We only need to show the following: \begin{align}
    & \mathbb{E}\left[ \left(\frac{A_{a_{1}}}{\sigma\left(x_{a_{1}}^{T} \hat{\beta}_{S_{c}}\right)}-1\right)\left(\frac{A_{a_{2}}}{\sigma\left(x_{a_{2}}^{T} \hat{\beta}_{S_{c}}\right)}-1\right)\left(\frac{A_{c_{1}}}{\sigma\left(x_{c_{1}}^{\top} \hat{\beta}_{S_{a}}\right)}-1\right)\left(\frac{A_{c_{2}}}{\sigma\left(x_{c_{2}}^{\top} \hat{\beta}_{S_{a}}\right)}-1\right)\right. \nonumber \\
    \cdot & \left. x_{a_{1}}^{\top} x_{c_{1}}x_{a_{2}}^{\top} x_{c_{2}} \right] \nonumber \\
    =&  \mathbb{E}^2\left[ \left(\frac{A_{a_{1}}}{\sigma\left(x_{a_{1}}^{T} \hat{\beta}_{S_{c}}\right)}-1\right)\left(\frac{A_{c_{1}}}{\sigma\left(x_{c_{1}}^{\top} \hat{\beta}_{S_{a}}\right)}-1\right)\left(x_{a_{1}}^{\top} x_{c_{1}}\right) \right]   + o(n^{-2}). \label{lemma_to_show}
\end{align}

Define $$ \xi_{i}=\frac{1}{\sigma\left(x_{i}^{T} \hat{\beta}_{S_{c}}\right)}-\frac{1}{\sigma\left(x_{i}^{T} \hat{\beta}_{S_{c}}^{\left\{-c_1,-c_2\right\}}\right)}=e^{-x_{i}^{T} \hat{\beta}_{S_{c}}}-e^{-x_{i}^{T} \hat{\beta}_{S_{c}}^{\left\{-c_1, -c_2\right\}}}  \quad \forall i \in S_a, $$ $$\zeta_{j}=\frac{1}{\sigma\left(x_{j}^{T} \hat{\beta}_{S_{a}}\right)}-\frac{1}{\sigma\left(x_{j}^{T} \hat{\beta}_{S_{a}}^{\left\{-a_1,-a_2\right\}}\right)}=e^{-x_{j}^{T} \hat{\beta}_{S_{a}}}-e^{-x_{j}^{T} \hat{\beta}_{S_{a}}^{\left\{-a_1,-a_2\right\}}} \quad \forall j \in S_c.$$

$\forall i \in S_a$, by a Taylor expansion of $e^x$ around $ -x_{i}^{T} \hat{\beta}_{S_{c}}^{\{-c_1,-c_2\}} $, we have \begin{align*}
    & \xi_{i}=\frac{1}{\sigma\left(x_{i}^{T} \hat{\beta}_{S_{c}}\right)}-\frac{1}{\sigma\left(x_{i}^{T} \hat{\beta}_{S_{c}}^{\left\{-c_1,-c_2\right\}}\right)}=e^{-x_{i}^{T} \hat{\beta}_{S_{c}}}-e^{-x_{i}^{T} \hat{\beta}_{S_{c}}^{-\left\{c_1, c_2\right\}}} \\
   =& e^{-x_{i}^{T} \hat{\beta}_{S_{c}}^{\{-c_1,-c_2\}}} x_i^\top \left(\hat{\beta}_{S_{c}}^{\{-c_1,-c_2\}} - \hat{\beta}_{S_{c}}  \right) + O_p\left(n^{-1} \right)\\
    =& e^{-x_{i}^{T} \hat{\beta}_{S_{c}}^{\{-c_1,-c_2\}}} x_i^\top \left(\sum_{k \in S_{c}, k \neq c_1, c_2} x_{k} \sigma^{\prime}\left(\left(\hat{\beta}_{S_{c}}^{\{-c_1,-c_2\}}\right)^{\top} x_{k}\right) x_{k}^{\top}\right)^{-1} x_{c_1} \\
    &\left(A_{c_1} - \sigma\left( \operatorname{prox}_{ \lambda_c^* \rho} \left( x_{c_1 }^\top \hat{\beta}_{S_{c}}^{\{-c_1,-c_2\}}  +\lambda_c^* A_{c_1}  +  o_p(1) \right) \right)\right)\\
    +& e^{-x_{i}^{T} \hat{\beta}_{S_{c}}^{\{-c_1,-c_2\}}} x_i^\top \left(\sum_{k \in S_{c}, k \neq c_{1}, c_{2}} x_{k} \sigma^{\prime}\left(\left(\hat{\beta}_{S_{c}}^{\{-c_1, -c_2\}}\right)^{\top} x_{k}\right) x_{k}^{\top}\right)^{-1} x_{c_2} \\
    &\left(A_{c_2} - \sigma\left( \operatorname{prox}_{\lambda_c^* \rho} \left( x_{c_2 }^\top \hat{\beta}_{S_{c}}^{\{-c_1,-c_2\}} +\lambda_c^* A_{c_2} + o_p(1)  \right) \right)\right) +  O_p(n^{-1}).
\end{align*}

Similarly, define $$ \tilde{\xi}_{i}=\frac{1}{\sigma\left(x_{i}^{T} \hat{\beta}_{S_{c}}\right)}-\frac{1}{\sigma\left(x_{i}^{T} \hat{\beta}_{S_{c}}^{-\left\{c_1\right\}}\right)}  \quad \forall i \in S_a, \quad \tilde{\zeta}_{j}=\frac{1}{\sigma\left(x_{j}^{T} \hat{\beta}_{S_{a}}\right)}-\frac{1}{\sigma\left(x_{j}^{T} \hat{\beta}_{S_{a}}^{-\left\{a_1\right\}}\right)} \quad \forall j \in S_c.$$

Also, define $$B_i = \mathbb{E}\left[\frac{\sigma^{\prime}\left(Z_\beta\right)}{\sigma\left(Z_ {\hat{\beta}_{S_{i}}}\right)}\right], \quad C_{i}=-\mathbb{E}\left[\frac{\sigma\left(Z_{\beta}\right) \sigma^{\prime}\left(Z_{\hat{\beta}_{S_{i}}}\right)}{\sigma^{2}\left(Z_{\hat{\beta}_{S_{i}}}\right)}\right] \quad \forall i=1,2,3$$

Then the left hand side of (\ref{lemma_to_show}) is \begin{align*}
    & \mathbb{E}^2\left[\left(\frac{A_{a_{1}}}{\sigma\left(x_{a_{1}}^{T} \hat{\beta}_{S_{c}}\right)}-1\right)\left(\frac{A_{c_{1}}}{\sigma\left(x_{c_{1}}^{\top} \hat{\beta}_{S_{a}}\right)}-1\right)x_{a_{1}}^{\top} x_{c_{1}}\right] \\
    =& \mathbb{E}^2\left[\left(\frac{A_{a_{1}}}{\sigma\left(x_{a_{1}}^{T} \hat{\beta}_{S_{c}}^{\{-c_1\}}\right)}-1 + A_{a_1}\tilde{\xi}_{a_1} \right)\left(\frac{A_{c_{1}}}{\sigma\left(x_{c_{1}}^{\top} \hat{\beta}_{S_{a}}^{\{-a_1\}}\right)}-1 + A_{c_1}\tilde{\zeta}_{c_1} \right) x_{a_{1}}^{\top} x_{c_{1}}\right] \\
    =& \left(\mathbb{E}\left[\left(\frac{A_{a_{1}}}{\sigma\left(x_{a_{1}}^{T} \hat{\beta}_{S_{c}}^{\{-c_1\}}\right)}-1  \right)\left(\frac{A_{c_{1}}}{\sigma\left(x_{c_{1}}^{\top} \hat{\beta}_{S_{a}}^{\{-a_1\}}\right)}-1  \right)x_{a_{1}}^{\top} x_{c_{1}}\right] \right.\\
    +&  \mathbb{E}\left[A_{a_1}\tilde{\xi}_{a_1} \left( \frac{A_{c_{1}}}{\sigma\left(x_{c_{1}}^{\top} \hat{\beta}_{S_{a}}^{\{-a_1\}}\right)}-1 \right) x_{a_{1}}^{\top} x_{c_{1}}\right] \\
    +& \left. \mathbb{E}\left[ A_{c_1}\tilde{\zeta}_{c_1} \left( \frac{A_{a_{1}}}{\sigma\left(x_{a_{1}}^{T} \hat{\beta}_{S_{c}}^{\{-c_1\}}\right)}-1 \right) x_{a_{1}}^{\top} x_{c_{1}} \right] + \mathbb{E}\left[A_{a_{1}} \tilde{\xi}_{a_{1}} A_{c_{1}} \tilde{\zeta}_{c_{1}} x_{a_{1}}^{\top} x_{c_{1}} \right] \right)^{2} \\
    =& \mathbb{E}^2\left[\left(\frac{A_{a_{1}}}{\sigma\left(x_{a_{1}}^{T} \hat{\beta}_{S_{c}}^{\left\{-c_{1}\right\}}\right)}-1\right)\left(\frac{A_{c_{1}}}{\sigma\left(x_{c_{1}}^{\top} \hat{\beta}_{S_{a}}^{\left\{-a_{1}\right\}}\right)}-1\right) x_{a_{1}}^{\top} x_{c_{1}}\right]\\
    +&  \mathbb{E}^2\left[A_{a_{1}} \tilde{\xi}_{a_{1}}\left(\frac{A_{c_{1}}}{\sigma\left(x_{c_{1}}^{\top} \hat{\beta}_{S_{a}}^{\left\{-a_{1}\right\}}\right)}-1\right) x_{a_{1}}^{\top} x_{c_{1}}\right]\\
    +& \mathbb{E}^2\left[A_{c_{1}} \tilde{\zeta}_{ c_{1}}\left(\frac{A_{a_{1}}}{\sigma\left(x_{a_{1}}^{T} \hat{\beta}_{S_{c}}^{\left\{-c_{1}\right\}}\right)}-1\right) x_{a_{1}}^{\top} x_{c_{1}}\right] \\
    +& 2 \mathbb{E}\left[A_{a_{1}} \tilde{\xi}_{a_{1}}\left(\frac{A_{c_{1}}}{\sigma\left(x_{c_{1}}^{\top} \hat{\beta}_{S_{a}}^{\left\{-a_{1}\right\}}\right)}-1\right) x_{a_{1}}^{\top} x_{c_{1}}\right] \mathbb{E}\left[A_{c_{1}} \tilde{\zeta} _{c_{1}}\left(\frac{A_{a_{1}}}{\sigma\left(x_{a_{1}}^{T} \hat{\beta}_{S_{c}}^{\left\{-c_{1}\right\}}\right)}-1\right) x_{a_{1}}^{\top} x_{c_{1}}\right]\\
    +& 2 \mathbb{E}\left[A_{a_{1}} \tilde{\xi}_{a_{1}}\left(\frac{A_{c_{1}}}{\sigma\left(x_{c_{1}}^{\top} \hat{\beta}_{S_{a}}^{\left\{-a_{1}\right\}}\right)}-1\right) x_{a_{1}}^{\top} x_{c_{1}}\right]\\
    \cdot& \mathbb{E}\left[\left(\frac{A_{a_{1}}}{\sigma\left(x_{a_{1}}^{T} \hat{\beta}_{S_{c}}^{\left\{-c_{1}\right\}}\right)}-1\right)\left(\frac{A_{c_{1}}}{\sigma\left(x_{c_{1}}^{\top} \hat{\beta}_{S_{a}}^{\left\{-a_{1}\right\}}\right)}-1\right) x_{a_{1}}^{\top} x_{c_{1}}\right]\\
    +& 2 \mathbb{E}\left[A_{c_{1}} \tilde{\zeta}_{c_{1}}\left(\frac{A_{a_{1}}}{\sigma\left(x_{a_{1}}^{T} \hat{\beta}_{S_{c}}^{\left\{-c_{1}\right\}}\right)}-1\right) x_{a_{1}}^{\top} x_{c_{1}}\right] \\
    \cdot& \mathbb{E}\left[\left(\frac{A_{a_{1}}}{\sigma\left(x_{a_{1}}^{T} \hat{\beta}_{S_{c}}^{\left\{-c_{1}\right\}}\right)}-1\right)\left(\frac{A_{c_{1}}}{\sigma\left(x_{c_{1}}^{\top} \hat{\beta}_{S_{a}}^{\left\{-a_{1}\right\}}\right)}-1\right) x_{a_{1}}^{\top} x_{c_{1}}\right] \\
    +& 2 \mathbb{E}\left[A_{a_{1}} \tilde{\xi}_{a_{1}} A_{c_{1}} \tilde{\zeta}_{c_{1}} x_{a_{1}}^{\top} x_{c_{1}}\right] \mathbb{E}\left[\left(\frac{A_{a_{1}}}{\sigma\left(x_{a_{1}}^{T} \hat{\beta}_{S_{c}}^{\left\{-c_{1}\right\}}\right)}-1\right)\left(\frac{A_{c_{1}}}{\sigma\left(x_{c_{1}}^{\top} \hat{\beta}_{S_{a}}^{\left\{-a_{1}\right\}}\right)}-1\right) x_{a_{1}}^{\top} x_{c_{1}}\right]\\
    +& o(n^{-2}).
\end{align*}

The right hand side of (\ref{lemma_to_show}) is \begin{align*}
    & \mathbb{E}\left[ \left(\frac{A_{a_{1}}}{\sigma\left(x_{a_{1}}^{T} \hat{\beta}_{S_{c}}\right)}-1\right)\left(\frac{A_{a_{2}}}{\sigma\left(x_{a_{2}}^{T} \hat{\beta}_{S_{c}}\right)}-1\right)\left(\frac{A_{c_{1}}}{\sigma\left(x_{c_{1}}^{\top} \hat{\beta}_{S_{a}}\right)}-1\right)\left(\frac{A_{c_{2}}}{\sigma\left(x_{c_{2}}^{\top} \hat{\beta}_{S_{a}}\right)}-1\right)\right.\\
    \cdot& \left. x_{a_{1}}^{\top} x_{c_{1}}x_{a_{2}}^{\top} x_{c_{2}} \right]\\
    =&\mathbb{E}\left[ \left(\frac{A_{a_{1}}}{\sigma\left(x_{a_{1}}^{T} \hat{\beta}_{S_{c}}^{\{-c_1,-c_2\}}\right)}-1 + A_{a_1}\xi_{a_1} \right)\left(\frac{A_{a_{2}}}{\sigma\left(x_{a_{2}}^{T} \hat{\beta}_{S_{c}}^{\{-c_1,-c_2\}}\right)}-1+ A_{a_2}\xi_{a_2}\right) \right.\\
    \cdot & \left.
    \left(\frac{A_{c_{1}}}{\sigma\left(x_{c_{1}}^{\top} \hat{\beta}_{S_{a}}^{\{-a_1,-a_2\}}\right)}-1 + A_{c_1}\zeta_{c_1} \right)\left(\frac{A_{c_{2}}}{\sigma\left(x_{c_{2}}^{\top} \hat{\beta}_{S_{a}}^{\{-a_1,-a_2\}}\right)}-1 +  A_{c_2}\zeta_{c_2}\right)x_{a_{1}}^{\top} x_{c_{1}}x_{a_{2}}^{\top} x_{c_{2}} \right]\\
    =& \mathbb{E}\left[ \left(\frac{A_{a_{1}}}{\sigma\left(x_{a_{1}}^{T} \hat{\beta}_{S_{c}}^{\{-c_1,-c_2\}}\right)}-1\right)\left(\frac{A_{a_{2}}}{\sigma\left(x_{a_{2}}^{T} \hat{\beta}_{S_{c}}^{\{-c_1,-c_2\}}\right)}-1\right)\left(\frac{A_{c_{1}}}{\sigma\left(x_{c_{1}}^{\top} \hat{\beta}_{S_{a}}^{\{-a_1,-a_2\}}\right)}-1\right) \right. \nonumber \\
    \cdot& \left. \left(\frac{A_{c_{2}}}{\sigma\left(x_{c_{2}}^{\top} \hat{\beta}_{S_{a}}^{\{-a_1,-a_2\}}\right)}-1\right) x_{a_{1}}^{\top} x_{c_{1}}x_{a_{2}}^{\top} x_{c_{2}} \right] \\
    +& \mathbb{E}\left[A_{a_{1}} \xi_{a_{1}}\left(\frac{A_{a_{2}}}{\sigma\left(x_{a_{2}}^{T} \hat{\beta}_{S_{c}}^{\left\{-c_{1},-c_{2}\right\}}\right)}-1\right)\left(\frac{A_{c_{1}}}{\sigma\left(x_{c_{1}}^{\top} \hat{\beta}_{S_{a}}^{\left\{-a_{1},-a_{2}\right\}}\right)}-1\right)\right.\\
    \cdot& \left.\left(\frac{A_{c_{2}}}{\sigma\left(x_{c_{2}}^{\top} \hat{\beta}_{S_{a}}^{\left\{-a_{1},-a_{2}\right\}}\right)}-1\right) x_{a_{1}}^{\top} x_{c_{1}} x_{a_{2}}^{\top} x_{c_{2}}\right]\\
    +& \mathbb{E}\left[A_{a_{2}} \xi_{a_{2}}\left(\frac{A_{a_{1}}}{\sigma\left(x_{a_{1}}^{T} \hat{\beta}_{S_{c}}^{\left\{-c_{1},-c_{2}\right\}}\right)}-1\right)\left(\frac{A_{c_{1}}}{\sigma\left(x_{c_{1}}^{\top} \hat{\beta}_{S_{a}}^{\left\{-a_{1},-a_{2}\right\}}\right)}-1\right)\right.\\
    \cdot& \left.\left(\frac{A_{c_{2}}}{\sigma\left(x_{c_{2}}^{\top} \hat{\beta}_{S_{a}}^{\left\{-a_{1},-a_{2}\right\}}\right)}-1\right) x_{a_{1}}^{\top} x_{c_{1}} x_{a_{2}}^{\top} x_{c_{2}}\right]\\
    +& \mathbb{E}\left[A_{c_{1}} \zeta_{c_{1}}\left(\frac{A_{a_{1}}}{\sigma\left(x_{a_{1}}^{T} \hat{\beta}_{S_{c}}^{\left\{-c_{1},-c_{2}\right\}}\right)}-1\right)\left(\frac{A_{a_{2}}}{\sigma\left(x_{a_{2}}^{T} \hat{\beta}_{S_{c}}^{\left\{-c_{1},-c_{2}\right\}}\right)}-1\right)\right.\\
    \cdot& \left.\left(\frac{A_{c_{2}}}{\sigma\left(x_{c_{2}}^{\top} \hat{\beta}_{S_{a}}^{\left\{-a_{1},-a_{2}\right\}}\right)}-1\right) x_{a_{1}}^{\top} x_{c_{1}} x_{a_{2}}^{\top} x_{c_{2}}\right]\\
    +& \mathbb{E}\left[A_{c_{2}} \zeta_{c_{2}}\left(\frac{A_{a_{1}}}{\sigma\left(x_{a_{1}}^{T} \hat{\beta}_{S_{c}}^{\left\{-c_{1},-c_{2}\right\}}\right)}-1\right)\left(\frac{A_{a_{2}}}{\sigma\left(x_{a_{2}}^{T} \hat{\beta}_{S_{c}}^{\left\{-c_{1},-c_{2}\right\}}\right)}-1\right)\right.\\
    \cdot& \left.\left(\frac{A_{c_{1}}}{\sigma\left(x_{c_{1}}^{\top} \hat{\beta}_{S_{a}}^{\left\{-a_{1},-a_{2}\right\}}\right)}-1\right) x_{a_{1}}^{\top} x_{c_{1}} x_{a_{2}}^{\top} x_{c_{2}}\right]\\
    +&  \mathbb{E}\left[A_{a_{1}} \xi_{a_{1}}A_{a_{2}} \xi_{a_{2}}\left(\frac{A_{c_{1}}}{\sigma\left(x_{c_{1}}^{\top} \hat{\beta}_{S_{a}}^{\left\{-a_{1},-a_{2}\right\}}\right)}-1\right)\left(\frac{A_{c_{2}}}{\sigma\left(x_{c_{2}}^{\top} \hat{\beta}_{S_{a}}^{\left\{-a_{1},-a_{2}\right\}}\right)}-1\right) x_{a_{1}}^{\top} x_{c_{1}} x_{a_{2}}^{\top} x_{c_{2}}\right]\\
    +& \mathbb{E}\left[A_{c_{1}} \zeta_{c_{1}}A_{c_{2}} \zeta_{c_{2}}\left(\frac{A_{a_{1}}}{\sigma\left(x_{a_{1}}^{T} \hat{\beta}_{S_{c}}^{\left\{-c_{1},-c_{2}\right\}}\right)}-1\right)\left(\frac{A_{a_{2}}}{\sigma\left(x_{a_{2}}^{T} \hat{\beta}_{S_{c}}^{\left\{-c_{1},-c_{2}\right\}}\right)}-1\right) x_{a_{1}}^{\top} x_{c_{1}} x_{a_{2}}^{\top} x_{c_{2}}\right]\\
    +& \mathbb{E}\left[A_{a_{1}} \xi_{a_{1}} A_{c_{1}} \zeta_{c_{1}} \left(\frac{A_{a_{2}}}{\sigma\left(x_{a_{2}}^{T} \hat{\beta}_{S_{c}}^{\left\{-c_{1},-c_{2}\right\}}\right)}-1\right)\left(\frac{A_{c_{2}}}{\sigma\left(x_{c_{2}}^{\top} \hat{\beta}_{S_{a}}^{\left\{-a_{1},-a_{2}\right\}}\right)}-1\right) x_{a_{1}}^{\top} x_{c_{1}} x_{a_{2}}^{\top} x_{c_{2}}\right]\\
    +& \mathbb{E}\left[A_{a_{1}} \xi_{a_{1}} A_{c_{2}} \zeta_{c_{2}} \left(\frac{A_{a_{2}}}{\sigma\left(x_{a_{2}}^{T} \hat{\beta}_{S_{c}}^{\left\{-c_{1},-c_{2}\right\}}\right)}-1\right)\left(\frac{A_{c_{1}}}{\sigma\left(x_{c_{1}}^{\top} \hat{\beta}_{S_{a}}^{\left\{-a_{1},-a_{2}\right\}}\right)}-1\right) x_{a_{1}}^{\top} x_{c_{1}} x_{a_{2}}^{\top} x_{c_{2}}\right]\\
    +& \mathbb{E}\left[A_{a_{2}} \xi_{a_{2}} A_{c_{2}} \zeta_{c_{2}} \left(\frac{A_{a_{1}}}{\sigma\left(x_{a_{1}}^{T} \hat{\beta}_{S_{c}}^{\left\{-c_{1},-c_{2}\right\}}\right)}-1\right)\left(\frac{A_{c_{1}}}{\sigma\left(x_{c_{1}}^{\top} \hat{\beta}_{S_{a}}^{\left\{-a_{1},-a_{2}\right\}}\right)}-1\right) x_{a_{1}}^{\top} x_{c_{1}} x_{a_{2}}^{\top} x_{c_{2}}\right]\\
    +& \mathbb{E}\left[A_{a_{2}} \xi_{a_{2}} A_{c_{1}} \zeta_{c_{1}} \left(\frac{A_{a_{1}}}{\sigma\left(x_{a_{1}}^{T} \hat{\beta}_{S_{c}}^{\left\{-c_{1},-c_{2}\right\}}\right)}-1\right)\left(\frac{A_{c_{2}}}{\sigma\left(x_{c_{2}}^{\top} \hat{\beta}_{S_{a}}^{\left\{-a_{1},-a_{2}\right\}}\right)}-1\right) x_{a_{1}}^{\top} x_{c_{1}} x_{a_{2}}^{\top} x_{c_{2}}\right]\\ +& o\left(n^{-2}\right).
\end{align*}

We study each term in the expression above. We know \begin{align}
    & \mathbb{E}\left[ \left(\frac{A_{a_{1}}}{\sigma\left(x_{a_{1}}^{T} \hat{\beta}_{S_{c}}^{\{-c_1,-c_2\}}\right)}-1\right)\left(\frac{A_{a_{2}}}{\sigma\left(x_{a_{2}}^{T} \hat{\beta}_{S_{c}}^{\{-c_1,-c_2\}}\right)}-1\right)\left(\frac{A_{c_{1}}}{\sigma\left(x_{c_{1}}^{\top} \hat{\beta}_{S_{a}}^{\{-a_1,-a_2\}}\right)}-1\right) \right. \nonumber \\
    \cdot& \left. \left(\frac{A_{c_{2}}}{\sigma\left(x_{c_{2}}^{\top} \hat{\beta}_{S_{a}}^{\{-a_1,-a_2\}}\right)}-1\right) x_{a_{1}}^{\top} x_{c_{1}}x_{a_{2}}^{\top} x_{c_{2}} \right] \nonumber \\
    =& \mathbb{E}\left[ \left(\frac{\sigma\left(x_{a_1}^\top \beta\right)}{\sigma\left(x_{a_{1}}^{T} \hat{\beta}_{S_{c}}^{\{-c_1,-c_2\}}\right)}-1\right)\left(\frac{\sigma\left(x_{a_2}^\top \beta\right)}{\sigma\left(x_{a_{2}}^{T} \hat{\beta}_{S_{c}}^{\{-c_1,-c_2\}}\right)}-1\right)\left(\frac{\sigma\left(x_{c_1}^\top \beta\right)}{\sigma\left(x_{c_{1}}^{\top} \hat{\beta}_{S_{a}}^{\{-a_1,-a_2\}}\right)}-1\right) \right. \nonumber \\
    \cdot& \left. \left(\frac{\sigma\left(x_{c_2}^\top \beta\right)}{\sigma\left(x_{c_{2}}^{\top} \hat{\beta}_{S_{a}}^{\{-a_1,-a_2\}}\right)}-1\right) x_{a_{1}}^{\top} x_{c_{1}}x_{a_{2}}^{\top} x_{c_{2}} \right] \nonumber \\
    =& \mathbb{E}\left[\left[ \left(\frac{\sigma\left(x_{a_1}^\top \beta\right)}{\sigma\left(x_{a_{1}}^{T} \hat{\beta}_{S_{c}}^{\{-c_1,-c_2\}}\right)}-1\right)\left(\frac{\sigma\left(x_{a_2}^\top \beta\right)}{\sigma\left(x_{a_{2}}^{T} \hat{\beta}_{S_{c}}^{\{-c_1,-c_2\}}\right)}-1\right)\left(\frac{\sigma\left(x_{c_1}^\top \beta\right)}{\sigma\left(x_{c_{1}}^{\top} \hat{\beta}_{S_{a}}^{\{-a_1,-a_2\}}\right)}-1\right) \right. \right. \nonumber \\
    \cdot& \left.\left. \left(\frac{\sigma\left(x_{c_2}^\top \beta\right)}{\sigma\left(x_{c_{2}}^{\top} \hat{\beta}_{S_{a}}^{\{-a_1,-a_2\}}\right)}-1\right) x_{a_{1}}^{\top} x_{c_{1}}x_{a_{2}}^{\top} x_{c_{2}} \mid \hat{\beta}_{S_a}^{ \{-a_1, -a_2\}}, \hat{\beta}_{S_c}^{ \{-c_1, -c_2\}}\right]\right] \nonumber \\
    =& \mathbb{E}\left[ \mathbb{E}\left[ \left(\frac{\sigma\left(x_{a_{1}}^{\top} \beta\right)}{\sigma\left(x_{a_{1}}^{T} \hat{\beta}_{S_{c}}^{\left\{-c_{1},-c_{2}\right\}}\right)}-1\right) \left(\frac{\sigma\left(x_{c_{1}}^{\top} \beta\right)}{\sigma\left(x_{c_{1}}^{\top} \hat{\beta}_{S_{a}}^{\left\{-a_{1},-a_{2}\right\}}\right)}-1\right) x_{a_{1}}^{\top} x_{c_{1}} \mid \hat{\beta}_{S_a}^{ \{-a_1, -a_2\}}, \hat{\beta}_{S_c}^{ \{-c_1, -c_2\}} \right] \right. \nonumber \\
    \cdot & \left. \mathbb{E}\left[ \left(\frac{\sigma\left(x_{a_{2}}^{\top} \beta\right)}{\sigma\left(x_{a_{2}}^{T} \hat{\beta}_{S_{c}}^{\left\{-c_{1},-c_{2}\right\}}\right)}-1\right) \left(\frac{\sigma\left(x_{c_{2}}^{\top} \beta\right)}{\sigma\left(x_{c_{2}}^{\top} \hat{\beta}_{S_{a}}^{\left\{-a_{1},-a_{2}\right\}}\right)}-1\right)x_{a_{2}}^{\top} x_{c_{2}} \mid \hat{\beta}_{S_a}^{ \{-a_1, -a_2\}}, \hat{\beta}_{S_c}^{ \{-c_1, -c_2\}} \right] \right] \nonumber\\
    =& \mathbb{E}\left[ \mathbb{E}\left[ \left(\frac{\sigma\left(x_{a_{1}}^{\top} \beta\right)}{\sigma\left(x_{a_{1}}^{T} \hat{\beta}_{S_{c}}^{\left\{-c_{1},-c_{2}\right\}}\right)}-1\right) \left(\frac{\sigma\left(x_{c_{1}}^{\top} \beta\right)}{\sigma\left(x_{c_{1}}^{\top} \hat{\beta}_{S_{a}}^{\left\{-a_{1},-a_{2}\right\}}\right)}-1\right) x_{a_{1}}^{\top} x_{c_{1}}  \right] \right. \nonumber \\
    \cdot & \left. \mathbb{E}\left[ \left(\frac{\sigma\left(x_{a_{2}}^{\top} \beta\right)}{\sigma\left(x_{a_{2}}^{T} \hat{\beta}_{S_{c}}^{\left\{-c_{1},-c_{2}\right\}}\right)}-1\right) \left(\frac{\sigma\left(x_{c_{2}}^{\top} \beta\right)}{\sigma\left(x_{c_{2}}^{\top} \hat{\beta}_{S_{a}}^{\left\{-a_{1},-a_{2}\right\}}\right)}-1\right)x_{a_{2}}^{\top} x_{c_{2}}  \right] \right] + o(n^{-2}) \nonumber\\
    =&  \mathbb{E}^2\left[ \left(\frac{\sigma\left(x_{a_{1}}^{\top} \beta\right)}{\sigma\left(x_{a_{1}}^{T} \hat{\beta}_{S_{c}}^{\left\{-c_{1},-c_{2}\right\}}\right)}-1\right) \left(\frac{\sigma\left(x_{c_{1}}^{\top} \beta\right)}{\sigma\left(x_{c_{1}}^{\top} \hat{\beta}_{S_{a}}^{\left\{-a_{1},-a_{2}\right\}}\right)}-1\right) x_{a_{1}}^{\top} x_{c_{1}}  \right] + o(n^{-2}) \nonumber \\
    =& \mathbb{E}^2\left[\left(\frac{A_{a_{1}}}{\sigma\left(x_{a_{1}}^{T} \hat{\beta}_{S_{c}}^{\{-c_1\}}\right)}-1\right)\left(\frac{A_{c_{1}}}{\sigma\left(x_{c_{1}}^{\top} \hat{\beta}_{S_{a}}^{\{-a_1\}}\right)}-1\right)x_{a_{1}}^{\top} x_{c_{1}}\right] + o(n^{-2}), \label{lemma_sub_res_0}
\end{align} where the order of error terms in the last two steps can be obtained by applying Stein's Lemma.

By apply Stein's Lemma twice we have \begin{align}
    & \mathbb{E}\left[ A_{a_1}\xi_{a_{1}}\left( \frac{A_{a_{2}}}{\sigma\left(x_{a_{2}}^{T} \hat{\beta}_{S_{c}}^{\{-c_1,-c_2\}}\right)}-1\right) \left( \frac{A_{c_{1}}}{\sigma\left(x_{c_{1}}^{\top} \hat{\beta}_{S_{a}}^{\{-a_1,-a_2\}}\right)}-1\right) \right.  \nonumber\\
    \cdot& \left.\left( \frac{A_{c_{2}}}{\sigma\left(x_{c_{2}}^{\top} \hat{\beta}_{S_{a}}^{\{-a_1,-a_2\}}\right)}-1  \right) x_{a_{1}}^{\top} x_{c_{1}} x_{a_{2}}^{\top} x_{c_{2}} \right] \nonumber \\
    =&\mathbb{E}\left[ \mathbb{E}\left[ A_{a_1}\xi_{a_{1}}\left( \frac{A_{a_{2}}}{\sigma\left(x_{a_{2}}^{T} \hat{\beta}_{S_{c}}^{\{-c_1,-c_2\}}\right)}-1\right) \left( \frac{A_{c_{1}}}{\sigma\left(x_{c_{1}}^{\top} \hat{\beta}_{S_{a}}^{\{-a_1,-a_2\}}\right)}-1\right) \right.\right. \nonumber\\
    \cdot& \left.\left. \left( \frac{A_{c_{2}}}{\sigma\left(x_{c_{2}}^{\top} \hat{\beta}_{S_{a}}^{\{-a_1,-a_2\}}\right)}-1  \right)  x_{a_{1}}^{\top} x_{c_{1}} x_{a_{2}}^{\top} x_{c_{2}} \mid S_c, x_{a_1}, A_{a_1},  x_{c_1}, A_{c_1},\hat{\beta}_{S_a}^{\{-a_1, -a_2\}} \right]\right] \nonumber \\
    =&  \mathbb{E}\left[A_{a_1}\xi_{a_{1}} \left( \frac{A_{c_{1}}}{\sigma\left(x_{c_{1}}^{\top} \hat{\beta}_{S_{a}}^{\{-a_1,-a_2\}}\right)}-1\right) x_{a_1}^\top x_{c_1} \left( \frac{A_{c_{2}}}{\sigma\left(x_{c_{2}}^{\top} \hat{\beta}_{S_{a}}^{\{-a_1,-a_2\}}\right)}-1  \right)   \right. \nonumber\\
    \cdot& \left.\mathbb{E}\left[ \left( \frac{A_{a_{2}}}{\sigma\left(x_{a_{2}}^{T} \hat{\beta}_{S_{c}}^{\{-c_1,-c_2\}}\right)}-1\right)   x_{a_{2}}^{\top} x_{c_{2}} \mid S_c, x_{a_1}, A_{a_1}, x_{c_1}, A_{c_1}, \hat{\beta}_{S_a}^{\{-a_1, -a_2\}} \right]\right] \nonumber \\
    =& \frac{1}{n} \mathbb{E}\left[A_{a_1}\xi_{a_{1}} \left( \frac{A_{c_{1}}}{\sigma\left(x_{c_{1}}^{\top} \hat{\beta}_{S_{a}}^{\{-a_1,-a_2\}}\right)}-1\right) x_{a_1}^\top x_{c_1} \left( \frac{A_{c_{2}}}{\sigma\left(x_{c_{2}}^{\top} \hat{\beta}_{S_{a}}^{\{-a_1,-a_2\}}\right)}-1  \right)\right. \nonumber \\
    \cdot & \left. \left(C_c x_{c_2}^\top \hat{\beta}_{S_c}^{\{-c_1, -c_2\}} + B_c x_{c_2}^\top \beta\right) \right] + o(n^{-2}) \nonumber\\
    =& \frac{1}{n} \mathbb{E}\left[ \left(C_c x_{c_2}^\top \hat{\beta}_{S_c}^{\{-c_1, -c_2\}} + B_c x_{c_2}^\top \beta\right) \left( \frac{A_{c_{1}}}{\sigma\left(x_{c_{1}}^{\top} \hat{\beta}_{S_{a}}^{\{-a_1,-a_2\}}\right)}-1\right)\left( \frac{A_{c_{2}}}{\sigma\left(x_{c_{2}}^{\top} \hat{\beta}_{S_{a}}^{\{-a_1,-a_2\}}\right)}-1  \right) \right. \nonumber \\
    \cdot& \left. \mathbb{E}\left[A_{a_1}\xi_{a_{1}}  x_{a_1}^\top x_{c_1}  \mid S_c , \hat{\beta}_{S_{a}}^{\left\{-a_{1},-a_{2}\right\}}  \right]\right] + o(n^{-2}) \nonumber\\
    =& \frac{1}{n^2} \mathbb{E}\left[ \left(C_c x_{c_2}^\top \hat{\beta}_{S_c}^{\{-c_1, -c_2\}} + B_c x_{c_2}^\top \beta\right) \left( \frac{A_{c_{1}}}{\sigma\left(x_{c_{1}}^{\top} \hat{\beta}_{S_{a}}^{\{-a_1,-a_2\}}\right)}-1\right) \right. \nonumber\\
    &\left. \left( \frac{A_{c_{2}}}{\sigma\left(x_{c_{2}}^{\top} \hat{\beta}_{S_{a}}^{\{-a_1,-a_2\}}\right)}-1  \right) C_c x_{c_1}^\top \left( \hat{\beta}_{S_c} - \hat{\beta}_{S_c}^{\{-c_1, -c_2\}} \right) \right] 
    + o(n^{-2}) \nonumber \\
    =& \frac{1}{n^2} \mathbb{E}\left[ \left(C_c x_{c_2}^\top \hat{\beta}_{S_c}^{\{-c_1, -c_2\}} + B_c x_{c_2}^\top \beta\right) \left( \frac{A_{c_{1}}}{\sigma\left(x_{c_{1}}^{\top} \hat{\beta}_{S_{a}}^{\{-a_1,-a_2\}}\right)}-1\right)
    \right. \nonumber \\
    &\left. \left( \frac{A_{c_{2}}}{\sigma\left(x_{c_{2}}^{\top} \hat{\beta}_{S_{a}}^{\{-a_1,-a_2\}}\right)}-1  \right) C_c x_{c_1}^\top \left( \hat{\beta}_{S_c}^{\{-c_2\}} - \hat{\beta}_{S_c}^{\{-c_1, -c_2\}} \right) \right] 
    + o(n^{-2}) \nonumber \\
    =& \frac{1}{n^2}\mathbb{E}\left[ \mathbb{E}\left[\left(C_{c} x_{c_{2}}^{\top} \hat{\beta}_{S_{c}}^{\left\{-c_{1},-c_{2}\right\}}+B_{c} x_{c_{2}}^{\top} \beta\right) \left(\frac{A_{c_{2}}}{\sigma\left(x_{c_{2}}^{\top} \hat{\beta}_{S_{a}}^{\left\{-a_{1},-a_{2}\right\}}\right)}-1\right) \mid \hat{\beta}_{S_{c}}^{\left\{-c_{1},-c_{2}\right\}} \right]\right. \nonumber \\
    \cdot& \left. \mathbb{E}\left[ \left(\frac{A_{c_{1}}}{\sigma\left(x_{c_{1}}^{\top} \hat{\beta}_{S_{a}}^{\left\{-a_{1},-a_{2}\right\}}\right)}-1\right) C_{c} x_{c_{1}}^{\top}\left(\hat{\beta}_{S_{c}}^{\left\{-c_{2}\right\}}-\hat{\beta}_{S_{c}}^{\left\{-c_{1},-c_{2}\right\}}\right) \mid \hat{\beta}_{S_{c}}^{\left\{-c_{1},-c_{2}\right\}} \right] \right] + o({n^{-2}}) \nonumber\\
    =& \frac{1}{n^2}\mathbb{E}\left[ \mathbb{E}\left[\left(C_{c} x_{c_{2}}^{\top} \hat{\beta}_{S_{c}}^{\left\{-c_{2}\right\}}+B_{c} x_{c_{2}}^{\top} \beta\right) \left(\frac{A_{c_{2}}}{\sigma\left(x_{c_{2}}^{\top} \hat{\beta}_{S_{a}}^{\left\{-a_{1},-a_{2}\right\}}\right)}-1\right)\right]\right. \nonumber \\
    \cdot& \left. \mathbb{E}\left[ \left(\frac{A_{c_{1}}}{\sigma\left(x_{c_{1}}^{\top} \hat{\beta}_{S_{a}}^{\left\{-a_{1},-a_{2}\right\}}\right)}-1\right) C_{c} x_{c_{1}}^{\top}\left(\hat{\beta}_{S_{c}}-\hat{\beta}_{S_{c}}^{\left\{-c_{1}\right\}}\right)  \right] \right] + o({n^{-2}}) \nonumber\\
    =& \mathbb{E}\left[A_{a_{1}} \xi_{a_{1}}\left(\frac{A_{c_{1}}}{\sigma\left(x_{c_{1}}^{\top} \hat{\beta}_{S_{a}}^{\left\{-a_{1}\right\}}\right)}-1\right) x_{a_{1}}^{\top} x_{c_{1}}\right] \label{lemma_sub_res_1} \\
    \cdot& \mathbb{E}\left[\left(\frac{A_{a_{1}}}{\sigma\left(x_{a_{1}}^{T} \hat{\beta}_{S_{c}}^{\left\{-c_{1}\right\}}\right)}-1\right)\left(\frac{A_{c_{1}}}{\sigma\left(x_{c_{1}}^{\top} \hat{\beta}_{S_{a}}^{\left\{-a_{1}\right\}}\right)}-1\right) x_{a_{1}}^{\top} x_{c_{1}}\right] + o\left(n^{-2}\right) \nonumber  
\end{align}

Similarly, we can show \begin{align}
    & \mathbb{E}\left[ A_{a_1}\xi_{a_{1}} A_{a_2}\xi_{a_{2}} \left( \frac{A_{c_{1}}}{\sigma\left(x_{c_{1}}^{\top} \hat{\beta}_{S_{a}}^{\{-a_1,-a_2\}}\right)}-1\right) \left( \frac{A_{c_{2}}}{\sigma\left(x_{c_{2}}^{\top} \hat{\beta}_{S_{a}}^{\{-a_1,-a_2\}}\right)}-1  \right) x_{a_{1}}^{\top} x_{c_{1}}x_{a_{2}}^{\top} x_{c_{2}} \right] \nonumber \\
    =& \mathbb{E}\left[ A_{a_1}\xi_{a_{1}} \left(\frac{A_{j_{1}}}{\sigma\left(x_{c_{1}}^{\top} \hat{\beta}_{S_{a}}^{\{-a_1,-a_2\}}\right)}-1\right) x_{a_{1}}^{\top} x_{c_{1}} \right] \nonumber \\
    \cdot& \mathbb{E}\left[ A_{a_2}\xi_{a_{2}} \left(\frac{A_{c_{1}}}{\sigma\left(x_{c_{2}}^{\top} \hat{\beta}_{S_{a}}^{\{-a_1,-a_2\}}\right)}-1\right) x_{a_{2}}^{\top} x_{c_{2}} \right]+ o(n^{-2}) \nonumber \\
    =& \mathbb{E}^2\left[A_{a_1} \xi_{a_{1}} \left(\frac{A_{c_{1}}}{\sigma\left(x_{c_{1}}^{\top} \hat{\beta}_{S_{a}}^{\{-a_1\}}\right)}-1\right) x_{a_{1}}^{\top} x_{c_{1}} \right] + o(n^{-2}).  \label{lemma_sub_res_2}
\end{align}

Next we show \begin{align}
    & \mathbb{E}\left[A_{a_1} \xi_{a_{1}}  A_{c_2}\zeta_{c_{2}} \left(\frac{A_{a_{2}}}{\sigma\left(x_{a_{2}}^{T} \hat{\beta}_{S_{c}}^{\{-c_1,-c_2\}}\right)}-1\right) \left(\frac{A_{c_{1}}}{\sigma\left(x_{c_{1}}^{\top} \hat{\beta}_{S_{a}}^{\{-a_1,-a_2\}}\right)}-1\right) x_{a_{1}}^{\top} x_{c_{1}}x_{a_{2}}^{\top} x_{c_{2}} \right] \nonumber \\
    =& \mathbb{E}\left[A_{a_1}{\xi}_{a_1}\left(\frac{A_{c_{1}}}{\sigma\left(x_{c_{1}}^{\top} \hat{\beta}_{S_{a}}^{\{-a_1\}}\right)}-1\right)x_{a_{1}}^{\top} x_{c_{1}}\right]\label{lemma_sub_res_3}  \\
    \cdot& \mathbb{E}\left[A_{c_2}{\zeta}_{c_2}\left(\frac{A_{a_{2}}}{\sigma\left(x_{a_{2}}^{\top} \hat{\beta}_{S_{a}}^{\{-c_2\}}\right)}-1\right)x_{a_{2}}^{\top} x_{c_{2}}\right] + o(n^{-2}).\nonumber 
\end{align}

We have 
\begin{align*}
    & \mathbb{E}\left[ A_{a_1} \xi_{a_{1}} A_{c_2}\zeta_{c_{2}}\left(\frac{A_{a_{2}}}{\sigma\left(x_{a_{2}}^{T} \hat{\beta}_{S_{c}}^{\{-c_1,-c_2\}}\right)}-1\right)\left(\frac{A_{c_{1}}}{\sigma\left(x_{c_{1}}^{\top} \hat{\beta}_{S_{a}}^{\{-a_1,-a_2\}}\right)}-1\right)x_{a_{1}}^{\top} x_{c_{1}}x_{a_{2}}^{\top} x_{c_{2}}\right]\\
    =& \mathbb{E}\left[\mathbb{E}\left[A_{a_1}\xi_{a_{1}}A_{c_2} \zeta_{c_{2}}\left(\frac{A_{a_{2}}}{\sigma\left(x_{a_{2}}^{T} \hat{\beta}_{S_{c}}^{\{-c_1,-c_2\}}\right)}-1\right)\left(\frac{A_{c_{1}}}{\sigma\left(x_{c_{1}}^{\top} \hat{\beta}_{S_{a}}^{\{-a_1,-a_2\}}\right)}-1\right) \right. \right. \\
    &\left. \left. x_{a_{1}}^{\top} x_{c_{1}}x_{a_{2}}^{\top} x_{c_{2}} \mid S_a, A_{c_2}, x_{c_2}, \hat{\beta}_{S_c}^{\{-c_1,-c_2\}} \right]\right]\\
    =&\mathbb{E}\left[A_{c_2}\zeta_{c_{2}}\left(\frac{A_{a_{2}}}{\sigma\left(x_{a_{2}}^{T} \hat{\beta}_{S_{c}}^{\{-c_1,-c_2\}}\right)}-1\right)x_{a_{2}}^{\top} x_{c_{2}}  \mathbb{E}\left[A_{a_1}\xi_{a_{1}}\left(\frac{A_{c_{1}}}{\sigma\left(x_{c_{1}}^{\top} \hat{\beta}_{S_{a}}^{\{-a_1,-a_2\}}\right)}-1\right) \right. \right. \\
    &\left. \left. x_{a_{1}}^{\top} x_{c_{1}} \mid S_a, A_{c_2}, x_{c_2}, \hat{\beta}_{S_c}^{\{-c_1,-c_2\}} \right]\right] 
   + o(n^{-2})\\
    =& \mathbb{E}\left[ \mathbb{E}\left[A_{c_{2}} \zeta_{c_{2}}\left(\frac{A_{a_{2}}}{\sigma\left(x_{a_{2}}^{T} \hat{\beta}_{S_{c}}^{\left\{-c_{1},-c_{2}\right\}}\right)}-1\right) x_{a_{2}}^{\top} x_{c_{2}} \mid S_c, A_{a_1},x_{a_1}, \hat{\beta}_{S_{a}}^{\left\{-a_{1},-a_{2}\right\}} \right]  \right.\\
    \cdot& \left. \mathbb{E}\left[ A_{a_{1}} \xi_{a_{1}}\left(\frac{A_{c_{1}}}{\sigma\left(x_{c_{1}}^{\top} \hat{\beta}_{S_{a}}^{\left\{-a_{1},-a_{2}\right\}}\right)}-1\right) x_{a_{1}}^{\top} x_{c_{1}} \mid S_{a}, A_{c_{2}}, x_{c_{2}}, \hat{\beta}_{S_{c}}^{\left\{-c_{1},-c_{2}\right\}}  \right]    \right] + o(n^{-2}), 
\end{align*} where the last step is obtained by first conditioning on everything other than $x_{a_2}, A_{a_2}$.

Note that \begin{align*}
    & \mathbb{E}\left[ A_{a_{1}} \xi_{a_{1}}\left(\frac{A_{c_{1}}}{\sigma\left(x_{c_{1}}^{\top} \hat{\beta}_{S_{a}}^{\left\{-a_{1},-a_{2}\right\}}\right)}-1\right) x_{a_{1}}^{\top} x_{c_{1}} \mid S_{a}, A_{c_{2}}, x_{c_{2}}, \hat{\beta}_{S_{c}}^{\left\{-c_{1},-c_{2}\right\}}  \right] \\
    =& \mathbb{E}\left[ A_{a_{1}} \left( \frac{1}{\sigma\left(x_{a_1}^{T} \hat{\beta}_{S_{c}}\right)}-\frac{1}{\sigma\left(x_{a_1}^{T} \hat{\beta}_{S_{c}}^{-\left\{c_{1}, c_{2}\right\}}\right)} \right)\left(\frac{A_{c_{1}}}{\sigma\left(x_{c_{1}}^{\top} \hat{\beta}_{S_{a}}^{\left\{-a_{1},-a_{2}\right\}}\right)}-1\right) \right.\\
    \cdot & \left. x_{a_{1}}^{\top} x_{c_{1}} \mid S_{a}, A_{c_{2}}, x_{c_{2}}, \hat{\beta}_{S_{c}}^{\left\{-c_{1},-c_{2}\right\}}  \right]\\
    =& \mathbb{E}\left[ A_{a_{1}} \left( \frac{1}{\sigma\left(x_{a_1}^{T} \hat{\beta}_{S_{c}}\right)}-\frac{1}{\sigma\left(x_{a_1}^{T} \hat{\beta}_{S_{c}}^{-\left\{c_{1}\right\}}\right)} \right)\left(\frac{A_{c_{1}}}{\sigma\left(x_{c_{1}}^{\top} \hat{\beta}_{S_{a}}^{\left\{-a_{1},-a_{2}\right\}}\right)}-1\right) \right.\\
    \cdot & \left. x_{a_{1}}^{\top} x_{c_{1}} \mid S_{a}, A_{c_{2}}, x_{c_{2}}, \hat{\beta}_{S_{c}}^{\left\{-c_{1},-c_{2}\right\}}  \right]\\
    +& \mathbb{E}\left[ A_{a_{1}} \left( \frac{1}{\sigma\left(x_{a_1}^{T} \hat{\beta}_{S_{c}}^{-\left\{c_{1}\right\}}\right)}-\frac{1}{\sigma\left(x_{a_1}^{T} \hat{\beta}_{S_{c}}^{-\left\{c_{1}, c_2\right\}}\right)} \right)\left(\frac{A_{c_{1}}}{\sigma\left(x_{c_{1}}^{\top} \hat{\beta}_{S_{a}}^{\left\{-a_{1},-a_{2}\right\}}\right)}-1\right) \right.\\
    \cdot & \left. x_{a_{1}}^{\top} x_{c_{1}} \mid S_{a}, A_{c_{2}}, x_{c_{2}}, \hat{\beta}_{S_{c}}^{\left\{-c_{1},-c_{2}\right\}}  \right].
\end{align*}

We know \begin{align*}
    & \mathbb{E}\left[ A_{a_{1}} \left( \frac{1}{\sigma\left(x_{a_1}^{T} \hat{\beta}_{S_{c}}^{-\left\{c_{1}\right\}}\right)}-\frac{1}{\sigma\left(x_{a_1}^{T} \hat{\beta}_{S_{c}}^{-\left\{c_{1}, c_2\right\}}\right)} \right)\left(\frac{A_{c_{1}}}{\sigma\left(x_{c_{1}}^{\top} \hat{\beta}_{S_{a}}^{\left\{-a_{1},-a_{2}\right\}}\right)}-1\right) \right.\\
    \cdot & \left. x_{a_{1}}^{\top} x_{c_{1}} \mid S_{a}, A_{c_{2}}, x_{c_{2}}, \hat{\beta}_{S_{c}}^{\left\{-c_{1},-c_{2}\right\}}  \right]\\
    =&  A_{a_{1}} \left( \frac{1}{\sigma\left(x_{a_1}^{T} \hat{\beta}_{S_{c}}^{-\left\{c_{1}\right\}}\right)}-\frac{1}{\sigma\left(x_{a_1}^{T} \hat{\beta}_{S_{c}}^{-\left\{c_{1}, c_2\right\}}\right)} \right) \mathbb{E}\left[ \left(\frac{A_{c_{1}}}{\sigma\left(x_{c_{1}}^{\top} \hat{\beta}_{S_{a}}^{\left\{-a_{1},-a_{2}\right\}}\right)}-1\right) \right.\\
    \cdot & \left. x_{a_{1}}^{\top} x_{c_{1}} \mid S_{a}, A_{c_{2}}, x_{c_{2}}, \hat{\beta}_{S_{c}}^{\left\{-c_{1},-c_{2}\right\}}  \right] =  o_p(n^{-1}).
\end{align*}

Also, we know the limit of \begin{align*}
    & \mathbb{E}\left[A_{a_{1}}\left(\frac{1}{\sigma\left(x_{a_{1}}^{T} \hat{\beta}_{S_{c}}\right)}-\frac{1}{\sigma\left(x_{a_{1}}^{T} \hat{\beta}_{S_{c}}^{-\left\{c_{1}\right\}}\right)}\right)\left(\frac{A_{c_{1}}}{\sigma\left(x_{c_{1}}^{\top} \hat{\beta}_{S_{a}}^{\left\{-a_{1},-a_{2}\right\}}\right)}-1\right) \right.\\
    \cdot & \left. x_{a_{1}}^{\top} x_{c_{1}} \mid S_{a}, A_{c_{2}}, x_{c_{2}}, \hat{\beta}_{S_{c}}^{\left\{-c_{1},-c_{2}\right\}}\right]
\end{align*}  only depends on $ x_{a_1}, A_{a_1}$. Thus the limit of $$\mathbb{E}\left[A_{a_{1}} \xi_{a_{1}}\left(\frac{A_{c_{1}}}{\sigma\left(x_{c_{1}}^{\top} \hat{\beta}_{S_{a}}^{\left\{-a_{1},-a_{2}\right\}}\right)}-1\right) x_{a_{1}}^{\top} x_{c_{1}} \mid S_{a}, A_{c_{2}}, x_{c_{2}}, \hat{\beta}_{S_{c}}^{\left\{-c_{1},-c_{2}\right\}}\right]$$ only depends on $ x_{a_1}, A_{a_1}$. Similarly, we know the limit of $$\mathbb{E}\left[A_{c_{2}} \zeta_{c_{2}}\left(\frac{A_{a_{2}}}{\sigma\left(x_{a_{2}}^{T} \hat{\beta}_{S_{c}}^{\left\{-c_{1},-c_{2}\right\}}\right)}-1\right) x_{a_{2}}^{\top} x_{c_{2}} \mid S_{c}, A_{a_{1}}, x_{a_{1}}, \hat{\beta}_{S_{a}}^{\left\{-a_{1},-a_{2}\right\}}\right]$$ only depends on $x_{c_2}, A_{c_2}$. 

This implies \begin{align*}
    & \mathbb{E}\left[ \mathbb{E}\left[A_{c_{2}} \zeta_{c_{2}}\left(\frac{A_{a_{2}}}{\sigma\left(x_{a_{2}}^{T} \hat{\beta}_{S_{c}}^{\left\{-c_{1},-c_{2}\right\}}\right)}-1\right) x_{a_{2}}^{\top} x_{c_{2}} \mid S_c, A_{a_1},x_{a_1}, \hat{\beta}_{S_{a}}^{\left\{-a_{1},-a_{2}\right\}} \right]  \right.\\
    \cdot& \left. \mathbb{E}\left[ A_{a_{1}} \xi_{a_{1}}\left(\frac{A_{c_{1}}}{\sigma\left(x_{c_{1}}^{\top} \hat{\beta}_{S_{a}}^{\left\{-a_{1},-a_{2}\right\}}\right)}-1\right) x_{a_{1}}^{\top} x_{c_{1}} \mid S_{a}, A_{c_{2}}, x_{c_{2}}, \hat{\beta}_{S_{c}}^{\left\{-c_{1},-c_{2}\right\}}  \right]    \right]\\
    =&  \mathbb{E}\left[ \mathbb{E}\left[A_{c_{2}} \zeta_{c_{2}}\left(\frac{A_{a_{2}}}{\sigma\left(x_{a_{2}}^{T} \hat{\beta}_{S_{c}}^{\left\{-c_{1},-c_{2}\right\}}\right)}-1\right) x_{a_{2}}^{\top} x_{c_{2}} \mid S_c, A_{a_1},x_{a_1}, \hat{\beta}_{S_{a}}^{\left\{-a_{1},-a_{2}\right\}} \right]  \right]  \\
    \cdot&  \mathbb{E}\left[  \mathbb{E}\left[ A_{a_{1}} \xi_{a_{1}}\left(\frac{A_{c_{1}}}{\sigma\left(x_{c_{1}}^{\top} \hat{\beta}_{S_{a}}^{\left\{-a_{1},-a_{2}\right\}}\right)}-1\right) x_{a_{1}}^{\top} x_{c_{1}} \mid S_{a}, A_{c_{2}}, x_{c_{2}}, \hat{\beta}_{S_{c}}^{\left\{-c_{1},-c_{2}\right\}}  \right] \right]  + o(n^{-2})\\
    =& \mathbb{E}\left[A_{c_{2}} \zeta_{c_{2}}\left(\frac{A_{a_{2}}}{\sigma\left(x_{a_{2}}^{T} \hat{\beta}_{S_{c}}^{\left\{-c_{1},-c_{2}\right\}}\right)}-1\right) x_{a_{2}}^{\top} x_{c_{2}} \right]\\
    \cdot &  \mathbb{E}\left[ A_{a_{1}} \xi_{a_{1}}\left(\frac{A_{c_{1}}}{\sigma\left(x_{c_{1}}^{\top} \hat{\beta}_{S_{a}}^{\left\{-a_{1},-a_{2}\right\}}\right)}-1\right) x_{a_{1}}^{\top} x_{c_{1}}  \right] + o(n^{-2})\\
    =& \mathbb{E}\left[A_{a_{1}} \xi_{a_{1}}\left(\frac{A_{c_{1}}}{\sigma\left(x_{c_{1}}^{\top} \hat{\beta}_{S_{a}}^{\left\{-a_{1}\right\}}\right)}-1\right) x_{a_{1}}^{\top} x_{c_{1}}\right]\\
    \cdot & \mathbb{E}\left[A_{c_{2}} \zeta_{c_{2}}\left(\frac{A_{a_{2}}}{\sigma\left(x_{a_{2}}^{\top} \hat{\beta}_{S_{a}}^{\left\{-c_{2}\right\}}\right)}-1\right) x_{a_{2}}^{\top} x_{c_{2}}\right]+o\left(n^{-2}\right),
\end{align*} which completes the proof for (\ref{lemma_sub_res_3}).

Similarly, we can show \begin{align}
    & \mathbb{E}\left[A_{a_{1}} \xi_{a_{1}} A_{c_{1}} \zeta_{c_{1}}\left(\frac{A_{a_{2}}}{\sigma\left(x_{a_{2}}^{T} \hat{\beta}_{S_{c}}^{\left\{-c_{1},-c_{2}\right\}}\right)}-1\right)\left(\frac{A_{c_{2}}}{\sigma\left(x_{c_{2}}^{\top} \hat{\beta}_{S_{a}}^{\left\{-a_{1},-a_{2}\right\}}\right)}-1\right) x_{a_{1}}^{\top} x_{c_{1}} x_{a_{2}}^{\top} x_{c_{2}}\right] \nonumber \\
    =& \mathbb{E}\left[A_{a_{1}} \tilde{\xi}_{a_{1}} A_{c_{1}} \tilde{\zeta}_{c_{1}} x_{a_{1}}^{\top} x_{c_{1}}\right] \mathbb{E}\left[\left(\frac{A_{a_{1}}}{\sigma\left(x_{a_{1}}^{T} \hat{\beta}_{S_{c}}^{\left\{-c_{1}\right\}}\right)}-1\right)\left(\frac{A_{c_{1}}}{\sigma\left(x_{c_{1}}^{\top} \hat{\beta}_{S_{a}}^{\left\{-a_{1}\right\}}\right)}-1\right) x_{a_{1}}^{\top} x_{c_{1}}\right] \label{lemma_sub_res_4}\\
    +& o\left(n^{-2}\right) \nonumber
\end{align}

 Now, plug in equations (\ref{lemma_sub_res_0}), (\ref{lemma_sub_res_1}), (\ref{lemma_sub_res_2}), (\ref{lemma_sub_res_3}), and (\ref{lemma_sub_res_4}) to the right hand side of (\ref{lemma_to_show}), we can recover the left hand side of (\ref{lemma_to_show}). This completes the proof of the lemma.

\subsection{Proof of Lemma \ref{lemma 25 second}} 
First we find the expectation of the left hand side.
  
  Conditioned on $ \hat{\beta}_{S_{a}} $, by Stein’s Lemma, we know \begin{align*}
      & \mathbb{E}\left[ \sum_{i \in S_{c}} \frac{A_{i} x_{i}}{\sigma\left(x_{i}^{T} \hat{\beta}_{S_{a}}\right)} \right] = n_c \mathbb{E}\left[  \frac{ \sigma\left(x_{i}^{T} \beta\right) x_{i}}{\sigma\left(x_{i}^{T} \hat{\beta}_{S_{a}}\right)} \right] = \frac{n_c}{n} \mathbb{E}\left[\boldsymbol{\nabla}_{x} \frac{\sigma\left(x_{i}^{T} \beta\right)}{\sigma\left(x_{i}^{T} \hat{\beta}_{S_{a}}\right)} \right]\\
      =& \frac{n_c}{n} \mathbb{E}\left[\frac{\sigma^{\prime}\left(x_{1}^{T} \beta\right)}{\sigma\left(x_{1}^{T} \hat{\beta}_{S_{a}}\right)} \right] \beta-\frac{n_c}{n} \mathbb{E}\left[\frac{\sigma\left(x_{1}^{T} \beta\right) \sigma^{\prime}\left(x_{1}^{T} \hat{\beta}_{S_{a}}\right)}{\sigma^{2}\left(x_{1}^{T} \hat{\beta}_{S_{a}}\right)} \right] \hat{\beta}_{S_{a}}\\
      =& r_c \mathbb{E}\left[ \frac{\sigma^{'}(Z_\beta)}{\sigma(Z_{\hat{\beta}_{S_a}})} \right] \beta - r_c  \mathbb{E}\left[ \frac{\sigma(Z_\beta)\left(1- \sigma(Z_{\hat{\beta}_{S_a}})\right)}{\sigma(Z_{\hat{\beta}_{S_a}})} \right] \hat{\beta}_{S_a} \\
      =& r_c \mathbb{E}\left[ \frac{\sigma^{'}(Z_\beta)}{\sigma(Z_{\hat{\beta}_{S_a}})} \right] \beta - r_c  \mathbb{E}\left[ \frac{\sigma(Z_\beta)}{\sigma(Z_{\hat{\beta}_{S_a}})} - \frac{1}{2} \right] \hat{\beta}_{S_a}   := C_1 \beta - C_2 \hat{\beta}_{S_{a}}.
  \end{align*} Thus \begin{align*}
      & \mathbb{E}\left[ \frac{1}{n}\left(\sum_{j \in S_{c}} \frac{A_{j} x_{j}}{\sigma\left(x_{j}^{T} \hat{\beta}_{S_{a}}\right)}\right)^{T}\left(\sum_{i \in S_{b}} A_{i} x_{i} x_{i}^{\top}\right)^{-1}\left(\sum_{i \in S_{b}} \frac{A_{i} x_{i}}{\sigma\left(x_{i}^{T} \hat{\beta}_{S_{a}}\right)}\right)  \right]\\
      =& \mathbb{E}\left[ \frac{1}{n}\left( \mathbb{E}\left[ \sum_{j \in S_{c}} \frac{A_{j} x_{j}}{\sigma\left(x_{j}^{T} \hat{\beta}_{S_{a}}\right)} \right]\right)^{T}\left(\sum_{i \in S_{b}} A_{i} x_{i} x_{i}^{\top}\right)^{-1}\left(\sum_{i \in S_{b}} \frac{A_{i} x_{i}}{\sigma\left(x_{i}^{T} \hat{\beta}_{S_{a}}\right)}\right)  \right]\\
      =& r_b \mathbb{E}\left[ \left( \mathbb{E}\left[ \sum_{j \in S_{c}} \frac{A_{j} x_{j}}{\sigma\left(x_{j}^{T} \hat{\beta}_{S_{a}}\right)} \right]\right)^{T}\left(\sum_{i \in S_{b}} A_{i} x_{i} x_{i}^{\top}\right)^{-1} \frac{A_{b_1} x_{b_1}}{\sigma\left(x_{b_1}^{T} \hat{\beta}_{S_{a}}\right)} \right]\\
      =&  r_b \mathbb{E}\left[ \frac{1}{\sigma\left(x_{b_1}^{T} \hat{\beta}_{S_{a}}\right)} \left(\mathbb{E}\left[\sum_{j \in S_{c}} \frac{A_{j} x_{j}}{\sigma\left(x_{j}^{T} \hat{\beta}_{S_{a}}\right)}\right]\right)^{T}  \frac{ \left(\sum_{i =2}^{ n_{b}} A_{i} x_{i} x_{i}^{\top}\right)^{-1} A_{b_1}x_{b_1}}{1 + A_{b_1} x_{b_1}^\top \left(\sum_{i=2}^{n_{b}} A_{i} x_{i} x_{i}^{\top}\right)^{-1} x_{b_1}}  \right]\\
      =&  r_b \mathbb{E}\left[ \frac{A_{b_1}}{\sigma\left(x_{b_1}^{T} \hat{\beta}_{S_{a}}\right)} \left(\mathbb{E}\left[\sum_{j \in S_{c}} \frac{A_{j} x_{j}}{\sigma\left(x_{j}^{T} \hat{\beta}_{S_{a}}\right)}\right]\right)^{T}  x_{b_1} \cdot \frac{ \frac{2n}{n_b-2p} }{1 + \frac{2 n}{n_{b}-2 p} A_{b_1} \| x_{b_1} \|^2} \right] + o(1)\\
      =&  r_b \mathbb{E}\left[ \frac{A_{b_1}}{\sigma\left(x_{b_1}^{T} \hat{\beta}_{S_{a}}\right)} C_1  x_{b_1}^\top \beta \cdot \frac{ \frac{2n}{n_b-2p} }{1 + \frac{2 n}{n_{b}-2 p} A_{b_1} \| x_{b_1} \|^2} \right]\\
      -& \mathbb{E}\left[ \frac{A_{b_1}}{\sigma\left(x_{b_1}^{T} \hat{\beta}_{S_{a}}\right)} C_2  x_{b_1}^\top \hat{\beta}_{S_a} \cdot \frac{ \frac{2n}{n_b-2p} }{1 + \frac{2 n}{n_{b}-2 p} A_{b_1} \| x_{b_1} \|^2} \right] + o(1)\\
      =&  r_b \mathbb{E}\left[ \frac{ \sigma\left( x_{b_1}^\top \beta \right) }{\sigma\left(x_{b_1}^{T} \hat{\beta}_{S_{a}}\right)} C_1  x_{b_1}^\top \beta \cdot \frac{ \frac{2n}{n_b-2p} }{1 + \frac{2 n}{n_{b}-2 p}  \kappa} \right] - \mathbb{E}\left[ \frac{ \sigma\left(x_{b_1}^{\top} \beta\right) }{\sigma\left(x_{b_1}^{T} \hat{\beta}_{S_{a}}\right)} C_2  x_{b_1}^\top \hat{\beta}_{S_a} \cdot \frac{ \frac{2n}{n_b-2p} }{1 + \frac{2 n}{n_{b}-2 p} \kappa } \right] + o(1)\\
      =& \frac{\frac{2 n}{n_{b}-2 p} r_b}{1+\frac{2 n}{n_{b}-2 p} \kappa} \left\{ C_1 \mathbb{E}\left[ \frac{\sigma(Z_\beta)Z_\beta}{\sigma(Z_{\hat{\beta}_{S_a}})} \right] - C_2 \mathbb{E}\left[\frac{\sigma(Z_\beta) Z_{\hat{\beta}_{S_a}}}{\sigma\left(Z_{\hat{\beta}_{S_a}}\right)}\right] \right \}\\
      =&  2  C_1 \mathbb{E}\left[ \frac{\sigma(Z_\beta)Z_\beta}{\sigma(Z_{\hat{\beta}_{S_a}})} \right] -2 C_2 \mathbb{E}\left[\frac{\sigma(Z_\beta) Z_{\hat{\beta}_{S_a}}}{\sigma\left(Z_{\hat{\beta}_{S_a}}\right)}\right].
  \end{align*}

  We then show the variance converges to 0.

  Define $$T = \frac{1}{n}\left(\sum_{j \in S_{c}} \frac{A_{j} x_{j}}{\sigma\left(x_{j}^{T} \hat{\beta}_{S_{a}}\right)}\right)^{T}\left(\sum_{i \in S_{b}} A_{i} x_{i} x_{i}^{\top}\right)^{-1}\left(\sum_{i \in S_{b}} \frac{A_{i} x_{i}}{\sigma\left(x_{i}^{T} \hat{\beta}_{S_{a}}\right)}\right),$$ $$ T^{\{ -c \}}  = \frac{1}{n}\left(\sum_{j =2}^{ n_{c}} \frac{A_{j} x_{j}}{\sigma\left(x_{j}^{T} \hat{\beta}_{S_{a}}\right)}\right)^{T}\left(\sum_{i \in S_{b}} A_{i} x_{i} x_{i}^{\top}\right)^{-1}\left(\sum_{i \in S_{b}} \frac{A_{i} x_{i}}{\sigma\left(x_{i}^{T} \hat{\beta}_{S_{a}}\right)}\right), $$ $$T^{\{ -b\}} =  \frac{1}{n}\left(\sum_{j \in S_{c}} \frac{A_{j} x_{j}}{\sigma\left(x_{j}^{T} \hat{\beta}_{S_{a}}\right)}\right)^{T}\left(\sum_{i =2}^{ n_{b}} A_{i} x_{i} x_{i}^{\top}\right)^{-1}\left(\sum_{i=2}^{ n_{b}} \frac{A_{i} x_{i}}{\sigma\left(x_{i}^{T} \hat{\beta}_{S_{a}}\right)}\right). $$
  
  Then by Efron-Stein, we only need to show $$ \mathbb{E}\left[\left(T-T^{\{-b\}}\right)^{2}+\left(T-T^{\{-c\}}\right)^{2}\right]=o\left(n^{-1}\right). $$
  
  We first show $$ \mathbb{E}\left[\left(T-T^{\{-c\}}\right)^{2}\right]=o\left(n^{-1}\right) .$$
  
  We have \begin{align*}
      & T-T^{\{-c\}} = \frac{1}{n} \frac{A_{c_1} }{\sigma\left(x_{c_1}^{T} \hat{\beta}_{S_{a}}\right)}x_{c_1}^{T}\left(\sum_{i \in S_{b}} A_{i} x_{i} x_{i}^{\top}\right)^{-1}\left(\sum_{i \in S_{b}} \frac{A_{i} x_{i}}{\sigma\left(x_{i}^{T} \hat{\beta}_{S_{a}}\right)}\right).
  \end{align*} Thus \begin{align*}
      & \mathbb{E}\left[\left(T-T^{\{-c\}}\right)^{2}\right] = \mathbb{E}\left[ \left( \frac{1}{n} \frac{A_{c_1}}{\sigma\left(x_{c_1}^{T} \hat{\beta}_{S_{a}}\right)} x_{c_1}^{T}\left(\sum_{i \in S_{b}} A_{i} x_{i} x_{i}^{\top}\right)^{-1}\left(\sum_{i \in S_{b}} \frac{A_{i} x_{i}}{\sigma\left(x_{i}^{T} \hat{\beta}_{S_{a}}\right)}\right) \right)^2 \right]\\
      \leq &  \frac{1}{n}\sqrt{ \mathbb{E}\left[ \frac{A_{c_1}}{\sigma^4\left(x_{c_1}^{T} \hat{\beta}_{S_{a}}\right)} \right] } \cdot \sqrt{\mathbb{E}\left[ \left( \frac{1}{\sqrt{n}} x_{c_1}^{T}\left(\sum_{i \in S_{b}} A_{i} x_{i} x_{i}^{\top}\right)^{-1}\left(\sum_{i \in S_{b}} \frac{A_{i} x_{i}}{\sigma\left(x_{i}^{T} \hat{\beta}_{S_{a}}\right)}\right) \right)^4 \right]}.
  \end{align*}
  
    Thus we only need to show $$ \frac{1}{\sqrt{n}} x_{c_1}^{T}\left(\sum_{i \in S_{b}} A_{i} x_{i} x_{i}^{\top}\right)^{-1}\left(\sum_{i \in S_{b}} \frac{A_{i} x_{i}}{\sigma\left(x_{i}^{T} \hat{\beta}_{S_{a}}\right)}\right) \rightarrow 0 \quad \text{a.s.} $$
    
    We know the expectation of the left hand side is 0 due to symmetry, and the variance also goes to 0 (by law of iterated expectation and conditioning on $S_b$). Thus we have shown $$\mathbb{E}\left[\left(T-T^{\{-c\}}\right)^{2}\right]=o\left(n^{-1}\right).$$
    
    We then show $$\mathbb{E}\left[\left(T-T^{\{-b\}}\right)^{2}\right]=o\left(n^{-1}\right).$$
    
    We have \begin{align*}
        & T - T^{\{-b\}} = \frac{1}{n}\left(\sum_{j \in S_{c}} \frac{A_{j} x_{j}}{\sigma\left(x_{j}^{T} \hat{\beta}_{S_{a}}\right)}\right)^{T}\left(\sum_{i \in S_{b}} A_{i} x_{i} x_{i}^{\top}\right)^{-1} \frac{A_{b_1} x_{b_1}}{\sigma\left(x_{b_1}^{T} \hat{\beta}_{S_{a}}\right)}\\
        -& \frac{1}{n}\left(\sum_{j \in S_{c}} \frac{A_{j} x_{j}}{\sigma\left(x_{j}^{T} \hat{\beta}_{S_{a}}\right)}\right)^{T}  \frac{\left(\sum_{i =2}^{ n_{b}} A_{i} x_{i} x_{i}^{\top}\right)^{-1} A_{b_1} x_{b_1}  x_{b_1}^\top \left(\sum_{i =2}^{ n_{b}} A_{i} x_{i} x_{i}^{\top}\right)^{-1}}{1 +A_{b_1} x_{b_1}^\top \left(\sum_{i =2}^{ n_{b}} A_{i} x_{i} x_{i}^{\top}\right)^{-1} x_{b_1}}   \left(\sum_{i=2}^{n_{b}} \frac{A_{i} x_{i}}{\sigma\left(x_{i}^{T} \hat{\beta}_{S_{a}}\right)}\right).
    \end{align*}
    
    We know \begin{align*}
        & \mathbb{E}\left[\left(T-T^{\{-b\}}\right)^{2}\right] \leq 2 \mathbb{E}\left[ \left( \frac{1}{n}\left(\sum_{j \in S_{c}} \frac{A_{j} x_{j}}{\sigma\left(x_{j}^{T} \hat{\beta}_{S_{a}}\right)}\right)^{T}\left(\sum_{i \in S_{b}} A_{i} x_{i} x_{i}^{\top}\right)^{-1} \frac{A_{b_1} x_{b_1}}{\sigma\left(x_{b_1}^{T} \hat{\beta}_{S_{a}}\right)}   \right)^2 \right] \\
        +& 2 \mathbb{E}\left[ \left( \frac{1}{n}\left(\sum_{j \in S_{c}} \frac{A_{j} x_{j}}{\sigma\left(x_{j}^{T} \hat{\beta}_{S_{a}}\right)}\right)^{T} \frac{\left(\sum_{i=2}^{n_{b}} A_{i} x_{i} x_{i}^{\top}\right)^{-1} A_{b_1} x_{b_1} x_{b_1}^{\top}\left(\sum_{i=2}^{n_{b}} A_{i} x_{i} x_{i}^{\top}\right)^{-1}}{1+A_{b_1} x_{b_1}^{\top}\left(\sum_{i=2}^{n_{b}} A_{i} x_{i} x_{i}^{\top}\right)^{-1} x_{b_1}}\right.\right.\\
        \cdot & \left.\left. \left(\sum_{i=2}^{n_{b}} \frac{A_{i} x_{i}}{\sigma\left(x_{i}^{T} \hat{\beta}_{S_{a}}\right)}\right) \right)^2 \right].
    \end{align*}
  
  We can apply Cauchy-Schwarz to show the second term is $o(n^{-1})$. Thus it remains to show $$ \mathbb{E}\left[\left(\frac{1}{\sqrt{n}}\left(\sum_{j \in S_{c}} \frac{A_{j} x_{j}}{\sigma\left(x_{j}^{T} \hat{\beta}_{S_{a}}\right)}\right)^{T}\left(\sum_{i \in S_{b}} A_{i} x_{i} x_{i}^{\top}\right)^{-1} \frac{A_{b_1} x_{b_1}}{\sigma\left(x_{b_1}^{T} \hat{\beta}_{S_{a}}\right)}\right)^{2}\right] = o(1). $$
  
 Since \begin{align*}
      & \frac{1}{\sqrt{n}}\left(\sum_{j \in S_{c}} \frac{A_{j} x_{j}}{\sigma\left(x_{j}^{T} \hat{\beta}_{S_{a}}\right)}\right)^{T}\left(\sum_{i \in S_{b}} A_{i} x_{i} x_{i}^{\top}\right)^{-1} \frac{A_{b_1} x_{b_1}}{\sigma\left(x_{b_1}^{T} \hat{\beta}_{S_{a}}\right)} \\
      =& \frac{1}{\sqrt{n}}\left(\sum_{j \in S_{c}} \frac{A_{j} x_{j}}{\sigma\left(x_{j}^{T} \hat{\beta}_{S_{a}}\right)}\right)^{T}\frac{\left(\sum_{i=2}^{ n_{b}} A_{i} x_{i} x_{i}^{\top}\right)^{-1}}{1 +A_{b_1}x_{b_1}^\top \left(\sum_{i=2}^{ n_{b}} A_{i} x_{i} x_{i}^{\top}\right)^{-1} x_{b_1}} \frac{A_{b_1} x_{b_1}}{\sigma\left(x_{b_1}^{T} \hat{\beta}_{S_{a}}\right)}\\
      =& \frac{1}{\sqrt{n}}\left(\sum_{j \in S_{c}} \frac{A_{j} x_{j}}{\sigma\left(x_{j}^{T} \hat{\beta}_{S_{a}}\right)}\right)^{T}\frac{C_3}{1 +A_{b_1}C_3 \|x_{b_1}\|^2} \frac{A_{b_1} x_{b_1}}{\sigma\left(x_{{b_1}}^{T} \hat{\beta}_{S_{a}}\right)} + o(1) \quad \text{ for some constant }C_3\\
      =& \frac{1}{\sqrt{n}}\left(\sum_{j \in S_{c}} \frac{A_{j} x_{j}}{\sigma\left(x_{j}^{T} \hat{\beta}_{S_{a}}\right)}\right)^{T} x_{b_1} \frac{A_{b_1}C_3 }{ 1+ A_{b_1}C_3 \kappa} \frac{1}{\sigma\left(x_{b_1}^{T} \hat{\beta}_{S_{a}}\right)} + o(1).
  \end{align*}
  
  Thus by Cauchy-Schwarz, \begin{align*}
      & \mathbb{E}\left[\left(\frac{1}{\sqrt{n}}\left(\sum_{j \in S_{c}} \frac{A_{j} x_{j}}{\sigma\left(x_{j}^{T} \hat{\beta}_{S_{a}}\right)}\right)^{T}\left(\sum_{i \in S_{b}} A_{i} x_{i} x_{i}^{\top}\right)^{-1} \frac{A_{b_1} x_{b_1}}{\sigma\left(x_{b_1}^{T} \hat{\beta}_{S_{a}}\right)}\right)^{2}\right]\\
      =& \mathbb{E}\left[ \left( \frac{1}{\sqrt{n}}\left(\sum_{j \in S_{c}} \frac{A_{j} x_{j}}{\sigma\left(x_{j}^{T} \hat{\beta}_{S_{a}}\right)}\right)^{T} x_{b_1} \frac{A_{b_1} C_3}{1+A_{b_1} C_3 \kappa} \frac{1}{\sigma\left(x_{b_1}^{T} \hat{\beta}_{S_{a}}\right)} \right)^2 \right]\\
      \leq & \sqrt{\mathbb{E}\left[\left(\frac{1}{\sqrt{n}}\left(\sum_{j \in S_{c}} \frac{A_{j} x_{j}}{\sigma\left(x_{j}^{T} \hat{\beta}_{S_{a}}\right)}\right)^{T} x_{b_1}\right)^4 \right] \cdot \mathbb{E}\left[ \left( \frac{A_{b_1} C_3}{1+A_{b_1} C_3 \kappa} \frac{1}{\sigma\left(x_{b_1}^{T} \hat{\beta}_{S_{a}}\right)} \right)^4 \right] }=o(1).
  \end{align*}
  
  Thus we have completed the proof for the lemma.

\subsection{Proof of Lemma \ref{lemma 26} } 
First we find the expectation of the left hand side. We have \begin{align*}
    & \mathbb{E}\left[ \frac{1}{n}\left(\sum_{j \in S_{c}} \frac{A_{j} x_{j}}{\sigma\left(x_{j}^{T} \hat{\beta}_{S_{a}}\right)}- x_j\right)^{T}\left(\sum_{i \in S_{b}} A_{i} x_{i} x_{i}^{\top}\right)^{-1}\left(\sum_{i \in S_{b}} \frac{A_{i} x_{i}}{\sigma\left(x_{i}^{T} \hat{\beta}_{S_{c}}\right)}\right) \right]\\
    =& n_b r_c \mathbb{E}\left[  \left( \frac{A_{c_1} }{\sigma\left(x_{c_1}^{T} \hat{\beta}_{S_{a}}\right)}-1\right)x_{c_1}^{T}\left(\sum_{i \in S_{b}} A_{i} x_{i} x_{i}^{\top}\right)^{-1} \frac{A_{b_1} x_{b_1}}{\sigma\left(x_{b_1}^{T} \hat{\beta}_{S_{c}}\right)} \right]\\
    =&  n_b r_c \mathbb{E}\left[   \left( \frac{A_{c_1} }{\sigma\left(x_{c_1}^{T} \hat{\beta}_{S_{a}}\right)}-1\right)x_{c_1}^{T}\left(\sum_{i =2}^{n_{b}} A_{i} x_{i} x_{i}^{\top}\right)^{-1} \frac{A_{b_1} x_{b_1}}{\sigma\left(x_{b_1}^{T} \hat{\beta}_{S_{c}}\right)} \right.\\
    \cdot& \left. \frac{1}{1 + A_{b_{1}} x_{b_{1}}^\top \left(\sum_{i=2}^{n_{b}} A_{i} x_{i} x_{i}^{\top}\right)^{-1}  x_{b_1}} \right]\\
    =&  n_b r_c \mathbb{E}\left[   \left( \frac{A_{c_1} }{\sigma\left(x_{c_1}^{T} \hat{\beta}_{S_{a}}\right)}-1\right)x_{c_1}^{T}\left(\sum_{i =2}^{n_{b}} A_{i} x_{i} x_{i}^{\top}\right)^{-1} \frac{\sigma\left( x_{b_1}^\top \beta \right) x_{b_1}}{\sigma\left(x_{b_1}^{T} \hat{\beta}_{S_{c}}\right)} \right.\\
    \cdot& \left. \frac{1}{1 + x_{b_{1}}^\top \left(\sum_{i=2}^{n_{b}} A_{i} x_{i} x_{i}^{\top}\right)^{-1}  x_{b_1}} \right]\\
    =& n_b r_c \frac{2 n}{n_{b}-2 p}  \frac{1}{1 + \frac{2 n \kappa }{n_{b}-2 p} } \mathbb{E}\left[   \left( \frac{A_{c_1} }{\sigma\left(x_{c_1}^{T} \hat{\beta}_{S_{a}}\right)}-1\right) \frac{\sigma\left( x_{b_1}^\top \beta \right) }{\sigma\left(x_{b_1}^{T} \hat{\beta}_{S_{c}}\right)} x_{c_1}^{T}x_{b_1}  \right] + o(1)\\
    =& 2nr_c \mathbb{E}\left[\left(\frac{A_{c_{1}}}{\sigma\left(x_{c_{1}}^{T} \hat{\beta}_{S_{a}}\right)}-1\right) \frac{\sigma\left(x_{b_{1}}^{\top} \beta\right)}{\sigma\left(x_{b_{1}}^{T} \hat{\beta}_{S_{c}}\right)} x_{c_{1}}^{T} x_{b_{1}}\right]+o(1).
\end{align*}

  Note that \begin{align*}
      & \mathbb{E}\left[   \left( \frac{A_{c_1} }{\sigma\left(x_{c_1}^{T} \hat{\beta}_{S_{a}}\right)}-1\right) \frac{\sigma\left( x_{b_1}^\top \beta \right) }{\sigma\left(x_{b_1}^{T} \hat{\beta}_{S_{c}}\right)} x_{c_1}^{T}x_{b_1}  \right]\\
      =& \mathbb{E}\left[  \mathbb{E}\left[   \left( \frac{A_{c_1} }{\sigma\left(x_{c_1}^{T} \hat{\beta}_{S_{a}}\right)}-1\right) \frac{\sigma\left( x_{b_1}^\top \beta \right) }{\sigma\left(x_{b_1}^{T} \hat{\beta}_{S_{c}}\right)} x_{c_1}^{T}x_{b_1} \mid S_c, S_a \right]\right]\\
      =& \mathbb{E}\left[  \left( \frac{A_{c_1} }{\sigma\left(x_{c_1}^{T} \hat{\beta}_{S_{a}}\right)}-1\right) x_{c_1}^\top \mathbb{E}\left[    \frac{\sigma\left( x_{b_1}^\top \beta \right) }{\sigma\left(x_{b_1}^{T} \hat{\beta}_{S_{c}}\right)} x_{b_1} \mid S_c \right]\right]\\
      =& \mathbb{E}\left[\left( \frac{A_{c_1} }{\sigma\left(x_{c_1}^{T} \hat{\beta}_{S_{a}}\right)}-1\right)x_{c_{1}}^{\top}  \left( \frac{1}{n} \mathbb{E}\left[\frac{\sigma^{\prime}\left(x_{b_1}^{T} \beta\right)}{\sigma\left(x_{b_1}^{T} \hat{\beta}_{S_{c}}\right)}\right] \beta-\frac{1}{n} \mathbb{E}\left[\frac{\sigma\left(x_{b_1}^{T} \beta\right) \sigma^{\prime}\left(x_{b_1}^{T} \hat{\beta}_{S_{c}}\right)}{\sigma^{2}\left(x_{b_1}^{T} \hat{\beta}_{S_{c}}\right)}\right] \hat{\beta}_{S_{c}}\right)  \right]
  \end{align*} where the last line is due to Stein's lemma.
  
  Thus \begin{align*}
      & \mathbb{E}\left[\frac{1}{n}\left(\sum_{i \in S_{c}} \frac{A_{i} x_{i}}{\sigma\left(x_{i}^{T} \hat{\beta}_{S_{a}}\right)} - x_i\right)^{T}\left(\sum_{i \in S_{b}} A_{i} x_{i} x_{i}^{\top}\right)^{-1}\left(\sum_{i \in S_{b}} \frac{A_{i} x_{i}}{\sigma\left(x_{i}^{T} \hat{\beta}_{S_{c}}\right)}\right)\right]\\
      =& 2  r_{c} \left\{ \mathbb{E}\left[\frac{\sigma^{\prime}\left(x_{b_{1}}^{T} \beta\right)}{\sigma\left(x_{b_{1}}^{T} \hat{\beta}_{S_{c}}\right)}\right] \mathbb{E}\left[ \left(\frac{A_{c_{1}}}{\sigma\left(x_{c_{1}}^{T} \hat{\beta}_{S_{a}}\right)}-1\right) x_{c_{1}}^{\top} \beta \right] \right. \\
      & \left. - \mathbb{E}\left[\frac{\sigma\left(x_{b_{1}}^{T} \beta\right) \sigma^{\prime}\left(x_{b_{1}}^{T} \hat{\beta}_{S_{c}}\right)}{\sigma^{2}\left(x_{b_{1}}^{T} \hat{\beta}_{S_{c}}\right)}\right] \mathbb{E}\left[ \left(\frac{A_{c_{1}}}{\sigma\left(x_{c_{1}}^{T} \hat{\beta}_{S_{a}}\right)}-1\right)x_{c_{1}}^{\top} \hat{\beta}_{S_{c}} \right] \right\} + o(1)\\
      =& 2 r_c \left\{ \mathbb{E}\left[\frac{\sigma^{\prime}\left(x_{b_{1}}^{T} \beta\right)}{\sigma\left(x_{b_{1}}^{T} \hat{\beta}_{S_{c}}\right)}\right] \mathbb{E}\left[ \frac{A_{c_{1}}}{\sigma\left(x_{c_{1}}^{T} \hat{\beta}_{S_{a}}\right)} x_{c_{1}}^{\top} \beta \right] \right. \\
      -& \left.  \mathbb{E}\left[\frac{\sigma\left(x_{b_{1}}^{T} \beta\right) \sigma^{\prime}\left(x_{b_{1}}^{T} \hat{\beta}_{S_{c}}\right)}{\sigma^{2}\left(x_{b_{1}}^{T} \hat{\beta}_{S_{c}}\right)}\right] \mathbb{E}\left[ \frac{A_{c_{1}}}{\sigma\left(x_{c_{1}}^{T} \hat{\beta}_{S_{a}}\right)}x_{c_{1}}^{\top} \hat{\beta}_{S_{c}} \right] \right\} + o(1).
  \end{align*}
  
  Moreover, by Lemma 21 and Theorem 13 in \cite{Sur14516}, we have \begin{align*}
      & \mathbb{E}\left[\left(\frac{A_{c_{1}}}{\sigma\left(x_{c_{1}}^{T} \hat{\beta}_{S_{a}}\right)}-1\right) x_{c_{1}}^{\top} \hat{\beta}_{S_{c}}\right] \nonumber \\
      =& \mathbb{E}\left[\left(\frac{A_{c_{1}}}{\sigma\left(x_{c_{1}}^{T} \hat{\beta}_{S_{a}}\right)}-1\right) x_{c_{1}}^{\top} \hat{\beta}_{S_{c}}^{\{-c_1\}}\right] + \mathbb{E}\left[\left(\frac{A_{c_{1}}}{\sigma\left(x_{c_{1}}^{T} \hat{\beta}_{S_{a}}\right)}-1\right) x_{c_{1}}^{\top} \left( \hat{\beta}_{S_c} -  \hat{\beta}_{S_{c}}^{\{-c_1\}}\right)\right]\\
      =& \mathbb{E}\left[\left(\frac{A_{c_{1}}}{\sigma\left(x_{c_{1}}^{T} \hat{\beta}_{S_{a}}\right)}-1\right)x_{c_{1}}^{\top} \hat{\beta}_{S_{c}}^{\{-c_1\}}\right] \\ 
      +& \mathbb{E}\left[\left(\frac{A_{c_{1}}}{\sigma\left(x_{c_{1}}^{T} \hat{\beta}_{S_{a}}\right)}-1\right) \lambda_c^*\left( A_{c_1} - \sigma \left( \text{prox}_{ \lambda_c^* \rho } \left( x_{c_1} \hat{\beta}_{S_c}^{\{-c_1\}} + q_{c_1}A_{c_1} \right)  \right)\right) \right] + o(1)\\
      =& \mathbb{E}\left[\left(\frac{\sigma(x_{c_1}^\top \beta)}{\sigma\left(x_{c_{1}}^{T} \hat{\beta}_{S_{a}}\right)}-1\right)x_{c_{1}}^{\top} \hat{\beta}_{S_{c}}^{\{-c_1\}}\right] \\
      +& \mathbb{E}\left[ \sigma(x_{c_1}^\top \beta)\left(\frac{1}{\sigma\left(x_{c_{1}}^{T} \hat{\beta}_{S_{a}}\right)}-1\right) \lambda_c^*\left( 1 - \sigma \left( \text{prox}_{\lambda_c^*\rho } \left( x_{c_1} \hat{\beta}_{S_c}^{\{-c_1\}} + q_{c_1} \right)  \right)\right) \right]\\
      +& \mathbb{E}\left[ \left(1 - \sigma(x_{c_1}^\top \beta) \right) \lambda_c^* \sigma\left(\operatorname{prox}_{\lambda_c^* \rho}\left(x_{c_{1}} \hat{\beta}_{S_{c}}^{\{-c_1\}}\right)\right) \right] +  o(1)\\
      =&   \mathbb{E}\left[ \sigma(Z_\beta)\left(\frac{1}{\sigma\left(Z_{\hat{\beta}_{S_{a}}}\right)}-1\right)  \lambda_c^*\left( 1 - \sigma \left( \text{prox}_{\lambda_c^*\rho } \left(Z_{\hat{\beta}_{S_c}} + \lambda_c^* \right)  \right)\right) \right]\\
      +&\mathbb{E}\left[\left(\frac{\sigma(Z_\beta)}{\sigma\left(Z_{\hat{\beta}_{S_{a}}}\right)}-1\right)Z_{\hat{\beta}_{S_{c}}}\right] + \mathbb{E}\left[ \left(1 - \sigma(Z_\beta) \right) \lambda_c^* \sigma\left(\operatorname{prox}_{\lambda_c^* \rho}\left(Z_{\hat{\beta}_{S_{c}}}\right)\right) \right] +  o(1)\\
      =& \mathbb{E}\left[\frac{\sigma(Z_\beta)}{\sigma\left(Z_{\hat{\beta}_{S_{a}}}\right)}Z_{\hat{\beta}_{S_{c}}}\right] + \mathbb{E}\left[ \sigma(Z_\beta)\left(\frac{1}{\sigma\left(Z_{\hat{\beta}_{S_{a}}}\right)}-1\right) \lambda_c^* \left( 1 - \sigma \left( \text{prox}_{\lambda_c^*\rho } \left(Z_{\hat{\beta}_{S_c}} + \lambda_c^* \right)  \right)\right) \right]\\
      +& \mathbb{E}\left[ \left(1 - \sigma(Z_\beta) \right) \lambda_c^* \sigma\left(\operatorname{prox}_{\lambda_c^* \rho}\left(Z_{\hat{\beta}_{S_{c}}}\right)\right) \right] +  o(1).
  \end{align*}
  
  Hence we have \begin{align*}
      & \mathbb{E}\left[\frac{1}{n}\left(\sum_{j \in S_{c}} \frac{A_{j} x_{j}}{\sigma\left(x_{j}^{T} \hat{\beta}_{S_{a}}\right)}-x_{j}\right)^{T}\left(\sum_{i \in S_{b}} A_{i} x_{i} x_{i}^{\top}\right)^{-1}\left(\sum_{i \in S_{b}} \frac{A_{i} x_{i}}{\sigma\left(x_{i}^{T} \hat{\beta}_{S_{c}}\right)}\right)\right]\\
      =&  2 r_c \mathbb{E}\left[\frac{\sigma^{\prime}\left(Z_\beta\right)}{\sigma\left(Z_{\hat{\beta}_{S_{c}}}\right)}\right] \mathbb{E}\left[\frac{\sigma(Z_\beta)}{\sigma\left(Z_{\hat{\beta}_{S_{a}}}\right)} Z_\beta\right]-   2 r_c \mathbb{E}\left[\frac{\sigma\left(Z_\beta\right) \sigma^{\prime}\left(   Z_{\hat{\beta}_{S_{c}}}\right)}{\sigma^{2}\left(Z_{\hat{\beta}_{S_{c}}}\right)}\right] \left\{ \mathbb{E}\left[\frac{\sigma\left(Z_{\beta}\right)}{\sigma\left(Z_{\hat{\beta}_{S_{a}}}\right)} Z_{\hat{\beta}_{S_{c}}}\right] \right. \\
      +&  \mathbb{E}\left[\sigma\left(Z_{\beta}\right)\left(\frac{1}{\sigma\left(Z_{\hat{\beta}_{S_{a}}}\right)}-1\right) \lambda_c^*\left(1-\sigma\left(\operatorname{prox}_{\lambda_c^* \rho}\left(Z_{\hat{\beta}_{S_{c}}+\lambda_c^*}\right)\right)\right)\right] \\
      +& \left. \mathbb{E}\left[\left(1-\sigma\left(Z_{\beta}\right)\right) \lambda_c^* \sigma\left(\operatorname{prox}_{\lambda_c^* \rho}\left(Z_{\hat{\beta}_{S_{c}}}\right)\right)\right] \right\}+ o(1).
  \end{align*}
  
  We then prove $$ \text{Var}\left(\frac{1}{n}\left(\sum_{j \in S_{c}} \frac{A_{j} x_{j}}{\sigma\left(x_{j}^{T} \hat{\beta}_{S_{a}}\right)} - x_j\right)^{T}\left(\sum_{i \in S_{b}} A_{i} x_{i} x_{i}^{\top}\right)^{-1}\left(\sum_{i \in S_{b}} \frac{A_{i} x_{i}}{\sigma\left(x_{i}^{T} \hat{\beta}_{S_{c}}\right)}\right) \right) \rightarrow 0.$$

  This is equivalent to show \begin{align*}
      & \mathbb{E}^2\left[ \frac{1}{n}\left(\sum_{j \in S_{c}} \frac{A_{j} x_{j}}{\sigma\left(x_{j}^{T} \hat{\beta}_{S_{a}}\right)}- x_j\right)^{T}\left(\sum_{i \in S_{b}} A_{i} x_{i} x_{i}^{\top}\right)^{-1}\left(\sum_{i \in S_{b}} \frac{A_{i} x_{i}}{\sigma\left(x_{i}^{T} \hat{\beta}_{S_{c}}\right)}\right) \right] \\
      =& \mathbb{E}\left[ \left( \frac{1}{n}\left(\sum_{j \in S_{c}} \frac{A_{j} x_{j}}{\sigma\left(x_{j}^{T} \hat{\beta}_{S_{a}}\right)}- x_j\right)^{T}\left(\sum_{i \in S_{b}} A_{i} x_{i} x_{i}^{\top}\right)^{-1}\left(\sum_{i \in S_{b}} \frac{A_{i} x_{i}}{\sigma\left(x_{i}^{T} \hat{\beta}_{S_{c}}\right)}\right) \right)^2 \right] + o(1).
  \end{align*}
  
  Note that \begin{align}
      & \mathbb{E}\left[ \left( \frac{1}{n}\left(\sum_{j \in S_{c}} \frac{A_{j} x_{j}}{\sigma\left(x_{j}^{T} \hat{\beta}_{S_{a}}\right)}- x_j\right)^{T}\left(\sum_{i \in S_{b}} A_{i} x_{i} x_{i}^{\top}\right)^{-1}\left(\sum_{i \in S_{b}} \frac{A_{i} x_{i}}{\sigma\left(x_{i}^{T} \hat{\beta}_{S_{c}}\right)}\right) \right)^2 \right] \nonumber \\
      =& r_b r_c \mathbb{E}\left[ \left( \left(\frac{A_{c_1} x_{c_1}^{T}}{\sigma\left(x_{c_1}^{T} \hat{\beta}_{S_{a}}\right)}-x_{c_1}\right)\left(\sum_{i \in S_{b}} A_{i} x_{i} x_{i}^{\top}\right)^{-1} \frac{A_{b_1} x_{b_1}}{\sigma\left(x_{b_1}^{T} \hat{\beta}_{S_{c}}\right)}\right)^2 \right] \nonumber\\
      +& r_b r_c (n_c-1) \mathbb{E}\Big[  \left(\frac{A_{c_1} x_{c_1}^{T}}{\sigma\left(x_{c_1}^{T} \hat{\beta}_{S_{a}}\right)}-x_{c_1}\right)\left(\sum_{i \in S_{b}} A_{i} x_{i} x_{i}^{\top}\right)^{-1} \frac{A_{b_1} x_{b_1}}{\sigma\left(x_{b_1}^{T} \hat{\beta}_{S_{c}}\right)}\cdot  \nonumber \\
      &\left(\frac{A_{c_2} x_{c_2}^{T}}{\sigma\left(x_{c_2}^{T} \hat{\beta}_{S_{a}}\right)}-x_{c_2}\right)\left(\sum_{i \in S_{b}} A_{i} x_{i} x_{i}^{\top}\right)^{-1} \frac{A_{b_1} x_{b_1}}{\sigma\left(x_{b_1}^{T} \hat{\beta}_{S_{c}}\right)} \Big] \nonumber\\
      +& r_b r_c (n_b-1) \mathbb{E}\Big[ \left(\frac{A_{c_1} x_{c_1}^{T}}{\sigma\left(x_{c_1}^{T} \hat{\beta}_{S_{a}}\right)}-x_{c_1}\right)\left(\sum_{i \in S_{b}} A_{i} x_{i} x_{i}^{\top}\right)^{-1} \frac{A_{b_1} x_{b_1}}{\sigma\left(x_{b_1}^{T} \hat{\beta}_{S_{c}}\right)} \cdot  \nonumber \\
      & \left(\frac{A_{c_1} x_{c_1}^{T}}{\sigma\left(x_{c_1}^{T} \hat{\beta}_{S_{a}}\right)}-x_{c_1}\right)\left(\sum_{i \in S_{b}} A_{i} x_{i} x_{i}^{\top}\right)^{-1} \frac{A_{b_2} x_{b_2}}{\sigma\left(x_{b_2}^{T} \hat{\beta}_{S_{c}}\right)} \Big]\nonumber \\
      +& r_br_c (n_b-1) (n_c-1) \mathbb{E}\Big[ \left(\frac{A_{c_1} x_{c_1}^{T}}{\sigma\left(x_{c_1}^{T} \hat{\beta}_{S_{a}}\right)}-x_{c_1}\right)\left(\sum_{i \in S_{b}} A_{i} x_{i} x_{i}^{\top}\right)^{-1} \frac{A_{b_1} x_{b_1}}{\sigma\left(x_{b_1}^{T} \hat{\beta}_{S_{c}}\right)}  \nonumber\\
      \cdot&  \left(\frac{A_{c_2} x_{c_2}^{T}}{\sigma\left(x_{c_2}^{T} \hat{\beta}_{S_{a}}\right)}-x_{c_2}\right)\left(\sum_{i \in S_{b}} A_{i} x_{i} x_{i}^{\top}\right)^{-1} \frac{A_{b_2} x_{b_2}}{\sigma\left(x_{b_2}^{T} \hat{\beta}_{S_{c}}\right)} \Big]. \label{last_eq_lemma}
  \end{align}

 First let's consider the third term in the expression above. By Sherman–Morrison formula we have 
 \begin{align*}
     & r_b r_c (n_b-1) \mathbb{E}\Big[ \left(\frac{A_{c_1} x_{c_1}^{T}}{\sigma\left(x_{c_1}^{T} \hat{\beta}_{S_{a}}\right)}-x_{c_1}\right)\left(\sum_{i \in S_{b}} A_{i} x_{i} x_{i}^{\top}\right)^{-1} \frac{A_{b_1} x_{b_1}}{\sigma\left(x_{b_1}^{T} \hat{\beta}_{S_{c}}\right)} \cdot \\
     &\left(\frac{A_{c_1} x_{c_1}^{T}}{\sigma\left(x_{c_1}^{T} \hat{\beta}_{S_{a}}\right)}-x_{c_1}\right)\left(\sum_{i \in S_{b}} A_{i} x_{i} x_{i}^{\top}\right)^{-1} \frac{A_{b_2} x_{b_2}}{\sigma\left(x_{b_2}^{T} \hat{\beta}_{S_{c}}\right)} \Big]\\
     =& r_b r_c (n_b-1) \mathbb{E}\left[ \left(\frac{A_{c_1} x_{c_1}^{T}}{\sigma\left(x_{c_1}^{T} \hat{\beta}_{S_{a}}\right)}-x_{c_1}\right)\frac{\left(\sum_{i \in S_{b}, i \neq 1} A_{i} x_{i} x_{i}^{\top}\right)^{-1}}{1 + A_{b_1}x_{b_1}^\top \left(\sum_{i \in S_{b}, i \neq 1} A_{i} x_{i} x_{i}^{\top}\right)^{-1} x_{b_1} } \frac{A_{b_1} x_{b_1}}{\sigma\left(x_{b_1}^{T} \hat{\beta}_{S_{c}}\right)} \right.\\
    \cdot& \left. \left(\frac{A_{c_1} x_{c_1}^{T}}{\sigma\left(x_{c_1}^{T} \hat{\beta}_{S_{a}}\right)}-x_{c_1}\right)\frac{\left(\sum_{i \in S_{b}, i \neq 2} A_{i} x_{i} x_{i}^{\top}\right)^{-1}}{1 + A_{b_2}x_{b_2}^\top \left(\sum_{i \in S_{b}, i \neq 2} A_{i} x_{i} x_{i}^{\top}\right)^{-1} x_{b_2} }  \frac{A_{b_2} x_{b_2}}{\sigma\left(x_{b_2}^{T} \hat{\beta}_{S_{c}}\right)} \right]\\
     =& r_b r_c (n_b-1) \left( \frac{2 n}{n_{b}-2 p} \right)^2 \mathbb{E}\left[ \left(\frac{A_{c_1} x_{c_1}^{T}}{\sigma\left(x_{c_1}^{T} \hat{\beta}_{S_{a}}\right)}-x_{c_1}\right)\frac{1}{1 + A_{b_1}\frac{2 \kappa n}{n_{b}-2 p} } \frac{A_{b_1} x_{b_1}}{\sigma\left(x_{b_1}^{T} \hat{\beta}_{S_{c}}\right)}  \right.\\
     \cdot& \left. \left(\frac{A_{c_1} x_{c_1}^{T}}{\sigma\left(x_{c_1}^{T} \hat{\beta}_{S_{a}}\right)}-x_{c_1}\right)\frac{1}{1 + A_{b_2}\frac{2 n \kappa}{n_{b}-2 p} }  \frac{A_{b_2} x_{b_2}}{\sigma\left(x_{b_2}^{T} \hat{\beta}_{S_{c}}\right)}\right] + o(1).
 \end{align*}
 
 Moreover, \begin{align*}
     & \mathbb{E}\left[\left(\frac{A_{c_{1}} x_{c_{1}}^{T}}{\sigma\left(x_{c_{1}}^{T} \hat{\beta}_{S_{a}}\right)}-x_{c_1}\right) \frac{1}{1+A_{b_{1}} \frac{2 \kappa n}{n_{b}-2 p}} \frac{A_{b_{1}} x_{b_{1}}}{\sigma\left(x_{b_{1}}^{T} \hat{\beta}_{S_{c}}\right)}\right.\\ \cdot& \left.\left(\frac{A_{c_{1}} x_{c_{1}}^{T}}{\sigma\left(x_{c_{1}}^{T} \hat{\beta}_{S_{a}}\right)}-x_{c_1}\right) \frac{1}{1+A_{b_{2}} \frac{2 n \kappa}{n_{b}-2 p}} \frac{A_{b_{2}} x_{b_{2}}}{\sigma\left(x_{b_{2}}^{T} \hat{\beta}_{S_{c}}\right)}\right]\\
     =& \mathbb{E}\left[\mathbb{E}\left[\left(\frac{A_{c_{1}} x_{c_{1}}^{T}}{\sigma\left(x_{c_{1}}^{T} \hat{\beta}_{S_{a}}\right)}-x_{c_1}\right) \frac{1}{1+A_{b_{1}} \frac{2 \kappa n}{n_{b}-2 p}} \frac{A_{b_{1}} x_{b_{1}}}{\sigma\left(x_{b_{1}}^{T} \hat{\beta}_{S_{c}}\right)} \right. \right.  \\ 
      \cdot& \left.\left. \left(\frac{A_{c_{1}} x_{c_{1}}^{T}}{\sigma\left(x_{c_{1}}^{T} \hat{\beta}_{S_{a}}\right)} -x_{c_1}\right) \frac{1}{1+A_{b_{2}} \frac{2 n \kappa}{n_{b}-2 p}} \frac{A_{b_{2}} x_{b_{2}}}{\sigma\left(x_{b_{2}}^{T} \hat{\beta}_{S_{c}}\right)} \mid S_c, \hat{\beta}_{S_a}\right]\right]\\
     =& \mathbb{E}\left[\mathbb{E}^2\left[\left(\frac{A_{c_{1}} x_{c_{1}}^{T}}{\sigma\left(x_{c_{1}}^{T} \hat{\beta}_{S_{a}}\right)} -x_{c_1}\right)\frac{1}{1+A_{b_{1}} \frac{2 \kappa n}{n_{b}-2 p}} \frac{A_{b_{1}} x_{b_{1}}}{\sigma\left(x_{b_{1}}^{T} \hat{\beta}_{S_{c}}\right)}\  \mid S_c, \hat{\beta}_{S_a}\right]\right]\\
     =& \mathbb{E}\left[\mathbb{E}^2\left[\left(\frac{A_{c_{1}} x_{c_{1}}^{T}}{\sigma\left(x_{c_{1}}^{T} \hat{\beta}_{S_{a}}\right)} -x_{c_1}\right)\frac{1}{1+ \frac{2 \kappa n}{n_{b}-2 p}} \frac{\sigma(x_{b_1}^\top \beta) x_{b_{1}}}{\sigma\left(x_{b_{1}}^{T} \hat{\beta}_{S_{c}}\right)}\  \mid S_c, \hat{\beta}_{S_a}\right]\right] = O(n^{-2}),
 \end{align*} where the last step is obtained by applying Stein's Lemma.
 
 Hence the third term in (\ref{last_eq_lemma}) is \begin{align*}
     &r_{b} r_{c}\left(n_{b}-1\right) \mathbb{E}\left[\left(\frac{A_{c_{1}} x_{c_{1}}^{T}}{\sigma\left(x_{c_{1}}^{T} \hat{\beta}_{S_{a}}\right)}-x_{c_1}\right)\left(\sum_{i \in S_{b}} A_{i} x_{i} x_{i}^{\top}\right)^{-1} \frac{A_{b_{1}} x_{b_{1}}}{\sigma\left(x_{b_{1}}^{T} \hat{\beta}_{S_{c}}\right)} \cdot \right. \\
    &\left.  \left(\frac{A_{c_{1}} x_{c_{1}}^{T}}{\sigma\left(x_{c_{1}}^{T} \hat{\beta}_{S_{a}}\right)}-x_{c_1}\right)\left(\sum_{i \in S_{b}} A_{i} x_{i} x_{i}^{\top}\right)^{-1} \frac{A_{b_{2}} x_{b_{2}}}{\sigma\left(x_{b_{2}}^{T} \hat{\beta}_{S_{c}}\right)}\right] = o(1).
      \end{align*}
 
We can obtain similar results for the second term in (\ref{last_eq_lemma}): \begin{align*}
    & r_{b} r_{c}(n_{c}-1) \mathbb{E}\Big[\left(\frac{A_{c_{1}} x_{c_{1}}^{T}}{\sigma\left(x_{c_{1}}^{T} \hat{\beta}_{S_{a}}\right)} -x_{c_1} \right) \left(\sum_{i \in S_{b}} A_{i} x_{i} x_{i}^{\top}\right)^{-1} \frac{A_{b_{1}} x_{b_{1}}}{\sigma\left(x_{b_{1}}^{T} \hat{\beta}_{S_{c}}\right)} \cdot  \\
     & \left(\frac{A_{c_{2}} x_{c_{2}}^{T}}{\sigma\left(x_{c_{2}}^{T} \hat{\beta}_{S_{a}}\right)}-x_{c_2} \right)\left(\sum_{i \in S_{b}} A_{i} x_{i} x_{i}^{\top}\right)^{-1} \frac{A_{b_{1}} x_{b_{1}}}{\sigma\left(x_{b_{1}}^{T} \hat{\beta}_{S_{c}}\right)}\Big] \\
    =&  r_{b} r_{c}\left(n_{c}-1\right) \left(\frac{2 n}{n_{b}-2 p}\right)^{2} \\
    \cdot&   \mathbb{E}\Big[\mathbb{E}\Big[\left(\frac{A_{c_{1}} x_{c_{1}}^{T}}{\sigma\left(x_{c_{1}}^{T} \hat{\beta}_{S_{a}}\right)} -x_{c_1} \right)\frac{1}{1+A_{b_{1}} \frac{2 \kappa n}{n_{b}-2 p}} \frac{A_{b_{1}} x_{b_{1}}}{\sigma\left(x_{b_{1}}^{T} \hat{\beta}_{S_{c}}\right)} \cdot \\
     & \left(\frac{A_{c_{2}} x_{c_{2}}^{T}}{\sigma\left(x_{c_{2}}^{T} \hat{\beta}_{S_{a}}\right)}-x_{c_2} \right) \frac{1}{1+A_{b_{1}} \frac{2 n \kappa}{n_{b}-2 p}} \frac{A_{b_{1}} x_{b_{1}}}{\sigma\left(x_{b_{1}}^{T} \hat{\beta}_{S_{c}}\right)} \mid S_{c}, \hat{\beta}_{S_{a}}\Big]\Big]\\
    =&  r_{b} r_{c}\left(n_{c}-1\right) \left(\frac{2 n}{n_{b}-2 p}\right)^{2} \left( \frac{1}{1+\frac{2 \kappa n}{n_{b}-2 p}}\right)^2 \cdot \\
    &\mathbb{E}\left[\mathbb{E}\left[\left(\frac{A_{c_{1}} x_{c_{1}}^{T}}{\sigma\left(x_{c_{1}}^{T} \hat{\beta}_{S_{a}}\right)} - x_{c_1}\right) \frac{ \sigma(x_{b_1}^\top \beta) x_{b_{1}}}{\sigma^2\left(x_{b_{1}}^{T} \hat{\beta}_{S_{c}}\right)} \cdot\left(\frac{A_{c_{2}} x_{c_{2}}^{T}}{\sigma\left(x_{c_{2}}^{T} \hat{\beta}_{S_{a}}\right)} - x_{c_2} \right) { x_{b_{1}}} \mid S_{c}, \hat{\beta}_{S_{a}}\right]\right]\\
    =& r_{b} r_{c}\left(n_{c}-1\right) \left(\frac{2 n}{n_{b}-2 p}\right)^{2} \left( \frac{1}{1+\frac{2 \kappa n}{n_{b}-2 p}}\right)^2  \cdot \\
      &\mathbb{E}\left[\left(\frac{A_{c_{1}} x_{c_{1}}^{T}}{\sigma\left(x_{c_{1}}^{T} \hat{\beta}_{S_{a}}\right)}- x_{c_1} \right)  \frac{ \sigma(x_{b_1}^\top \beta) x_{b_{1}}}{\sigma^2\left(x_{b_{1}}^{T} \hat{\beta}_{S_{c}}\right)} \cdot\left(\frac{A_{c_{2}} x_{c_{2}}^{T}}{\sigma\left(x_{c_{2}}^{T} \hat{\beta}_{S_{a}}\right)} - x_{c_2} \right) { x_{b_{1}}} \right]\\
    =& r_{b} r_{c}\left(n_{c}-1\right) \left(\frac{2 n}{n_{b}-2 p}\right)^{2} \left( \frac{1}{1+\frac{2 \kappa n}{n_{b}-2 p}}\right)^2 \cdot  \\
    &\mathbb{E}\left[\left(\frac{A_{c_{1}} x_{c_{1}}^{T}}{\sigma\left(x_{c_{1}}^{T} \hat{\beta}_{S_{a}}\right)}  - x_{c_1} \right) \frac{ \sigma(x_{b_1}^\top \beta) x_{b_{1}}}{\sigma^2\left(x_{b_{1}}^{T} \hat{\beta}_{S_{c}}^{\{-c_1,-c_2\}}\right)} \cdot\left(\frac{A_{c_{2}} x_{c_{2}}^{T}}{\sigma\left(x_{c_{2}}^{T} \hat{\beta}_{S_{a}}\right)} - x_{c_2} \right) { x_{b_{1}}} \right] + o(1).
\end{align*}

The last step is obtained by applying Lemma 21 in \cite{Sur14516}, and then applying Stein's Lemma after conditioning on $S_b, \hat{\beta}_{S_a}$.
 Moreover, we know \begin{align*}
     &  \mathbb{E}\left[\left(\frac{A_{c_{1}} x_{c_{1}}^{T}}{\sigma\left(x_{c_{1}}^{T} \hat{\beta}_{S_{a}}\right)}  - x_{c_1} \right)\frac{ \sigma(x_{b_1}^\top \beta) x_{b_{1}}}{\sigma^2\left(x_{b_{1}}^{T} \hat{\beta}_{S_{c}}^{\{-c_1,-c_2\}}\right)} \cdot\left(\frac{A_{c_{2}} x_{c_{2}}^{T}}{\sigma\left(x_{c_{2}}^{T} \hat{\beta}_{S_{a}}\right)}  - x_{c_2} \right){ x_{b_{1}}} \right]\\
     =& \mathbb{E}\left[  \mathbb{E}\left[\left(\frac{A_{c_{1}} x_{c_{1}}^{T}}{\sigma\left(x_{c_{1}}^{T} \hat{\beta}_{S_{a}}\right)} - x_{c_1} \right)  \frac{ \sigma(x_{b_1}^\top \beta) x_{b_{1}}}{\sigma^2\left(x_{b_{1}}^{T} \hat{\beta}_{S_{c}}^{\{-c_1,-c_2\}}\right)}\mid S_b, \hat{\beta}_{S_a} \right] \right. \\
     \cdot & \left.  \mathbb{E}\left[ \left(\frac{A_{c_{2}} x_{c_{2}}^{T}}{\sigma\left(x_{c_{2}}^{T} \hat{\beta}_{S_{a}}\right)} - x_{c_2} \right) { x_{b_{1}}} \mid S_b, \hat{\beta}_{S_a} \right]\right] = O(n^{-2}),
 \end{align*} where the last step is again obtained by applying Stein's Lemma.
 
 Thus we have shown the second term \begin{align*}
     & r_{b} r_{c}\left(n_{c}-1\right) \mathbb{E}\left[\left(\frac{A_{c_{1}} x_{c_{1}}^{T}}{\sigma\left(x_{c_{1}}^{T} \hat{\beta}_{S_{a}}\right)}-x_{c_{1}}\right)\left(\sum_{i \in S_{b}} A_{i} x_{i} x_{i}^{\top}\right)^{-1} \frac{A_{b_{1}} x_{b_{1}}}{\sigma\left(x_{b_{1}}^{T} \hat{\beta}_{S_{c}}\right)} \cdot\right. \\
     &\left. \left(\frac{A_{c_{2}} x_{c_{2}}^{T}}{\sigma\left(x_{c_{2}}^{T} \hat{\beta}_{S_{a}}\right)}-x_{c_{2}}\right)\left(\sum_{i \in S_{b}} A_{i} x_{i} x_{i}^{\top}\right)^{-1} \cdot  \frac{A_{b_{1}} x_{b_{1}}}{\sigma\left(x_{b_{1}}^{T} \hat{\beta}_{S_{c}}\right)}\right] =  o(1).
 \end{align*}
 
 Similarly we can show the first term in (\ref{last_eq_lemma}) is $o(1)$. So it remains to consider the last term: \begin{align*}
&r_{b} r_{c}\left(n_{b}-1\right)\left(n_{c}-1\right) \mathbb{E}\left[\left(\frac{A_{c_{1}} x_{c_{1}}^{T}}{\sigma\left(x_{c_{1}}^{T} \hat{\beta}_{S_{a}}\right)}-x_{c_{1}}\right)\left(\sum_{i \in S_{b}} A_{i} x_{i} x_{i}^{\top}\right)^{-1} \frac{A_{b_{1}} x_{b_{1}}}{\sigma\left(x_{b_{1}}^{T} \hat{\beta}_{S_{c}}\right)}\right. \\
&\left.\left(\frac{A_{c_{2}} x_{c_{2}}^{T}}{\sigma\left(x_{c_{2}}^{T} \hat{\beta}_{S_{a}}\right)}-x_{c_{2}}\right)\left(\sum_{i \in S_{b}} A_{i} x_{i} x_{i}^{\top}\right)^{-1} \frac{A_{b_{2}} x_{b_{2}}}{\sigma\left(x_{b_{2}}^{T} \hat{\beta}_{S_{c}}\right)}\right]\\
=& r_{b} r_{c}\left(n_{b}-1\right)\left(n_{c}-1\right) \left(\frac{2 n}{n_{b}-2 p}\right)^{2}  \left(\frac{1}{1+\frac{2 \kappa n}{n_{b}-2 p}}\right)^2\\
     \cdot & \mathbb{E}\left[\mathbb{E}\left[ \left(\frac{A_{c_{1}} x_{c_{1}}^{T}}{\sigma\left(x_{c_{1}}^{T} \hat{\beta}_{S_{a}}\right)} - x_{c_1}\right) \frac{\sigma\left(x_{b_{1}}^{\top} \beta\right) x_{b_{1}}}{\sigma\left(x_{b_{1}}^{T} \hat{\beta}_{S_{c}}\right)} \mid S_{c}, \hat{\beta}_{S_{a}}\right]\right.\\
     \cdot&  \left. \mathbb{E}\left[\left(\frac{A_{c_{2}} x_{c_{2}}^{T}}{\sigma\left(x_{c_{2}}^{T} \hat{\beta}_{S_{a}}\right)} - x_{c_2}\right) \frac{\sigma\left(x_{b_{2}}^{\top} \beta\right) x_{b_{2}}}{\sigma\left(x_{b_{2}}^{T} \hat{\beta}_{S_{c}}\right)} \mid S_{c}, \hat{\beta}_{S_{a}}\right]\right]\\
     =&  \left(4 r_c^2 n^2 + o(1) \right)\mathbb{E}\Big[\mathbb{E}\left[ \left(\frac{A_{c_{1}} x_{c_{1}}^{T}}{\sigma\left(x_{c_{1}}^{T} \hat{\beta}_{S_{a}}\right)} - x_{c_1}\right) \frac{\sigma\left(x_{b_{1}}^{\top} \beta\right) x_{b_{1}}}{\sigma\left(x_{b_{1}}^{T} \hat{\beta}_{S_{c}}\right)} \mid S_{c}, \hat{\beta}_{S_{a}}\right] \cdot\\
     &\mathbb{E}\Big[\left(\frac{A_{c_{2}} x_{c_{2}}^{T}}{\sigma\left(x_{c_{2}}^{T} \hat{\beta}_{S_{a}}\right)} - x_{c_2}\right) \frac{\sigma\left(x_{b_{2}}^{\top} \beta\right) x_{b_{2}}}{\sigma\left(x_{b_{2}}^{T} \hat{\beta}_{S_{c}}\right)} \mid S_{c}, \hat{\beta}_{S_{a}}\Big]\Big]\\
     =& 4r_c^2 \mathbb{E}\left[ \left(\frac{A_{c_{1}}}{\sigma\left(x_{c_{1}}^{T} \hat{\beta}_{S_{a}}\right)}-1\right) x_{c_{1}}^{\top}\left( \mathbb{E}\left[\frac{\sigma^{\prime}\left(x_{b_{1}}^{T} \beta\right)}{\sigma\left(x_{b_{1}}^{T} \hat{\beta}_{S_{c}}\right)}\right] \beta- \mathbb{E}\left[\frac{\sigma\left(x_{b_{1}}^{T} \beta\right) \sigma^{\prime}\left(x_{b_{1}}^{T} \hat{\beta}_{S_{c}}\right)}{\sigma^{2}\left(x_{b_{1}}^{T} \hat{\beta}_{S_{c}}\right)}\right] \hat{\beta}_{S_{c}}\right) \right.\\
     \cdot& \left. \left(\frac{A_{c_{2}}}{\sigma\left(x_{c_{2}}^{T} \hat{\beta}_{S_{a}}\right)}-1\right) x_{c_{2}}^{\top}\left( \mathbb{E}\left[\frac{\sigma^{\prime}\left(x_{b_{2}}^{T} \beta\right)}{\sigma\left(x_{b_{2}}^{T} \hat{\beta}_{S_{c}}\right)}\right] \beta- \mathbb{E}\left[\frac{\sigma\left(x_{b_{2}}^{T} \beta\right) \sigma^{\prime}\left(x_{b_{2}}^{T} \hat{\beta}_{S_{c}}\right)}{\sigma^{2}\left(x_{b_{2}}^{T} \hat{\beta}_{S_{c}}\right)}\right] \hat{\beta}_{S_{c}}\right)  \right] + o(1) \\
     =& 4r_c^2\mathbb{E}\left[\mathbb{E}\left[ \left(\frac{A_{c_{1}}}{\sigma\left(x_{c_{1}}^{T} \hat{\beta}_{S_{a}}\right)}-1\right) x_{c_{1}}^{\top}\left( \mathbb{E}\left[\frac{\sigma^{\prime}\left(x_{b_{1}}^{T} \beta\right)}{\sigma\left(x_{b_{1}}^{T} \hat{\beta}_{S_{c}}\right)}\right] \beta- \mathbb{E}\left[\frac{\sigma\left(x_{b_{1}}^{T} \beta\right) \sigma^{\prime}\left(x_{b_{1}}^{T} \hat{\beta}_{S_{c}}\right)}{\sigma^{2}\left(x_{b_{1}}^{T} \hat{\beta}_{S_{c}}\right)}\right] \hat{\beta}_{S_{c}}\right) \right. \right.\\
     \cdot& \left.\left. \left(\frac{A_{c_{2}}}{\sigma\left(x_{c_{2}}^{T} \hat{\beta}_{S_{a}}\right)}-1\right) x_{c_{2}}^{\top}\left( \mathbb{E}\left[\frac{\sigma^{\prime}\left(x_{b_{2}}^{T} \beta\right)}{\sigma\left(x_{b_{2}}^{T} \hat{\beta}_{S_{c}}\right)}\right] \beta- \mathbb{E}\left[\frac{\sigma\left(x_{b_{2}}^{T} \beta\right) \sigma^{\prime}\left(x_{b_{2}}^{T} \hat{\beta}_{S_{c}}\right)}{\sigma^{2}\left(x_{b_{2}}^{T} \hat{\beta}_{S_{c}}\right)}\right] \hat{\beta}_{S_{c}}\right)  \mid \hat{\beta}_{S_{c}}^{\{-c_1,-c_2\}}\right]\right] \\
     +& o(1) .
 \end{align*}

Define $$t_1 =    \left(\frac{A_{c_{1}}}{\sigma\left(x_{c_{1}}^{T} \hat{\beta}_{S_{a}}\right)}-1\right) x_{c_{1}}^{\top}\left(\mathbb{E}\left[\frac{\sigma^{\prime}\left(x_{b_{1}}^{T} \beta\right)}{\sigma\left(x_{b_{1}}^{T} \hat{\beta}_{S_{c}}\right)}\right] \beta-\mathbb{E}\left[\frac{\sigma\left(x_{b_{1}}^{T} \beta\right) \sigma^{\prime}\left(x_{b_{1}}^{T} \hat{\beta}_{S_{c}}\right)}{\sigma^{2}\left(x_{b_{1}}^{T} \hat{\beta}_{S_{c}}\right)}\right] \hat{\beta}_{S_{c}}\right) ,$$ $$ t_2 =  \left(\frac{A_{c_{2}}}{\sigma\left(x_{c_{2}}^{T} \hat{\beta}_{S_{a}}\right)}-1\right) x_{c_{2}}^{\top}\left(\mathbb{E}\left[\frac{\sigma^{\prime}\left(x_{b_{2}}^{T} \beta\right)}{\sigma\left(x_{b_{2}}^{T} \hat{\beta}_{S_{c}}\right)}\right] \beta-\mathbb{E}\left[\frac{\sigma\left(x_{b_{2}}^{T} \beta\right) \sigma^{\prime}\left(x_{b_{2}}^{T} \hat{\beta}_{S_{c}}\right)}{\sigma^{2}\left(x_{b_{2}}^{T} \hat{\beta}_{S_{c}}\right)}\right] \hat{\beta}_{S_{c}}\right)  .$$
 
    By a leave-two-out argument, we know $t_1$ and $t_2$ are independent in the limit, given $\hat{\beta}_{S_{c}}^{\left\{-c_{1},-c_{2}\right\}} $. Further, the limits of $\mathbb{E}\left[ t_1 \mid \hat{\beta}_{S_{c}}^{\left\{-c_{1},-c_{2}\right\}} \right]$ and $\mathbb{E}\left[ t_2 \mid \hat{\beta}_{S_{c}}^{\left\{-c_{1},-c_{2}\right\}} \right]$ do not depend on $\hat{\beta}_{S_c}^{\{-c_1,-c_2\}}$. Therefore \begin{align*}
        & 4r_c^2 \mathbb{E}\left[ \mathbb{E}\left[ t_1 t_2  \mid \hat{\beta}_{S_{c}}^{\left\{-c_{1},-c_{2}\right\}}  \right] \right] = \left(2 r_{c} \mathbb{E}\left[t_{1}\right]\right)^{2} + o(1),
    \end{align*} where the right hand side is the same as the square of $$\lim_{n \rightarrow \infty} \mathbb{E}\left[\frac{1}{n}\left(\sum_{j \in S_{c}} \frac{A_{j} x_{j}}{\sigma\left(x_{j}^{T} \hat{\beta}_{S_{a}}\right)}-x_{j}\right)^{T}\left(\sum_{i \in S_{b}} A_{i} x_{i} x_{i}^{\top}\right)^{-1}\left(\sum_{i \in S_{b}} \frac{A_{i} x_{i}}{\sigma\left(x_{i}^{T} \hat{\beta}_{S_{c}}\right)}\right)\right].$$
    This implies that $$\operatorname{Var}\left(\frac{1}{n}\left(\sum_{j \in S_{c}} \frac{A_{j} x_{j}}{\sigma\left(x_{j}^{T} \hat{\beta}_{S_{a}}\right)}-x_{j}\right)^{T}\left(\sum_{i \in S_{b}} A_{i} x_{i} x_{i}^{\top}\right)^{-1}\left(\sum_{i \in S_{b}} \frac{A_{i} x_{i}}{\sigma\left(x_{i}^{T} \hat{\beta}_{S_{c}}\right)}\right)\right) \rightarrow 0, $$ which completes the proof of the Lemma.
 
\section{Proof of Theorem \texorpdfstring{\MakeLowercase{\ref{thm:dr_distribution_ridge}}}{}}\label{appendix g}

In this section, we first state the full result of Theorem \ref{thm:dr_distribution_ridge}. To this end, we need some
additional notation.

Fix any ridge penalty parameter $\lambda > 0$. For any $i = 1, 2, 3$, let $\left(\alpha_{i}^{(\lambda)}, \sigma_{i}^{(\lambda)}, \lambda_{i}^{(\lambda)}\right)$ be the solution to the system of equations (4.12) in \cite{Sur14516}, where the covariates are of dimension $n_{i} \times p$ with i.i.d. entries $\sim \mathcal{N}\left(0, \frac{1}{n_i}\right)$, the signal
strength is $\gamma^2$, and the ridge penalty parameter is $\lambda$. Moreover, define $$ \tilde{\alpha}_{i}^{(\lambda)} = \alpha_{i}^{(\lambda)} \cdot \frac{\kappa_{a}-\lambda \lambda_{a}^{\left(\lambda \right)}}{\kappa_{a}}, \quad \tilde{\sigma}_{i}^{(\lambda)} = \sigma_{i}^{(\lambda)} \cdot \frac{\kappa_{a}-\lambda \lambda_{a}^{\left(\lambda \right)}}{\kappa_{a}},\quad \tilde{\lambda}_i^{(\lambda)} =\lambda_i^{(\lambda)} . $$

Then the asymptotic variance $\left(\sigma_{c f}^{(\lambda)}\right)^{2}$ in Theorem \ref{thm:dr_distribution_ridge} is the same as the expression for $\sigma_{c f}^2$ in Theorem \ref{thm:dr_distribution}, except that $(\alpha_i^*, \sigma_i^*, \lambda_i^*)$ is replaced by $ \left(\tilde{\alpha}_{i}^{(\lambda/r_i)},  \tilde{\sigma}_{i}^{(\lambda/r_i)},\tilde{\lambda}_{i}^{(\lambda/r_i)} \right) \;\; \forall i \in \{1,2,3\}$.

The proof of Theorem \ref{thm:dr_distribution_ridge} can be obtained by modifying the proof of Theorem \ref{thm:dr_distribution}. In particular, we only need to modify the lemmas involved in the proof of Theorem \ref{thm:dr_distribution} that use properties of MLE for logistic regression models.

 \begin{lemma} \label{lemma 3 ridge}
Fix any ridge penalty parameter $\lambda > 0$. Fix any $(a,b,c)$, a permutation of $(1,2,3)$.  Assume the setting described in section \ref{sec:setup}. Let $x \in \mathbb{R}^{p}$ be any sample in $S_c$. Then the random vector $\left(\begin{array}{c}
x^{\top} \beta \\
x^{\top} \hat{\beta}_{S_a}^{(\lambda)}
\end{array}\right)$ converges in distribution to a multivariate normal distribution with mean zero and variance $\Sigma$, where $$\Sigma = \begin{bmatrix}
\gamma^2 &  \tilde{\alpha}^{(\lambda / r_a)}_{a}\gamma^2  \\
\tilde{\alpha}^{(\lambda / r_a)}_{a}\gamma^2 &  \quad\kappa_a \left(\tilde{\sigma}^{(\lambda / r_a)}_a\right)^2 + \left(\tilde{\alpha}_a^{(\lambda / r_a)}\right)^2 \gamma^2
\end{bmatrix} .$$
  \end{lemma}
  \begin{proof}
  Conditioned on $S_a, S_b$, \begin{align*}
    \left(\begin{array}{c}
x^{\top} \beta \\
x^{\top} \hat{\beta}_{S_a}^{(\lambda)} 
\end{array}\right) \sim N \left( \mathbf{0}, \frac{1}{n} \begin{bmatrix}
\|\beta\|^2 & \beta^{\top} \hat{\beta}_{S_a}^{(\lambda)} \\
\beta^{\top} \hat{\beta}_{S_a}^{(\lambda)} & \|\hat{\beta}_{S_a}^{(\lambda)}\|^2 
\end{bmatrix}  \right).
\end{align*}

By assumption $$\frac{\|\beta\|^{2}}{n}  \rightarrow \gamma^2 . $$

Note that $\hat{\beta}_{S_a}^{(\lambda)}$ is estimated from $S_a$, where $X_a \in \mathbb{R}^{n_a \times p}$ with i.i.d. entries $\sim N(0, \frac{1}{n})$. Define $$\beta^{'} = \sqrt{r_a} \beta, \quad \hat{\beta}_{S_a}^{'} =\sqrt{r_a} \hat{\beta}_{S_a}^{(\lambda)} , \quad X_a^{'} = \frac{1}{\sqrt{r_a}} X_a $$

Then $ X_a^{'} \in  \mathbb{R}^{n_a \times p}$ with i.i.d. entries $\sim N(0, \frac{1}{n_a})$, and $ X_a \beta =X_a^{'} \beta^{'},  X_a \hat{\beta }_{S_a}^{(\lambda)} =X_a^{'} \hat{\beta}_{S_a}^{'}$. Moreover, $$   \frac{\|\beta^{'}\|^{2}}{n_a} = \frac{r_a\|\beta\|^2}{n_a} = \frac{\|\beta\|^{2}}{n}  \rightarrow   \gamma^2.$$


According to Chapter 4, Result 1 in \cite{Sur14516}, $$ \frac{1}{p} (\beta^{'})^{\top} \hat{\beta}_{S_a}^{'}  \xrightarrow{a.s.} \frac{\alpha_a^{(\lambda/r_a)}}{p} \left(\beta^{\prime}\right)^{\top} \beta^{\prime} \cdot \frac{\kappa_a -\lambda \lambda_a^{(\lambda/r_a)}}{\kappa_a} .$$

This implies $$ \lim_{n \rightarrow \infty} \frac{1}{n} \beta^{\top} \hat{\beta}_{S_a}^{(\lambda)}   = \alpha_{a}^{(\lambda/r_a)} \gamma^2 \cdot \frac{\kappa_{a}-\lambda \lambda_{a}^{(\lambda/r_a)}}{\kappa_{a}} \quad \text{almost surely.} $$

Moreover, applying Chapter 4, Result 1 in \cite{Sur14516} with $ \psi(t,u) = t^2 $ yields $$      \lim_{n \rightarrow \infty}  \frac{\|\hat{\beta}_{S_a}^{'}\|^2}{p} =\left[(\sigma_a^{(\lambda/r_a)})^2 +  (\alpha_a^{(\lambda/r_a)})^2 \frac{r_a \gamma^2}{\kappa}\right] \cdot \left(\frac{\kappa_{a}-\lambda \lambda_{a}^{\left(\lambda / r_{a}\right)}}{\kappa_{a}}\right)^2 \quad \text{almost surely.}  $$

Thus
$$ \frac{\|\hat{\beta}_{S_a}^{(\lambda)} \|^2}{n} = \frac{\|\hat{\beta}_{S_a}^{\prime}\|^{2}}{p}  \cdot \frac{p}{n} \cdot \frac{1}{r_a} \xrightarrow{a.s.}  \left[\kappa_a (\sigma_a^{\left(\lambda / r_{a}\right)})^2 +  (\alpha_a^{\left(\lambda / r_{a}\right)})^2\gamma^2 \right] \cdot \left(\frac{\kappa_{a}-\lambda \lambda_{a}^{\left(\lambda / r_{a}\right)}}{\kappa_{a}}\right)^2 .$$

 Hence we know 
\begin{align*}
    \left(\begin{array}{c}
x^{\top} \beta \\
x^{\top} \hat{\beta}_{S_a}^{(\lambda)} 
\end{array}\right)  \xrightarrow{d} N\left( \mathbf{0}, \begin{bmatrix}
\gamma^2 &  \tilde{\alpha}^{(\lambda / r_a)}_{a}\gamma^2  \\
\tilde{\alpha}^{(\lambda / r_a)}_{a}\gamma^2 &  \quad\kappa_a \left(\tilde{\sigma}^{(\lambda / r_a)}_a\right)^2 + \left(\tilde{\alpha}_a^{(\lambda / r_a)}\right)^2 \gamma^2
\end{bmatrix} \right),
\end{align*} which does not depend on $S_a,S_b$. This completes the proof of the lemma.
  \end{proof}

\begin{proof}[Proof of Theorem \ref{thm:dr_distribution_ridge}]

Below we state the properties of MLE for logistic
regression models that are used in the proof of Theorem \ref{thm:dr_distribution}, and show how they can be modified to prove Theorem \ref{thm:dr_distribution_ridge}.

\noindent {(i)} The proof of Lemma \ref{lemma_t1_terms_limit} and \ref{vec_lemma_2} uses the following property: $$ \frac{1}{n} \hat{\beta}_{S_{a}}^{T} \hat{\beta}_{S_{c}} =\alpha_{a}^{*} \alpha_{c}^{*} \gamma^{2}+o(1) \quad \text{a.s.} \quad \forall a, c \in \{1,2,3\}. $$

Similar to the proof of Lemma \ref{lemma 3 ridge}, with an extension of Chapter 4, Result 1 in \cite{Sur14516}, we can show $$ \frac{1}{n} \left(\hat{\beta}_{S_{a}}^{(\lambda)}\right)^{T} \hat{\beta}_{S_{c}}^{(\lambda)} =\tilde{\alpha}_{a}^{(\lambda / r_a)} \tilde{\alpha}_{c}^{(\lambda / r_c)} \gamma^{2}+o(1) \quad \text{a.s.} \quad \forall a, c \in \{1,2,3\}. $$

\noindent {(ii)} The proof of Lemma \ref{vec_lemma_2} and \ref{lemma 26} uses Lemma 21 in \cite{Sur14516}. Since the ridge penalty is strongly convex, its Hessian is positive. Thus Lemma 16 in \cite{Sur14516} still holds if we replace the negative log-likelihood with the sum of the negative log-likelihood and the ridge penalty. The remaining proof for Lemma 21 in \cite{Sur14516} also holds, which implies that Lemma 21 in \cite{Sur14516} also applies to ridge regularized estimates.\\

\noindent {(iii)} Lemma \ref{lemma 3} characterizes the asymptotic distribution of $\left(\begin{array}{c}
x^{\top} \beta \\
x^{\top} \hat{\beta}_{S_{a}}
\end{array}\right),$ where $x \in \mathbb{R}^{p} $ is any sample in  $S_{c}$, and $(a,b,c)$ is any permutation of $(1,2,3)$. Correspondingly, Lemma \ref{lemma 3 ridge} characterizes the asymptotic distribution of $\left(\begin{array}{c}
x^{\top} \beta \\
x^{\top} \hat{\beta}_{S_{a}}^{(\lambda)}
\end{array}\right)$. 

After modifying the lemmas mentioned above, we can prove Theorem \ref{thm:dr_distribution_ridge} in a similar fashion as the proof of Theorem \ref{thm:dr_distribution}.

\end{proof}

\end{appendix}

\end{document}